\documentclass[11pt]{amsart}
\usepackage{amsmath}
\usepackage[scaled]{helvet} % ss
\usepackage[T1]{fontenc}
\usepackage{textcomp}
\normalfont
\usepackage{amssymb}
\usepackage{color}
\usepackage{graphicx}
\usepackage{tikz}
\usepackage[utf8]{inputenc}
\usepackage[english]{babel}
\usepackage{csquotes}
\usepackage{amssymb}
\usepackage{tikz-cd}
\usepackage{csquotes}
\usepackage{verbatim}

\usepackage{amsmath, amsfonts,amsthm,graphicx,fullpage,etex,tikz,hyperref}

\hypersetup{
    colorlinks=true,
    linkcolor=blue,
    filecolor=magenta,      
    urlcolor=cyan,
}
\usepackage{amssymb}

\paperwidth=\dimexpr \paperwidth + 3cm\relax
\oddsidemargin=\dimexpr\oddsidemargin + 2cm\relax
\evensidemargin=\dimexpr\evensidemargin + 2cm\relax
\marginparwidth=\dimexpr \marginparwidth + 1cm\relax

\usepackage{enumerate}
\usepackage{enumitem}
\newtheorem{theorem}{Theorem}[section]
\newtheorem{lemma}[theorem]{Lemma}
\newtheorem{proposition}[theorem]{Proposition}
\newtheorem{corollary}[theorem]{Corollary}
\newtheorem{conjecture}[theorem]{Conjecture}

\newtheorem{definition}[theorem]{Definition}

\theoremstyle{remark}
\newtheorem{remark}[theorem]{Remark}
\newtheorem{partremark}{Remark}[part]

\numberwithin{equation}{section}

\newcommand{\Aff}{\mathrm{Aff}}
\newcommand{\Aut}{\mathrm{Aut}}

\newcommand{\id}{\mathrm{Id}}
\newcommand{\Z}{\mathbb{Z}}

\newcommand{\Q}{\mathbb{Q}}
\newcommand{\R}{\mathbb{R}}

\newcommand{\N}{\mathbb{N}}

\newcommand{\mc}{\mathcal}
\newcommand{\set}[1]{\left\lbrace #1 \right\rbrace}
\newcommand{\mf}{\mathfrak}
\newcommand{\abs}[1]{\left| #1 \right|}
\newcommand{\norm}[1]{\left|\left| #1 \right|\right|}
\newcommand{\ve}{\varepsilon}
\newcommand{\of}{\circ}
\newcommand{\tth}{^{\mbox{\normalfont\tiny th}}}
\newcommand{\mbf}{\mathbf}

\newcommand{\smooth}{ C^{\infty}}

\definecolor{darkcyan}{rgb}{0. 0.65, 0.65}

\def\h-rk{rk ^{h}}

\newtheorem{Structural Stability Theorem}[theorem]{Structural Stability Theorem}

\def\bt{\begin{theorem}}
\def\et{\end{theorem}}
\def\bd{\begin{definition}}
\def\ed{\end{definition}}
\def\bl{\begin{lemma}}
\def\el{\end{lemma}}

\def\Sym{\operatorname{Sym}}
\def\Diff{\operatorname{Diff}}
\def\Homeo{\operatorname{Homeo}}
\def\Lie{\operatorname{Lie}}
\def\Stab{\operatorname{Stab}}
\def\loc{\operatorname{loc}}
\def\Heis{\operatorname{Heis}}
\def\rank{\operatorname{rank}}
\def\ad{\operatorname{ad}}
\def\Ad{\operatorname{Ad}}
\def\diag{\operatorname{diag}}

\def\Isom{\operatorname{Isom}}
\def\Fix{\operatorname{Fix}}

\def\Per{\operatorname{Per}}
\def\Hom{\operatorname{Hom}}
\def\Rig{\operatorname{Rig}}

\begin{document}

\title{Cartan Actions of Higher Rank Abelian Groups and their Classification}

\author{RALF SPATZIER $^\ast$}
\address{DEPARTMENT OF MATHEMATICS, UNIVERSITY OF MICHIGAN,
    ANN ARBOR, MI, 48109.}
\email{spatzier@umich.edu}

\author{KURT VINHAGE $^\dagger$}
\address{DEPARTMENT OF MATHEMATICS, UNIVERSITY OF UTAH,
 SALT LAKE CITY, UT 84112}
\email{vinhage@math.utah.edu}

\thanks{$^\ast$ Supported in part by NSF grants DMS 1607260 and DMS 2003712}

\thanks{$^{\dagger}$ Supported in part by NSF grant DMS  1604796}

\maketitle

\begin{abstract}
 We study $\R ^k \times \Z^\ell$ actions on arbitrary compact manifolds with a projectively dense set of Anosov elements and 1-dimensional coarse Lyapunov foliations. Such actions are called totally Cartan actions. We completely classify such actions as built from low-dimensional Anosov flows and diffeomorphisms and affine actions, verifying the Katok-Spatzier conjecture for this class. This is achieved by introducing a new tool, the action of a dynamically defined topological group describing paths in coarse Lyapunov foliations, and understanding its generators and relations. We obtain applications to the Zimmer program.
 \end{abstract}
   
\tableofcontents

%\part*[Introduction]{\Large Introduction and Main Results}

\section{Introduction}
 Hyperbolic actions of higher rank abelian groups are markedly different from single hyperbolic diffeomorphisms and  flows and display a multitude of rigidity properties such as measure rigidity and cocycle rigidity.  We refer to the survey \cite{spatzier-invitation2004} as a starting point for these topics, and \cite{venkatesh08,einsiedler09,EinLin2015,brown19} and \cite{KS1,damjanovic-katok10,Vinhage:2015aa} for more recent developments concerning measure and cocycle rigidity, respectively. In this paper, we concentrate on the third major rigidity property:   classification and global differential rigidity of such actions.

 Smale already conjectured in 1967 that generic diffeomorphisms only commute with their  iterates \cite{Smale}.  In the hyperbolic case 
 this was proved by Palis and Yoccoz in 1989 \cite{Palis-Yoccoz1,Palis-Yoccoz2}.   Recently,  Bonatti, Crovisier and Wilkinson  proved this in full generality \cite{BCW2}.
 Much refined  local rigidity properties were found 
  in the works of  Hurder and then  Katok and Lewis on deformation and local  rigidity of the standard action of $SL(n,\Z)$ on the $n$-torus in which  they used properties of the action of higher rank abelian subgroups \cite{Hurder,KL1,KL}.  
   This led to investigating such actions more systematically.  
   
   Let us first coin some terminology.  Given a foliation $\mc{F}$, a diffeomorphism $a$ acts {\em normally hyperbolically} w.r.t. $\mc{F}$ if $\mc{F}$ is invariant under $a$ and if there exists a $C^0$ splitting of the tangent space $TM = E^s  _a\oplus T \mc{F} \oplus  E^u _a$ into  stable and unstable subspaces for $a$ and the tangent space of the leaves of $\mc{F}$.  More precisely, we assume that $E^s _a$ and $E^u _a$  are contracted uniformly in either forward or backward time by $a$ (see Definition \ref{def:normally hyperbolic} for  precise definitions).  
  
All manifolds discussed will be assumed to be connected unless otherwise stated. This will often be superfluous, since we often assume the existence of  a continuous action of a connected group with a dense orbit.
  
\begin{definition}   \label{Anosov action}
Let  $k,l \in \Z_{\ge 0}$, $k+l \ge 1$, and $\R^k \times \Z^\ell \curvearrowright X$ be a locally free $C^1$ action  on a compact manifold $X$.   If $a \in \R^k \times \Z^\ell$ acts normally hyperbolic to the orbit foliation of $\R^k$, we call $a$ an  {\em Anosov element}, and $\alpha$ an {\em Anosov action}. If $k + \ell \ge 2$, the action is called {\normalfont higher rank}.

Let $\mc A \subset \R^k \times \Z^\ell$ denote the set of Anosov elements, and $p : \R^{k+\ell} \setminus \set{0} \to \R\mathbb{P}^{k+\ell-1}$ denote the projection onto real projective space. We say that an action is {\normalfont totally Anosov} if $p(\mc A)$ is dense in $\R\mathbb{P}^{k+\ell-1}$.

An action $\R^k \times \Z^\ell \curvearrowright X$ if called {\em transitive} is there exists a point $x \in X$ such that $(\R^k \times \Z^\ell) \cdot x$ is dense in $X$.  An Anosov action is {\em cone transitive} if there exists an open cone $C \subset \R^{k+\ell}$ such that $\overline{C} \cap \R^k \times \Z^\ell$ consists of Anosov element, and there exists some $x \in X$ such that $(C \cap (\R^k \times \Z^\ell)) \cdot x$ is dense in $X$.
\end{definition}

\begin{remark}
Note that  higher rank Anosov $\Z^\ell$-actions correspond to $\ell$ commuting diffeomorphisms with at least one being Anosov. %with only commutator relations among them and  at least one of them Anosov. In particular, %We state the following definition for $\R^k$ actions (the definition for $\Z^k$ actions can be easily easily deduced, although one usually allows passing to a finite index subgroup).
However, when $k \ge 1$, Anosov elements of the action are not Anosov diffeomorphisms, as they necessarily act isometrically along the $\R^k$-orbit foliation. For example, an $\R^1$ Anosov action is an Anosov flow, and in this case every nonzero element will be Anosov. In general, the set of Anosov elements is a union of open cones which we will explain later.
\end{remark}

The following definition is critical to the structure theory of higher rank abelian actions.

\begin{definition}
\label{def:rank-one-factor}
Let $k+\ell \ge 1$ and $\R^k\times \Z^\ell \curvearrowright X$ be a $C^r$ action on a $C^\infty$ manifold $X$, $r \ge 1$. A {\normalfont $C^s$-rank one factor}, $s \le r$ is 

\begin{itemize}
\item a $C^\infty$ manifold $Y$ and a $C^s$ submersion $\pi : X \to Y$, 
\item a  surjective homomorphism $\sigma : \R^k \times \Z^\ell \to A$, where $A$ is a compact extension of $\R$ or $\Z$, and 
\item a locally free, faithful $C^s$ action $A \curvearrowright Y$ such that $\pi(a\cdot x) = \sigma(a) \cdot \pi(x)$.
\end{itemize}

The factor is called a {\normalfont Kronecker factor} if  $Y = \mathbb{T}^d$ for some $d \in \Z_+$, and the action of $A$ is by translations and transitive (in the group-theoretic sense).
% The rank one factor is measurable, topological or $C^r$ if the map $\pi$, space $X$ and flow on $X$ are measurable, topological or $C^r$, respectively. An action has {\normalfont no topological} or {\normalfont no $C^r$ rank one factor} if the only topological or $C^r$ rank-one factor is the trivial action on the one-point space.

%In the case of measurable factors, we assume $\R^k \curvearrowright M$ has a distinguished invariant probability measure $\mu$. Then each measurable rank one factor has an associated invariant measure $\pi_*\mu$. In this case, the action is said to have {\normalfont no measurable rank one factor (with respect to $\mu$)} if every measurable rank one factor with the measure $\pi_*\mu$ is measurably isomorphic to the trivial action on a one point space. 
\end{definition}

\begin{remark}
If  $X$ is a compact manifold, $\R^k \times \Z^\ell \curvearrowright X$ is an Anosov action, and $\pi : X \to Y$ is a Kronecker factor, it is not difficult to see that any two elements related by a stable or unstable manifold must be sent to the same point. In particular, there are no Kronecker factors of an Anosov $\Z^\ell$ action, and the only Kronecker factors of an Anosov $\R^k$ action correspond to projecting onto the coordinates of the acting group (and any such action is transitive on $\mathbb{T}^d$).
\end{remark}

Our definition coincides with another one that is often used: one may alternatively define a $C^s$-rank one factor as a factor for which a cocompact subgroup $A' \subset \R^k \times \Z^l$ acts by a single flow or diffeomorphism after projecting ($(\R^k \times \Z^\ell) / A'$ corresponds to compact extension of Definition \ref{def:rank-one-factor}).

The lack of a rank one factor can be interpreted as an irreducibility condition: many actions with rank one factors are represented as skew-products in the corresponding category where the base of the skew product is a flow. % (in general, one can only claim that the space is a fibration and that the $\R^k$ action intertwines the fibers and projects to a flow on the base). 
One may similarly define measurable and continuous ($C^0$) factors of the action, in which the regularity conditions are relaxed. It is clear that every smooth rank one factor is topological, and every topological rank one factor is measurable. Forbidding smooth rank one factors is therefore weakest of the possible no rank-one factor assumptions. %It is clear that every smooth rank one factor is a topological rank one factor, and if $\mu$ is a fully supported ergodic invariant measure, every topological rank one factor is a measurable rank one factor with respect to $\mu$. Therefore, having no measurable rank one factor with respect to a fully supported ergodic invariant measure $\mu$ is the strongest condition and having no smooth rank one factor is the weakest for an $\R^k$ action.

% If $k \geq2$, we call the action  {\em higher rank}, and otherwise {\em rank one}. 
%Finally, a higher  rank Anosov action  {\em irreducible} if no finite cover has a $C^{1, \theta}$ rank one action on a compact manifold as a quotient, i.e. becomes  rank one after factoring out by the kernel.

%  Irreducibility of actions could alternatively be defined by precluding different types of rank one factors, in particular measurable, topological and differentiable ($C^{1,\theta}$ or $C^{\infty}$).  Asking for differentiable and especially $\smooth$ factors is clearly the most restrictive  condition.  

There are natural examples of higher rank actions coming from homogeneous actions, e.g. by the diagonal subgroup of $SL(n,\R)$ on a compact quotient  $SL(n,\R)/\Gamma$ or actions by automorphisms of Lie groups, e.g. by commuting toral automorphisms. More generally, genuinely higher rank actions should be algebraic in the following sense (see Section \ref{subsec:examples} for examples of such Anosov actions):

\begin{definition}
An {\normalfont algebraic action} is an action of $\R^k \times \Z^l$, $k + l \ge 1$, by compositions of translations and automorphisms on a compact homogeneous space $G / \Gamma$. The action is {\normalfont homogeneous} if it consists only of translations.
\end{definition}

\begin{remark}
In other settings, algebraic actions allow for actions by translations and automorphisms on a {\it double} homogeneous space $K \backslash G / \Gamma$, where $K$ is compact. In particular, this often appears for Anosov actions, even in rank one. However, in the Cartan setting, we will always have that $K = \set{e}$.
\end{remark}
    
Katok and Spatzier proved local $\smooth$ rigidity of so-called standard actions without rank one factors in \cite{KS97}, generalizing the  earlier works by Hurder \cite[Theorem 2.19]{Hurder} respectively Katok and Lewis \cite[Theorem 4.2]{KL1}.    This gives evidence for  the following conjecture of Katok and Spatzier (posed as a question  in \cite{Burns-Katok1985} and as a conjecture in \cite[Conjecture 16.8]{problemsMSRI2007}).

\begin{conjecture}[Katok-Spatzier] \label{conjecture:Katok-Spatzier}
All  higher rank $C^\infty$ Anosov actions on any compact manifold without $C^\infty$ rank one factors 
are $C^\infty$ conjugate to an algebraic action after passing to finite covers.
\end{conjecture}

 Due to an example recently discovered by the second author, the conjecture as stated here is false. However, a weakened version of the conjecture may hold, in particular, by strengthening the Anosov assumption to totally Anosov. Furthermore, the counterexample can be constructed only for actions of $\R^k \times \Z^\ell$, $k \ge 2$. When $k \le 1$, the conjecture may still hold. This paper also contains several new examples of totally Cartan actions, which illustrate that each of our assumptions is necessary, which are described in Section \ref{sec:exotic}.%The main theorem of this paper verifies this with an additional simplifying assumption (see Definition \ref{Cartan action}).

%A $C^r$ version of the conjecture may also hold, possibly with loss of regularity in the conjugacy. 
Weakened versions may also hold in the $C^r$ setting, possibly with loss of regularity in the conjugacy. This loss of regularity occurs in many rigidity theorems for partially hyperbolic actions, see, e.g.  \cite{damjanovic-katok10}. The minimal value of $r$ for which rigidity holds is still mysterious.

Conjecture \ref{conjecture:Katok-Spatzier} is  reminiscent  of the longstanding conjecture by Anosov and Smale  that    Anosov diffeomorphisms are topologically conjugate to an automorphism of an infra-nilmanifold \cite{Smale}.  The higher rank and irreducibility assumptions allow for much more dramatic conclusions:  First, the  conjugacy is claimed to be smooth.  For single Anosov diffeomorphisms, smooth rigidity results are not  possible as one can change derivatives at fixed points 
by local changes. Furthermore, Farrell and Jones and later Farrell and Gogolev constructed Anosov diffeomorphisms on exotic tori \cite{FarrellJonesExoticExpanding,FarrellJones,FarrellGogolev} (although after passing to a finite cover, the exotic structure becomes standard).  Thus not even the differentiable structure of the underlying manifold is determined. Secondly, the conjecture applies equally well  to    $\Z^k$ and $\R ^k$-actions. There is no version of the Anosov-Smale conjecture for Anosov flows. Even topological rigidity is out of reach as there are many flows which are not even orbit-equivalent to Anosov flows, see for instance the Handel-Thurston examples \cite{handel-thurston1980}. In addition, Anosov flows may not be transitive, such as the Franks-Williams examples \cite{Franks-Williams1980}. 

Significant  progress has been made on this conjecture in the last decade.  For higher rank Anosov $\Z ^k$ actions on tori and nilmanifolds, Rodriguez Hertz and Wang \cite{RH-W} have proved the ultimate result: global rigidity assuming only one Anosov element, and that its linearization does not have affine rank one factors. This improves on the results of \cite{Hertz}, and together they are the only known results that assume the existence of only one Anosov element in the action.    In fact it followed  earlier work by Fisher, Kalinin and Spatzier \cite{FKS} for totally Anosov actions, i.e. Anosov actions with a dense set of Anosov elements.  
As knowledge of the underlying manifold is required, this is really a global rigidity theorem.   Finally, in the same vein, Spatzier and Yang classified nontrivially commuting expanding maps in \cite{Spatzier-Yang2017}.  This was possible as expanding maps were known to be $C^0$ conjugate to endomorphisms of nilmanifolds (up to finite cover) thanks to work of Gromov and Shub \cite{GromovPolyGrowth,ShubExpandingMaps1970}.  Of course, a positive resolution of the Anosov-Smale conjecture, combined with the results above, would automatically prove the Katok-Spatzier conjecture for higher rank $\Z ^k$ actions.

Much less is known for $\R ^k$ actions or when the underlying manifold is not a torus or nilmanifold.   
We will concentrate on the so-called totally Cartan actions. 
For a  totally Anosov action $\alpha$,  a {\em coarse Lyapunov foliation} is a foliation whose leaf at $x$ is the path component containing $x$ of the maximal (nontrivial) intersections of stable manifolds $\cap_i \mc{W} ^s _{a_i}(x)$ for some fixed Anosov elements $a_i$ (see Section \ref{sec:normal-hyp} for a thorough discussion and definition).  If $\alpha$ preserves a measure $\mu$ of full support, this can be phrased in terms of Lyapunov exponents (see Section \ref{subsec:lyapunov-prelim}), thus the name. 

\begin{definition} \label{Cartan action}
A (totally) Anosov action of $A = \R^k \times \Z^l$, $k + l \ge 1$, $A \curvearrowright X$, is called {\em (totally) Cartan} if all coarse Lyapunov foliations are one dimensional. %We call the action {\em Lyapunov orientable} if every coarse Lyapunov foliation can be given an orientation.
 \end{definition}

We remark that while our theorems hold when $k + l =1$, their conclusions reduce to well-established theorems about Anosov diffeomorphisms of surfaces and Anosov flows on 3-manifolds. Indeed, when $k + l = 1$, the Cartan condition will force the stable and unstable manifolds to have dimension one.

The notion of a Cartan action has had varying incarnations: one may insist that each coarse Lyapunov foliation is 1-dimensional {\it and} that for each coarse Lyapunov foliations $\mc W$, there exists $ a \in \R^k \times \Z^\ell$ such that $\mc W = W^u_a$ (as defined in \cite{Hurder}); or one may define a Cartan action to be an action of $\Z^{n-1}$ on an $n$-manifold (as defined in \cite{KL1} and \cite{KalKatRod09}); or we may use our definition, which coincides with the one given in \cite{KaSp04}. It by far is the weakest of the possible definitions, is implied by all the others and there are many examples which satisfy Definition \ref{Cartan action} which do not satisfy the alternatives.

Kalinin and Spatzier  classified $C^\infty$ totally Cartan actions of $\R ^k$ for $k \geq 3$  in \cite{KaSp04} on arbitrary manifolds under the additional hypothesis that every  one-parameter subgroup of $\R ^k$ acts transitively and $\alpha$  preserves an ergodic probability  measure $\mu$ of full support.  Later, 
Kalinin and Sadovskaya proved several strong results in this direction for totally nonsymplectic (TNS) actions of $\Z^\ell$, i.e. actions for which no two Lyapunov exponents (thought of as linear functionals on $\Z^\ell$) are negatively proportional.
 They also treated higher dimensional coarse Lyapunov spaces but required additional conditions such as joint integrability of coarse Lyapunov foliations or non-resonance conditions, as well as quasi-conformality of the action on the coarse Lyapunov foliations \cite{KalSad07, KalSad06}.  Recently, Damjanovic and Xu  \cite{Damjanovic:2018aa} generalized their results relaxing the  quasi-conformality conditions but still requiring joint integrability or non-resonance.  
 The joint integrability or non-resonance conditions are useful to force the coarse Lyapunov subspaces to ``commute,'' which guarantees that the action is on a torus.

We improve on the assumptions appearing in these works in the following ways:

\begin{itemize}
\item We establish a new method of building homogeneous structures, which allows for resonances and non-integrability, allowing for complete classification.
\item We do not require the existence of an ergodic invariant measure of full support (we only require cone transitivity of the $\R^k \times \Z^\ell$ action).
\item We do not directly require (cone) transitivity or ergodicity for subactions (e.g. the ergodicity of one-parameter subgroups), instead deducing their transitivity by understanding their relationship to rank one factors.
\item Our conditions are purely dynamical and do not require topological assumptions or an underlying model.
\item We build rank one factors of the action using only dynamical input, and still obtain a classification in their presence.
\end{itemize}

%Remarkably, these results by Kalinin and Sadovskaya or Damjanovic and Xu hold also   for $k=2$ when making  additional  conditions such as non-resonance.  Naturally, one hopes to extend at least the classification of Cartan actions to $k=2$.   The techniques used in \cite{KaSp04}  in rank $\geq 3$ break down completely.  Indeed, the  main idea there was to create isometries on integrable hulls of coarse Lyapunov spaces using the dynamics of  elements in the intersection of kernels of Lyapunov exponents.  This works well in rank at least 3 but obviously not for rank 2.  

%The main achievement of this paper is to classify Cartan actions of higher rank abelian groups without any of the 
%integrability or TNS assumptions of prior works as well as replacing the ergodicity assumptions with a no rank one factor assumption. Crucially, we require no transitivity assumptions for subgroups of the acting group.  %We also give a new proof for actions of rank 3 or higher groups.   In addition, we can  drop the regularity requirements for such actions from $C ^{\infty}$ to $C^{1+\theta}$ which is not possible with the arguments in \cite{KaSp04}.
We begin by establishing the main result, where the action has no rank one factors. We actually prove more, allowing Kronecker factors of the action, ie factors which are translations of a torus. This allows us to discuss the case of $\Z^k$ actions by passing to a their suspension (the suspension will always have a transitive $\mathbb{T}^k$ factor). We discuss suspensions and their relationship to Kronecker factors in Section \ref{sec:susp}.

\begin{theorem}[Main Theorem]
\label{thm:big-main}
Let  $\R^k \times \Z^\ell \curvearrowright X$ be a %$C^{1,\theta}$ 
{  $C^2$}
cone transitive, totally Cartan action, and assume no finite cover of the action has a non-Kronecker $C^{1,\theta}$ rank 1 factor {  for some $\theta \in (0,1)$.}
Then the action is %$C^{1,\theta}$
{  $C^2$}-conjugate to an affine action (up to finite cover). %for some $\theta \in (0,1)$. 
Furthermore, if the action is $C^\infty$, and one instead assumes that there are no non-Kronecker $C^\infty$ rank one factors of finite covers of the action, then a $C^\infty$ conjugacy exists.
\end{theorem}

\begin{remark}
The assumption of cone transitivity follows from transitivity with some additional assumptions (such as the existence of a measure of full support, or the existence of a transitive Anosov element). For more on the relationship between these transitivity conditions for Anosov actions, see Lemma \ref{lem:cone-transitive}.
\end{remark}

\begin{remark}
Due to the examples constructed in \cite{vinhage22}, the totally Cartan assumption may not be relaxed to the Cartan assumption. This shows that the assumption of totally Cartan is optimal.
\end{remark}

We will use Theorem \ref{thm:big-main} in the proof of the full classification result. In addition, we describe all homogeneous (in fact, affine) totally Cartan actions, see Section \ref{sec:affine-classification}. If one assumes that the action is Lyapunov orientable (Definition \ref{Cartan action}), then one may replace the assumption that every finite cover of the action has no non-Kronecker rank one factors with the assumption that the action itself has no non-Kronecker rank one factors. Therefore, in general, one only needs to check up to $2^{\dim(X) - k}$-fold covers of $X$. Alternatively, one may assume that there are no orbifold factors of a given action.

%While the case of $\Z^\ell$ actions has simplest assumptions and conclusion, general $\R^k \times \Z^\ell$ actions are much less clean, so we instead describe them in relation to other actions. We say that an action $\R^k \times \Z^\ell \curvearrowright X$ {\it embeds into} an action $\R^m \curvearrowright Y$ if there exists a homomorphism $\sigma : \R^k \times \Z^\ell \to \R^m$ and an embedding $\varphi : X \to Y$ such that $\varphi(a \cdot x) = \sigma(a) \cdot \varphi(x)$. This definition allows us to consider only $\R^k$ actions, since $\R^k \times \Z^\ell$ actions will be embedded in their suspension. 

To describe the structure of a general totally Cartan $\R^k \times \Z^\ell$ action, we first need a few additional structural features, which we now describe. Let $\Delta$ be the set of coarse Lyapunov exponents of the totally Cartan action (see Section \ref{subsec:lyapunov-prelim}). Then let $S \subset \R^k$ be defined by $S = \bigcap_{\beta \in \Delta} \ker \beta$. We call $S$ the {\it Starkov component} of the $\R^k$-action.

\begin{remark}
\label{rem:starkov}
One may alternatively describe the Starkov component as the set of $a \in \R^k$ satisfying any (and hence all) of the following conditions:

\begin{enumerate}
\item $a$ has zero topological entropy.
\item $\set{a^k : k \in \Z}$ is an equicontinuous family.
\item $a$ does not have sensitive dependence on initial conditions.
\item $a$ is not partially hyperbolic (with nontrivial stable and unstable bundles).
\end{enumerate}

{  See Lemma \ref{lem:starkov-conditions}.} Therefore, the condition that an action has {\it trivial} Starkov component is equivalent to 0 being the only element satisfying any of the above conditions.
\end{remark}

The Starkov component can be considered the ``trivial'' part of the $\R^k$ action, and we can safely ignore it after passing to a factor. Lemma \ref{lem:starkov-factor} constructs the {\it Starkov factor} of $\R^k \curvearrowright X$, which we denote by $\R^{k-\ell} \curvearrowright \bar{X} = X / S$.
%\begin{theorem}
%\label{thm:starkov-factor}
%Let $\R^k \curvearrowright X$ be a $C^2$ ($C^\infty$) transitive, totally Cartan action.  Then the Starkov component factors through a $\mathbb{T}^\ell$ action, and the quotient $\bar{X} = X / \mathbb{T}^\ell$ is a $C^\infty$ manifold and the projection is $C^2$ ($C^\infty$). Any complementary $\R^{k-\ell}$ acts on $\bar{X}$ by a transitive, totally Cartan action $C^2$ ($C^\infty$) with trivial Starkov component. %Furthermore, $X$ is a principle $\mathbb{T}^\ell$ bundle over $\bar{X}$, and the Starkov component is $\ell$-dimensional and acts locally freely by translations in each fiber.
%\end{theorem}
 The Starkov factor carries all of the revelant hyperbolic behavior on $X$.

 In addition to Theorem \ref{thm:big-main}, we can completely describe the structure of totally Cartan actions (even with rank one factors). %We also obtain the following description of a general cone transitive, totally Cartan action. 
For a more precise statement, see Theorem \ref{thm:big-headache}.

\begin{theorem}[Structure Theorem]
\label{thm:mini-headache}
If $\R^k \curvearrowright X$ is a %$C^r$, $r = (1,\theta)$ 
{ $C^\infty$ (or $C^2$)}, cone transitive, totally Cartan action with trivial Starkov component, then $X$ is a homogeneous { $C^\infty$ (resp. $C^{1,\theta}$ for any $1>\theta > 0$)} fiber bundle over a closed submanifold of product of 3-manifolds. The action covers a (rescriction of a) product of Anosov 3-flows, preserves the fiber bundle, and both the action and holonomies are algebraic maps when restricted to fibers.
\end{theorem}

%One cannot relax the totally Cartan condition to the Cartan condition: we in fact construct an example of an $\R^2$ action which is Cartan, but not totally Cartan (see Section \ref{app:anosov-not-totally}). We furthermore 
We cannot dispose of the condition that the Starkov factor is trivial, or that the actions covers a {\it restriction} of the product of Anosov flows, as shown in Sections \ref{sec:exotic} and \ref{sec:embedded-ex}.
The case of $\Z^\ell$ actions has a simpler statement, due to the absence of a Starkov component.

\begin{corollary}[Rigidity of $\Z^\ell$ Actions]
\label{cor:Zk-actions}
Let $\alpha : \Z^\ell \curvearrowright X$ be a { $C^2$}  cone transitive, totally Cartan action. Then, up to finite cover, $X$ is a $C^{1,\theta}$ fiber bundle {  (for any $1>\theta > 0$)} over $(\mathbb{T}^2)^m$ for some $m$, whose fibers have a canonical $C^{1,\theta}$ nilmanifold structure (ie, are parameterized by a transitive $C^{1,\theta}$ nilpotent group action). The $\Z^\ell$ action maps fibers to fibers, and descends to (a restriction of) the product of $m$ $C^{1,\theta}$ Anosov diffeomorphisms on $(\mathbb{T}^2)^m$.

 Furthermore, each $\alpha(a)$ and the stable and unstable holonomies are induced by automorphisms of the nilmanifold structure when restricted to fibers. Finally, if the action is $C^\infty$, so are the Anosov diffeomorphisms, the submersion and the nilmanifold structure. %Then there exists an action $\Z^m \curvearrowright (\mathbb{T}^2)^r \times Y$ for some nilmanifold $Y$, given by the product of Anosov diffeomorphisms of $\mathbb{T}^2$ and an affine action on $Y$, such that $X$ is a $C^{1,\theta}$ finite-to-one factor of $(\mathbb{T}^2)^r \times Y$, and a finite index subgroup of $\alpha(\Z^\ell)$ coincides with a restriction of the induced $\Z^m$ action. If the action is $C^\infty$, so is the factor map. 
\end{corollary}

We can also apply Theorem \ref{thm:big-main} to obtain a global rigidity result for totally Cartan actions on spaces of the form $X = M \backslash G / \Gamma$, with $G$ a semisimple Lie group and $M$ a compact subgroup. This result is complementary to the global rigidity results for Anosov $\Z^k$ actions (for which the best results are given in \cite{RH-W}, see the discussion above).

\begin{corollary}[Global Rigidity for Weyl Chamber Flows]
\label{cor:global}
Let $G$ be a semisimple linear Lie group of noncompact type without compact factors, $\Gamma \subset G$ be a cocompact lattice such that $\Gamma$ projects densely onto any rank one factor  of $G$,  and $M$ be a connected compact subgroup of $G$ such that $M \cap \Gamma = \set{e}$. Define $X = M \backslash G / \Gamma$ to be the corresponding double-homogeneous space.  Fix any $k \ge 1$. Then there is at most one $C^\infty$ ({or  $C^2$}),  transitive, totally Cartan $\R^k$-action on $X$, up to $C^\infty$ ({  resp. $C^{1,\theta}$ for every $\theta > 0$}) conjugacy and linear time change. Such an action exists if and only if $G$ is $\R$-split, and in this case,  $M = \set{e}$, $k = \rank_\R(G)$ and the action is the Weyl chamber flow on $G / \Gamma$. 
%Let $G$ be a semisimple linear Lie group with all factors having $\R$-rank at least two, $\Gamma\subset G$ be a lattice, and $M = G / \Gamma$. Let $\R^\ell \curvearrowright M$ be the Weyl chamber flow on $M$ (see Section \ref{app:weyl-ch-flows}). If another action $\R^k \curvearrowright M$ is a transitive, totally Cartan action, then $k = \ell$, $G$ is $\R$-split and there exists a map $h : M \to M$ which is $C^{1,\theta}$ for any $\theta \in (0,1)$ and conjugates the two actions (which is $C^\infty$ if the $\R^k$ action is $C^\infty$). {  proof needs updating}
\end{corollary}

%As a consequence of Corollary \ref{cor:global}, we get that on each double homogeneous space $M \backslash G / \Gamma$ of the form described in the corollary, there are either no totally Cartan, cone transitive actions, or exactly one, up to linear time change and smooth change of coordinates.

%\begin{remark}
%The Weyl chamber flow on a semisimple homogeneous space $G / \Gamma$ is Cartan if and only if $G$ is $\R$-split. Note that in Corollary \ref{cor:global}, we do not assume that $G$ is $\R$-split. As a result, one may conclude that if $G$ is not $\R$-split, then $G / \Gamma$ does not have any Cartan actions satisfying the assumptions of Theorem \ref{thm:big-main}. Furthermore, if $G$ is $\R$-split, then there is a unique totally Cartan transitive action on $G / \Gamma$ up to conjugacy.
%\end{remark}

Theorem \ref{thm:big-main} also has an immediate application to the Zimmer program of classifying actions of higher rank semisimple Lie groups and their lattices on compact manifolds.  We refer to Fisher's recent surveys \cite{Fisher-survey2011,Fisher:2017aa} for a more extensive discussion.  This program has been a main impetus for seeking rigidity results for hyperbolic actions of higher rank abelian groups. We consider the following class of actions

\begin{definition}
Let $G$ be a semisimple Lie group. We say that a $C^r$ action $G \curvearrowright X$ is {\normalfont totally Cartan} if it is locally free, and there exists a connected abelian subgroup $A \subset G$ such that the restriction of the $G$-action to an $A$-action is totally Cartan. We say that the action is {\normalfont strongly transitive} if for every one-parameter subgroup $L \subset G$ which is not in the kernel of the projection onto a simple factor of $G$, $L$ has a dense orbit.
\end{definition}

\begin{remark}
If $G$ preserves a ergodic fully supported measure, then any action of $G$ is strongly transitive by considering the Howe-Moore theorem. Furthermore, if $G$ has a totally Cartan action, then $G$ is $\R$-split, and strong transitivity of the $G$-action implies cone transitivity of the restriction.
\end{remark}

\begin{corollary}[Classification of totally Cartan $G$-actions]
\label{cor:semisimple-actions}
Let $G$ be a connected semisimple Lie group without factors locally isomorphic to $PSL(2,\R)$.  Suppose that $G \curvearrowright X$ is a ${ C^2}$ strongly transitive, totally Cartan action of $G$. %Suppose $G$ acts on a compact manifold $X$ by $C^{1,\theta}$  diffeomorphisms such that the restriction to a split Cartan $A \subset G$ is a {  cone} transitive, totally Cartan action of $A$.  %Also assume that the action preserves a fully supported measure $\mu$ on $M$, and that every simple factor of $G$ acts ergodically w.r.t. $\mu$.  
Then a finite cover of the action is $C^{1,\theta}$ {  (for any $1>\theta >0$)} conjugate to a
homogeneous action of $G$ on a homogeneous space $H/\Lambda$ via an embedding $G \rightarrow H$. If the action is $C^\infty$, so is the conjugacy.
\end{corollary}

\begin{remark}
We still get rigidity when $\rank_\R(G) = 1$ here, but for uninteresting reasons. Indeed, the Cartan condition will require that such a $G$ is locally isomorphic to $PSL(2,\R)$, which we assume does not occur.
\end{remark}

 {\sl Acknowledgements:} The authors would like to first and foremost thank the referee for a very detailed reading of the paper and extensive comments. We are also grateful to Mladen Bestvina, Aaron Brown, Danijela Damjanovic, Boris Hasselblatt, Boris Kalinin, Karin Melnick, Victoria Sadovskaya, Amie Wilkinson and Disheng Xu, for their interest and various conversations, especially about normal forms and cohomological dimension. { We would also like to thank Ekaterina Shchetka for identifying a gap in an argument.} Finally, we would like to thank David Fisher and Federico Rodriguez Hertz for comments on the paper and suggestions for improvements. We also  thank University of Chicago, Pennsylvania State University and University of Michigan for support and  hospitality  during various trips.

\section{Summary and structure of the paper}
\label{sec:summary}

The paper is organized into 5 parts, with the proofs in the last three. In Part I, we review several important preliminary topics, as well as prove some ``folklore'' theorems (for instance, the generalization of the Anosov closing lemma to $\R^k$ actions, Theorem \ref{thm:anosov-closing}). Many such results had been proven in specialized cases or stated without proof in other works. In Part II, we recall some standard constructions to build totally Cartan $\R^k$-actions, as well as some new constructions of exotic examples. In Part III, we discuss the relationship between transitivity conditions for subactions and the existence of rank one factors. In Part IV, we show that if an action satisfies a transitivity condition equivalent to the non-existence of rank one factors, then the action is homogeneous. In Part V, we use the structural results of the first two parts to prove the full structure theorem, as well as establish applications of the main theorems.

\subsection{Parts I and II: Developing the Toolbox}

Parts I and II appear before the main arguments and results of the paper. Part I begins by summarizing necessary algebraic tools in Section \ref{subsec:free-prods}, which recalls results in topological groups which we will use in Part IV.  %Section \ref{sec:affine-classification} contains a classification of homogeneous Cartan actions.

%We continue to establish many basic technical tools, especially  from Anosov flows and diffeomorphisms. Many, but not all, such results have proofs which are immediate adaptations from the rank one to the higher rank setting. We will use them freely, after establishing them in Section \ref{subsec:dynamical-prelim}. %  In particular,  we have decompositions into $\R^k$-Oseledets spaces   and Lyapunov exponents for every invariant measure.  The latter are   also called {\em weights}, and we view  them as linear maps from the acting group $\R ^k$ to $\R$.   We refer to \cite{Brown:2016aa} for a detailed presentation. 

Section \ref{subsec:dynamical-prelim} contains many results which are folklore theorems and/or adaptations of classical theorems in Anosov flows to Anosov $\R^k$-actions. This includes the Anosov closing lemma (Theorem \ref{thm:anosov-closing}), which appeared in \cite{KS1} without proof and in a more specialized setting in \cite[Lemma 4.5 and Theorem 4.8]{MR808219}. We also prove a spectral decomposition (Theorem \ref{thm:spectral}). Unlike its counterpart in Anosov flows, the proof that $\R^k$-periodic orbits are dense (Lemma \ref{lem:cone-transitive}) is not immediate from transitivity, even with the closing lemma. This is the only place where the condition of cone transitivity over transitivity is required.

We also generalize notions of stable manifolds called {\it coarse Lyapunov foliations}, with associated linear functionals called {\it weights}. In the presence of an invariant measure, the weights are related to Lyapunov exponents coming from higher-rank adaptations of Pesin theory, see \cite{Brown:2016aa} for a detailed presentation.

 We also recall recent developments on normal forms, prove some folklore theorems related to the construction of Lyapunov hyperplanes and show that Anosov flows on 3-manifolds are virtually self-centralizing (Theorem \ref{thm:3flow-centralizer}). 

Section \ref{sec:closing} contains several of the crucial new tools we develop and use in Parts III-V. The core of the arguments throughout the paper is isometric-like behavior along coarse Lyapunov foliations and moving by their associated hyperplanes. This adapts arguments from \cite{KaSp04} which produced a family of metrics on the coarse Lyapunov distributions such that the derivative along each such distribution is given by a functional independent of $x$. This is a feature of algebraic systems (indeed, our arguments show that it is a characterizing feature!). Following \cite{KaSp04}, such a metric is constructed in Section \ref{subsec:metrics-prelim} and used extensively in Part IV. When the action has rank one factors, one does not expect such a metric, since the derivative cocycle may not be cohomologous to a constant. However, the Lyapunov hyperplanes still have a weaker property: uniformly bounded derivatives along hyperplanes. This is proved in Lemma \ref{lem:uniformly bounded derivative}, and serves as an essential tool in Parts III and V.

The last tool we develop is the {\em   geometric commutator}, which is essential in understanding the way in which the coarse Lyapunov foliations interact. When the coarse Lyapunov distributions are smooth, one can take Lie brackets to answer such questions. %For this, one needs to prove the corresponding vector fields have smoothness properties. 
In our case, and most dynamical cases, the vector fields corresponding to dynamical foliations are only H\"older continuous. This motivates another approach. %, and work with the usual setting when the coarse Lyapunov spaces and these group actions  are only H\"older continuous transversally.  w
We replace the Lie bracket with a coarser geometric version of the commutator of two foliations as $W ^{\alpha}$ and $W ^{\beta}$ as follows: create a path by  following the $\alpha$ and $\beta$ coarse Lyapunov spaces to create a ``rectangle.'' We produce a canonical way to close this path up  using legs from other coarse Lyapunov spaces. Those combined paths define the {\em geometric commutator}. We will often call the individual segments in a coarse Lyapunov space {\em legs} of the geometric commutator.

Part \ref{subsec:examples} has several examples of actions which exhibit the features of the theorems as well as illustrate their optimality. We also prove a classification of affine actions. In particular, we begin by recalling several classically studied examples in Section \ref{sec:algebraic-exs}. We then prove a structure theorem for homogeneous Cartan actions in Section \ref{sec:affine-classification}, which will be important in applications and shows that the classical examples are (nearly) exhaustive.

Finally, we discuss several exotic constructions of Cartan actions in Section \ref{sec:exotic}. Many of these examples, such as the one due to A. Starkov discussed in Section \ref{subsec:starkov}, were known previously. There are also several new examples presented which show that some assumptions cannot be relaxed, such as that of being {\it totally} Cartan (Section \ref{app:anosov-not-totally}). This example was further developed by the second author to produce a non-homogeneous totally Cartan action without rank one factors. Furthermore, some conclusions cannot be strengthened, such as the need for a skew product structure described in Theorem \ref{thm:mini-headache}, rather than a direct product structure (Section \ref{sec:nonproduct-ex}). A similar example was recently constructed by Damjanovic, Wilkinson and Xu, which we summarize in Section \ref{sec:DWX}.

\subsection{Part III: Rank One Factors and Transitivity of Hyperplane Actions}

One of the key novelties, and difficulties, of our results is that we obtain classification results on arbitrary manifolds, without requiring irreducibility conditions, such as ergodicity of one parameter subgroups. This allows for products and skew products to appear (as described in the statement of Theorem \ref{thm:mini-headache}, and later Theorem \ref{thm:big-headache}). Indeed, since not every Anosov flow on a 3-manifold is homogeneous, we cannot hope to prove Theorem \ref{thm:big-main} without assumptions to rule them out. If such a rank one factor exists, by Definition \ref{def:rank-one-factor}, there is an associated codimension one subgroup $H = \ker \sigma$ which acts trivially on the factor, and therefore preserves each fiber. Therefore, if the action has a rank one factor, there exists a codimension one hyperplane which does not have a dense orbit. In fact, generically, the $H$-orbit closures are exactly the fibers of the factor. 

%The aim of the first part of the paper is therefore the following highly nontrivial result, which gives the converse of this observation. 
The converse of this observation is much more difficult to prove: one must build a factor under the assumption that certain hyperplane orbits are never dense. The aim of Part III is exactly this undertaking, culminating in Theorem \ref{thm:main-anosov}. We restrict our attention to $\R^k$-actions, and treat $\R^k \times \Z^\ell$-actions by considering their suspensions. Indeed, the relationship between rank one factors of the $\R^k \times \Z^\ell$ actions and factors of the suspension is established in Lemma \ref{lem:susp}.

\begin{theorem}
\label{thm:main-anosov}
Let $\R^k \curvearrowright X$ be a $C^2$ ($C^\infty$) cone transitive totally Cartan action on a compact manifold $X$, and fix a Lyapunov hyperplane $H \subset \R^k$. Then {  for some $1>\theta >0$} either:

\begin{itemize}
\item the action lifts to some finite cover $\R^k \curvearrowright \tilde{X}$, and there is a $C^{1,\theta}$ ($C^\infty$) non-Kronecker rank one factor of the lifted action, or
\item there is a {  $C^{1,\theta}$}
($C^\infty$) %submersion $\pi : X \to \mathbb{T}^1$ and homomorphism $f : \R^k \to \R$ such that $\pi(a \cdot x) = \pi(x) + f(a)$
Kronecker factor onto $\mathbb{T}^1$, and there exists $x \in \pi^{-1}(0)$ such that $\overline{H \cdot x} = \pi^{-1}(0)$, or 
\item there exists a point $x$ such that $\overline{H \cdot x} = X$.%with corresponding foliations $\mc W$, there is a metric $\norm{\cdot}_{\mc W}$ on $T\mc W$ which is invariant under $H$.
\end{itemize}
%coarse Lyapunov foliation $\mc W$, there exists a linear functional $\alpha : \R^k \to \R$ and a H\"older metric $\norm{\cdot}_\mc W$ on $T\mc W$ such that

%\begin{equation}\label{eq:equivariance}
% \norm{a_* v}_{\mc W} = e^{\alpha(a)} \norm{v}_{\mc W} \mbox{ for all }a \in \R^k, v \in T\mc W
%\end{equation}
\end{theorem}

There are several important ingredients in proving Theorem \ref{thm:main-anosov}. We assume throughout that $H$ does not have a dense orbit, and aim to obtain one of the first two cases of the theorem. First, we have to build a model for the factor. We do so by first showing that if the action $H$ is not transitive, then among the set of $\R^k$-periodic orbits, many of them are also $H$-periodic. Using normal forms, we are able to show that saturations of $H$-periodic orbits by the corresponding coarse Lyapunov foliations are actually fixed point sets for regular elements of $H$ fixing the periodic orbit itself. We call this set $M^\alpha(p)$. They therefore carry a canonical smooth structure, and are $(k+2)$-dimensional. We further show that the restriction of the $\R^k$-action to this invariant set is Cartan, and that the $H$-action is the Starkov component for the restricted action. Therefore, after factoring by the $H$-action, which is through a $(k-1)$-dimensional torus, one obtains a 3-manifold onto which we factor.

With the model in hand, we must define the projection. Let $\alpha$ denote a weight linked to $H$. Intuitively, in the case of a direct product, moving by the $\beta$-foliations, $\beta \not= \pm c\alpha$ and $H$-orbits does not change the position on the factor. Therefore, it is natural to consider the sets $W^H(x)$, which are the saturations of points $x$ by their $H$-orbits and $\beta$-foliations, $\beta \not= \pm c\alpha$. The main idea is then to show that $W^H(x)$ intersects the set $M^\alpha(p)$ in a single point (or more generally, in a finite set) for every $x$, and defining the projection by $x \mapsto W^H(x) \cap M^\alpha(p)$.

Building structure on such intersections is the main technical work of Part I. Crucially, the control of derivatives for elements of Lyapunov hyperplanes translates into control of derivatives for holonomies along coarse Lyapunov foliations, so that given a path in $W^H$ which begins and ends on the same $M^\alpha(p)$, one can construct a holonomy map and bound its derivative based on the number of times the path ``switches'' between exponents. One can use this control to show that the holonomies must converge, and using normal forms that any sequence of holonomies limiting to the same coarse Lyapunov leaf induces an affine map. Rigidity properties for subgroups of the affine groups allow us to conclude that $\abs{W^H(x) \cap M^\alpha(p)} = \infty$, and with more work, that $\overline{W^H(x)} \supset M^\alpha(p)$. In this case, we are able to show that $\overline{W^H(x)} = X$, and hence that the action of $H$ is not transitive. Furthermore, if the action of $H$ is not transitive, we show that $\abs{W^H(x) \cap M^\alpha(p)}$ is independent of $x$. This implies that there is a well-defined factor of the $\R^k$-action, which is a finite-to-one factor of the induced Anosov flow on $M^\alpha(p)$.

%\begin{corollary}
%\label{cor:main-cartan}
%If $\Z^k \curvearrowright X$ is a totally Cartan action with a dense orbit, then either there is a $C^\infty$ rank one factor, or for every coarse Lyapunov foliation $\mc W$, there exists a linear functional $\alpha : \R^k \to \R$ and a H\"older metric $\norm{\cdot}_\mc W$ on $T\mc W$ such that

%\begin{equation}\label{eq:equivariance}
% \norm{a_* v}_{\mc W} = e^{\alpha(a)} \norm{v}_{\mc W} \mbox{ for all }a \in \R^k, v \in T\mc W.
%\end{equation}
%\end{corollary}

%Recall that an affine action on a homogeneous space $X = G / \Gamma$ is an action given by compositions of translations and automorphisms.

\subsection{Part IV: Constructing a Homogeneous Structure}
\label{sec:partii-sum}

In Part IV, we prove Theorem \ref{thm:big-main}. The main difficulty in producing a conjugacy with a homogeneous action is building a homogeneous structure from purely dynamical, and not geometric or topological, input. Indeed, we do not assume anything about the underlying manifold. In the case of $\R^k$, $k \ge 3$ actions, the construction was carried out in \cite{KaSp04} using a Lie algebra of vector fields, under strong assumptions. The case of $\R ^2$ actions is significantly more difficult than that of actions of  rank 3 and higher groups, and previous methods remain inadequate in our generality.  Indeed, we need to develop several completely new tools. They allow us to remove several restrictive assumptions, such as ergodicity of one-parameter subgroups with respect to a fully supported invariant measure, and even the existence of such a measure at all. 
We expect these new ideas to be useful in other situations as well. % Let us outline some of the steps and tools of our argument.  
%\vspace{.2em}

%{\sl Basic Higher Rank Tools:} Many basic technical tools, especially  from Pesin theory,  have been adapted from single diffeomorphisms to higher rank actions. We will use them freely, after summarizing them  below in Section \ref{subsec:dynamical-prelim}.  In particular,  we have decompositions into $\R^k$-Oseledets spaces   and Lyapunov exponents for every invariant measure.  The latter are   also called {\em weights}, and we view  them as linear maps from the acting group $\R ^k$ to $\R$.   We refer to \cite{Brown:2016aa} for a detailed presentation.  
%\vspace{.2em}

The starting point of our argument is  the  main technical result about the H\"{o}lder cohomology of the derivative cocycle, which was first shown in \cite[Theorem 1.2]{KaSp04} under much stronger assumptions. We show that it also holds for totally Cartan actions of any higher rank abelian group without rank one factors.  It implies immediately the following theorem, which gives a canonical homogeneous structure on each coarse Lyapunov leaf:

\begin{theorem}
\label{thm:main-cartan}
Let $\R^k \times \Z^\ell \curvearrowright X$ be a $C^2$ ($C^\infty$) cone transitive totally Cartan action on a compact manifold $X$. Fix a coarse Lyapunov foliation $\mc W$. Then either the action lifts to a finite cover $\R^k \times \Z^l \curvearrowright \tilde{X}$ and there is a {  $C^{1,\theta}$ (for some $1>\theta>0$)} ($C^\infty$) non-Kronecker rank one factor of the lifted action such that $\mc W$ projects to the stable foliation, or there exists a linear functional $\alpha : \R^k \to \R$ and a H\"older metric $\norm{\cdot}_\mc W$ on $T\mc W$ such that

\begin{equation}\label{eq:equivariance}
 \norm{a_* v}_{\mc W} = e^{\alpha(a)} \norm{v}_{\mc W} \mbox{ for all }a \in \R^k, v \in T\mc W.
\end{equation}

\noindent In the latter case, the norm $\norm{\cdot}_{\mc W}$ is unique up to global scalar.
\end{theorem}

\vspace{.2em}

%{\sl Path and Cycle Groups and the Lie criteria:} % By our ergodicity assumptions we get  transitive isometric group actions of $\R$ on each coarse Lyapunov manifold, as in \cite{KaSp04}.  
The H\"older metrics make each coarse Lyapunov manifold isometric to $\R$. After passing to a cover, we may define H\"older flows along each such manifold which act by translations in each leaf. We now discuss the main novelty of this paper: how to ``glue'' groups actions which parameterize only {\it topological} foliations with smooth leaves into a larger Lie group action that actually gives the total space a homogeneous structure. The method is as follows:

 %Together with the $\R ^k$ action, 
The $\R$-actions parameterizing coarse Lyapunov foliations allow us to define an action of the free product $\mc{P} :=\R ^{*d}$ on $X$ %and its universal cover $\tilde{M}$ 
where $d = \dim M - k$.  Elements of $\mc P$ are formal products of elements coming from each copy of $\R$, which we call {\it legs}. Each copy of $\R$ will correspond to a flow along a corresponding coarse Lyapunov foliation. We call $\mc P$ the {\em path group}. With the free product topology, $\mc{P}$ becomes a connected and path connected topological group which, when combined with the $\R^k$ action in a precise way, gives a transitive topological group action %acts transitively 
on $X$.  This group $\mc{P}$ is enormous, infinite dimensional for sure, and not a Lie group (see Section \ref{subsec:free-prods}).    Our main achievement   is to show that this action  factors through  the action of a Lie group.  The rough idea is to show that  the {\em cycle subgroups}, the stabilizers of  $\mc{P}$ at $x$, are normal.  Such an action of $\mc P$ is said to have {\it constant cycle structures}, see Corollary \ref{cor:lie-from-const}. We will not be able to show this exactly, but instead show that every stabilizer contains a fixed normal subgroup $\mc C$ independent of $x$, for which $\mc P / \mc C$ has has canonical local Euclidean coordinates (see Corollary \ref{cor:GhatLie}). Since $\mc{P}$ acts transitively on $M$ by construction, $M$ will be a $G$-homogeneous space. Moreover, the original action of $\R ^k$ naturally relates to $\mc{P}$, and becomes part of the homogeneous action.

The idea of using free products to build a homogeneus structures was first explored by the second author in \cite{vinhageJMD2015} when proving local rigidity of certain algebraic actions.  Basically it is a new tool to build global homogeneous structures from partial ones on complementary subfoliations.   
The main insight of Part IV is that these technical and seemingly narrow algebraic techniques can actually be applied to actions on {\it any} manifold, only using dynamical structures.

%That this can be done  away from algebraic actions, assuming just uniform hyperbolicity, is the main insight of Part II.  

From this perspective, the main goal becomes proving constancy of the cycle relations.  There are two particularly important cases: cycles consisting only of two proportional (``symplectic'') weights, and commutator cycles of non-proportional weights. We call these  types of cycles  {\em pairwise cycle structures} (since they involve commutators of negatively proportional weights $\alpha$ and $-c\alpha$ or linearly independent weights $\alpha$ and $\beta$). From commutator cycles of linearly independent weights, we are able to obtain constancy of  
 cycles whose legs belong to a stable set of weights $E$, where we call $E$ {\em stable}  if for some  $ a \in \R ^k$, $\lambda (a) <0$ for all $\lambda \in E$. 
%For more details see below. 
% ``symplectic'' cycles coming from proportional Lyapunov functionals, and finally stable cycles which arise from a  stable set $E$ of  weights.  Here we call $E$ {\em stable}  if there exists $ a \in \R ^k$ such that  $\lambda (a) <0$ for all $\lambda \in E$. 
%We will explain in more detail below. The commutator and symplectic cycles are  two types of relations generated by a pair of weights $\alpha $ and $\beta$, and we call   these types of cycles  {\em pairwise cycle structures}.  

Once  constancy of symplectic and stable cycles  is accomplished, we combine this information  to prove constancy of all cycles in Section \ref{sec:fibers}.  In other related works on local rigidity problems, the latter was  achieved via K-theoretic arguments (this first appeared in \cite{DamjanovicKatok2005}). We found a new way to do this, avoiding the intricate K-theory arguments, by showing constancy of an open dense subset of relations using explicit relations between stable and unstable horocycle flows in $PSL(2, \R)$.  The K-theory argument was used in the past to treat the remaining potential relations.  However, density of the good relations makes this unnecessary.  %This also simplifies the K-theoretic arguments in past works, e.g. %\todo{not sure these references are correct - also not sure if we maybe want to just mention one or two if any}\cite{DamjanovicKatok2005,DamjanovicKatok2007,Damjanovic2007,DamjanovicKatok2011,DamjKatok2011,Vinhage:2015aa,vinhageJMD2015,Vinhage:2015aa}.
%\vspace{.2em}

%Next,  say that a set of weights $E$ is {\em stable}  if there exists $ a \in \R ^k$ such that  $\lambda (a) <0$ for all $\lambda \in E$. 

  We first introduce a cyclic ordering of the Lyapunov hyperplanes (kernels of weights) to handle stable cycles (see Definition \ref{def:circular-ordering}).  %For  $\R^ k $ actions with $k \geq 3$, restrict to a generic 2-plane in $\R ^k$ which is not contained in $\ker \alpha$ for any weight $\alpha$ to get an ordering. 
Then we simplify 
products $\eta ^{\lambda _{i_1} } _{t_{i_1}} * \ldots * \eta ^{\lambda _{i_l}}_{t_{i_l}} $  with all $\lambda_{i_j}$ in a stable subset $E$ by putting the $\lambda_{i_j}$ into cyclic order by commuting them with other $\lambda_{i_k}$.  Assuming commutator cycles are constant we can then easily show that stable cycles are constant using the dynamics of the action, cf. Section \ref{subsec:sun-presentations}. Combining this with the uniqueness of presentation of elements of the stable set in the order determined by a circular ordering, we get an injective continuous map from $\mc P / \mc C$. We can then apply another Lie criterion developed by Gleason and Palais (Theorem \ref{lem:gleason-palais}) to get a group structure on stable manifolds.

To prove constancy of pairwise cycle structures coming from negatively proportional weights $\alpha, - c \,\alpha$ for $c>0$, we imitate the procedure in \cite{GS1,KaSp04}, replacing their use of isometric returns with the tools of path and cycle groups. Using a point $x_0$ with a dense $\ker \alpha$-orbit, we are able to take any cycle with legs in $\alpha$ and $- c \, \alpha$, and obtain it as a cycle at every point in $M$. We can then apply the Gleason-Palais Lie criterion theorem (again) to get a Lie group structure.  We use this homogeneous structure, to understand the cycles coming from $\alpha, - c \alpha$ (see Lemma \ref{lem:symplectic}). 
\vspace{.2em}

%Proving constancy of stable cycles is fairly simple.  :  we first observe that just order the Lyapunov hyperplanes (kernels of weights) in a cynic order, and then commute 
 
 %{\sl Commutator Cycles:}   Finally, we discuss the commutator cycles of two weights $\alpha$ and $ \beta$ which are not proportional.  This is by far the most difficult part of our argument, and requires yet again a set of diverse new tools. 
 %\vspace{.2em}

%{\em Higher Rank at least 3:} For $\R ^k$ actions with $k \geq 3$, $\ker \alpha \cap \ker \beta$ has dimension at  least 1. Under the stronger assumption that such intersections also have dense orbits, we immediately get constancy of the commutator of  $\alpha$ and $\beta$, finishing the argument in this case. However, the condition that $\ker \alpha \cap \ker \beta$ has dense orbits is similar to a no rank {\it two} factor condition, which is stronger than necessary. In particular, if the action is rank two, then these arguments fail completely. Instead, we need to find new ways to build homogeneous structures. We provide an outline below only for rank 2 actions for simplicity. % which we outline for simplicity in the rank two case below.
% \vspace{.2em}

%{\em $\R ^2$ actions:} From now on, consider $\R^2$ actions.   As the argument above fails with weaker assumptions (in particular, in rank 2), we have to proceed entirely differently.   
 %The simplest case is when for all weights $\lambda$, the coarse Lyapunov foliation $\mc{W} ^{\lambda}$ has a dense leaf in $M$.  
 The other important relation to establish is that the geometric commutators are independent of $x$. For simplicity suppose that for some $\alpha$ and $\beta$, there exists a unique weight $\gamma$ which is a positive linear combination of $\alpha$ and $\beta$.  Then the geometric commutators between $\alpha$ and $\beta$ satisfy an intertwining property (see \eqref{eq:rho-equivariance}), as well as a cocycle relation. By using these, we are able to establish regularity properties of the geometric commutator. In fact, we deduce that they are always polynomials in the lengths of the legs and independent of the basepoint. A more complicated argument always works by induction, replacing the cocycle property by a cocycle ``up to polynomial'' property (c.f. Lemma \ref{eq:cocycle-like-bb}). 

 \vspace{.2em}

\subsection{Part V: Structure of General Totally Cartan Actions}

In Part V, we prove the most general structural result of the paper, Theorem \ref{thm:big-headache}, which describes the general structure of a totally Cartan action in terms of its rank one factors. We also classify when a rank one factor ``commutes'' with the rest of the action (ie, when it is a direct product of a subaction of rank $k-1$). Definition \ref{def:splitting}, Lemma \ref{lem:splitting-equivalence} and Theorem \ref{thm:splitting-implies-direct} give checkable criteria for detecting when a rank one factor splits as a direct product. %We begin by establishing some structure theorems which reduce a proof of Theorem \ref{thm:full-classify} to Theorem \ref{thm:main-self-centralizing}. 

 We begin in Section \ref{sec:starkov}, by describing how the Starkov component acts. The uniform boundedness lemma (Lemma \ref{lem:uniformly bounded derivative}) will imply that the action of the Starkov component is equicontinuous, and factors through a torus action. When combining Corollary \ref{cor:factor-by-starkov} with Theorem \ref{thm:big-headache}, one can obtain a general structure theorem of actions with nontrivial Starkov component, by describing their Starkov factor, then claiming the original action is an extension by a transitive torus action.
 
 To prove Theorem \ref{thm:big-headache}, we consider all rank one factors of the action simultaneously. While the map to all such factors may fail to be onto, it fails in a controlled way by being a submersion onto a closed submanifold (for an example of this phenomenon, see Section \ref{sec:embedded-ex}). By Theorem \ref{thm:main-cartan}, any coarse Lyapunov foliation whose leaves are contained in the fibers of the submersion has an associated metric. We call these the ``homogeneous weights,'' and show that the collection is invariant under {\it all} geometric  commutators. We are then able to apply the arguments of Part IV again to build a homogeneous structure on the fibers, as well as prove that holonomies along the stable and unstable manifolds of the rank one factors are affine in this homogeneous structure.

We then turn to the centralizers of cone transitive, totally Cartan actions in Section \ref{sec:centralizer}. Understanding the structure of the centralizer in this setting will be important when considering holonomies induced by motion along the stable and unstable manifolds of a rank one factor. %In particular, we show that such centralizers act uniformly hyperbolically.The difficulty is showing that the action by the centralizer is still totally Cartan. 
We consider the action of the centralizer on periodic orbits, and show that an element of the centralizer is determined by its multipliers along finitely many such orbits (Lemmas \ref{lem:psi-injective} and \ref{lem:injective-derivative}). An important consequence will be that there are only finitely many elements of the centralizer whose iterates have uniformly bounded derivatives (Corollary \ref{cor:q-finite}).

The main application of the ideas of Section \ref{sec:centralizer} is achieved in Section \ref{sec:prod-struct}, where the main difficulty is the analysis of compositions of stable and unstable holonomies coming from a rank one factor. We show that a direct product structure must come from the action of such holonomies, and classify when they do in Lemma \ref{lem:splitting-equivalence} and Theorem \ref{thm:splitting-implies-direct}.

The end of Part V also includes proofs of the applications of the main theorem (Section \ref{sec:corollaries}). The application to discrete group actions is an easy consequence of the main theorems and the preliminary work done in Section \ref{sec:susp}. The application to global rigidity of Weyl chamber flows (Corollary \ref{cor:global}) is established by first excluding rank one factors by carefully considering how a rank one factor interacts with the fundamental group and using the Margulis normal subgroup theorem. Once the group $G$ giving $X$ its homogeneous structure coincides with the group produced by Theorem \ref{thm:big-main}, we may apply Mostow rigidity to see that the two groups are the same, and conclude global rigidity. The final application to totally Cartan $G$-actions is basically immediate.

\subsection{Proof Locations}
For convenience, the following table contains the locations of the proofs of each of the main results stated in these introductory sections:

\begin{center}
\begin{tabular}{|l|l|}
\hline
Result & Location \\
\hline
\hline
Theorem \ref{thm:big-main} & Section \ref{sec:proofs} \\
\hline
Theorem \ref{thm:mini-headache} & Section \ref{sec:main-structure} \\
\hline
%Theorem \ref{thm:splitting-implies-direct} & Section \ref{sec:prod-struct} \\
%\hline
Corollaries \ref{cor:Zk-actions}-\ref{cor:semisimple-actions} & Section \ref{sec:corollaries} \\
\hline
Theorem \ref{thm:main-anosov}, \ref{thm:main-cartan} & Section \ref{sec:part1-proofs} \\
\hline
\end{tabular}
\end{center}

\part[Preliminaries and Standard Structure Theory]{\Large Preliminaries and Standard Structure Theory}

\section{Topological group and Lie group foundations}
\label{subsec:free-prods}

\subsection{Free products of topological groups}
Let $U_1,\dots,U_r$ be topological groups. The {\it topological free product} of the $U_i$, denoted $\mc P = U_1 * \dots * U_r$ is a topological group whose underlying group structure is exactly the usual free product of groups. That is, elements of $\mc P$ are given by

\[ u_1^{(i_1)} * \dots * u_N^{(i_N)} \]

\noindent where each $i_k \in \set{1,\dots,r}$ and each $u_k \in U_{i_k}$. We call the sequence $(i_1,\dots,i_k)$ the {\it combinatorial pattern} of the word. Each term $u_k^{(i_k)}$ is also called a {\it leg} and each word is also called a {\it path}. This is because in the case of a free product of connected Lie groups, the word can be represented by a path beginning at $e$, moving to $u_N^{(i_N)}$, then to $u_{N-1}^{(i_{N-1})} * u_N^{(i_N)}$, and so on through the truncations of the word. The multiplication is given by concatenation of words, and the only group relations are given by

\begin{eqnarray}
\label{eq:freep-rel1} u^{(i)} * v^{(i)} & = & (uv)^{(i)} \\
\label{eq:freep-rel2} e^{(i)} & = & e \in \mc P
\end{eqnarray}

Notice that the relations \eqref{eq:freep-rel1} and \eqref{eq:freep-rel2} give rise to canonical  injective homomorphisms of each $U_i$ into $\mc P$. We therefore identify each $U_i$ with its image in $\mc P$. The usual free product is characterized by a universal property: given a group $H$ and any collection of homomorphisms $\varphi_i : U_i \to H$, there exists a unique homomorphism $\Phi : \mc P \to H$ such that $\Phi|_{U_i} = \varphi_i$. The group topology on $\mc P$ may be similarly defined by a universal property, as first proved by Graev \cite{graev1950}:

\begin{proposition}
There exists a unique topology $\tau$ on $\mc P$ (called the {\normalfont  free product topology}) such that

\begin{enumerate}
\item each inclusion $U_i \hookrightarrow \mc P$ is a homeomorphism onto its image, and
\item if $\varphi_i : U_i \to H$ are continuous group homomorphisms to a topological group $H$, then the unique extension $\Phi$ is continuous with respect to $\tau$.
\end{enumerate}
\end{proposition}

In the case when each $U_i$ is a Lie group (or more generally, a CW-complex), Ordman found a more constructive description of the topology \cite{ordman1974}. Indeed, the free product of Lie groups is covered by a disjoint union of {\it combinatorial cells}. 

\begin{definition}
\label{def:comb-cells}
Let $\mc P$ be the free products of groups $U_{\beta}$, where the $\beta$ range over some indexing set $\Delta$. A {\normalfont combinatorial pattern} in $\Delta$ is a finite sequence $\bar{\beta} = (\beta_1,\dots,\beta_n)$ such that $\beta_i \in \Delta$ for $i=1,\dots,n$. For each combinatorial pattern $\bar{\beta}$, there is an associated {\normalfont combinatorial cell} $C_{\bar{\beta}} = U_{\beta_1} \times \dots \times U_{\beta_n}$. If each $U_\beta$ is a topological group, $C_{\bar{\beta}}$ carries the product topology from the topologies on $U_{\beta_i}$. Notice that each $C_{\bar{\beta}}$ has a map $\pi_{\bar{\beta}} : C_{\bar{\beta}} \to \mc P$ given by $(u_1,\dots,u_N) \mapsto u_1^{(\beta_1)} * \dots * u_N^{(\beta_N)}$. Furthermore, if $C = \bigsqcup_{\bar{\beta}} C_{\bar{\beta}}$ and $\pi : C \to \mc P$ is defined by setting $\pi(x) = \pi_{\bar{\beta}}(x)$ when $x \in C_{\bar{\beta}}$, then $\pi$ is onto.
%For each finite word $\bar{i} = (i_1,\dots,i_N)$, let $C_{\bar{i}} = U_{i_1} \times \dots \times U_{i_N}$ be the combinatorial cell for $\bar{i}$, with its usual product topology. Notice that each $C_{\bar{i}}$ has a map $\pi_{\bar{i}} : C_{\bar{i}} \to \mc P$ given by $(u_1,\dots,u_N) \mapsto u_1^{(i_1)} * \dots * u_N^{(i_N)}$. Furthermore, if $C = \bigsqcup_{\bar{i}} C_{\bar{i}}$ and $\pi : C \to \mc P$ is defined by setting $\pi(x) = \pi_{\bar{i}}(x)$ when $x \in C_{\bar{i}}$, then $\pi$ is onto.
\end{definition}

\begin{lemma}[\cite{vinhageJMD2015} Proposition 4.2]
\label{lem:continuity-criterion}
If each $U_i$ is a Lie group, $\tau$ is the quotient topology on $\mc P$ induced by $\pi$. In particular, $f : \mc P \to Z$ is a continuous function to a topological space $Z$ if and only if its pullback $f \of \pi_{\bar{\beta}}$ to $C_{\bar{\beta}}$ is continuous for every combinatorial pattern $\bar{\beta}$.
\end{lemma}

\begin{corollary}
\label{cor:connected}
If each $U_i$ is a connected Lie group, $\mc P$ is path-connected and locally path-connected.
\end{corollary}

\begin{proof}
To see that $\mc P$ is path-connected, notice that if $\rho \in \mc P$, it must belong to a combinatorial cell $C_{\bar{\beta}}$ for some $\bar{\beta}$. Since each such cell is a product of finitely many connected Lie groups, it is path-connected, so there is a path from the identity to $\rho$. Similarly, any open neighborhood of $e \in \mc P$ contains a product of open neighborhoods in each cell. Since each cell is locally path-connected, one may consider the product of such neighborhoods in $\mc P$ to build arbitrarily small path-connected neighborhoods of $e$. Since $\mc P$ is a topological group, this is sufficient.
\end{proof}

Let $\mc P = \R^{*r}$ be the $r$-fold free product of $\R$. Given continuous flows $\eta^1,\dots,\eta^r$ on a space $X$, we may induce a continuous action $\tilde{\eta}$ of $\mc P$ on $X$ by setting:

\[ \tilde{\eta}(t_1^{(i_1)} * \dots t_n^{(i_n)})(x) = \eta^{i_1}_{t_1} \of \dots \of \eta^{i_n}_{t_n}(x) .\]

This can be observed to be an action of $\mc P$ immediately, and continuity can be checked with either the universal property (considering each $\eta^i$ as a continuous function from $\R$ to $\Homeo(X)$) or directly using the criterion of Lemma \ref{lem:continuity-criterion}. Given a word $t_1^{(i_1)} * \dots * t_N^{(i_N)}$ (which we often call a path as discussed above), we may associate a path in $X$ defined by:

\[ \gamma\left(\frac{s+k-1}{m}\right) = \eta^{i_k}_{st_k}(x_{k-1}), \; s \in [0,1], \; k = 1,\dots,N, \]

\noindent where $x_0$ is a base point and $x_k = \eta^{i_k}_{t_k}(x_{k-1})$. This gives more justification for calling each term $t_k^{(i_k)}$ a leg. The points $x_k$ are called the {\it break points} or {\it switches} of the path. 

\begin{definition}
\label{def:cycles}
Given an action $\mc P \curvearrowright X$, the {\normalfont cycle subgroup at $x$}, or just {\normalfont cycles at $x$}, are defined as $\mc C_x := \Stab(x)$.
\end{definition}

Notice that cycles are exactly the paths whose endpoints and start points coincide.

\begin{remark}
\label{rem:paths-no-action}
Given a collection of oriented one-dimensional foliations $\mc W^{\beta_1},\dots,W^{\beta_N}$ and some Riemannian metric on $X$, one can define flows $\eta^i_t$ according to unit speed flow along $\mc W^{\beta_i}$. Even if there is no distinguished Riemannian metric to use, we will often use the terms paths, legs, break points and cycles in the foliations $\mc W^{\beta_i}$ (especially in Parts III and V). The combinatorial patterns of such paths still make sense even without reference to a fixed parameterization.
\end{remark}
%\begin{corollary}
%\label{cor:lie-from-const}
%If $\mc P$ is the free product of Lie groups, and $\eta : \mc P \curvearrowright X$ is an action of $\mc P$ on a manifold $X$ such that $\Stab_\eta(x)$ is independent of $x$, then $\eta$ factors through the action of a Lie group $H$.
%\end{corollary}

%\begin{proof}
%Let $\mc C = \Stab_\eta(x)$ be the common stabilizer of every point in $X$, and $H = \mc P / \mc C$.
%\end{proof}

Given $\lambda = (\lambda_1,\dots,\lambda_r) \in \R^r$, let $\psi_\lambda$ denote the automorphism of $\mc P$ defined in the following way. Suppose that $t_1^{(i_1)} * \dots * t_m^{(i_m)} \in \mc P$ is an element of combinatorial length $m$. Then define:

\[ \psi_\lambda : t_1^{(i_1)} * t_2^{(i_2)} * \dots * t_m^{(i_m)} \mapsto  (\lambda_{i_1}t_1)^{(i_1)} * (\lambda_{i_2}t_2)^{(i_2)} * \dots * (\lambda_{i_m}t_m)^{(i_m)}. \]

\begin{definition}
\label{def:P-groups}
Fix a finite collection $\Delta = \set{\alpha_1,\dots,\alpha_r} \subset (\R^k)^*$, and let $a \in \R^k$. Define $\psi_a = \psi_{\big(e^{\alpha_1(a)},\dots,e^{\alpha_r(a)}\big)}$, and $\hat{\mc P} = \R^k \ltimes \mc P = \R^k \ltimes \R^{*r}$, with the semidirect product structure given by

\[ (a_1,\rho_1) \cdot (a_2,\rho_2) = (a_1a_2,\psi_{a_2}^{-1}(\rho_1)* \rho_2) \]

\noindent where the group operation of $\R^k$ is written multiplicatively.
\end{definition}

\begin{proposition}
\label{prop:integrality}
%Let $\psi_a$ denote the automorphism of $\mc P$ as defined in Definition \ref{def:P-groups}. Let $\mc C$ be any closed normal subgroup of $\mc P$, and
Let $\mc C$ be a closed, normal subgroup of $\mc P$, and $H = \mc P / \mc C$ be the corresponding topological group factor of $\mc P$. If $\psi_a(\mc C) = \mc C$ for all $a \in \R^k$, then $\psi_a$ descends to a continuous homomorphism $\bar{\psi}_a$ of $H$. Furthermore, if $H$ is a Lie group with Lie algebra $\mf h$, then

\begin{enumerate}
\item each generating copy of $\R$ in $\mc P$ projects to a one parameter subgroup whose generating element of $\mf h$ is an eigenvector for $d\bar{\psi}_a$ on $H$,
\item if the $i\tth$ and $j\tth$ copies of $\R$ do not commute, the Lie algebra commutator of their generators, $[X_i,X_j]$ in $H$ is an eigenspace of $\bar{\psi}_a$ with eigenvalue $e^{\alpha_i(a) + \alpha_j(a)}$, and
\item if $Y = [Z_1,[Z_2,\dots,[Z_N,Z_0]\dots]]$, with $Z_k = X_i$ or $X_j$ for every $k$, then $Y$ is an eigenvector of $\bar{\psi}_a$ with eigenvalue of the form $e^{u\alpha_i(a)+v\alpha_j(a)}$ with $u,v \in \Z_+$.
\end{enumerate}
\end{proposition}

\begin{proof}
If $\rho \in \mc P$ and $h = \rho \mc C$ is the corresponding element of the quotient group $H$, define $\bar{\psi}_a(h) = \psi_a(\rho \mc C) = \psi_a(\rho) \mc C$. Let $\pi : \mc P \to H$ denote the projection from $\mc P$ to $H$. For each $i$, let $f_i : \R \to \mc P$ denote the inclusion of $\R$ into the $i\tth$ copy of $\R$ generating $\mc P$. Then $\pi \of f_i : \R \to H$ is a one-parameter subgroup of $H$, and we denote its corresponding Lie algebra element by $X_i$. Observe that since $\bar{\psi}_a \of \pi \of f_i(t) = \pi \of\psi_a \of f_i =\pi \of f_i(e^{\alpha_i(a)}t)$,  and hence $d\bar{\psi}_a(X_i) = e^{\alpha_i(a)}X_i$, proving (1).

To see (2), observe that 

\[ d\bar{\psi}_a[X_i,X_j] = [d\bar{\psi}_aX_i,d\bar{\psi}_aX_j] = [e^{\alpha_i(a)}X_i,e^{\alpha_j(a)}X_j] = e^{\alpha_i(a)+\alpha_j(a)}[X_i,X_j] .\]

A similar argument shows (3).
\end{proof}

\subsection{Lie Criteria}

In this subsection, we recall a deep result for Lie criteria of topological groups. %The first was obtained by Gleason and Yamabe \cite[Proposition 6.0.11]{hilbert5-2014}:
%\bt[Gleason-Yamabe]  \label{thm:gleason-yamabe}
%Let $G$ be a locally compact group. Then there exists an open subgroup $G' \subset G$ such that, for any open neighborhood $U$ of the identity in $G '$ there exists a compact normal subgroup $K \subset U \subset G'$  such that $G'/K$ is isomorphic to a Lie group. Furthermore, if $G$ is connected, $G'=G$.
%\et
%Recall that a locally compact group $G$ has the {\em no small subgroups property} if for $G'$ as in Theorem \ref{thm:gleason-yamabe}, a small enough neighborhood $U \subset G'$ as above, $U$ does not contain any compact normal subgroup besides $\{1\}$.  Such a group $G$ then is automatically a Lie group, by  Theorem \ref{thm:gleason-yamabe}.  
 %The second criterion
 The main criterion we use was obtained by Gleason and Palais:

\begin{theorem}[Gleason-Palais]
\label{lem:gleason-palais}
If $G$ is a locally path-connected topological group which admits an injective continuous map from a neighborhood of $ e \in G$ into a finite-dimensional topological space, then $G$ is a Lie group.
\end{theorem}

Theorem \ref{lem:gleason-palais} has an immediate corollary for actions of the path group $\mc P =  \R^{*r}$:

\begin{corollary}
\label{cor:lie-from-const0}
If $\eta : \mc P \curvearrowright X$ is a continuous group action on a topological space $X$, and there exists $x_0 \in X$ such that $\mc C_{x_0}  \subset \mc C_x$ for every $x \in \mc P \cdot x_0$, then $\mc C_{x_0}$ is normal and the $\mc P$-action $\eta$ descends to an action of $\mc P / \mc C_{x_0}$ on $\mc P \cdot x_0$.  If there is an injective continuous map from $\mc P /\mc C$ to a finite-dimensional space $Y$, then $\mc P / \mc C_{x_0}$ is a Lie group.
\end{corollary}

\begin{proof}
We first show that $\mc C$ is normal. Let $\sigma \in \mc C$ and $\rho \in \mc P$. Then $\sigma\cdot(\rho \cdot x_0) = \rho \cdot x_0$, since $\sigma$ stabilizes every point of $X$. Therefore, $\rho^{-1}\sigma\rho \cdot x_0 = x_0$, and $\rho^{-1}\sigma \rho \in \mc C$, and $\mc C$ is a closed normal subgroup. By Corollary \ref{cor:connected}, $\mc P$, and hence all of its factors, are locally path-connected. Therefore, by Theorem \ref{lem:gleason-palais}, $\mc P / \mc C$ is a Lie group if it admits an injective continuous map to a finite-dimensional space.
\end{proof}

\begin{definition}
Let $X$ be a metric space and $U_1,\dots,U_m$ be Lie groups with faithful continuous actions $U_i \curvearrowright X$. We say that these actions {\normalfont generate the action of a Lie group on $X$} or that $U_1 * \dots * U_m$ {\normalfont factors through the action of a Lie group on $X$} if

\begin{enumerate}
\item there exists a Lie group $G$ with a faithful continuous action $G \curvearrowright X$,
\item there exist continuous embeddings $f_i : U_i \hookrightarrow G$ such that $f_i(u) \cdot x = u \cdot x$ for all $u \in U_i$ and $x \in X$, and
\item $\bigcup_{i = 1}^m f_i(U_i)$ generates $G$.
\end{enumerate}
\end{definition}

The following is an immediate reformulation of Corollary \ref{cor:lie-from-const0} when each $U_i \cong \R$ and $X$ is finite-dimensional:

\begin{corollary}
\label{cor:lie-from-const}
Let $\eta^{\alpha_1}, \dots, \eta^{\alpha_m}$ be continuous, fixed-point free flows on a finite-dimensional space $X$ and $\mc P \curvearrowright X$ be the corresponding action of the free product. If $\mc C_{x_0} \subset \mc C_x$ for all $x \in \mc P \cdot x_0$, the flows $\eta^{\alpha_i}$ generate the action of a Lie group on $\mc P \cdot x_0$. 
\end{corollary}

\subsection{H\"older transitivity of generating subgroups}
\label{app:holder-transitivity}

In this section, we prove a folklore theorem about generating subgroups of a Lie group $G$. This property was used in \cite{DamjanovicKatok2005}, and follows from arguments of transitivity of foliations. A similar statement also holds for a family of foliations such that the sum of their distributions is everywhere transverse and totally non-integrable.

\begin{definition}
A family $\mc F_1,\dots,\mc F_n$ of foliations on a manifold $X$ with metric $d_X$ are called {\it locally $\theta$-H\"older transitive} if there exists an $\epsilon > 0$ and $C > 0$ such that if $x,y \in X$ satisfy $d(y,x) < \epsilon$, there are points $x=x_0,x_1,\dots,x_k = y$ such that:

\begin{enumerate}
\item $x_{i+1} \in \mc F_{m_i}(x_i)$ for some $1 \le m_i \le n$
\item $\sum d(x_i,x_{i+1}) < C \cdot d(x,y)^\theta$.
\end{enumerate}
\end{definition}

\begin{lemma}
\label{lem:coset-theta-holder}
If $U_1,\dots,U_n \subset G$ are generating subgroups of a Lie group $G$, then the coset foliations of $U_i$ are $\theta$-H\"older transitive for some $\theta > 0$.
\end{lemma}

\begin{proof}
Pick vectors $V_1,\dots,V_m \in \Lie(G)$ such that each $V_i \in \Lie(U_{k_i})$ for some $k_i$ and the $V_i$ generate $\Lie(G)$ as a Lie algebra. Then there are finitely many $W_1,\dots,W_\ell$ such that each $W_j$ is an iterated bracket of the elements of $\set{V_i}$ and such that $\set{V_i,W_j}$ generate $\Lie(G)$ as a vector space. If $W_j = [V_{i_0},[V_{i_1},\dots,[V_{i_{q-1}},V_{i_q}],\dots,]]$, let 

\[\varphi_j(t) = [\exp(t^{1/2}V_{i_0}),[\exp(t^{1/4}V_{i_1}),\dots,[\exp(t^{1/2^{q}}V_{i_{q-1}}),\exp(t^{1/2^{q}}V_{i_q})],\dots,]].\]

The brackets in the definition of $\varphi_j$ are commutators in the Lie group, not the Lie algebra. Notice that $\varphi_j(t)$ satisfies $\varphi_j'(0) = W_j$. Therefore, the map:

\[ (s_1,\dots,s_m,t_1,\dots,t_\ell) \mapsto \exp(s_1V_1)\dots\exp(s_mV_m)\varphi_1(t_1)\dots \varphi_n(t_\ell) \]

\noindent has full-rank derivative at 0 and is onto a neighborhood of $e \in G$. If the number of required brackets to express each $Q_j$ is $\le q$, this gives $2^{-q}$-H\"older transitivity.
\end{proof}

Lemma \ref{lem:coset-theta-holder} immediately implies the following useful result. We say that an action $G \curvearrowright X$ is locally Lipschitz (H\"older) if there exists a neighborhood $U \subset G$ of the identity such that the action map $U \times X \to X$ is Lipschitz (H\"older).

\begin{lemma}
\label{lem:lip-to-holder}
Suppose that a Lie group $G$ is generated by subgroups $U_1,\dots,U_n$, and that $\eta : G \curvearrowright X$ is an action of $G$ on a compact metric space. If the restriction of $\eta$ to each subgroup $U_i$ is locally Lipschitz, then $\eta$ is a locally H\"older action.
\end{lemma}

\section{Dynamical foundations of hyperbolic abelian group actions}
\label{subsec:dynamical-prelim}

\subsection{Normal Hyperbolicity}
\label{sec:normal-hyp}

 Let us first recall the definition of normally hyperbolic transformations, which was used in Definition \ref{Anosov action} of an Anosov action. If $X$ is a compact manifold without boundary, $G \curvearrowright X$ is a group action and $\mc F$ is a foliation of $X$, we say that $\mc F$ is {\it invariant under $G$} if for every $x \in X$, $\mc F(g\cdot x) = g \cdot \mc F(x)$. We say that the foliation is invariant under a transformation $f$ it is invariant under the group generated by $f$.

\begin{definition}			 \label{def:normally hyperbolic}
Let $X$ be a compact manifold without boundary, equipped with a Riemannian metric $\langle\:, \: \rangle$. 
Let $\mc{F}$ be a continuous foliation of $X$ with $C^1$ leaves.  We call a $C^1$ diffeomorphism $a: X \to X$ {\em normally hyperbolic} with respect to $\mc{F}$ if $\mc{F}$ is an $a$-invariant foliation and there exist constants $C>0$, and $\lambda_1 > \lambda_2 >0$ and a splitting of the tangent bundle into continuous subbundles $T_xM = E^s _a(x) \oplus T_x \mc{F}(x) \oplus E^u _a(x)$ such that for all $n \geq 0$ and $v^s \in E^s _a(x)$, $v^u \in E^u _a(x)$ and $v^0 \in T\mc{F}(x)$:
\begin{eqnarray} 
\label{eq:NH1}\| da ^n_x (v^s)  \| \leq C e^{- \lambda_1 n} \|v^s\|, \phantom{\mbox{ and}}   \\
  \| da ^{-n}_x (v^u)  \| \leq C e^{- \lambda_1 n} \|v^u\|, \mbox{ and}  \\
\label{eq:NH3}  \| da ^{\pm n}_x (v^0)  \| \geq C e^{- \lambda_2 n} \|v^0\|.\phantom{\mbox{ and}}
\end{eqnarray}
\end{definition}

\begin{remark}
In fact, continuity of the bundles $E^s_a$ and $E^u_a$ follows from \eqref{eq:NH1}-\eqref{eq:NH3}, so one may drop this assumption from the definition. The proof is identical to the hyperbolic case, see for instance \cite[Lemma 6.2.15]{KatokHasselBook} or \cite[Proposition 5.2.1]{Brin-Stuck}.
\end{remark}

For the trivial foliation $\mc{F}(x) =\{x\}$, a normally hyperbolic diffeomorphism w.r.t. $\mc{F}$ is just an Anosov diffeomorphism.  
We refer to \cite{HPS1970,HertzHertzUres2007} for the standard facts about normally hyperbolic transformations. Most importantly, we obtain the existence of foliations $W^s_a$ and $W^u_a$ of $X$ tangent to  $E^s_a$ and $E^u_a$, respectively. The foliations are H\"older with $C^r$ leaves if $a$ is $C^r$. We prove a technical lemma:

\begin{lemma}
\label{lem:coarse-refinement}
Suppose that $f$ acts normally hyperbolically with respect to a foliation $\mc F$ of a compact manifold $X$ and that $V_i \subset TX$ are continuous $f$-invariant subbundles such that $TX = T \mc F \oplus \bigoplus_i V_i$. Then $E^* = \bigoplus_i (V_i \cap E^*)$ for $* = s,u$, and $TX = T\mc F \oplus \bigoplus_i \big((V_i \cap E^s) \oplus  (V_i \cap E^u)\big)$.
\end{lemma}

\begin{remark}
Each intersection $V_i \cap E^u$ may be trivial, and this will happen in many examples. However, the stable and unstable bundles will be refined by taking their intersections with any $V_i$ which is nontrivial, which is the main tool in producing coarse Lyapunov foliations.
\end{remark}

\begin{proof}
%For each $x \in X$, notice that the splitting $T_xX =E^s(x) \oplus T \mc F(x) \oplus E^u(x)$ induces projections $\pi_x^{u/s} : T_xX \to E^{u/s}(x)$ and $p_x : T_xX \to T\mc F$, each of which is onto. Furthermore, since $T\mc F(x) \subset \ker \pi_x^u$, the restriction of $\pi_x^u$ to $\bigoplus_i V_i(x)$ is also onto since it is complementary.
We claim that if $\pi_x^i : T_xX \to V_i$ and $p : T_xX \to T\mc F(x)$ are the projections induced by the splitting $T_xX = T\mc F(x) \oplus \bigoplus_i V_i(x)$, %then $\pi_x^i \pi_x^u = \pi_x^u \pi_x^i$. Indeed, suppose that $v$ is decomposed according to the splitting $T\mc F(x) \oplus \bigoplus_i V_i(x)$, so that $v = w + \sum_i v_i$, where $w \in T\mc F(x)$ and $v_i \in V_i(x)$. Then $\pi_x^i(v) = v_i$, and $\pi_x^u(\pi_x^i(v)) = \pi_x^i(v_i)$. In the other composition, since $(f^n)_*v = (f^n)_*w + \sum_i (f^n)_*v_i$ must converge to 0 as $n \to \infty$, and the subspaces remain uniformly transverse by assumption, $\pi_x^u(v) = \pi_x^u(w) + \sum_i \pi_x^u(v_i)$. 
then $\pi_x^i(E^s) \subset E^s$ and $p(E^s) = 0$. Indeed, if $v \in E^s(x)$, then $v = w + \sum_i v_i$, where $w \in T\mc F(x)$ and $v_i \in V_i(x)$. Since $(f^n)_*v = (f^n)_*w + \sum_i (f^n)_*v_i$ must converge to 0 as $n \to \infty$, and the subspaces remain uniformly transverse by compactness of $X$ and continuity of the splitting, we conclude that $\pi_x^i(E^s) \subset E^s$ and $p(E^s) = 0$. Symmetric arguments show the same properties for $E^u$, and this immediately implies that $E^s \oplus E^u = \bigoplus_i V_i$. Now, if $v \in E^u(x)$, then according to the decomposition $T_xX = T\mc F(x) \oplus \bigoplus_i V_i(x)$,

\[ v = p_x(v) + \sum_i \pi_x^i(v) = p_x(v) + \sum_i \pi_x^i(v) = \sum_i \pi_x^i(v). \]

Since each $\pi_x^i(v) \in V_i$ by definition of $\pi_x^i$, and $\pi_x^i(v) \in E^u(x)$ by the remarks above, $\pi_x^i(v) \in E^u(x) \cap V_i(x)$, and $E^u(x) = \sum E^u(x) \cap V_i(x)$. The sum is clearly direct, since $V_i \cap V_j = \set{0}$ for every $i \not= j$. A symmetric argument follows for $E^s(x)$, and the lemma follows.
\end{proof}

 The following has been claimed for Anosov $\R^k$ actions. We provide a careful proof for completeness.  Given  a collection of Anosov elements $\set{a_1,\dots,a_n}$, define their {\it common stable manifold} to be $W^s_{a_1,\dots,a_n}(x)$, the path component containing $x$ of the intersection $\bigcap_i W^s_{a_i}(x)$.

\begin{lemma}
\label{lem:coarse-lyapunov}
If $r \ge 1$, $\R^k \curvearrowright X$ is a $C^r$ Anosov action, and $\set{a_i}$ be a collection of Anosov elements. Then the common stable manifolds of $\set{a_i}$ define a H\"older foliation with  $C^r$ leaves.
\end{lemma}

 As in the introduction, we define a {\it coarse Lyapunov foliation} to be a common stable foliation of smallest possible dimension. As an immediate corollary, we obtain:

\begin{corollary}
\label{cor:coarse-lyapunov}
If $r \ge 1$, $\R^k \curvearrowright X$ is a $C^r$ Anosov action, each coarse Lyapunov foliation is a H\"older foliation with $C^r$ leaves. The coarse Lyapunov foliations are all uniformly transverse, and $T_xX$ is the direct sum of their tangent bundles, and the tangent bundle to the $\R^k$-orbit foliations.
\end{corollary}

\begin{proof}[Proof of Lemma \ref{lem:coarse-lyapunov}]
We prove the result by induction on the cardinality of $\set{a_i}$. Let $\mc O$ denote the $\R^k$-orbit foliation, so that the Anosov elements of the action are those which act normally hyperbolically with respect to $\mc O$. Choose an Anosov element $a_1$, and let $V_1^{(1)}(x) = E^s_{a_1}(x)$ and $V_2^{(1)}(x) = E^u_{a_1}(x)$, so that $T_xX = T\mc O(x) \oplus E^s_{a_1}(x) \oplus E^u_{a_1}(x) = T\mc O(x) \oplus V_1^{(1)}(x) \oplus V_2^{(1)}(x)$.

We will proceed by induction on $q \in \N$ based on the following hypotheses:

\begin{itemize}
\item There is a chosen collection $S_q = \set{a_1,\dots,a_q} \subset \R^k$ of Anosov elements.
\item There is a splitting $T_xX = T\mc O(x) \oplus \bigoplus_{i=1}^{n(q)} V_i^{(q)}(x)$ such that $V_i^{(q)}(x) = \bigcap_j E^*_{a_{c_j}}$ for some subcollection $\set{a_{c_j}} \subset S_q$, {   (here, $* =$ $s$ or $u$ is allowed to depend on $i$ and $j$ and the choice of indices $c_j$ depends on $i$)}.
\item There is a H\"older foliation $W_i^{(q)}$ with $C^r$ leaves such that $T_xW_i^{(q)}(x) = V_i^{(q)}(x)$ for every $i = 1,\dots,n(q)$.
\end{itemize}

We have established the base of the induction. The induction terminates when, for every Anosov element $a\in \R^k$, $E^s_a(x)$ and $E^u_a(x)$ can be written as a sum of some collection of subbundles $V_i(x)$. If the process has not terminated after $q$ steps, there exists $a_{q+1} \in \R^k$ such $E^s(x)$ cannot be written as a sum of some subcollection of bundles $V_i^{(q)}$. Then according to Lemma \ref{lem:coarse-refinement}, one may refine the splitting of $T_xX$ into:

\[ T_xX = T \mc O(x) \oplus \bigoplus_{i=1}^{n(q)} \Big(\big(V_i^{(q)}(x) \cap E^s_{a_{q+1}}(x)\big) \oplus \big(V_i^{(q)}(x) \cap E^u_{a_{q+1}}(x)\big)\Big). \]

Let $V_1^{(q+1)},\dots,V_{n(q+1)}^{(q+1)}$ enumerate the bundles $V_i^{(q)}(x) \cap E^s_{a_{q+1}}$ and $V_i^{(q)}(x) \cap E^u_{a_{q+1}}$ which are nontrivial. Since by induction, each $V_i^{(q)}$ is an intersection of stable  distributions, so is $V_i^{(q+1)}$. Finally, we prove the foliation property. Notice that the leaves $W_i^{(q)}(x)$ and $W^s_{a_{q+1}}(x)$ are both H\"older foliations with $C^r$ leaves, and that the intersection of their bundles is exactly one of the $V_j^{(q+1)}(x)$. We claim that if $W_j^{(q+1)}(x)$ is the path component containing $x$ of the intersection $W_i^{(q)}(x) \cap W^s_{a_{q+1}}(x)$, then $W_j^{(q+1)}$ defines a foliation of $X$ with the same properties. Indeed, this follows from Theorem 2.6 of \cite{bonatti-gomez-martinez20} for $\R^k$-actions (one may still directly apply the result by suspending $\R^k \times \Z^l$-actions). We give a brief sketch of a proof for completeness. The main idea is to build the local manifold at $x$ by applying the Hadamard-Perron Theorem (Theorem 5.6.1 of \cite{Brin-Stuck} or Theorem 6.2.8 of \cite{KatokHasselBook}) to the sequence of maps $a_{q+1} : W_i^{(q)}({a_{q+1}}^n\cdot x) \to W_i^{(q)}({a_{q+1}}^{n+1}\cdot x)$. This yields $C^r$ local $W_i^{(q)}(x)$-stable manifolds, which we denote by $W^{(q+1)}_{j,\operatorname{loc}}(x)$, the local leaf at $x$. Note that since they are tangent to H\"older distributions, they vary H\"older continuously.

 To get the global leaves, suppose that $y$ is in the path component of $W^s_{a_{q+1}}(x) \cap W^{(q)}_i(x)$ containing $x$, so that there exists $\gamma : [0,1] \to W^s_{a_{q+1}}(x) \cap W^{(q)}_i(x)$ such that $\gamma(0) = x$ and $\gamma(1) = y$. Then since $a_{q+1}$ contracts $W^s_{a_{q+1}}$, there exists $n$ such that the image of $a_{q+1}^n \of \gamma$ is contained in $B(a_{q+1}^n\cdot x, \ve)$. In particuarly, it must be in the local leaves we have constructed. Therefore the leaf of $W^{(q+1)}_j(x)$ is given by $\bigcup_{n \ge 1} a_{q+1}^{-n} W^{(q+1)}_{j,\operatorname{loc}}(a_{q+1}^n\cdot x)$.
 %Indeed, such a leaf is a $C^\infty$ manifold, since the leaf $W^s_{a_{q+1}}(x)$ can be (locally) written as $\varphi^{-1}(0)$ for some $C^\infty$ function $\varphi : X \to \R^b$, with $\ker \varphi_*(y) = E^s_{a_{q+1}}(y)$ for every $y \in W^s_{a_{q+1}}$. Therefore, $\varphi|_{W_i^{(q)}(x)}$ is a $C^\infty$ function of constant rank.
%First, notice that if $a$ is any Anosov element, and $b \in \R^k$, then $E^s_a$ and $E^u_a$ are both $b$-invariant subbundles, and uniformly transverse to the $\R^k$-orbit. Therefore, so is $\bigcap_i E^s_{a_i}$ for any collection $\set{a_i} \subset \R^k$ of Anosov elements. One may therefore continually refine the splitting of $TX$ into smaller subbundles, until it is no longer possible. Such subbundles are the candidates for the tangent bundles to the leaves $\mc W(x)$.
%Each $E^s_{a_i}$ integrates to a foliation $W^s_{a_i}$ which is H\"older with smooth leaves. Assuming no two are equal, and that the collection $\set{a_i}$ is as small as possible, $W^s_{a_i}$ is transverse to $\bigcap_{j=1}^{i-1}W^s_{a_j}$. The transversality is uniform by compactness. Therefore, each leaf $\mc W(x)$ is a smooth manifold, and since each $W^s_{a_i}$ is a H\"older foliation, the leaves $\mc W(x)$ vary H\"older continuously in $x$.
\end{proof}

%\begin{definition}
%An action $\R^k \curvearrowright X$ is called {\it Anosov} if there exists an element $a \in \R^k$ which acts normally hyperbolically 
%\end{definition}

\subsection{Lyapunov Functionals and Coarse Lyapunov Spaces}
\label{subsec:lyapunov-prelim}

Let $\R^k \curvearrowright X$ be a $C^1$ action. In the study of $\R^k$ actions, one often has a preferred invariant measure, such as an invariant volume. While we make no such assumption, invariant measures for actions of $\R^k$ always exist, and we discuss their properties here. %For a more careful overview of the discussion, see \cite{Brown:2016aa}.  
If $\R^k$ preserves an ergodic invariant measure $\mu$, then  there are linear functionals $\lambda: \R^k \to \R$ and  a measurable $\R^k$-invariant splitting of the tangent bundle $TM = \oplus E^{\lambda}$  such that for all $0 \neq v \in E^{\lambda}$  and $a \in \R^k$, the Lyapunov exponent of $v$ is $\lambda (a)$.  These objects may depend on the measure $\mu$, and exist only on a set of full $\mu$-measure.
We call this the {\em Oseledets} or {\it Lyapunov splitting} of $TM$ for $\R^k$, and each $\lambda$ a {\em Lyapunov functional} or simply {\em weight} of the action. Each $E^\lambda$ is called a {\it Lyapunov distribution}. We let $\Delta$ denote the collection of Lyapunov functionals. The Lyapunov splitting is a refinement of the Oseledets splitting for any single $a \in \R^k$. For each $\lambda \in \Delta$, we let $\bar{E}^\lambda = \bigoplus_{t > 0} E^{t\lambda}$ be the {\it coarse Lyapunov distributions}, which in general still depend on the measure $\mu$ and exist only on a set of full $\mu$-measure. We will see that they coincide with the coarse Lyapunov distributions defined after Lemma \ref{lem:coarse-lyapunov} when the action is totally Anosov. We refer to \cite{Brown:2016aa} for an extensive discussion of all these topics. 

 The dependence of the Lyapunov functionals and distributions on the measure $\mu$ is typical for Anosov flows and diffeomorphisms, hence also for products of such flows. In the totally  Anosov setting, we also have some additional structures, which do not assume any properties of invariant measures. Recall that we defined a coarse Lyapunov foliation of a totally Anosov action to be a H\"older foliation with smooth leaves whose leaves are locally defined as intersections of local stable manifolds, and has associated H\"older distributions $T\mc W$ on $M$, which do {\it not} depend on the choice of measure $\mu$ (Corollary \ref{cor:coarse-lyapunov}). Therefore, the splitting of the tangent bundle according into coarse Lyapunov distributions with respect to any invariant measure $\mu$ must coincide with the splitting into the coarse Lyapunov distributions coming from Corollary \ref{cor:coarse-lyapunov}. Indeed, when the set of Anosov elements is dense, any pair of linearly independent functionals can be made to have the opposite sign with an Anosov element. We assume that we have a fixed Riemannian metric $\norm{\cdot}$ on $M$.

\begin{lemma}
\label{lem:lyap-hyp}
If $\R^k \curvearrowright X$ is a totally Anosov action, then for each coarse Lyapunov foliation $\mc W$, there exists a unique hyperplane $H \subset \R^k$ such that:

\begin{enumerate}
\item  for every $x \in X$, $H = \set{ a \in \R^k : \lim_{t\to\infty} \frac{1}{t}\log \norm{(ta)_*|_{T\mc W}(x)} = 0}$, and
\item if $\mu$ is an $\R^k$-invariant measure, and $\lambda$ is the Lyapunov functional associated to $\mc W$, then $\ker \lambda = H$.
\end{enumerate}
\end{lemma}

The hyperplane $H$ is called the {\it Lyapunov hyperplane} corresponding to $\mc W$. %It is important to note that Lemma \ref{lem:lyap-hyp} does not require that $\R^k \curvearrowright X$ has no rank one factors, it is a feature of {\it totally} Anosov actions.
 An example of the failure of Lemma \ref{lem:lyap-hyp} for Cartan, but not totally Cartan actions, can be found in Section \ref{app:anosov-not-totally}.

\begin{proof}
Notice that if $a \in \R^k$ is Anosov, $\lim_{t \to\infty} \frac{1}{t}\log \norm{d(ta)|_{T\mc W}} \not= 0$, since $T\mc W \subset TW^s$ or $TW^u$. Let $U$ denote the set of elements for which this limit is positive and $V$ denote the set of elements for which the limit is negative. Then $U \cup V$ contains the set of Anosov elements, and since such elements are assumed to be dense ($\R^k \curvearrowright X$ is totally Anosov), $\overline{U \cup V} = \R^k$. The set $H = \R^k \setminus (U \cup V)$ separates $\R^k$ into two components. Since $U$ and $V$ are convex, it follows that $H$ is a hyperplane (by the hyperplane separation theorem, see \cite[Section 14.5]{royden10}). 
%{\cb \cite[Theorem III.1.2]{barvinok-convexity-book}}. 
Assertion (2) follows immediately.
\end{proof}

\begin{remark}
%The hyperplane $H$ appearing in Lemma \ref{lem:lyap-hyp} is often called a {\em Weyl chamber wall}
After removing all  Lyapunov hyperplanes from  $\R ^k $, we call the connected components of the remainder the {\em Weyl chambers} of the action  (this terminology comes from the case of Weyl chamber flows on semisimple Lie groups, see Section \ref{app:weyl-ch-flows}).  For actions which are Anosov but not totally Anosov, we still call the connected components of the set of Anosov elements Weyl chambers, even though their complement may not be a union of hyperplanes. %It is easy to see that  Weyl chambers and their walls are independent of the measure when the action is totally Anosov.  In fact, having an Anosov element in every Weyl chamber is sufficient.   Then the whole 
In general, each Weyl chamber is a convex set of Anosov elements, and the Anosov elements are the union of the Weyl chambers.
\end{remark}

\begin{definition}
\label{def:regular}
Let $\Delta$ denote a set of functionals such that for every coarse Lyapunov foliation $\mc W$, there exists a unique functional $\alpha \in \Delta$ such that $H = \ker \alpha$ and any $a$ for which $\alpha(a) > 0$ expands $\mc W$. An element $a \in \R^k$ will be called {\normalfont regular} if %$\alpha(a) \not= \beta(a)$ for all $\alpha \not= \beta \in \Delta$ and 
$\alpha(a) \not= 0$ for all $\alpha \in \Delta$.
\end{definition}

It is clear that the set of regular elements of $\R^k$ is open and dense.

\begin{remark}
\label{rem:lyap-coefficients1}
The functionals $\alpha$ are defined only up to positive scalar multiple in this context. Therefore, without loss of generality, we may assume that if there are Lyapunov foliations whose corresponding half-spaces are opposites, that their corresponding functionals are exact opposites, $\alpha$ and $-\alpha$ (this simplifies notation). In Part IV, the precise exponents will exist and we will need them in our analysis, so we will not be able to make this assumption, see Remark \ref{rem:lyap-coefficients2}.
\end{remark}

\subsection{Complementary distributions to slow exponents}
\label{sec:slow-foliations}

Let $E \subset TX$ be a subbundle. We say that $E$ is {\it integrable} if there exists a unique H\"older foliation $\mc F$ of $X$ with smooth leaves such that $T\mc F = E$. {  Notice that if $E$ is integrable and invariant under a totally Anosov action $\R^k \curvearrowright X$, then so is the corresponding foliation by uniqueness.} %We say that $E$ is {\it globally integrable} if its leaves are compact. 
 Fix a Lyapunov hyperplane $H \subset \R^k$, with corresponding foliations $W^{\pm \alpha}$. % Let $W^{-\alpha}$ denote the coarse Lyapunov foliation whose hyperplane coincides with $\alpha$, but whose contracting elements are the opposite half-space (if this foliation exists). 
Let $E^H$ be the subbundle of $TX$ which is the sum of the tangent bundle to the $H$-orbits, and the tangent bundles $E^\beta$ to the foliations $W^\beta$, $\beta \not= \pm \alpha$. That is,

\[ E^H_x = T_x(H \cdot x) \oplus \bigoplus_{\beta\not=\pm \alpha} E^\beta_x. \]

If $a \in \R^k$ is regular, let $E^+_a$ denote the unstable bundle of $a$ and $\Delta^+(a) := \set{\beta \in \Delta : W^\beta \subset W^u_a}$, so that $E^+_a = T\mc W^u_a = \bigoplus_{\beta \in \Delta^+(a)} E^\beta$. Fix a finite collection of weights $\Omega \subset \Delta^+(a)$ such that $\bigoplus_{\beta \in \Omega} E^{\beta}$ is integrable to a foliation $\mc F$ {  (we call such a collection of weights integrable)}. 

\begin{definition}
\label{def:slow-foliation}
If $\alpha \in \Omega$, let $\widehat{E^\alpha} := \widehat{E^\alpha}(\Omega) = \bigoplus_{\beta \in {  \Omega}\setminus \alpha} E^\beta$ be the dynamically defined complementary bundle. When $\Omega$ is fixed, we will suppress the dependence on $\Omega$. Notice that while each $E^\beta$ is integrable, $\widehat{E^\alpha}$ may, and often does, fail to be integrable. We say that $W^\alpha$ is the {\normalfont slow foliation for $a$ in $\mc F$} if $\alpha \in \Omega$ and there exists $C > 0$, $\lambda < 1$ such that 

\begin{equation}
\label{eq:dominated}
\dfrac{\norm{(ta)_*|_{E^\alpha}}}{m((ta)_*|_{\widehat{E^\alpha}})} \le C\lambda^t
\end{equation}

 for all $t \ge 0$, where $m(A) = \norm{A^{-1}}^{-1}$ is the conorm of $A$ . 
\end{definition}

\begin{remark}
\label{rem:fast-distribution}
The notion of a slow foliation makes sense whenever there are two invariant complementary distributions $E_1$ and $E_2$ such that $T\mc F = E_1 \oplus E_2$, and the dynamics of $a$ on $E_1$ is strictly slower that the dynamics on $E_2$ as in \eqref{eq:dominated}. Any splitting $T\mc F = E_1 \oplus E_2$ which satisfies \eqref{eq:dominated} with $E_1$ and $E_2$ playing the roles of $E^\alpha$ and $\widehat{E}^\alpha$ is called a dominated splitting (see, eg, \cite[p. 905]{psw04}).
\end{remark}

Notice that a priori, there may not exist a slow coarse Lyapunov foliation at all (this is the case when the action is Anosov, but not totally Anosov, as in Section \ref{app:anosov-not-totally}, or the example constructed in \cite{vinhage22}). The following lemma shows that they are guaranteed to exist for elements $a$ sufficiently close to a Lyapunov hyperplane. The smoothness and integrability conclusions of the complementary ``fast'' foliation are true generally for a dominated splitting.

\begin{lemma}
\label{lem:slow-foliations}
Let $\R^k \curvearrowright X$ be a totally Anosov action. Fix any {  integrable} collection $\Omega \subset \Delta^+(a)$ for some Anosov $a \in \R^k$ with { foliation $\mc F$ associated to $\Omega$} as described above. If $W^\alpha$ is the slow foliation of $a$ in $\mc F$, then $\widehat{E^\alpha}$ is integrable to a foliation $\widehat{\mc W^\alpha}$. Furthermore, if $y \in \mc F(x)$, then $W^\alpha(x)$ and $\widehat{\mc W^\alpha}(y)$ has a unique intersection point, varying continuously with $y$. Finally, given any pair of linearly independent weights $\alpha,\beta \in \Delta$ there exists a regular $a \in \R^k$ such that $\alpha,\beta \in \Delta^+(a)$ and $W^\alpha$ is the slow foliation {  for $a$} in any subfoliation $\mc F \subset W^s_a$ as described above such that $E^\alpha,E^\beta \subset T \mc F$.
\end{lemma}

\begin{proof}
We first prove existence of the element $a$. Begin by choosing $a_0$ such that $a_0$ belongs to the Lyapunov hyperplane for $\alpha$, and uniformly expands $\beta$. Let $a$ be a small perturbation of $a_0$ such that $a$ now expands $W^\alpha$. By Lemma \ref{lem:lyap-hyp}(1), if the perturbation is small enough, $W^\alpha$ is the slow foliation for $a$.

The existence and properties of the foliation $\widehat{\mc W^\alpha}$ come from the standard works in partially hyperbolic dynamics. The definition of slow foliation immediately implies that the splitting ${  T\mc F} = E^\alpha \oplus \widehat{E^\alpha}$ is a dominated splitting for the diffeomorphism $a$. The existence of the foliation $\widehat{\mc W^\alpha}$ can then be deduced from \cite{psw04}.

The unique intersection property follows locally from the transversality of the foliations. By applying a sufficiently large multiple of $-a \in \R^k$, any intersection can be made local, so there is only one global intersection.
\end{proof}

\subsection{Closing lemmas}

The main result of this section will be familiar to those fluent in Anosov flows. We prove a higher-rank version of the Anosov closing lemma, which has been used in \cite{KS1} without proof. A proof appeared in \cite[Lemma 4.5 and Theorem 4.8]{MR808219} in a more geometric setting.  We also prove this more generally by closing pseudo-orbits, which is used in the proof of Theorem \ref{thm:spectral}.

\begin{definition}
\label{def:periodic}
If $\R^k \curvearrowright X$ is a locally free action, we say that $x \in X$ is $\R^k$-periodic if $\Stab_{\R^k}(x)$ is a lattice in $\R^k$. We call $\Stab_{\R^k}(x)$ the {\it period} of $x$. If $a \in \R^k$, we say that a sequence of points $x_i$ and times $t_i$ are an $(a,\ve)$-pseudo-orbit if $d((t_ia)\cdot x_i, x_{i+1}) < \ve$.
\end{definition}

The Anosov closing lemma will require one additional assumption compared to the rank one setting. Indeed, consider the case of a product of two Anosov flows: $Y = X_1 \times X_2$. Then there are two Lyapunov exponents, which are dual to the standard basis, call them $\alpha$ and $\beta$. If $(p,x) \in Y$ is a point such that $p$ is periodic and $x$ is not, then $\Stab_{\R^2}(p,x) \cong \Z (t_0,0)$ for some $t_0 \in \R$. That is, the returns to $(p,x)$ occur along the Weyl chamber wall $\ker \beta$. While the first factor of $Y$ can be made periodic, we have no control over the behavior of the second coordinate. We show that returning near the Weyl chamber walls is the only potential obstacle to finding a nearby $\R^2$-periodic orbit.

\begin{theorem}[Anosov Closing Lemma]
\label{thm:anosov-closing}
Let $\R^k \curvearrowright X$ be an Anosov action, $\lambda > 0$, and $a \in \R^k$ be an Anosov element satisfying:

\begin{equation}
\label{eq:uniform-expansion}
\norm{da|_{E^u_a}(x)} > e^{\lambda\norm{a}} \qquad \mbox{and} \qquad \norm{da|_{E^s_a}(x)} < e^{-\lambda\norm{a}}.
\end{equation}

 Let $C \subset \R^k$ be the open cone 

\begin{multline}
 C = \set{ b \in \R^k : nb \mbox{ satisfies } \eqref{eq:uniform-expansion} \mbox{ for some }n \in \N} \\
\cap \set{b \in \R^k : d(b,t{ a}) < t/(10\lambda) \mbox{ for some }t \in \R}
\end{multline}

\noindent so that $\overline{C} \setminus \set{0}$ consists of Anosov elements. Then for every $\delta > 0$, there exists $\ve > 0$, $T > 0$ such that if $\set{(x_i,t_i)}_{i=1}^n$ is an $(a,\ve)$-pseudo-orbit, $t_i \ge T$ for all such $t_i$, and $d(x_1,(t_na)\cdot x_n) < \ve$, then there exists $x' \in X$ such that $x'$ is $\R^k$-periodic and for every $i =1,\dots,n$, there exists $b_i \in \R^k$ such that $d(b_i \cdot x',x_i) < \delta$ and $\norm{(b_i - b_{i-1}) - t_ia} < 2\delta$ (where we set $b_0 = 0$). If there is only one orbit segment of the pseudo-orbit (ie, $n = 1$), then $d(b\cdot x, b \cdot x') < \delta$ for all $b \in C \cap B(0,t/2)$.
%$x \in X$ satisfies $d(a_0^t \cdot x,x) < \ve$ for some $t \ge 1$, then there exists $x' \in X$ such that $x'$ is $\R^k$-periodic, and $d(b\cdot x, b \cdot x') < K\ve$ for all $b \in C \cap B(0,t/2)$.
\end{theorem}

\begin{proof}
The proof is very similar to the case of Anosov flows (see, eg, \cite[Theorem 6.4.15]{KatokHasselBook}), with an additional observation. For simplicity of exposition, we first treat the case of a single orbit segment, the case of a pseudo-orbit follows similarly by setting up a sequence of contractions (rather than a single contraction, which we describe below). By local product structure, one may take $W^s_{a,\loc}(x)$ then saturating each $y \in W^s_{a,\loc}(x)$ with each of their local center-unstable manifolds, $W^{cu}_{a,\loc}(y)$ to obtain a neighborhood of $x$. If $ta \cdot x$ is sufficiently close to $x$, consider the following induced map on $W^s_{a,\loc}(x)$: if $y \in W^s_{a,\loc}(x)$, then $ta \cdot y$ is also within a neighborhood of $x$. Define $f(y) = W^{cu}_{a,\loc}(ta \cdot y) \cap W^s_{a,\loc}(x)$. By the usual arguments, if $t$ is sufficiently large (which determines $T$), the contraction on the stable manifolds is such that $f$ is a contraction on $W^s_{a,\loc}(x)$. Therefore, $f$ has some fixed point $x_1 \in W^s_{a,\loc}(x)$. Then $x_1$ satisfies $ta\cdot x_1 \in W^{cu}_{a,\loc}(x_1)$. Looking at $-ta$ on $W^u_{a,\loc}(x_1)$, we may again find a fixed point to get $x'$ such that $ta \cdot x' = b \cdot x'$ for some $b \in B(0,\ve) \subset \R^k$.

If $\ve$ is sufficiently small and $t$ is sufficiently large, we claim that $a' = ta-b$ is still an Anosov element of $\R^k$. Indeed, once an element of $\R^k$ is Anosov, so is the line passing through it (minus 0). Furthermore, the set of Anosov elements is open, so if $\delta_1 > 0$ is such that $B(a,\delta_1)$ are all Anosov elements, so are $B(ta,t\delta_1)$.  By choosing $T$ such that $T\delta_1 > \ve$ (and hence $t\delta_1 > \ve$), we obtain that  that $a'$ is Anosov. Consider $\Fix(a')$, the fixed point set of $a'$. This set is compact, and $\R^k \cdot x' \subset \Fix(a')$. Furthermore, since $a'$ is Anosov, $\Fix(a')$ is a finite union of $\R^k$-orbits (since if two points are both in $\Fix(a')$ and are sufficiently close, they must be related by an element of the central direction, ie $\R^k$). This immediately implies that $x'$ is an $\R^k$-periodic point.

 We sketch the proof for the case of a general pseudo-orbit. One proceeds as above, except that one replaces the map $f$ by a sequence of maps $f_i : W^s_{a,\loc}(x_i) \to W^s_{a,\loc}(x_{i+1})$, defined by $f_i(y) = W^{cu}_{a,\loc}(t_ia \cdot y) \cap W^s_{a,\loc}(x_{i+1})$. These maps are again contractions. As $d((t_ia)\cdot x_i,x_{i+1})$ is very small, each $f_i$ is well-defined on some neighborhood $U_i \subset W^s_{a,\loc}(x_i)$ with $f_i(U_i) \subset U_{i+1}$. Analyzing their composition yields a map from $U_0$ to $U_0$, and one again arrives at a fixed point $y^s$. We then symmetrize the construction to find a point of $W^u_{a,\loc}(x_0)$ which is fixed by the analogously defined composition of the inverse maps, which we denote by $y^u$. Then $W^{cu}_{a,\loc}(y^s) \cap W^{cs}_{a,\loc}(y^u)$ is again a piece of an $\R^k$ orbit, which has a point $x'$ that is fixed by an Anosov element. Now, any point $x_i$ could have served as the reference to find the fixed point, and $t_ia$ takes the local central manifolds near $x_{i-1}$ to those near $x_i$. Let $x_i'$ denote a point on this local central manifold near $x_i$. Now, $t_ia \cdot W^c_{\loc}(x_{i-1}')$ is an orbit piece of radius $\delta$ that intersects the distinguished local central manifold near $x_i$. In particular, there exists $c_i \in \R^k$ with $\norm{c_i} < 2\delta$ such that $t_ia \cdot x_{i-1}' = c_i x_i'$, and we set $b_i = t_ia - c_i + b_{i-1}$. By induction, this yields the desired sequence of elements $b_i$.

Let us now obtain the estimates on $d(b\cdot x,b \cdot x')$  in the case of only one leg. Notice that by the construction above, there exists $a' \in \R^k$ such that $d(ta,a') < \ve$ and $a' \cdot x' = x'$. Notice that $x$ and $x'$ are connected by a short piece of stable manifold, and a short piece of unstable manifold, both of which have length at most $\ve$. Similarly, $a' \cdot x$ and $x' = a' \cdot x'$ are connected by a short piece of stable manifold and short piece of unstable manifold, both of which have length at most $\ve$. Hence, $d((sa') \cdot x, (sa') \cdot x') < \max\set{e^{-\lambda \norm{a'} s}\ve,e^{-\lambda(1-s)\norm{a'}}\ve}$, which is at most $e^{-\lambda s\norm{a'}}\ve$ if $s < 1/2$. Notice that $C' =  \set{b \in \R^k : d(b,sa') < t/(2\lambda) \mbox{ for some }s \in \R_+}$ contains $C$ if $a'$ is sufficiently close to $a$, which can be achieved by choosing $\ve$ sufficiently small. But if $b \in C'$, and $d(b,sa') < s/(2\lambda)$, then if $b' = b - ta'$:

\[ d(b\cdot x,b \cdot x') \le e^{\lambda\norm{b'}} d(sa'\cdot x, sa' \cdot x') < e^{s/2} \cdot e^{-\lambda s\norm{a'}}\ve < \ve. \]
 %Call this set $\Sigma$. Every point $z$ sufficiently close to $x$ has a unique $a$ such that $a \cdot z \in \Sigma$. We may without loss of generality assume that $a^t\cdot x \in \Sigma$. Then project
\end{proof}

%\begin{lemma}
%Let $\R^k \curvearrowright X$ be a transitive action, and $C \subset \R^k$ be any open cone. For every $\ve > 0$, there exists some $x \in X$ such that $C \cdot X$ is $\ve$-dense.
%\end{lemma}

%\begin{proof}
%Choose any point $x$ such that $\R^k \cdot x$ is dense. Then for every $\ve > 0$, there exists $R >0$ such that $B_R \cdot x$ is $\ve$-dense, where $B_R$ is the ball of radius $R$ in $\R^k$. Since $C$ is an open cone, there exists $a \in \R^k$ such that $B_R \subset C + a$. Therefore, if $y  = a \cdot x$, $C \cdot y$ is $\ve$-dense, as claimed.
%\end{proof}

\subsection{The spectral decomposition and transitivity properties of $\R^k$ actions}

This section provides a spectral decomposition, which we will need in the subsequent section. It may also be of independent interest as it provides some structure to $\R^k$-actions when transitivity assumptions are dropped.  Non-transitive actions can be constructed by taking the direct product of a non-transitive Anosov flow with any Anosov $\R^{k-1}$-action.

%Let $\R^k \curvearrowright X$ be an Anosov action, and $a \in \R^k$ be an Anosov element. A closed $\R^k$-invariant subset $\Lambda \subset X$ is said to be $a$-locally maximal if there exists a neighborhood $U$ of $\Lambda$ such that $\Lambda = \bigcap_{t\in \R} (ta) \cdot U$.

Recall that if $\varphi_t : X \to X$ is a continuous flow on a compact metric space $X$, then the {\it nonwandering set}, denoted $NW(\set{\varphi_t})$, is the set of points $x$ satisfying: for every $\ve >0$, there exist arbitrarily large $t \in \R_+$ such that $\varphi_t(B(x,\ve)) \cap B(x,\ve) \not= \emptyset$.

\begin{theorem}
\label{thm:spectral}
Let $\R^k \curvearrowright X$ be an Anosov $\R^k$ action. For any Anosov element $a$, there exist finitely many closed $\R^k$-invariant sets $\Lambda_i = \Lambda_i(a) \subset X$, $i = 1,\dots,n$ such that if $\Per(\R^k)$ denotes the set of $\R^k$-periodic points:

\[ NW(\set{ta}) = \overline{\Per(\R^k)} = \bigcup_{i=1}^n \Lambda_i. \]

Furthermore, for a generic set of Anosov elements $a' \in \R^k$ in the Weyl chamber for $a$, $a'$ is transitive on each $\Lambda_i$.%, and each $\Lambda_i$ is locally maximal.}
\end{theorem}

Each $\Lambda_i$ is called an {\it $a$-basic set}.

\begin{proof}
First, note that from Theorem \ref{thm:anosov-closing}, we immediately get that $NW(\set{ta}) = \overline{\Per(\R^k)}$. To pick the elements $a'$, notice that there are at most countably many $\R^k$-periodic orbits. Therefore, the set of elements which have dense orbits in every $\R^k$-periodic orbit is generic, since these are exactly the complement of a countable union of hyperplanes (perhaps not through 0). In particular, we may choose such an $a'$ to have this property  be in the same Weyl chamber as $a$.

Define a  relation on periodic points. Say that $p \sim q$ if and only if $W^{cu}_a(p) \cap W^s_a(q) \not= \emptyset$ and $W^{cu}_a(q) \cap W^s_a(p) \not= \emptyset$. These intersections are transverse since $a$ is Anosov, and we claim that $\sim$ is an equivalence relation. Indeed, symmetry and reflexivity are trivial by construction. To see transitivity, suppose that $p \sim q$ and $q \sim r$, where $p,q,r \in \Per(\R^k)$. Since $p \sim q$, the weak unstable manifold of $p$ intersects the stable manifold of $q$. Therefore, under application of the flow, this intersection point approaches the corresponding orbit of $q$. Hence its weak unstable manifolds converge to the weak unstable manifold of $q$ on arbitrarily large compact subsets. Since this weak unstable intersects the stable manifold of $r$, by local product structure at the intersection point, we get that $p \sim r$.

Notice that by local product structure, there exists $\ve > 0$ such that if $p,q \in \Per(\R^k)$ and $d(p,q) < \ve$, then $p \sim q$. In particular, we know that there can be only finitely many equivalence classes of $\sim$. 

Let $E_i$ denote the union of the periodic orbits in each equivalence class, and $\Lambda_i = \overline{E_i}$. Let $p,q \in E_i$. Notice that since $W^{cu}_a(p) \cap W^s_a(q) \not= \emptyset$, there exists $p' = b \cdot p$ such that $W^u_a(p') \cap W^s_a(q) \not= \emptyset$ for some $b \in \R^k$. Since $\set{na' \cdot p'}$ is dense in $\R^k \cdot p$ by choice of $a'$, for any $\ve > 0$, there exists arbitrarily negative $s_1$ such that $d(s_1a' \cdot p',p) < \ve$ and arbitrarily positive $s_2$ such that $d(s_2a' \cdot p',p') < \ve$. Fix $z \in W^{u}_a(p') \cap W^s_a(q)$, and let $x_1 = s_1a' \cdot z$ and $t_1 = s_2-s_1$. Notice that since $W^u_a(p')$ is contracted by flowing backwards under $a'$, the point $z$ can be moved arbitrarily close to $p'$, and hence $p$. Similarly, $(t_1a') \cdot x_1$ can be moved arbitrarily close to $q$.

By symmetrizing the above argument, we may find $x_2$ and $t_2$ such that $x_2$ is arbitrarily close to $q$ and $(t_2a') \cdot x_2$ is arbitrarily close to $p$. Hence for every $\ve > 0$, there exists a closed $(a',\ve)$-pseudo orbit $\set{(x_1,t_1),(x_2,t_2)}$. By Theorem \ref{thm:anosov-closing}, there exists an $\R^k$-periodic orbit shadowing it.

Now, we recall the following criterion for topological transitivity of a flow on a compact metric space $\Lambda$: if $U$ and $V$ are arbitrary open subsets of $\Lambda$, then there exists $n > 0$ such that $(na') \cdot U \cap V \not= \emptyset$. In our case, since $\Lambda_i$ is the closure of a set of periodic orbits, there exists $p \in U$ and $q \in V$ such that $p$ and $q$ are $\R^k$-periodic. Therefore, by the arguments above, we may find another periodic orbit arbitrarily close to each, in particular one which intersects $U$ and $V$. Since by assumption, $a'$ is transitive on every periodic orbit, we conclude the criterion, and hence transitivity.
\end{proof}

\subsection{  Cone Transitivity, Periodic Orbits, and Measures}

Unlike its rank one counterpart, it is not clear from transitivity of the group action $\R^k \curvearrowright X$ that there exist forward orbits $L_a\cdot x = \set{ta : t > 0} \cdot x$ which are dense. In fact, when the group action splits as a product, there are singular directions for which this does not hold. There are examples of $\R^k$ actions for which there is a dense $\R^k$-orbit, but for which {\it no} direction $L_a$ has a dense orbit. One may build this by taking any transitive flow $\psi_t$, and the direct product with a parabolic flow on a circle with a unique fixed point. When $k = 1$ this is avoided by assuming that there are no open orbits of $\R$.

We therefore use the stronger condition of {\it cone transitivity} (recall Definition \ref{Anosov action}), which combined with the Anosov closing lemma (Theorem \ref{thm:anosov-closing}), will allow us to deduce several nice properties. We first establish a lemma which shows that the cone transitivity assumption is natural. We expect the following to hold when replacing cone transitive by transitive. One may compare its statement with the fact that the time-$t$ map of an Anosov flow is transitive for generic $t \in \R$ if the flow itself is transitive, or that, for measure-perserving transformations, the set of directions elements which are ergodic as transformations are the complement of a countable union of affine hyperplanes \cite{pugh-shub1971}.

\begin{lemma}
\label{lem:cone-transitive}
Let $\R^k \curvearrowright X$ be a $C^{1,\theta}$, totally Anosov action of $\R^k$. Then the following are equivalent:

\begin{enumerate}
\item The action is cone transitive.
\item The action has a dense set of periodic orbits. 
\item The set of elements $a \in \R^k$ with dense forward orbits is residual.
\item For every open cone $C \subset \R^k$, there exists $x \in X$ such that $C \cdot x$ is dense in $X$.
\item There exists an Anosov element $a$ with a dense forward orbit.
\item The action preserves a fully supported ergodic invariant measure.
\item For every $a \in \R^k$, $NW(a) = X$.
\end{enumerate}
\end{lemma}

We prove the following technical lemma which will be useful:

\begin{lemma}
\label{lem:cone-differences}
Let $C \subset \R^k$ be an open cone and $C'$ be any open cone containing $C$. Let $S \subset C$ be an unbounded set. Then for every $R > 0$, there exists $a_1,a_2 \in S$ such that $a_2-a_1 \in C'$ and $\norm{a_2-a_1} > R$.
\end{lemma}

\begin{proof}
Since a cone is invariant under positive scaling, it is determined by its intersection with $S^{k-1} \subset \R^k$, which is open and convex.
Fix $a_1 \in S$, and choose a sequence $b_n \in S$ such that $\norm{b_n} \to \infty$. Then $\norm{b_n - a_1} \to \infty$. Since $b_n \in S \subset C$, the normalized vector $b_n / \norm{b_n} \in S^{k-1} \cap C$. Without loss of generality, we may assume that the sequence $b_n / \norm{b_n}$ converges to some $b \in \overline{C} \subset C'$. Now, $(b_n-a_1)/\norm{b_n-a_1} \to b_n/\norm{b_n}$, since $a_1$ is fixed. Therefore, since $C'$ is open and $b_n / \norm{b_n} \to b \in C'$, there exists $n_0$ such that $(b_{n_0}-a_1)/\norm{b_{n_0}-a_1} \in C'$ and $\norm{b_{n_0} - a_1} > R$. Setting $a_2 = b_{n_0}$ yields the desired pair of elements.
\end{proof}

\begin{proof}[Proof of Lemma \ref{lem:cone-transitive}]
We will show that (1) $\implies$ (2) $\implies$ (3) $\implies$ (4)  $\implies$ (1), and that (3) $\implies$ (5) $\implies$ (6) $\implies$ (7) $\implies$ (2). This is sufficient, since the first chain of implications is a loop, and the second begins and ends on that loop.

To see that (1) $\implies$ (2), assume that $\R^k \curvearrowright X$ is cone transitive. Choose a cone $C$ whose closure consists of Anosov elements and associated point $x$ such that $C \cdot x$ is dense. Since $\overline{C}$ consists of Anosov elements, and the set of Anosov elements is open, we may choose an open cone $C'$ consisting of Anosov elements containing $\overline{C}$, and such that $\overline{C'}$ also consists of Anosov elements. We will show that if $y \in X$ and $\ve > 0$ is sufficiently small, then there exists an $\R^k$-periodic point $p$ such that $d(y,p) < \ve$. Notice that since $C \cdot x$ is dense, the orbit must visit $B(y,\ve/2)$ in an unbounded set (otherwise, there is an open orbit of $\R^k$). Therefore, by Lemma \ref{lem:cone-differences}, there exists some $a_1,a_2 \in C$ such that $a_2 - a_1 \in C'$ can be made arbitrarily large, and $a_1\cdot x,a_2 \cdot x \in B(y,\ve/2)$. Since $\overline{C'}$ consists of Anosov elements, we may assume $\ve > 0$ is small enough to be able to apply Theorem \ref{thm:anosov-closing}, and arrive at $\R^k$ periodic orbits arbitrarily close to $y$.

To see that (2) $\implies$ (3), assume that there are a dense set of $\R^k$-periodic orbits. Then according to the spectral decomposition (Theorem \ref{thm:spectral}) for any Anosov $a$, there is only one $a$-basic set. Therefore, in each Weyl chamber of Anosov elements, there is a residual set of transitive elements.

That (3) $\implies$ (4) is clear, since if there is a residual set of transitive elements, any open cone will contain positive multiples of some transitive element, and hence itself have a dense orbit. That (4) $\implies$ (1) is also clear, since we may chooose any cone of Anosov elements whose closure still consists of Anosov elements.

Now, we prove the second loop of implications. That (3) $\implies$ (5) is clear: any residual subset of $\R^k$ must contain an Anosov element. That (5) $\implies $ (6) follows from \cite[Theorem 4.7(2)]{CPZ-PH}. Such a measure was also constructed in \cite{BonGuiWei2021} for $C^\infty$ actions. Another construction was carried out in \cite{CRH2021} under additional assumptions.

 That (6) $\implies$ (7) follows from the Poincar\'{e} recurrence theorem applied to the fully supported invariant measure. Finally, that (7) $\implies$ (2) again follows from the Anosov closing lemma.
\end{proof}

We now establish several important applications of Lemma \ref{lem:cone-transitive}.

\begin{corollary}
\label{lem:residual-recurrence}
Let $\R^k \curvearrowright X$ be a cone transitive, totally Anosov action. Fix $a \in \R^k$. Then the set of points $x \in X$ such that there exists a sequence $n_k \to +\infty$ such that $a^{n_k} \cdot x \to x$ is a residual subset of $X$.
\end{corollary}

\begin{proof}
Every $\R^k$-periodic point is recurrent for every $a \in \R^k$, so this is immediate from the fact that the set of recurrent points is a $G_\delta$ set.
\end{proof}

%\begin{corollary}
%\label{cor:dense-1-param}
%Let $\R^k \curvearrowright X$ be a transitive, totally Cartan action. Then for a generic direction $a \in S^{k-1}$, there exists $x \in X$ such that $\set{(ta) \cdot x : t >0}$ is dense in $X$.
%\end{corollary}

%\begin{proof}
%If $p$ is an $\R^k$-periodic orbit, let $\pi(p) \subset S^{k-1}$ be the set of directions which lie in a rational subtorus. Then $\pi(p)$ is a countable union of closed, nowhere dense subsets (corresponding to the rational subtorus obtained as the closure). Let $P = \bigcup_p \pi(p)$ be the union of such directions. Then $S^{k-1} \setminus P$ is a dense $G_\delta$ subset of $S^{k-1}$. We claim that for any $a \in S^{k-1} \setminus P$, the forward orbit under $a$ is dense. It is enough to show that for any open sets $U,V \subset X$, there exists $t > 0$ such that $(ta) \cdot U \cap V \not= \emptyset$. By Theorem \ref{thm:dense-periodic}, for any $\ve > 0$, there exists an $\ve$-dense $\R^k$-periodic orbit. Since $U$ and $V$ are open, there exists some $\ve >0$ such that $U$ and $V$ both contain $\ve$-balls. Choose an $\ve/2$-dense periodic orbit which must intersect $U$ and $V$ by construction. By assumption, the forward orbit of $a$ is dense inside the periodic orbit, and therefore intersects both $U$ and $V$. This proves the result.
%\end{proof}

Recall that a set is $\ve$-dense if its $\ve$-neighborhood is $X$.

\begin{corollary}
\label{cor:eps-dense}
Let $\R^k \curvearrowright X$ be a cone transitive, totally Anosov action. Then for every $\delta > 0$, there exists an $\R^k$-periodic point $p$ whose orbit is $\delta$-dense.
\end{corollary}

\begin{proof}
Choose an Anosov element $a$ with a dense orbit (which is possible by Lemma \ref{lem:cone-transitive}), some $\delta > 0$ and choose $\ve$ as in Theorem \ref{thm:anosov-closing} for $\delta/2$. Then pick a point $x$ such that $\set{(ta)\cdot x : t > 0}$ is dense in $X$. For some $t_0 > 0$, $\set{(ta)\cdot x : t \in [0,t_0]}$ is $\delta/2$-dense in $X$. Furthermore, we may choose some $T > 2t_0$ such that $d((ta)\cdot x,x) < \ve$, so that there exists an $\R^k$ periodic orbit $x'$ with $d((ta)\cdot x',(ta)\cdot x) < \delta/2$. Then by construction, the $\R^k$-periodic orbit of $x'$ is $\delta$-dense.
\end{proof}

{ 
\subsubsection{SRB Measures and disintegrations}
\label{sec:SRB}
In the proof of Lemma \ref{lem:cone-transitive}, the fully supported ergodic invariant measure constructed from an Anosov element is actually the {\it SRB measure} for some Anosov element $a\in \R^k$. 

\begin{remark}
    %By uniqueness coming from \cite[Theorem 4.7]{CPZ-PH} (see also \cite{BonGuiWei2021} in the $C^\infty$ setting), this measure depends only on the Weyl chamber of the action.
    By uniqueness coming from \cite[Theorem 4.7]{CPZ-PH} (see also \cite{BonGuiWei2021} in the $C^\infty$ setting), this measure is invariant for any $b \in \R^k$.  If $b$ also belongs to the same Weyl chamber as $a$, then the SRB measure for $a$ is also the SRB measure for $b$ as the stable foliations are identical.  Hence this measure only depends on the Weyl chamber. 
\end{remark}

Let $\mc W \subset \R^k$ be a Weyl chamber, $\mu_{\mc W}$ denote the corresponding SRB measure and $\beta$ be some weight such that $\beta(\mc W) > 0$ and $\ker \beta$ bounds $\mc W$. 

Through standard constructions via expanding partitions subordinate to the leaves of a coarse Lyapunov foliation $W^\alpha$, one may build a family of measures $\mu^\alpha_{\mc W,x}$ on $W^\alpha(x)$, defined for $\mu$-almost every $x$, such that

\begin{itemize}
    \item $x \mapsto \mu^\alpha_{\mc W,x}$ is measurable,
    \item if $\Xi$ is a measurable partition subordinate to $W^\alpha$, $\mu_{\mc W,x,\xi}$ is a constant multiple of $\mu^\alpha_{\mc W,x}$ on each atom of $\Xi$,
    \item $\mu^\alpha_{\mc W,x}(B_{W^\alpha}(1,x)) = 1$ for $\mu$-almost every $x$, and
    \item $a_*\mu^\alpha_{\mc W,x}$ is proportional to $\mu^\alpha_{a\cdot x}$ for $\mu$-almost every $x$.
\end{itemize}

For more on constructing the conditional measures $\mu^\alpha_{\mc W,x}$ from conditional measures defined on subordinate partitions, we refer the reader to \cite[Section 3.2]{AVW} (and \cite[Section 2.2.3, Appendix B.5]{brown19}, where only algebraic actions are considered).

The action of $\ker \beta$ on $(X,\mu_{\mc W})$ may fail to be ergodic. Let $\mu_{\mc W} = \displaystyle\int_\Omega \nu_{\mc W,\beta,\omega} \, dm(\omega)$ denote the ergodic decomposition of $\mu_{\mc W}$ with respect to the action of $\ker \beta$ (here $\Omega$ an indexing set for the space of $\ker \beta$-ergodic measures and $m$ is a measure on it).  We have the following structures for ergodic components:

\begin{proposition} \label{prop:SRB}
Let $\R^k \curvearrowright X$ be a $C^{1,\theta}$  cone transitive, totally Anosov action and $\alpha$, $\beta$ and $\mc W$ be as above.  For every $\alpha$ such that $\alpha(\mc W) > 0$, $\alpha \not= \beta$, $\nu_{\mc W,\beta,\omega,x}^\alpha$ is absolutely continuous with respect to Lebesgue measure for $m$-almost every $\omega \in \Omega$ and $\nu_{\mc W,\beta,\omega}$-almost every $x \in X$. In particular, the absolute continuity holds for $\mu_{\mc W}$-almost every $x \in X$.
\end{proposition}

\begin{proof}
Fix some generic measure $\nu$ in the ergodic decomposition.  Let
$a \in \ker \beta$ on the boundary of $\mc W$ which does not belong to any other Lyapunov hyperplane. Since the set of ergodic elements for $\ker \beta$ acting on $(X,\nu)$ is the complement of countably many hyperplanes \cite[Theorem 1.2]{pugh-shub1971}, we may without loss of generality assume that $\nu$ is ergodic for $a$. Let $a_0$ be an element of $\mc W$ very close to $a$, so that $W^u_a$ is a fast foliation in $W^u_{a_0}$. 

Let $r = 2$ or $\infty$ based on the regularity of the action. We claim that $W^\alpha$ is a $C^r$ subfoliation of each leaf of $W^u_{a_0}$ for any $\alpha$ as described in the statement proposition.

Indeed, we show by induction that $W^\alpha$ is {  absolutely continuous in} common unstable manifolds of increasing dimension contained in $W^u_{a_0}$. {  By absolutely continuous, we mean that the disintegration of the Lebesgue measure class in the leaf of the common unstable manifold along leaves of $W^\alpha$ is again of Lebesgue measure class.} We use the following very useful fact: given two foliations $\mc F'$ and $\mc F$ such that $\mc F'$ refines $\mc F$, $\mc F'$ is fast in $\mc F$ (as in Section \ref{sec:slow-foliations}), and $W^\gamma$ is slow and complementary in $\mc F$ (with arbitrarily slow growth), then $\mc F'$ is $C^r$ in each leaf of $\mc F$. This follows from the $C^r$-section theorem of \cite{HPS1970} and the fact that the growth of the slow foliation can be made arbitrarily slow. We will use this argument again in Lemma \ref{lem:smooth-holonomy}.

Assume that $W^\alpha$ is an absolutely continuous subfoliation of some common stable foliation $\mc F$. We may find a weight $\gamma$ and common stable foliation $\mc F'$ such that $T\mc F' = T\mc F \oplus E^\gamma$. By choosing an element sufficiently close to $\ker \gamma$, we may make $\mc F$ be the fast foliation inside $\mc F'$. Since we may become arbitrarily close to $\ker \gamma$, it follows that $\mc F$ is $C^r$ in $\mc F'$. Since $W^\alpha$ is {  absolutely continuous} in $\mc F$, it follows that $W^\alpha$ is {  absolutely continuous} in $\mc F'$. By induction, we get that $W^\alpha$ is {  absolutely} in $W^u_{a_0}$.

{ 
\begin{remark}
    It is not true from the general theory that $W^\alpha$ is $C^r$ in the common stable manifolds. Indeed, it is easy to construct a one-parameter family of smooth foliations of $\R^2$ such which do not form a smooth foliation of $\R^3$.
\end{remark}
}
By definition of SRB measure, the conditional measures of $\mu_{\mc W}$ along $W^u_{a_0}$ are absolutely continuous with respect to Lebesgue. Since $W^\alpha$ is an {  absolutely continuous} subfoliation, it follows that $\mu^\alpha_{\mc W,x}$ is absolutely continuous with respect to Lebesgue for $\mu^\alpha_{\mc W,x}$-almost every $x$.

Finally, when considering the ergodic decomposition, $m$-almost every ergodic component $\nu$ must have $\nu^\alpha_x$ absolutely continuous with respect to Lebesgue, following the Hopf argument. Indeed, the forward Birkhoff averages of $a$ are constant along $W^\alpha$-leaves for continuous observables. Then using density of $C^0$ functions in $L^2$ and the description of ergodic components via this construction, one gets that the ergodic decomposition is refined by the partition into stable leaves. See  \cite[Theorem C.2.1]{brown19} for a more thorough discussion. 
\end{proof}
}
\subsection{Normal Forms Coordinates}
\label{sec:normal-forms}

Let $f : X \to X$ be a H\"older continuous transformation of a compact metric space $X$, $\mc L = X \times \R$ be the trivial line bundle over $f$.  A map $F : \mc L \to \mc L$ is called a (one-dimensional) {\it $C^r$-extension of $f$} if $F(x,t) = (f(x),F_x(t))$, where $F_x \in C^r(\R,\R)$. We say that $F$ is {\it uniformly contracting} if there exists $\mu \in (0,1)$ such that $\abs{F_x'(t)} < \mu$ for all $(x,t) \in \mc L$. We say that $F$ is {\it transversally H\"older} if $x \mapsto F_x \in C^r(\R,\R)$ is  H\"older continuous in the $C^r$ topology. A {\it $C^r$-normal form coordinate system} for $F$ is a map $H : \mc L \to \mc L$ of the form $H(x,t) = (x,\psi_x(t))$ such that:

\begin{enumerate}
\item ${ \psi}_x \in C^r(\R,\R)$ for every $x \in X$,
\item ${ \psi}_x(0) = 0$ and $\psi_x'(0) = 1$ for every $x \in X$,
\item $x \mapsto \psi_x$ is continuous from $X$ to $C^r(\R,\R)$, and
\item $HF(x,t) = \hat{F}H(x,t)$,
\end{enumerate}

\noindent where $\hat{F}(x,t) = (f(x),F_x'(0)t)$. Katok and Lewis were the first to construct normal forms on contracting foliations in the 1-dimensional setting for $C^\infty$ transformations \cite{KL1}. Guysinsky extended this to the $C^r$ setting \cite{Guys}. The theory of normal forms has since been extended to a much more general setting. The broadest results in the uniformly hyperbolic setting have been written by Kalinin in \cite{kal20} and even extended to the nonuniformly hyperbolic setting by Melnick \cite{melnick} and Kalinin and Sadovskaya \cite{Kalinin2017341}.

\begin{theorem}
\label{thm:normal-forms}
Let $F$ be a transversally H\"older, uniformly contracting one-dimensional $C^r$ extension of a $C^r$ diffeomorphism $f$, $r > 1$. Then $F$ has a unique $C^r$-normal form coordinate system, varying H\"older continuously in the $C^r$-topology.
%\item[(b)] If $g : X \to X$ is a continuous transformation commuting with $f$, with a $C^1$ extension $G$ commuting with $F$, then $G$ is also linear in the $C^r$-normal form coordinate system. 
%\item[(c)] Any $C^1$ normal form coordinate system coincides with the $C^r$ system.
%\end{enumerate}
\end{theorem}

Theorem \ref{thm:normal-forms} is an immediate corollary of \cite[Theorem 4.3]{kal20}. Indeed, in the 1-dimensional case, all narrow band spectrum assumptions become vacuous and the identity map is the unique subresonance polynomial whose derivative is 1 (see \cite{kal20} for the definitions of these technical conditions).

\begin{lemma}
\label{lem:nf-commuting}
Let $F$ be as in Theorem \ref{thm:normal-forms}, and $G$ be any one-dimensional $C^1$-extension of a homeomorphism $g$ which commutes with $f$ such that $G$ commutes with $F$. If $H$ is the $C^r$-normal forms coordinate system for $F$, $\hat{F} = HFH^{-1}$ and $\hat{G} = HGH^{-1}$, then $\hat{G}$ is linear.
\end{lemma}

\begin{proof}
Let $\hat{F}$ and $\hat{G}$ be as in the statement of the lemma, and notice that since $F$ commutes with $G$, $\hat{F}$ commutes with $\hat{G}$. Furthermore, $\hat{G}$ is still a $C^1$-extension of $g$. Let $\lambda(x,n) = (F_x^{(n)})'(0)$ denote the derivative  of $F^n$ along the bundle direction, so that $\hat{F}^n(x,t) = (x,\lambda(x,n)t)$. Then since $\hat{G}$ commutes with $\hat{F}$,

\[ \hat{G}(x,t) = \hat{F}^{-n}\hat{G}\hat{F}^n(x,t) = \hat{F}^{-n}\hat{G}(f^n(x),\lambda(x,n)t). \]

Now, since $\hat{G}$ is $C^1$, by the mean value theorem, there exists $s_n \in [0,\lambda(x,n)t]$ such that $\hat{G}(f^n(x),\lambda(x,n)t) = (gf^n(x),\hat{G}_{f^n(x)}'(s_n)\lambda(x,n)t)$. Notice that $s_n \to 0$ since $\lambda(x,n) \to 0$ exponentially in $n$. Then:

\begin{equation}
\label{eq:hatG-conj}
\hat{G}(x,t) =  \hat{F}^{-n}(gf^n(x),\hat{G}_{f^n(x)}'(s_n)\lambda(x,n)t) = (g(x),\lambda(g(x),n)^{-1}\hat{G}_{f^n(x)}'(s_n)\lambda(x,n)t) .
\end{equation}

Choose a subsequence $n_k$ such that $f^{n_k}(x)$ converges to some $y \in X$ (this is possible by compactness of $X$). Then

\[ \frac{\lambda(g(x),n_k)}{ \lambda(x,n_k)} = \frac{{\hat{F}}^{n_k\prime}_{g(x)}(0)}{{\hat{F}}^{n_k\prime}_x(0) } 
= \frac{(\hat{F}^{n_k} \of \hat{G})_x'(0)} { \hat{G}_x'(0)\cdot \hat{F}^{n_k\prime}_x(0) }
 = \frac{(\hat{G} \of \hat{F}^{n_k})_x'(0)}{ \hat{G}_x'(0)\cdot \hat{F}^{n_k\prime}_x(0) } 
 = \frac{\hat{G}_{f^n(x)}'(0)} { \hat{G}_x'(0)}
 \to \frac{\hat{G}_y'(0)}{\hat{G}_x'(0)}. \]
 
 Therefore, since \eqref{eq:hatG-conj} holds for every $n$ (and hence for every $n_k$), since $s_n \to 0$, and since all functions involved are continuous:
 
 \begin{equation}
 \label{eq:hatG-linear}
  \hat{G}(x,t) = \lim_{k \to \infty}   \left(g(x),\hat{G}_{f^{n_k}(x)}'(s_n)\frac{\lambda(x,n_k)}{\lambda(g(x),n_k)}t\right) = (g(x),G_x'(0)t).
  \end{equation}
 Therefore, $\hat{G}$ is linear at every point, as claimed.
\end{proof}

Theorem \ref{thm:normal-forms} establishes uniqueness of a $C^r$-normal form, but a different $C^1$ normal form could exist, since we require $r > 1$. The following establishes uniqueness of normal forms, even for $C^1$ extensions.

\begin{corollary}
Let $F$ be as described in Theorem \ref{thm:normal-forms} and  $H$ be the unique $C^r$-normal forms coordinate system for $F$. Then if $\bar{H}$ is a $C^1$-normal forms coordinate system for $F$, $\bar{H} = H$.
\end{corollary}

\begin{proof}
Let $G = \bar{H}^{-1} \of H$, so that $G : \mc L \to \mc L$ is a $C^1$ extension of the identity on $X$. Then

\[ G \of F = \bar{H}^{-1} \of H \of \hat{F} = \bar{H}^{-1} \of \hat{F} \of H= F \of G\]
since both $\bar{H}$ and $H$ conjugate $F$ to $\hat{F}$ by definition of normal forms charts. Since $G$ commutes with $F$, by Lemma \ref{lem:nf-commuting}, $G$ is linear in the normal forms coordinates provided by $H$. But since $G_x'(0) = 1$ for every $x$, $G = \id$ and $\bar{H} = H$.
\end{proof}

Our main application of normal forms will be to contracting one-dimensional oriented foliations of a smooth manifold $X$. Our notation of orientability here is sometimes called {\it tangential orientability}, and the definition is very easy for 1-dimensional foliations.

\begin{definition}
We say that a one-dimensional foliation $\mc F$ on a $C^\infty$ manifold $X$ is {\rm orientable} if there exists a continuous, nonvanishing vector field on $X$ which is everywhere tangent to the foliation $\mc F$. We call such a vector field an {\rm orienting field} and say that the foliation is {\rm oriented} if such a vector field has been fixed.
\end{definition}

%Notice that any oriented foliation has a notion of a ``forward'' and ``backward'' direction, corresponding to whether the direction is a positive or negative multiple of the orienting field. Thus, it makes sense to discuss the signed distance between two points on the same leaf of an oriented foliation.

  If $f : X \to X$ is a $C^r$ diffeomorphism ($r > 1$) preserving an oriented, H\"older foliation $\mc F$ with $C^r$ leaves, one may construct a $C^r$ extension in the following way. Fix a $C^r$ Riemannian metric on $X$, and normalize the orienting field to have unit length. Then let $x_t$ denote the point of $\mc F(x)$ which is the image of $x$ under the flow induced by the normalized orienting field at time $t$. Notice that pushforward of the constant vector field $\partial /\partial t$ on $\R$ under the map $t \mapsto x_t$ is exactly the normalized orienting field. Since $\mc F$ is $f$-invariant, the pushforward of $\partial / \partial t$ under the map $t \mapsto f(x_t)$ is a multiple of the normalizing orienting vector field at every point. Let $\lambda_{t,x}$ denote this ratio, and $\int_0^t \lambda_{\tau,x} \, d\tau$ be the {\it signed length }of the curve $\tau \mapsto f(x_\tau)$, $\tau \in [0,t]$.

  Define $F(x,t) = (f(x),s)$, where $s$ is the (signed) length of the curve $\tau \mapsto f(x_\tau)$, $\tau \in [0,t]$ (if the leaves are all noncompact, then $s \in \R$ uniquely satisfies $f(x)_s = f(x_t)$). Then $F$ is a contracting $C^r$-extension of $f$ if $\mc F$ is contracting under $f$. In this case, we have the following additional structures, which are easy corollaries of the preceding results with the exception of \ref{enum:nf-consistency}, which follows from, eg, \cite[Theorem 4.6]{kal20}:

\begin{theorem}
\label{thm:normal-forms-foliation}
If $f : X \to X$ is $C^r$ diffeomorphism ($r > 1$) of a $C^\infty$ Riemannian manifold $X$ and $\mc F$ is a {  one-dimensional}, contracting, $f$-invariant, oriented, H\"older foliation  with $C^r$ leaves, then there exists a family of $C^r$-maps $\psi_x^{\mc F} : \R \to \mc F(x)$, varying H\"older continuously in the $C^r$-topology, such that:

\begin{enumerate}[label=(\alph*)]
\item \label{enum:nf-immersion} for every $x$, $\psi_x^{\mc F}$ is a immersion whose image is $\mc F(x)$,
\item $(\psi_x^{\mc F})'$ is the normalized orienting vector field of $\mc F$,
\item if $g :X \to X$ is a $C^1$ diffeomorphism of $X$ commuting with $f$ (including $g = f$), $g(\psi_x^{\mc F}(t)) = \psi_{g(x)}^{\mc F}(\norm{d_{\mc F}g{ (x)}}t)$, where $d_{\mc F}$ denotes the derivative restricted to $T\mc F$, and
\item \label{enum:nf-consistency} if $y \in \mc F(x)$, then $\psi_x \of {\psi_y}^{-1}$ is {  an affine transformation of} $\R$.
\end{enumerate}

Furthermore, any collection of $C^1$ maps varying continuously in the $C^1$-topology satisfying \ref{enum:nf-immersion}-\ref{enum:nf-consistency} must coincide with $\psi_x^{\mc F}$.
\end{theorem}

\subsection{Centralizers of transitive Anosov flows in dimension 3}

We will build rank-one factors that are Anosov flows on 3-manifolds. Understanding their properties will be important to our analysis. The following is classical in the study of Anosov flows, and the existence of the family is proven in the work of Margulis, see  Theorem 1 of Section 3 and Theorem 11 of Section 8 in \cite{margulisbook04}, or \cite{hasselblatt89}. The family satisfying these properties is used to construct the measure of maximal entropy, which has by now many alternate constructions.  While for Anosov flows which are not transitive, the family is not unique, it is unique for transitive flows. We could not find a proof of this uniqueness, which we offer here for completeness.

Recall that an Anosov flow is a flow $(g_t)$ such that $g_1$ is normally hyperbolic with respect to the orbit foliation of the flow, and that there are foliations associated with $(g_t)$, the stable and unstable foliations $W^s$, $W^u$, as well as the center-stable and center-unstable foliations, $W^{cs}$ and $W^{cu}$, which are the saturations of the orbit foliations by $W^s$ and $W^u$, respectively.

\begin{theorem}[Margulis]
\label{lem:margulis-conditionals}
Let $(g_t)$ be a transitive Anosov flow on a compact manifold $Y$, and $\set{W^u(y) : y \in Y}$ be the unstable foliation of $(g_t)$. Then there exists a %continuously varying 
family of $\sigma$-finite, Borel measures $\mu^u_y$ defined on $W^u(y)$ such that

\begin{enumerate}
\item  for every $y \in Y$, $\mu^u_y$ is finite on compact subsets in the leaf topology,

\item for every $y \in Y$, $\mu^u_y$ is fully supported on $W^u(y)$ in the leaf topology,

\item the measures $\mu^{cu}_y$ on $W^{cu}(y)$ defined by 

\begin{equation}
\label{eq:weak-disintegration}
\int f \, d\mu^{cu}_y := \int \int f\big|_{W^u_{g_t(y)}} \, d\mu^u_{g_t(y)} \, dt
\end{equation}
 also have property (1),

\item if $z \in W^s(y)$, and $\pi : W^{cu}(y) \to W^{cu}(z)$ is the stable holonomy map, then $\pi_* \mu_y^{cu} = \mu_z^{cu}$,%, where $\mu_y^{cu}$ is the measure on $W^{cu}(y)$ obtained by integrating the family $\mu_y^u$.

\item the family $\set{\mu^u_y}$ satisfies

\begin{equation}
\label{eq:margulis-cocycle}
 (g_t)_* \mu_y^u = e^{th} \mu_{g_t(y)}^u,
\end{equation}
\noindent where $h$ is the topological entropy of $(g_1)$,

\item the families $\set{\mu^u_y}$ and $\set{\mu^{cu}_y}$ have the following continuity property: if $y_n \to y$ and $f_n$ is a sequence of compactly supported continuous functions on $W^*(y_n)$ such that $f_n \to f$, where $f$ is continuous and compactly supported on $W^*(y)$, then $\int f_n \, d\mu^*_{y_n} \to \int f \, d \mu^*_y$, where $* = u$ or $cu$.
\end{enumerate}

\noindent Furthermore, the family $\set{\mu_y^u}$ is determined up to global scalar by these properties.
\end{theorem}

\begin{remark}
The continuity property (5) follows from the other properties. Indeed, it follows for $\mu^{cu}_y$ from the holonomy invariance property (3), and it follows for $\mu^u_y$ by the continuity property for $\mu^{cu}_y$ and \eqref{eq:weak-disintegration}.
\end{remark}

\begin{remark}
A more general construction due to Rokhlin  also gives a similar family of measures by using measurable partitions subordinate to the unstable foliation (see, eg, Appendix B.5 of \cite{brown19}). However, this construction does not give a family of measures up to global scalar, but instead measures which are defined uniquely up to a scalar which may depend on the point $y$. Such measures are often normalized to give the unit ball in $W^u(y)$ measure 1, and therefore do not usually coincide with the measures above.
\end{remark}

We summarize the construction of the unique measure of the maximal entropy $\mu$ from the family $\mu_y^u$. First, notice that there exist analogous families $\mu_y^s$ and $\mu_y^{cs}$ for the stable and center-stable foliations, respectively, which can be obtained by considering the flow $g_{-t}$. We define the measure $\mu$ on subsets of $Y$ which are obtained by saturating a small open subset of a  center-stable leaf (centered at some point $y_0$) with small pieces of unstable manifolds. By the local product structure of these foliations, such sets are open in $Y$. For a continuous function $\varphi$ supported on such a neighborhood, we define:

\begin{equation}
\label{eq:integrated-margulis} \int \varphi \, d\mu := \int \int \varphi \big|_{W^u(y)}(z) \, d\mu^u_y(z) \, d\mu^{cs}_{y_0}(y).
\end{equation}

We call $\mu$ the {\it Margulis measure} or {\it measure of maximal entropy} for $f$. After normalizing to a probability measure, $\mu$ is the unique entropy maximizing measure \cite{bowen74}.

\begin{proof}[Proof of Uniqueness]
We wish to show that properties (1)-(5) characterize the family $\set{\mu_y^u}$ up to a global scalar. Suppose that $\set{\nu_y^u}$ is another such family. Then following the Margulis construction outlined above, one can construct another probability measure  $\nu$ on $Y$, whose local disintegrations along measurable partitions subordinate to the unstable foliation are absolutely continuous with respect to the respective families. Notice also that in both cases, $h_\mu(g_t) = h_\nu(g_t) = h_{\operatorname{top}}(g_t)$. Therefore, both $\mu$ and $\nu$ are measures of maximal entropy for $(g_t)$. Since the measure of maximal entropy is unique, the measure classes of the disintegrations along measurable partitions subordinate to $W^u$ coincide  for $\mu$-almost every atom of the measurable partition, which in particular implies that $\mu_y^u$ and $\nu_y^u$ are absolutely continuous with respect to each other. Let $f : Y \to \R$ be the Radon-Nikodym derivative, so that $\mu_y^u = e^f \cdot \nu_y^u$ and $f$ is defined $\mu$-almost everywhere. Then one can easily check that $f$ is invariant under $g_t$ by \eqref{eq:margulis-cocycle} for both $\mu_y^u$ and $\nu_y^u$. Therefore, $\mu_y^u = \lambda \nu_y^u$ for a fixed constant $\lambda$ on a set of full $\mu$-measure by ergodicity with respect to $\mu$. Therefore, since the measure of maximal entropy is fully supported for any transitive Anosov flow and the families $\set{\mu_y^u}$ and $\set{\nu_y^u}$ both have the continuity property (5), they agree everywhere up to a global constant.
\end{proof}

\begin{lemma}
\label{lem:margulis-rescaling}
If $(g_t)$ is a transitive Anosov flow on a compact manifold $Y$, $\set{\mu^u_x}$ and $\set{\mu^s_y}$ are the unstable and stable families of measures for $(g_t)$ provided by Theorem \ref{lem:margulis-conditionals}, and $f : Y \to Y$ is a homeomorphism commuting with $(g_t)$, then $f_*\mu = \mu$ and there exists $\lambda \in \R_+$ such that for every $y \in Y$, $f_*\mu^u_y = \lambda \mu^u_{f(y)}$ and $f_*\mu^s_y = \lambda^{-1} \mu^s_{f(y)}$.
\end{lemma}

\begin{proof}
Notice that since $f$ commutes with $(g_t)$, $\set{f_* \mu^u_{f^{-1}(x)}}$ also satisfies properties (1)-(5) of Theorem \ref{lem:margulis-conditionals}. Since the family is unique up to global scalar, it follows that $f_*\mu_x^u = \lambda_1\mu_x^u$ and $f_*\mu_x^s = \lambda_2 \mu_x^s$ for some constants $\lambda_1,\lambda_2 \in \R_+$ independent of $x$.  We wish to show that $\lambda_1\lambda_2 = 1$. Assume without loss of generality that the families $\set{\mu_x^s}$ and $\set{\mu_x^u}$ have already been chosen so that the measure defined by \eqref{eq:integrated-margulis} is a probability measure.

Since $f$ commutes with $(g_t)$, $f_*\mu$ is also an entropy maximizing measure for $(g_t)$, so $f_*\mu = \mu$ by the uniqueness of the measure of maximal entropy \cite{bowen74}. But \eqref{eq:integrated-margulis} implies that

\[ \int \varphi \, d\mu = \int \varphi \, d(f_*\mu) =  \int \int \varphi \big|_{W^u(y)}(z) \, d(f_*\mu^u_y)(z) \, d(f_*\mu^{cs}_{y_0})(y) = \lambda_1\lambda_2 \int \varphi \, d\mu. \]
Therefore, $\lambda_1\lambda_2 = 1$.
\end{proof}

We will use the following result, which may be of independent interest. It requires low dimension, since it is easy to construct examples of Anosov flows on manifolds of dimension greater than 3 which have infinite index in their centralizers (for instance, by suspending certain Anosov automorphisms of $\mathbb{T}^d$, $d \ge 3$).

\begin{theorem}
\label{thm:3flow-centralizer}
Let $(g_t)$ be a transitive Anosov flow on a compact 3-dimensional manifold $Y$, and $C_{\Homeo(Y)}(g_t)$ denote the group of homeomorphisms that commute with $(g_t)$. Then $C_{\Homeo(Y)}(g_t) = \set{g_t \of f : t\in \R, f \in F}$ for some finite group $F \subset \Homeo(Y)$.
\end{theorem}

\begin{proof}
Fix some $\phi \in \Homeo(Y)$ that commutes with $(g_t)$. Since $(g_t)$ is a transitive Anosov flow, it has a unique measure of maximal entropy, $\mu$. %In the setting of Anosov flows, the measure $\mu$ was first constructed by Margulis and Bowen using different methods \cite{margulisbook04,bowen72}. The Margulis construction gives a very detailed description of the disintegration of $\mu$ along the stable and unstable manifolds, $W^s$ and $W^u$. Then the measure $\mu$ has decompositions along $W^u$ and $W^s$, $\mu^{W^u}_y$ and $\mu^{W^s}_y$ defined for every $y \in Y$, that vary continuously in $y$ and satisfy $(g_t)_*\mu^{W^u}_y = e^{ht}\mu^{W^u}_{g_t(y)}$ (similarly for $W^s$), where $h$ is the topological entropy of $(g_t)$. Each of the leafwise measures are uniquely defined up to a global scalar (i.e., any other family $\set{\nu^{W^s}_y}$ of disintegrations must satisfy $\nu^{W^s}_y = c \mu^{W^s}_y$). That the leafwise measures are defined everywhere up to a global scalar and vary continuously is a special property of Anosov systems, see 
Let $\set{\mu_x^u}$ and $\set{\mu_x^s}$ denote the families of leaf-wise Margulis measures in Theorem \ref{lem:margulis-conditionals}.
 %Let $\ell$ be the length of the smallest periodic orbit of $(g_t)$, and $\mc O$ be such an orbit. Then $F^n(\mc O) = \mc O$ for some $n$, since the orbits of length $\ell$ must be permuted. Let $G = F^n$, and 
%Note that since $\mu$ is the unique measure of maximal entropy, and this measure is unique, $\phi_*\mu = \mu$. %$\phi$ must preserve the Margulis measure of $(g_t)$, since $\phi_*\mu$ is also a measure of maximal entropy for $(g_t)$. 
%It also preserves the stable and unstable foliations of $(g_t)$ and Lebesgue measure on the orbits of $(g_t)$. This implies that $\phi$ must also act on the measures $\mu^u_x$ and $\mu^s_x$. Since the Margulis measure has local product structure, and $\phi$ preserves the Margulis measures, $\phi_* \mu^u_x = \lambda_x \mu^u_{\phi(x)}$ and $\phi_* \mu^{W^s}_x = \lambda_x^{-1} \mu^{W^s}_{\phi(x)}$ for some fixed $\lambda_x \in \R$. % independent of $x$. 
%Notice also that $\lambda_x$ varies continuously since the measures $\mu_x^{W^*}$ do and is invariant under the flow, so it must be constant by transitivity.
 By Lemma \ref{lem:margulis-rescaling} and \eqref{eq:margulis-cocycle}, we may find a time $t$ of the flow such that $g_{-t} \of \phi$ preserves the Margulis measure, and the leaf-wise Margulis measures.

We claim that $F = \set{ f \in \Homeo(Y) : f_*\mu = \mu, f_*\mu^u_y = \mu_y^u, \mbox{ and }f_*\mu^s_y = \mu_y^s \mbox{ for every } y \in Y}$ is a finite group. Notice that one may use the leafwise measures to define a length on $W^u_y$ and $W^s_y$ for every $y \in Y$. Then define a length metric $\delta$ on $Y$ by declaring the length of a broken path in the $(g_t)$-orbit, $W^u$ and $W^s$-foliations to be the sum of the leaf-wise measures of each leg. Since the leaf-wise Margulis measures are fully supported, the new metrics defined on the stable and unstable foliations induce the same topology on them, and $\delta$ induces the same topology on $Y$. Hence, $(Y,\delta)$ is a connected, compact metric space, so $\Isom(Y,\delta)$ is a compact group \cite{dantzig-vdw}.

Finally, we show that $F := \Isom(Y,\delta) \cap C_{\Homeo(Y)}(g_t)$ is discrete, which is sufficient since any discrete, compact group is finite. Let $\phi \in F$ be $C^0$-close to the identity, so that  for every $x \in X$, $\phi(x)$ to a neighborhood  of $x$ with local product structure. Notice that since $\phi$ commutes with $(g_t)$, $d(g_t(\phi(x)),g_t(x)) = d(\phi(g_t(x)),g_t(x))$ is uniformly bounded above and below for all $t \in \R$.  Since $\phi(g_t(x))$ always lies in a neighborhood of $g_t(x)$ with local product structure, we conclude that that $\phi(x) \in W^{cs}(x) \cap W^{cu}(x)$. Therefore, $\phi$ fixes every orbit of $g_t$. Hence, there exists a continuous function $h : Y \to \R$ such that $\phi(x) = g_{h(x)}(x)$. Since $\phi$ commutes with $g_t$, $h(g_t(x)) = h(x)$. Therefore $h$ is constant since $g_t$ is transitive. It follows that $\phi$ is a time-$t$ map of the flow. By \eqref{eq:margulis-cocycle}, since $\phi \in F$, we conclude that $\phi$ is the identity.%, and we have shown that $g_{-t} \of \phi \in F$.
\end{proof}

\subsection{Rank one factors of $\Z^k$-actions and their suspensions}
\label{sec:susp}
We recall the suspension construction of a $\Z^k$ action, or more generally an $\R^k \times \Z^\ell$ action. Let $\R^k \times \Z^\ell \curvearrowright X$ be a smooth or continuous action, and let $\bar{X} = X \times \R^\ell$. Introduce an equivalence relation $\sim$ on $\bar{X}$, where $(x_1,a_1) \sim (x_2,a_2)$ if and only if there exists $a \in \Z^\ell$ such that $a \cdot x_2 = x_1$ and $a_2 = a_1 + a$. The space $\tilde{X} = \bar{X} / \sim$ is called the {\it suspension space} of $\R^k \times \Z^\ell \curvearrowright X$. A fundamental domain for this equivalence relation is $X \times [0,1]^\ell$. Notice that if $k = 0$ and $\ell = 1$, this is the standard suspension construction for the diffeomorphism generating $\Z$. The general construction is a higher-rank version of this. Notice that the map $x \mapsto (x,0)$ is an embedding of $X$ into $\tilde{X}$, and  the restriction of the $\R^{k+\ell}$ action to $\R^k \times \Z^\ell$ preserves the embedded image and is canonically conjugated to the action on $X$ by construction. Furthermore, the space $\tilde{X}$ carries a canonical $C^r$ structure if the $\R^k \times \Z^\ell$-action is $C^r$, since the relation $\sim$ will be given by $C^r$ diffeomorphisms. Furthermore, the $\R^{k+\ell}$-action on $\tilde{X}$ is $C^r$.

$\mathbb{T}^\ell$ is always a factor of the suspension action, since one may project onto the $\R^\ell / \Z^\ell$ component, and one may therefore factor onto $\R / \Z$ and obtain a rank one factor of the action, which is transitive (in the group theoretic sense). It therefore makes sense to allow Kronecker factors. %The following lemma does not require any hyperbolicity assumptions:

\begin{lemma}
\label{lem:susp}
Let $\R^{k+\ell} \curvearrowright \tilde{X}$ be the suspension of a cone transitive action $\R^k \times \Z^\ell \curvearrowright X$ by $C^r$ diffeomorphisms, $r \ge 1$.  If $\R^k \times \Z^\ell \curvearrowright X$ has a $C^r$ non-Kronecker rank one factor, then so does $\R^{k+\ell} \curvearrowright \tilde{X}$. If $\R^k \times \Z^\ell \curvearrowright X$ is Anosov and $\R^{k+\ell} \curvearrowright \tilde{X}$ has a $C^r$ non-Kronecker rank one factor, then so does $\R^k \times \Z^\ell \curvearrowright X$.
\end{lemma}

\begin{proof}

First, we show that if $\R^k \times \Z^\ell \curvearrowright X$ has a non-Kronecker rank one factor, then $\R^{k+\ell} \curvearrowright \tilde{X}$ also has a non-Kronecker rank one factor. Indeed, let $\pi : X \to Y$ denote the factor map, and $p : \R^k \times \Z^\ell \to A$ denote the corresponding surjective homomorphism (see Definition \ref{def:rank-one-factor}). There are two cases: $A \cong \R \times C$ and $A \cong \Z \times C$. In the latter case, we think of $\Z$ as sitting inside $\R$. Since $C$ is a compact abelian group, we may write $C$ as a product of a torus and groups $\Z / n_i \Z$. Let $C'$ be the product of the torus component of $C$ with a copy of $\mathbb{T}$ for each $\Z / n_i \Z$. Notice that $C$ canonically embeds in $C'$ since $\Z / n_i\Z$ is isomorphic to $\set{k/ni : k \in \Z} \subset \mathbb{T}$.

Since $C'$ is a torus, its universal cover is a vector space, call it $V$ (so that $C' = W / \Lambda$ for some $\Lambda \cong \Z^{\dim W}$). For each generator $e_i \in \Z^\ell$, choose some element $w_i \in \R \times W$ which projects to $p(e_i)$. Then there exists a unique homomorphism from $\R^{k+\ell}$ to $\R \times W$ which sends $e_i$ to $w_i$. Let $\hat{p}$ denote the composition of this homomorphism with the projection to $\R \times C'$, so that $\hat{p} : \R^{k+\ell} \to \R \times C'$ is an extension of $p : \R^k \times \Z^\ell \to A \subset \R \times C'$.

Recall that $C$ was the product of a torus with finitely many groups of the form $\Z / n_i\Z$, $ i = 1,\dots,N$. Let $\tilde{Y}$ be the suspension of the $\Z \times \Z^N$-action on $Y$ in the case when $A = \Z \times C$, or the $\Z^N$-action in the case when $A = \R \times C$, where the $i\tth$ generator of $\Z^N$ acts by the generator of $\Z / n_i\Z$. Notice that after doing so, $\tilde{Y}$ carries a canonical $\R \times C'$-action, into which $A$ embeds. Recall that $\tilde{X}$ is the factor of $X \times \R^\ell$ by the action of $\Z^\ell$ on $X$ in the first coordinates and by translations in the second coordinate. Then define $\tilde{\pi} : \tilde{X} \to \tilde{Y}$ by $\tilde{\pi}(x,a) = (\pi(x),\hat{p}(a))$. We claim that $\tilde{\pi}$ is well-defined. Indeed, if $(x,a)$ and $(y,b)$ represent the same point of $\tilde{X}$, then there exists $c\in \Z^\ell$ such that $y = c \cdot x$ and $a = b + c$. But then $\hat{p}(c) = p(c)$ (since $p$ is an extension of $\hat{p}$), so $\pi(y) = \hat{p}(c) \cdot \pi(x)$ and $\hat{p}(a) = \hat{p}(b) + \hat{p}(c)$ since $\hat{p}$ is a homomorphism. It now follows easily that $\tilde{Y}$ is a  rank one factor of $\tilde{X}$. If $\tilde{Y}$ were a Kronecker factor of $\tilde{X}$, then the $\R \times C'$-action is transitive (in the group-theoretic sense). But then the $A$-action on $Y$ is transitive (in the group-theoretic sense), so $Y$ is Kronecker. This is a contradiction, so $\tilde{Y}$ is a non-Kronecker rank one factor of $\tilde{X}$.

Now we show the converse: if the action is Anosov and the suspension action $\R^{k+\ell} \curvearrowright \tilde{X}$ has a non-Kronecker rank-one factor, then so does the original action $\R^k \times \Z^\ell \curvearrowright X$. Indeed, suppose that $\pi : \tilde{X} \to \tilde{Y}$ is such a factor, and that $A \curvearrowright \tilde{Y}$ is the associated action with projection $p : \R^{k+\ell} \to A$ (see Definition \ref{def:rank-one-factor}). For $t \in \mathbb{T}^\ell$, let $X_t \subset \tilde{X}$ denote the fiber above $t$ of the suspension, so that $X_t$ is canonically diffeomorphic to $X$, and preserved by $\R^k \times \Z^\ell$. Then define $Y_t \subset \tilde{Y}$ to be the image of $X_t$ under $\pi$.

We claim that $Y_0$ is a $C^\infty$ submanifold. Indeed, let $B = \set{ t \in \R^\ell : t \cdot Y_0 = Y_0}$, and $B_0 = \set{ t \in \R^{k+\ell} : t \cdot Y_0 = Y_0}$, so that $B = B_0 \cap \R^\ell$. Since the $\R^k$ action descends to the factor $\tilde{Y}$, it follows that $Y_{t+s} = s\cdot Y_t$, and for any $s \in \mathbb{T}^\ell$, $B = \set{t \in \R^\ell : t \cdot Y_s = Y_s}$. Then $B$ is a closed subgroup of $\R^\ell$ containing $\Z^\ell$, and $\R^\ell / B$ is a factor of $Y$. The factor map is given by sending $Y_s$ to $s \pmod B$. Now, fix a subgroup $V$ transverse to $B^\circ$, the connected component of $B$ in $\R^\ell$.  Then by construction, $Y = \R^\ell \cdot Y_0 = V \cdot Y_0$, and the action of $V$ on $Y$ is locally free, since it must locally permute the fibers. Therefore, the projection map $Y \to \R^\ell / B$  is a submersion, since it takes the tangent space of the $V$-orbits onto the tangent space to $\R^\ell / B$. This implies that $Y_0$, which is the inverse image of 0 under $Y \to \R^\ell / B$, is a smooth submanifold, and $\pi|_{X_0} :  X_0 \to Y_0$ is a $C^r$ submersion.

Now, fix $a \in \R^k\times \Z^\ell$ which is Anosov, so that $TX = E^s_a \oplus \mc O \oplus E^u_a$, where $\mc O$ is the orbit foliation of the $\R^k$-action.  Since $\pi|_{X_0}$ is a submersion, 
\begin{equation}\label{eq:factor-splitting}
TY_0 = \pi_*(E^s_a) + \pi_*(\mc O) + \pi_*(E^u_a).\end{equation}
 Since $a$ desends to a map on $Y_0$, under the corresponding induced map, it follows that nonzero vectors of $\pi_*(E^s_a)$ contract exponentially in the future, nonzero vectors of $\pi_*(\mc O)$ stay bounded and nonzero vectors of $\pi_*(E^u_a)$ contract exponentially in the past. In particular, the sum on $TY_0$ is direct.

Consider $H_0 = p(\R^k \times \Z^\ell) \subset \R \times C$. We now show that  $H_0$ is closed. Indeed, notice that $H_0$ preserves $Y_0$, so it follows that $H := \overline{H_0}$ preserves $Y_0$. Let $\mc H$ be the subbundle   of $TY_0$ tangent to the $H^\circ$-orbits. Since $H$ commutes with the $A$-action, $\mc H$ is invariant under the induced action by $a$ and its vectors stay bounded under its iterates. Since the decomposition \eqref{eq:factor-splitting} is direct, this implies that $\mc H \subset \pi_*(\mc O)$, and hence that the $H^\circ$-orbits are contained in the $p(\R^k)$-orbits. Since $p(\R^k) \subset H_0$, the $H^\circ$-orbits are contained in $H_0$ orbits, so $H_0$ is closed in $\R \times C$.

Finally, we claim that $H_0$ is either $\Z \times C'$ or $\R \times C'$ for some compact group $C'$. Indeed, the only other possibility for a closed subgroup of $\R \times C$ is that $H_0 = C'$ for some compact group $C'$. Since the action of $H_0$ must have a dense orbit, it follows that it must be transitive (in the group action sense). This implies that $\tilde{Y}$ is a Kronecker factor as well, which is a contradiction. Therefore, $Y_0$ is a non-Kronecker factor of $X_0$.

\end{proof}

The following can be thought of as a converse of Lemma \ref{lem:susp}, allowing us to claim that the existence of Kronecker rank one factors is equivalent to being conjugate to the suspension of an action with larger discrete component of the acting group. We do not use this result directly, but provide it to add additional structure to actions with Kronecker factors.

\begin{lemma}
Let $\R^k \times \Z^\ell \curvearrowright X$ be a $C^r$ cone transitive, totally Cartan action for $r = (1,\theta)$ or $\infty$, and $\pi : X \to \mathbb{T}$ be a $C^r$ Kronecker rank one factor. Then there exists a $C^r$ transitive, totally Cartan action $\R^{k-1} \times \Z^{\ell+1} \curvearrowright Y$ which is $C^r$-conjugate to the action of $A = \set{a \in \R^k \times \Z^\ell : a\cdot \pi^{-1}(0) = \pi^{-1}(0)}$ on $\pi^{-1}(0)$.
\end{lemma}

\begin{proof}
Consider the set $Y = \pi^{-1}(0) \subset X$. Since $\pi : X \to \mathbb{T}$ is a factor of the action, there is an associated homomorphism $\sigma : \R^k \to \R$ such that $\pi(a \cdot x) = \sigma(a) + \pi(x)$. If $X$ is the generator of a one-parameter subgroup transverse to $\ker \sigma$, we may conclude that $\pi_*X$ is nowhere vanishing, so $\pi$ is a submersion. Hence, $\pi^{-1}(0)$ is a $C^r$ submanifold. Furthermore, there exists a periodic $\R^k \times \Z^\ell$ orbit by the Anosov Closing Lemma (one must first pass to a suspension, see Theorem \ref{thm:anosov-closing}), so there exists a point $p \in X$ such that a finite index subgroup of $\Z^\ell$ fixes $p$. This implies that a finite index subgroup of $\Z^\ell$ fixes $\pi^{-1}(0)$, and hence $A$ is isomorphic to $\R^{k-1} \times \Z^{\ell+1}$.
Finally, it is clear the the action of $A$ on $\pi^{-1}(0)$ must be cone transitive, and totally Cartan. 
\end{proof}

\
\subsection{Journ\'{e}'s Theorem}

High regularity of conjugacies is a key feature of the rigidity for higher rank hyperbolic actions. We will later see that certain continuous rank one factors of such actions also enjoy high regularity. One of the key ingredients in proving the upgraded regularity is the following theorem of Journ\'{e}:

\begin{theorem}[\cite{journe88}]
\label{thm:journe}
Let $M$ be a smooth manifold, $q \ge 1$ be an integer, and $\theta \in (0,1)$. If $\mc F_1$ and $\mc F_2$ are two transverse foliations of $M$ with uniformly $C^{q,\theta}$ leaves, and $f : M \to \R$ is a function such that the restriction of $f$ to each leaf of $\mc F_1$ and $\mc F_2$ is uniformly $C^{q,\theta}$, then $f$ is $C^{q,\theta}$. 
\end{theorem}

\begin{remark}
While the main theorem of \cite{journe88} is stated for the case of $q = \infty$,  the intermediate regularity cases are also treated in the proof.
\end{remark}

\section{Rigidity of the derivative cocycle}
\label{sec:closing}

%Let $H = \ker \alpha$ be a Lyapunov hyperplane in $\R ^k$.  

Throughout this section, we fix some arbitrary Riemannian metric on $X$, $\abs{\cdot}$.

\subsection{Equicontinuity of Lyapunov hyperplane actions}

The following is not stated explicitly in \cite{KaSp04}, but follows from equation (1) and the related discussion in the proof of Proposition 3.1 of that paper. The proof of the proposition does {\it not} require the assumptions of the main theorem of that paper, but only the totally Cartan assumption.  Notably, one {\it does} require one-dimensionality of the foliations at a crucial part of the argument.

\begin{proposition}[Kalinin-Spatzier, \cite{KaSp04}]
\label{lem:kalinin-lem}
 Fix a totally Cartan action $\R^k \curvearrowright X$.
Let $W^\alpha$ be a coarse Lyapunov foliation with Lyapunov hyperplane $H$. If $a \in H$ does not belong to any other Lyapunov hyperplane, then for every sufficiently small $\beta > 0$ (depending on the H\"older regularity of the coarse Lyapunov foliations), there exists $\ve_0 > 0$ and $T \ge 0$ such that if $t \ge T$ and $x \in X$ satisfy $d(x,(ta) \cdot x) < \ve_0$, then
\begin{equation}
\label{eq:kasp1}
\big|\norm{(ta)_*|_{TW^\alpha}(x)} - 1\big| < d(x,(ta) \cdot x)^\beta.
\end{equation}
\end{proposition}

\begin{proof}[Outline of proof]
We consider a point $x$ which returns to a very small neighborhood of itself as in the statement of the proposition. One proceeds as in the proof of the Livsic theorem, by finding an associated periodic orbit where the derivative is controlled. Unfortunately, the standard higher-rank Anosov closing lemma (Theorem \ref{thm:anosov-closing}) requires hyperbolicity of the return, which is incompatible with the desired conclusion. However, in analyzing the proof of Theorem \ref{thm:anosov-closing}, one may find exactly how much hyperbolicity is needed. The main idea of the proof is to assume that \eqref{eq:kasp1} doesn't hold, and arrive at a contradiction.  Indeed, with sufficient contraction on the coarse Lyapunov distribution $E^\alpha$, one may close the orbit modulo possibly the manifold $W^{-\alpha}$ and some small $\delta \in \R^k$ whose smallness is controlled by $d(x,(ta)\cdot x)$ with H\"older estimates. If \eqref{eq:kasp1} fails, then by possibly looking at the inverse and reversing the roles of $x$ and $(ta) \cdot x$, we get that 

\begin{equation*}
%\label{eq:negation-kasp1}
\norm{(ta)_*|_{TW^\alpha(x)}} < 1 - d(x,(ta)\cdot x)^\beta,
\end{equation*}
 which for sufficiently small $\beta$ gives exactly the required hyperbolicity to build a more nuanced closing lemma. See Proposition 4.1 of \cite{KaSp04}.  Notice that this requires one-dimensionality of $W^\alpha$. % Indeed, one-dimensionality implies that $\norm{(ta)_*|_{TW^\alpha(x)}} \cdot \norm{(-ta)_*|_{TW^\alpha((ta) \cdot x)}} = 1$, so  \eqref{eq:negation-kasp1} is implied by $\norm{(-ta)_*|_{TW^\alpha}((ta) \cdot x)} > 1 + d(x,(ta)\cdot x)^\beta$. 
In general, with multi-dimensional coarse Lyapunov foliations, the condition that $\norm{(ta)_*} > 1 + d(x,(ta)\cdot x)^\beta$ does not give any estimates on the opposite direction.

 Now, choose any $c \in \R^k$ which is Anosov and $\alpha(a) > 0$. Then pick a sequence of times $s_n$ such that $(s_nc) \cdot x$ converges to some point $y$. Since $x$ returns to its own $W^{-\alpha}$-leaf under $ta+\delta$, $(ta+\delta) \cdot y = y$. Thus $y$ is fixed by $ta +\delta$, and $(ta+\delta)_*|_{TW^\alpha(y)} =1$, so $(ta)_*|_{TW^\alpha(y)}$ is very close to 1 (where the closeness is determined by $\norm{\delta}$, which has H\"older estimates). By using Livsic-type arguments, one may further show that the derivative of $ta+\delta$ along $E^\alpha$ at the original point is also close to 1 with H\"older estimates (notice we also require one-dimensionality here, as we rely on the derivative of $a$ along $E^\alpha$ being an abelian cocycle). Thus we recover the original estimate after all.

%First, apply a version closing lemmas with more precise estimates, and other dynamical estimates involving H\"older regularity of distributions. 
The arguments of \cite{KaSp04} work verbatim, with the exception of a normal forms result (Lemma 4.4 of \cite{KaSp04}) which needs a low regularity version (Theorem \ref{thm:normal-forms}). One may also notice that a reference to a paper of Qian \cite{qian94} can be replaced with our proof of the closing lemma (Theorem \ref{thm:anosov-closing}).
\end{proof}

%We summarize the argument here. One may obtain Closing lemmas for elements of $H$ which are sufficiently far away from Weyl chamber walls.

%{   Lemma 4.4 of K-S needs $C^2$.}

%{\color{olive} Need {\it either} this Lemma or the next, not both.
%\begin{lemma}
%\label{lem:ptwise-bounded}
%For every $x \in X$, there exists $A_x > 0$ such that for all $h \in H$, $\| dh|_{W^\alpha} \| \le A_x$.
%\end{lemma}
%}

%Let $H = \ker \alpha$ be a Lyapunov hyperplane in $\R ^k$.  

Proposition \ref{lem:kalinin-lem} combined with the following lemma are the main inputs for the Lemma \ref{lem:uniformly bounded derivative}, which plays a crucial role in Parts III and V of this paper. It establishes equicontinuity of hyperplane actions along their corresponding coarse Lyapunov foliations.

\begin{lemma}
\label{lem:bounded-cocycle}
Let $X$ be a compact metric space $X$, $f : X \to X$ be a homeomorphism, and $\varphi : X \to \R$ be continuous. % preserving $\mc W$ such that $B(x,\delta/2) \cap B(f(x),\delta) =\emptyset$ for some $\delta > 0$ and all $x \in X$.
Assume there exists some $\delta,\ve > 0$ such that  if $y \in X$ and $n \in \N$ satisfy $d(f^n(y),y) < \delta$, then

\begin{equation}
\label{eq:f-bdd-der}
\sum_{i=0}^{n-1} \varphi(f^i(y)) < \ve.
\end{equation}
 Then there exists $S  = S(\delta)$ such that  for every $x \in X$ and ${ N} \in \N$, $\sum_{i=0}^{N-1} \varphi(f^i(x)) \le S  \cdot(\ve + \sup \abs{\varphi})$.
\end{lemma}

\begin{proof}
Fix  $N \in \N$. We define a sequence of times $0 = m_0 < n_1 < m_1 < n_2 < m_2 < \dots$ depending on $N$, and which will terminate when either some $n_k= N$ or $m_k = N$. Define:

\[ n_1 = \inf \left(\set{ n > 0 : \exists m \mbox{ such that } n < m \le N \mbox{ and } d(f^n(x),f^m(x)) < \delta } \cup \set{N}\right). \]

Notice that if $0 \le m < n_1$ and $n \in [0,N]\setminus \set{m}$, then $B(f^n(x),\delta/2) \cap B(f^m(x),\delta/2) = \emptyset$. If $n_1 = N$, we terminate the sequence, otherwise we set

\[ m_1 = \sup \set{ n_1 < m \le N : d(f^{n_1}(x),f^m(x)) < \delta }. \]

As $n_1 \not= N$, then $m_1$ is well defined by definition, and at most $N$. Inductively, we define sequences $n_i$ and $m_i$ using the following (terminating when either $m_i$ or $n_i = N$):

\[ n_{i+1} = \inf \left(\set{ m_i < n  : \exists m > n \mbox{ such that } d(f^n(x),f^m(x)) < \delta } \cup \set{N} \right), \mbox{ and} \]

\[ m_{i+1} = \sup \set{n_{i+1} < m \le N : d(f^{n_{i+1}}{ (x)},f^m(x)) < \delta }. \]

Notice that, similarly to the observation after the definition of $n_1$, if $\ell_1,\ell_2 \in \bigcup_i [m_i,n_{i+1})$, then $B(f^{\ell_1}(x),\delta/2) \cap B(f^{\ell_2}(x),\delta/2) = \emptyset$. Since $X$ is compact, the cardinality of a $\delta/2$-separated set is bounded by some $S$. Therefore, $\sum n_{i+1} - m_i < S$.

We also claim that $\set{f^{n_i}(x)}$ is $\delta/2$-separated. Indeed, if $B(f^{n_i}(x),\delta/2) \cap B(f^{n_j}(x),\delta/2) \not=\emptyset$, then $d(f^{n_i}(x),f^{n_j}(x)) < \delta$. Without loss of generality assume $i < j$. This contradicts the choice of $m_{i+1}$, since $n_i < m_{i+1} < n_j$. Therefore, there are at most $S$ intervals of the form $[n_i,m_i]$.

We break the Birkhoff sum over the first $N$ iterates into intervals $(m_i,n_{i+1})$ and $[n_i,m_i]$. Then we get that:

\[ \sum_{i=0}^{N-1} \varphi(f^i(x)) \le S(\ve +  \sup \abs{\varphi}) \]
where we bound each sum $\sum_{j = n_i}^{m_i} \varphi(f^j(x))$ using \eqref{eq:f-bdd-der}.
\end{proof}

\begin{remark}
One may notice some similarities between our proof and a proof of the Birkhoff ergodic theorem. This is perhaps not surprising, since one obtains that the Birkhoff average of $\varphi$ tends to 0 at least linearly as a consequence of the statement, proving the Birkhoff ergodic theorem for any invariant measure of $f$. One may also notice that the output of Lemma \ref{lem:bounded-cocycle} is one of the assumptions of the Gottschalk-Hedlund theorem, together with the assumption that the system is minimal. While in our application, the map $f$ will not be minimal, we note that this allows for a weakening of the uniform boundedness assumption  in the Gottschalk-Hedlund theorem.
\end{remark}

\bl \label{lem:uniformly bounded derivative}
Let $\R^k \curvearrowright X$ be a $C^{1,\theta}$,   totally Cartan action and $\mc W$ be  a coarse Lyapunov foliation with coarse Lyapunov hyperplane $H$. 
Then there exists $A>0$ such that for  all $h \in H$ the  derivatives $\|  {h_*} |_{T\mc W}\| <A$ are uniformly bounded on $X$  independent of the base point and $h \in H$.  
\el

\begin{proof}
It suffices to show this for one-parameter subgroups as finitely many one-parameter subgroups generate $H$. We may furthermore assume that the one-parameter subgroups are far from the other Weyl chamber walls, so that Proposition \ref{lem:kalinin-lem} applies.   Fix any $e \not= h \in H$ not belonging to any other Lyapunov hyperplane. Notice that by fixing $\ve > 0$, one obtains \eqref{eq:f-bdd-der} for $\log \norm{h_*|_{T\mc W}}$ from Proposition \ref{lem:kalinin-lem}, using $\delta = \ve_0$ and $\ve = {\ve_0}^\beta$. Therefore, we may conclude that $\log \norm{(nh)_*|_{T \mc W}}$ is uniformly bounded for all $n \in \Z$ by Lemma \ref{lem:bounded-cocycle}. Since $\log\norm{(th)_*|_{T\mc W}}$ is uniformly bounded for $t \in [0,1]$ and $\R = \Z + [0,1]$, we get that $\log \norm{(th)_*|_{T \mc W}}$ is uniformly bounded for all $t \in \R$, as claimed.
\end{proof}

 We assume that each coarse Lyapunov foliation $W^\alpha$ is oriented by first passing to double-covers as necessary, and fixing an orientation on each such foliation. Notice that since $\R^k$ is a connected group, every element must automatically perserve the orientation of $W^\alpha$. Recall the normal forms for contracting foliations described in Section \ref{sec:normal-forms}, and let $\psi_x^\alpha$ be the normal forms chart for $W^\alpha(x)$ at $x$ provided by Theorem \ref{thm:normal-forms-foliation}.

\begin{lemma}
\label{lem:identity-return}
 Let $\R^k \curvearrowright X$ be a $C^{1,\theta}$ totally Cartan action. Suppose that $x \in X$, $W^\alpha$ is a coarse Lyapunov foliation, and $H$ is the corresponding Lyapunov hyperplane. Fix any Riemannian metric on $X$. Then

\begin{enumerate}
\item[(a)] if $b_n$ is a sequence of elements in $H$ such that $b_n \cdot x \to x$, and $y \in W^\alpha(x)$, then $b_n \cdot y \to y$, and
\item[(b)] if $b_n$ is a sequence of elements in $H$ such that $b_n \cdot x \to x' \in W^\alpha(x)$, then after passing to a subsequence, $(\psi_x^\alpha)^{-1} \of \lim { b_n|_{W^\alpha(x)}} \of \psi_x^\alpha$ is a translation  (as a map from $\R$ to itself). %and $y \in W^\alpha(x)$, then $b_n \cdot y \to y'$, where $y' \in W^\alpha(x)$ and $A^{-1} d(x,x') < d(y,y') < A \cdot d(x,x')$, where $A$ is as in Lemma \ref{lem:uniformly bounded derivative}. 
\end{enumerate}
\end{lemma}

\begin{proof}
For (a), according to Proposition \ref{lem:kalinin-lem}, the derivative of $b_n$ restricted to the $W^\alpha$-foliation must converge to 1. Therefore, by Theorem \ref{thm:normal-forms}, the corresponding linear map describing the normal forms must converge to $\id$. Since the normal forms and leaves of $W^\alpha$ vary continuously, we get that $b_n \cdot y \to y$ for all $y \in W^\alpha(x)$.

For (b), notice that by Lemma \ref{lem:uniformly bounded derivative}, the maps $(\psi_{b_n\cdot x}^\alpha)^{-1} \of b_n \of \psi_x^\alpha$ converge to a linear map $f_0 : \R \to \R$ such that $\lim b_n\cdot \psi_x^{\alpha}(t) = \psi_{x'}^\alpha(f_0(t))$. Since the normal forms charts at $x$ and $x'$ are related by an affine transformation, we conclude that $f(t) =  (\psi_x^\alpha)^{-1} \of \lim b_n \of \psi_x^\alpha(t)$ is affine. Write $f(t) = mx + b$, and suppose that $m \not= 1$. Then $f$ has a unique fixed point on $\R$, call it $z$. By definition, { $b_n \cdot \psi_x^\alpha(z) \to \psi_x^\alpha(z)$} and by part (a), this implies that $f = \id$, so $m = 1$. Since we assumed that $m \not= 1$, this is a contradiction, and $f$ is a translation  since it preserves the orientation on $W^\alpha$.
%Part (b) follows from a similar argument using Lemma \ref{lem:uniformly bounded derivative} instead.
\end{proof}

\begin{remark}
\label{rem:top-cart-ok}
Throughout the remainder of Section \ref{sec:closing}, we will prove several features of totally Cartan actions. When the action has no rank one factor, we will see that Lemma \ref{lem:uniformly bounded derivative} can be strengthened to yield isometric behavior. We axiomatize the features of such actions later in Definition \ref{def:top-cartan}, which gives a remarkable rigidity result in the $C^0$ category. The results below only use the axioms listed there, and we invite the reader to check this while reading to prevent the need to read twice.
\end{remark}

\subsection{Circular orderings}

 Let $\R^k \curvearrowright X$ be a totally Cartan, cone transitive $C^r$ action, with $r = (1,\theta)$ or $r = \infty$ (alternatively, a topological Cartan action as described in Definition \ref{def:top-cartan}). Fix an Anosov element $a \in \R^k$, then let $\Delta^-(a) = \set{ \beta \in \Delta : \beta(a) < 0}$, and $\Phi \subset \Delta^-(a)$ be a subset. % and distinct for all $\beta \in \Phi$. If $a$ is ultra-regular (Definition \ref{def:hyper-reg}), one may take $\Phi$ to be the set of {\it all} weights $\beta$ for which $\beta(a) < 0$. %Order the roots $\beta_n(a) < \beta_{n-1}(a) < \dots < \beta_1(a) < 0$ which have negative evaluation on $a$. Such a choice can be made since $\beta_1$ is the least negative root.
We introduce  a (not necessarily unique) order on the set $\Phi$. Choose $\R^2 \cong V \subset \R^k$ which contains $a$ and for which $\gamma_1|_V$ is proportional to $\gamma_2|_V$ if and only if $\gamma_1$ is proportional to $\gamma_2$ for all $\gamma_1,\gamma_2 \in \Phi$ (such choices of $V$ are open and dense). Fix some nonzero $\chi \in V^*$ such that $\chi(a) = 0$ ($\chi$ is not necessarily a weight). Then $\beta|_V \in V^* \cong \R^2$ for every $\beta \in \Phi$ and $\Phi|_V = \set{\beta|_V : \beta \in \Phi}$ is contained completely on one side of the line spanned by $\chi$.

\begin{definition}
\label{def:circular-ordering}
The ordering $\beta < \gamma$ if and only if $\angle(\chi,\beta|_V) < \angle(\chi,\gamma|_V)$ is called the {\normalfont circular ordering of $\Phi$ (induced by $\chi$ and $V$)} and is a total order of $\Phi$. If the set $\Phi$ is understood, we let $|\alpha,\beta|$ denote the set of weights $\gamma \in \Phi$ such that $\alpha \le \gamma \le \beta$.
%The order  on the weights in $\Phi$  the angle it makes with $\chi$, where $\chi$ itself having an angle of 0 will be the first element (while the angles may depend on the a choice of metric in $V^*$, the ordering does not). Such an ordering is called a {\it circular ordering (determined by $\chi$)}.
\end{definition}

\begin{remark}
We call this a circular ordering as we order the weights as they appear on the unit circle in $\R^2$, which is a circle. The order can be extended to all weights in $\Delta$, where it becomes a cyclic order. Because we restrict to a contracting submanifold, the circular order for the set of all weights  becomes a total order.
\end{remark}

 While each $\beta \in \Delta$ is only defined up to positive scalar multiple (see Remark \ref{rem:lyap-coefficients1}), this is still well-defined since the circular ordering on $\R^2$ is invariant under orientation-preserving linear maps.  Figure \ref{fig:circ-ordering} exhibits such a circular ordering, after intersecting each $\ker \chi_i$ with the generic 2-space $V$. The lines through 0 represent the kernels of the weights, and the arrows indicate the half space for which $\chi_i$ has positive evaluation.

\begin{figure}
\includegraphics[width=4in]{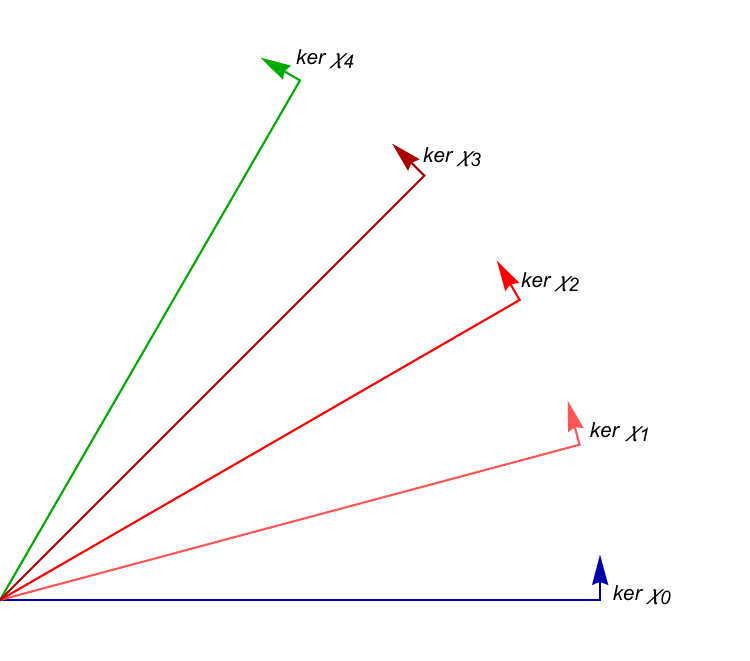}
\caption{Circular Ordering}
\label{fig:circ-ordering}
\end{figure}

\begin{definition}
\label{def:canonical-order}
If $\alpha,\beta \in \Delta$, let $D(\alpha,\beta) \subset \Delta$ (called the {\em $\alpha,\beta$-cone}) be the set of $\gamma \in \Delta$ such that $\gamma = t \alpha + s\beta$ for some $t,s \ge 0$. We may identify $D(\alpha,\beta)$ as a subset of the first quadrant of $\R^2$ by using the coordinates $(t,s)$. The {\em canonical circular ordering} on $D(\alpha,\beta)$ is the counterclockwise order in the first quadrant (where $t$ is horizontal and $s$ is vertical). 
\end{definition}

It is clear that the canonical circular ordering is itself a circular ordering by choosing $V$ in the following way: first choose $a$ and $b$ for which $\alpha(a) = -1$, $\alpha(b) = 0$, $\beta(a) = 0$ and $\beta(b) = -1$. Then perturb them to Anosov elements $a'$ and $b'$ which both contract $\alpha$ and $\beta$. Setting $V$ to be the span of $a'$ and $b'$ and $\chi$ to be any functional for which $\chi(a') = -1$ and $\chi(b') = 0$ obtains the canonical circular ordering on $D(\alpha,\beta)$. Also observe that the canonical circular ordering on $D(\alpha,\beta)$ is the opposite of the canonical circular ordering on $D(\beta,\alpha)$.

Let $a_1,\dots,a_n \in \R^k$ be Anosov elements, and $E^s_{\set{a_i}} = \bigcap_{i=1}^n E^s_{a_i}$ and $W^s_{\set{a_i}}(x) = \bigcap_{i=1}^n W^s_{a_i}(x)$. Each $W^s_{\set{a_i}}$ is a H\"older foliation with $C^r$ leaves by Lemma \ref{lem:coarse-lyapunov}. Let 

\[\Delta^-(\set{a_i}) = \set{\gamma \in \Delta : \gamma(a_i) < 0 \mbox{ for every }i}.\]

\begin{lemma}
\label{lem:root-comb}
Let $\alpha$ and $\beta$ be linearly independent weights and $W_1,\dots,W_m$ be the Weyl chambers of the $\R^k$ action such that $\alpha$ and $\beta$ are both negative on every $W_i$. Then if $a_j \in W_j$ are regular elements in each chamber, $D(\alpha,\beta) = \Delta^-(\set{a_i})$.
\end{lemma}

\begin{proof}
That $D(\alpha,\beta)$ is contained in the right hand side is obvious. For the other, suppose that $\chi$ satisfies $\chi(a_j) < 0$ for every $j = 1,\dots,m$. We first show $\chi$ is a linear combination of $\alpha$ and $\beta$. Suppose that $\chi$ is linearly independent of $\alpha$ and $\beta$. Then there exists $b \in \R^k$ such that $\chi(b) = 1$ and $\alpha(b) = \beta(b) = 0$. Choose any $a_j$, and consider $a_j + tb$. Notice that any Weyl chamber this curve passes through must be one of the $W_i$, since the evaluations of $\alpha$ and $\beta$ will still be negative. But for sufficiently large $t$, $\chi(a_j + tb) > 0$. Hence there is a Weyl chamber such that $\alpha$ and $\beta$ are both negative, but $\chi$ is positive.  This contradicts the fact that $\chi$ is negative on each $W_i$, so by contradiction, $\chi$ must be linearly dependent with $\alpha$ and $\beta$.

So we may write $\chi = t\alpha + s\beta$. We must show that $t,s \ge 0$. Indeed, assume $t < 0$. Since $\alpha$ and $\beta$ are linearly independent, it follows that there exists $b \in \R^k$ such that $\alpha(b) = 1$ and $\beta(b) = 0$. Choosing any $a_j$ from the given elements, the curve $a_j - rb$, $r \in \R$ satisfies $\beta(a_j- rb) = \beta(a_j) < 0$, $\alpha(a_j - rb) = \alpha(a_j) - r < 0$ if $r \ge 0$, and $\chi(a_j - rb) = t\alpha(a_j) - tr + s\beta(a_j)$. Thus, if $r$ is sufficiently large, $\alpha$ and $\beta$ are negative but $\chi$ is positive. Hence there is a Weyl chamber on which $\alpha$ and $\beta$ are negative but $\chi$ is positive. This is a contradiction, so $t \ge 0$. The argument for $s$ is completely symmetric.
\end{proof}

\begin{lemma}
\label{lem:extending-charts}
Let $\alpha,\beta \in \Delta$ satisfy $\alpha \not= \pm \beta$ { (recall Remark \ref{rem:lyap-coefficients1})}, $\set{a_i}$ be Anosov elements such that $\alpha,\beta \in \Delta^-(\set{a_i})$ and $\gamma_1,\dots,\gamma_n$ be all weights of $\Delta^-(\set{a_i})$ strictly between $\alpha$ and $\beta$ in a circular ordering such that $\alpha < \beta$. Let $\delta \in\Delta^-(\set{a_i})$ be either the weight immediately preceding $\alpha$ or following $\beta$. Then 

\[ TW^{|\alpha,\beta|} := TW^\alpha \oplus TW^{\beta} \oplus \bigoplus_{i=1}^n TW^{\gamma_i} \qquad \mbox{and} \qquad TW^{|\alpha,\beta|} \oplus TW^\delta\]

\noindent integrate to foliations $W^{|\alpha,\beta|}$ and $W'$, respectively. Furthermore, if $\psi_x : \R^{n+2} \to W^{|\alpha,\beta|}(x)$ and $\varphi_x : \R \to W^\delta(x)$ are a continuous family of parameterizations, then the map $\Phi_x : \R^{n+2} \times \R \to W'(x)$ defined by letting $\Phi_x(u,t) = z$, where $y = \psi_x(u)$ and $z = \varphi_y(t)$, is a homeomorphism onto its image.
\end{lemma}

\begin{proof}
Notice that we may find regular elements $a_1,\dots,a_m \in \R^k$ such that $W = \bigcap_{i=1}^m W^s_{a_i}$ is a foliation such that $TW^{|\alpha,\beta|} = TW$. We may do the same thing when adding $\delta$, since by assumption it is adjacent to $\alpha$ or $\beta$ in a circular ordering and the entire collection can be placed in a stable manifold. To see that the map $\Phi_x$ is a homeomorphism, observe first that it is onto a neighborhood of $x$ in $W'$ since the foliation $W^\delta$ is transverse to $W^{|\alpha,\beta|}$. We wish to show it is injective. Since $\delta$ is the boundary of a cone (either $\abs{\alpha,\delta}$ or $\abs{\delta,\beta}$), we may find $b \in \ker\delta$ such that $\alpha(b),\beta(b),\gamma_1(b),\dots,\gamma_n(b) < 0$. Assume there exists $(u,t)$ and $(v,s)$ such that $\Phi_x(u,t) = \Phi_x(v,s)$. Then iterating $b$ forward, the $\alpha,\gamma_1,\dots,\gamma_n,\beta$-legs will all contract exponentially, but the $\delta$ leg have its length bounded away from $0$ and $\infty$ by Lemma \ref{lem:uniformly bounded derivative}. By choosing a convergent subsequence, we may conclude that $t = s$. Then since $\psi_x$ is a coordinate chart itself, we conclude $u = v$. Therefore, the map is a local homeomorphism near 0.
That it is a global homeomorphism follows by intertwining with the hyperbolic dynamics.%fact that it can be shown to be a homeomorphism on arbitrarily large balls. To see this, we intertwine the chart $\Phi_x$ with some contracting dynamics and deduce the homeomorphism property on the larger ball from the shrunken one.
\end{proof}

\begin{lemma}[Weak geometric commutators]
\label{lem:geom-commutator1}
Let $\alpha,\beta \in \Delta$ be nonproportional weights and write $D(\alpha,\beta) = \set{\alpha,\gamma_1,\dots,\gamma_n,\beta}$ in the canonical circular ordering. Then given $x \in X$, $x' \in W^\alpha(x)$ and $x'' \in W^\beta(x')$, there exists

{ 
\begin{enumerate}[label=(\alph*)]
    \item a unique $y \in W^\beta(x)$ and $y' \in W^\alpha(y)$ such that $x''$ and $y'$ are connected by a broken path in the $W^{\gamma_i}$-foliations with combinatorial pattern $(\gamma_1,\dots,\gamma_n)$.
    \item a unique $z \in W^\alpha(x'')$, and $z' \in W^\beta(z)$ such that $x$ and $z'$  are connected by a broken path in the $W^{\gamma_i}$-foliations with combinatorial pattern $(\gamma_n,\dots,\gamma_1)$.
\end{enumerate} 
}

The break points constructed depend continuously on $x'$ and $x''$.
\end{lemma}

\begin{remark}
The structures here are reminiscent of a commutator, as given an $W^\alpha$ leg followed by a $W^\beta$ leg, it completes a cycle by adding the ``inverse'' $W^\alpha$-leg and $W^\beta$-leg which are then followed by things which appear in the ``commutator.'' We will see this is related to an actual commutator in the presence of more structure (see Lemma \ref{lem:comm-relation}).
\end{remark}

\begin{proof}
Notice that, by Lemma \ref{lem:root-comb}, $D(\alpha,\beta) = \Delta^-(\set{a_i})$ for some choice of Anosov elements $\set{a_i}$. We may iteratively apply Lemma \ref{lem:extending-charts} to obtain a chart for $W^{|\alpha,\beta|}$. Observe that we may begin with listing the $\gamma_i$-foliations listed in a circular ordering, adding them one at a time, to obtain a chart for $W^{|\gamma_1,\gamma_n|}$, which is obtained by moving along each of the foliations $\gamma_i$, one at a time. Since each $\alpha$ and $\beta$ bounds the collection, we may then move by $\beta$ and then by $\alpha$ to obtain a chart for $W^{|\alpha,\beta|}$ which first moves along the $\gamma_i$ foliations, then the $\alpha$-foliation and finally the $\beta$ foliation. {  That is, by iteratively applying Lemma \ref{lem:root-comb}, we may produce a map which parameterizes the foliation $W^{\abs{\alpha,\beta}}(x)$, $H : \R^n_\gamma \times \R_\beta \times \R_\alpha \to W^{\abs{\alpha,\beta}}$, where moving along a coordinate in the Euclidean model corresponds to moving along the corresponding coarse Lyapunov foliations as long as all coordinates appearing afterwards are 0. For (b), if ${H}^{-1}(x'') = (u,s_1,s_2)$, let $z' = H(u,0,0)$ and $y' = H(u,s_1,0)$. Continuity of the break points follows from the fact that $H$ is a homeomorphism.

For (a), use the coordinates based at $x''$ rather than $x$ and reverse the roles of $\alpha$ and $\beta$.}
\end{proof}

\subsection{Residual Properties of Foliations}

The following result is classical analog of Fubini's theorem, and usually stated for product spaces (see, e.g., \cite{oxtoby80}). Recall that a continuous foliation of a space $X$ is a collection of leaves which are locally homeomorphic a product of a leaf with the transversal. This immediately gives the following:

\begin{lemma}[Kuratowski-Ulam]
\label{lem:kur-ulam}
Let $X$ be a smooth manifold with continuous foliation $\mc F$ and $Y \subset X$ be a residual subset of $X$. If $B \subset X$ is a local transversal disc to $\mc F$, then the set of points $x \in B$ such that $\mc F(x) \cap Y$ is residual in the leaf topology of $\mc F(x)$ is residual in $B$.  
\end{lemma}

We also discuss accessibility properties and their relationship to paths in foliations (recall Remark \ref{rem:paths-no-action}), which will appear later in the discussion.

\begin{definition}
\label{def:sub-accessibility}
Let $\Omega$ index a set of continuous foliations $\set{\mc F_1,\dots, \mc F_n}$ of a smooth manifold $X$. Define $W^\Omega(x) = W^{\mc F_1,\dots,\mc F_n}$ to be the set of all points which can be reached using finitely many broken paths in the foliations $\set{\mc F_i}$.
\end{definition}

The proof of the following is identical to the proof that if a flow has a dense orbit, then the set of dense orbits is residual, so we omit it.

\begin{lemma}
\label{lem:dense-is-gdelta}
The set of points $x \in X$ such that $W^\Omega(x)$ is dense is a $G_\delta$ set. If it is nonempty, it is residual.
\end{lemma}

\subsection{Exact H\"older Metrics}
\label{subsec:metrics-prelim}

 Kalinin and Spatzier proved the following cocycle rigidity theorem in \cite[Theorem 1.2]{KaSp04} which is fundamental to the developments in Part IV.  We remark that this is not a general cocycle result but rather applies specifically to the derivative  cocycle.  The starting point of the argument is Lemma \ref{lem:kalinin-lem}. One then uses a Livsic-like argument to solve the derivative cocycle along a dense $H$-orbit, and extend it continuously to the closure.%: that the derivative cocycle of the Weyl chamber wall along a closed orbit has to be isometric since non-isometric subexponential growth along a periodic orbit is not possible (e.g. since the dynamics at closed orbits is $C^0$-conjugate to its derivative cocycle along the coarse Lyapunov foliation \cite{pugh-shub70}). This requires that  the Lyapunov functionals have the same kernels, independent of the invariant measure in question.  This is one of the places in the argument where we really need many Anosov elements.  It  implies that the Lyapunov functionals for different measures are positively proportional though not necessarily equal. 

\begin{theorem} [Exact H\"older Metrics] \label{thm:cocycle rigidity}
 Let $\alpha$ be a $C^{1,\theta}$, totally Cartan action of $\R ^k$, $k \geq 2 $, on a compact smooth manifold $M$ and $H$ be a Lyapunov hyperplane for a coarse Lyapunov foliation $\mc W$. If $H$ has a dense orbit, then there exists a H\"older  continuous Riemannian metric $g$ on $T\mc W$ and a functional $\chi : \R^k \to \R$ such that for any $a \in \R ^k$
 $$||da(v)||=e^{\chi (a)} ||v|| \qquad \text{ for any } \; v \in  T\mc W.$$ 
\end{theorem}

\begin{remark}
One may notice that the assumptions of \cite[Theorem 1.2]{KaSp04} are stronger than those given here. However, an inspection of the proof reveals that one does {\it not} need an invariant measure or dense orbits of one parameter subgroups: for this theorem in their paper, the conditions listed here are sufficient. %In view of Theorem \ref{thm:main-anosov}, Theorem \ref{thm:main-cartan} is a slightly stronger version of Theorem \ref{thm:cocycle rigidity}.
\end{remark}

\begin{remark}
Theorem \ref{thm:cocycle rigidity} can be regarded as a strengthening of Lemma \ref{lem:uniformly bounded derivative} under the assumption that a Lyapunov hyperplane has a dense orbit. In particular, it implies that the Lyapunov exponent corresponding to $\chi$ for every invariant measure is equal to $\chi$. %In fact, one may replace the transitivity assumptions of Theorem \ref{thm:big-main} with the conclusion of Theorem \ref{thm:cocycle rigidity} for every foliation $\mc W$.
\end{remark}

%\begin{theorem} %\label{thm:cocycle rigidity}
 %Let $\alpha$ be a totally Cartan action of $\R ^k, k \geq 2 $, on a compact smooth manifold $M$ preserving an ergodic probability measure $\mu$ with full support. Suppose that every Lyapunov hyperplane contains a generic one-parameter subgroup with a dense orbit.  Then there exists a \holder  continuous Riemannian metric $g$ on $M$ such that for any $a \in \R ^k$
% $$||a_\ast(v)||=e^{\chi (a)} ||v|| \qquad \text{ for any } \; v \in E_\chi.$$ 
%\end{theorem}

\begin{remark}
\label{def:hyper-reg}
If $\R^k \curvearrowright X$ is a totally Cartan action, and every Lyapunov hyperplane has a dense orbit, the functionals in Definition \ref{def:regular} can be chosen to be those of Theorem \ref{thm:cocycle rigidity}. See the discussion at the start of Part IV. %In this case, we say that $a \in \R^k$ is {\it ultra-regular} if it is regular, and $\alpha(a) \not= \beta(a)$ for all $\alpha\not= \beta \in \Delta$.
\end{remark}

\part[Examples of Cartan Actions and Classification of Affine Actions]{\Large Examples of Cartan Actions and Classification of Affine Actions}
\label{subsec:examples}

In this part, we exhibit several interesting examples which we believe will be useful to the reader, especially by comparing their features with the structures in the arguments for the remainder of this paper. A reader only interested in understanding the rigidity and structural arguments could skip directly to Part III.

In Section \ref{sec:algebraic-exs}, we summarize some well-known classes of Cartan actions, and in Section \ref{sec:affine-classification}, we give a general description of affine Cartan actions. In Section \ref{sec:exotic}, we outline several exotic examples. Some are merely lesser-known example, some are original to this paper, and some are recent developments which exhibit unexpected behavior (see \cite{DWX22} and \cite{vinhage22}).   Theorem \ref{thm:big-headache} gives a structure theorem for Cartan actions with trivial Starkov component, and many of the examples in Section \ref{sec:exotic} show that this is optimal.

\section{The Algebraic Actions} \label{sec:algebraic-exs} We first recall several algebraic actions, which are often called the {\it standard actions}. These are the principle ``building blocks'': (suspensions of) affine $\Z^k$ actions on nilmanifolds, Weyl chamber flows and Anosov flows on 3-manifolds.

\subsection{$\Z^k$ affine actions and their suspensions}
\label{subsubsec:suspensions}

Let $A_1,A_2,\dots,A_k \in SL(n,\Z)$ be a collection of commuting matrices such that $A_1^{m_1}\dots A_k^{m_k} \not= \id$ unless $m_i = 0$ for every $i$. Then there is an associated action $\Z^k \curvearrowright \mathbb{T}^n$ defined by $\mbf m \cdot (v + \Z^n) = A_1^{m_1}A_2^{m_2}\dots A_k^{m_k}v + \Z^n$, which is an action by automorphisms.

Notice $\R^n$ splits as a sum of common generalized eigenspaces for each matrix $A_i$ (where a generalized eigenspace can contain Jordan blocks, and we identify eigenvalues of the same modulus), $\R^n = \bigoplus_{i=1}^\ell E_\ell$. There are $\ell$ different functionals on $\Z^k$, which associate to $m \in \Z^k$ the modulus of the eigenvalue of $A_1^{m_1}\dots A_k^{m_k}$ on $E_i$, $e^{\chi_i(a)}$. If the functionals $\chi_i$ are all nonvanishing, the action is Anosov, and in this case, $\chi_i$ are the Lyapunov functionals for every invariant measure.

If each generalized eigenspace $E_i$ is 1-dimensional and there are no positively proportional Lyapunov functionals, the action is totally Cartan. Notice that the totally Cartan condition implies that there are no Jordan blocks of the action. %If every eigenspace is one-dimensional and the corresponding eigenvalues are positive and not identically 1 for every element of the action, we say that the action is {\it Cartan}.\footnote{This differs from the notion of Cartan defined by other authors, where it is assumed that $k = n-1$. Our condition is implied by $k = n-1$.} Then there is an automorphism action $\Z^k \curvearrowright \mathbb{T}^n$ defined by $\mbf m \cdot (v + \Z^n) = A_1^{m_1}A_2^{m_2}\dots A_k^{m_k}v + \Z^n$.
Analogous actions can be constructed when replacing $\mathbb{T}^n$ by a nilmanifold, see \cite{nantian95}. We will restrict our discussion to tori for simplicity.

We have already described a dynamical suspension construction for $\Z^k$ actions in Section \ref{sec:susp}. Here, we show that in the case of Cartan actions by automorphisms, this is also realized as a homogeneous flow on a solvable group $S = \R^k \ltimes \R^n$. We define the semidirect product structure of $S$ by fitting the $\Z^k$ subgroup of $SL(n,\Z)$ into an $\R^k$ subgroup.  We may assume, by passing to a finite index subgroup as needed, that the eigenvalues of each $A_i$ are all positive real numbers. Therefore, each $A_i$ fits into a one-parameter subgroup, $A_i = \exp(tX_i)$ with $X_i \in \mf{sl}(n,\R)$. Furthermore, since $[A_i,A_j] = e$, $[X_i,X_j] = 0$. %\footnote{This is why $k = n-1$ implies the action is Cartan: indeed, the $X_i$ generate a Cartan subalgebra of $\mf{sl}(n,\R)$.} 
Thus, there exists a homomorphism $f : \R^k \to SL(n,\R)$ such that $f(e_i) = A_i$ for every $i$ (so in particular, $f(\Z^k) \subset SL(n,\Z)$).

We are ready to define the semidirect product structure of $S$. Let $(a_i,x_i) \in \R^k \times \R^n$ for $i = 1,2$, and define

\[ (a_1,x_1) \cdot (a_2,x_2) = (a_1 + a_2,f(a_2)^{-1}x_1+x_2) .\]

Then $\Gamma = \Z^k \ltimes \Z^n$ is a cocompact subgroup of $\R^k \ltimes \R^n$ and $\R^k$ is an abelian subgroup. The translation action by $\R^k$ is a Cartan action, since the common eigenspaces in $\R^n$ are the coarse Lyapunov subspaces. Furthermore, the stabilizer of the subgroup $\mathbb{T}^n \subset S / \Gamma$ is exactly $\Z^k$, and by construction, if $v \in \R^n$ and $a \in \Z^k$:

\[ a \cdot v = f(a)v \cdot a \sim f(a)v .\]

Therefore, the translation action on $S / \Gamma$ is the suspension of the $\Z^k$ action on $\mathbb{T}^n$.

\subsection{Weyl Chamber Flows}
\label{app:weyl-ch-flows}

Let $\mf g$ be a semisimple Lie algebra and $\mf a \subset \mf g$ be an {\it $\R$-split Cartan subalgebra}, which unique up to automorphism of $\mf g$. Cartan subalgebras are characterized by the following: for every $X \in \mf a$, $\ad_X : \mf g \to \mf g$ diagonalizable over $\R$, and is the maximal abelian subalgebra satisfying this property. A semisimple algebra $\mf g$ is called ($\R$-)split if the centralizer of $\mf a$ (ie, the common zero eigenspace of $\ad_X$ for $X \in \mf a$) is $\mf a$ itself.

The semisimple split Lie algebras are well-classified, the most classical example being $\mf g = \mf{sl}(d,\R)$, with Cartan subalgebra $\mf a = \set{ \diag(t_1,t_2,\dots,t_d) : \sum t_i = 0} \cong \R^{d-1}$, which we will address directly now. Other examples include $\mf g = \mf{so}(m,n)$ with $\abs{m-n} \le 1$ and $\mf g = \mf{sp}(2n,\R)$, as well as split forms of the exotic algebras. In what follows, $SL(d,\R)$ may be replaced with an $\R$-split Lie group $G$, with its corresponding objects. The condition that all root spaces (which will be the Lyapunov subspaces) are 1-dimensional is equivalent to $G$ being $\R$-split.

Let $\Gamma \subset SL(d,\R)$ be a cocompact lattice, and define the {\it Weyl chamber flow} on $SL(d,\R)/\Gamma$ to be the translation action of $A = \set{\diag(e^{t_1},\dots,e^{t_d}) : \sum t_i = 0} \cong \R^{d-1}$. Notice that if $Y \in \mf{sl}(d,\R) = T_eSL(d,\R)$, and $a = \exp(X) \in A$ with $X \in \mf a$, then:

\[ da(Y) = \Ad_a(Y) = \exp(\ad_X)Y. \]

Therefore, if $Y$ is an eigenvector of $\ad_X$ with eigenvalue $\alpha(X)$, $da(Y) = e^{\alpha(X)}Y$. The eigenvectors are exactly the elementary matrices $Y_{ij}$, with all entries equal to 0 except for the $(i,j)\tth$ entry, which is 1. For a general split group, it is classical that the eigenspaces are 1-dimensional. By direct computation, if $X = \diag(t_1,\dots,t_d)$, then $\ad_X(Y) = XY-YX = (t_i-t_j)Y$. Therefore, the eigenvalue functionals $\alpha$ are exactly $\alpha(X) = t_i - t_j$. These functionals $\alpha$ are called the {\it roots} of $\mf g$, and are the weights of the Cartan action as considered above.

For a general semisimple Lie group, the translation action of an $\R$-split Cartan subgroup is not Cartan, or even Anosov. However, the centralizer of $\mf a$ is always isomorphic to $\mf a \oplus \mf m$, where $\mf m \subset \mf g$ is the Lie algebra of some compact Lie subgroup $M \subset G$. In this case, $\exp(\mf a)$ descends to a left translation action on the double quotient space $M \backslash G / \Gamma$, and this action is totally Anosov.  However, the subgroup $M$ is always finite if the group is $\R$-split, or equivalently if the Weyl chamber flow is Cartan.

\subsection{Twisted Weyl Chamber Flows}
\label{subsubsec:twisted-ex}

This example is a combination of the previous two examples. Let $G$ be an $\R$-split semisimple Lie group, and $\rho : G \to SL(N,\R)$ be a representation of $G$, which has an induced representation $\bar{\rho} : \mf{g} \to \mf{sl}(N,\R)$. Then $\bar{\rho}$ has a {\it weight} $\alpha : \mf a \to \R$ for each common eigenspace $E \subset \R^N$ for the transformations $\set{\rho(X) : X \in \mf a}$, which assigns to $X$ the eigenvalue of $\rho(X)$ on $E$. We call $\rho$ a {\it Cartan} representation if:

\begin{enumerate}
\item (No zero weights) $\mbf{0}$ is not a weight of $\bar{\rho}$
\item (Non-resonant) no weight $\alpha$ of $\bar{\rho}$ is proportional to a root of $\mf g$%\footnote{In fact, this is implied by (1).}
\item (One-dimensional) the eigenspaces $E_\alpha$ for each weight $\alpha$ are one-dimensional.
\end{enumerate}

Given an $\R$-split semisimple group with Cartan representation $\rho$, we may define a semidirect product group $G_\rho = G \ltimes \R^N$ which is topologically given by $G \times \R^N$, with multiplication defined by:

\[ (g_1,v_1) \cdot (g_2,v_2) = (g_1g_2,\rho(g_2)^{-1}v_1 + v_2). \]

We now assume that $\Gamma$ is a (cocompact) lattice in $G$ such that $\rho(\Gamma) \subset SL(N,\Z)$ (this severely restricts the possible classes of $\rho$ and $\Gamma$ one may take). Then $\Gamma_\rho = \Gamma \ltimes \Z^N \subset G_\rho$ is a (cocompact) lattice in $G_\rho$, and the translation action of the Cartan subgroup $A \subset G$ on $G_\rho / \Gamma_\rho$ is the {\it twisted Weyl chamber flow}. The weights of the action are exactly the roots of $\mf g$ and weights of $d\bar{\rho}$, which by the non-resonance condition implies that the coarse Lyapunov distributions are all one-dimensional.

When $\R^N$ is replaced by a nilpotent Lie group, one may replace the toral fibers with certain nilmanifold fibers. See \cite{nantian95}.

\section{Classification of Affine Cartan Actions}
\label{sec:affine-classification}
 Here we consider  affine actions %$\alpha$ 
of $\R ^k$ by left translations  or more generally $\R^k \times \Z^l $ by left translations and automorphisms on homogeneous spaces $G/\Gamma$ where $G$ is a connected Lie group, and $\Gamma \subset G$ a uniform lattice. Unlike for the rest of the paper, we do not require $k+l \ge 2$ to obtain classification under the homogeneity assumption.
%Recall that the $\alpha$ is called {\em affine} if each element acts via a product of left translations and automorphisms on  $G/\Gamma$.  
Below we describe the structure of the homogeneous space $G/\Gamma$ and the action $\alpha$  in terms of the Levi decomposition.  
%We also show that the connected component of the identity of $\Gamma$ is normal in $G$.  Passing to the quoient by $\Gamma ^0$, we may assume w.l.o.g. that $\Gamma$ is discrete and hence a lattice in $G$. %The latter follows easily from Borel's density theorem for $G$ semisimple but requires more work for a general connected Lie group $G$.
%This follows once we have an additional structural assumption on $G$ and its corresponding action.

% Fix an affine action on a homogeneous space $G / \Gamma$ which we write as  $a \cdot (g\Gamma) = (\varphi(a)\psi_a(g)) \Gamma$, where $\varphi : \R^k \times \Z^l \to G$ and $\psi_a$ is an automorphism of $G$ which fixes $\Gamma$. The action hence lifts to $G$, and has a linear part $E(a) = \Ad(\varphi(a)) \of d\psi_a$. Let $\mf u_+(a)$, $\mf u_-(a)$ and $\mf u_0(a)$ denote the eigenspaces with modulus greater than, less than, and equal to one for $E(a)$, respectively.  such that for every regular $a \in \R^k \times \Z^l$, $\exp(\mf u_*(a))$ acts locally freely on $G / \Gamma$ for $* = +,-,0$

\begin{theorem}
\label{thm:homo-classification}
 Suppose that $G$ is a connected Lie group,  and $\Gamma \subset G$ a cocompact lattice. 
Suppose $\R^k \times \Z^l 
\curvearrowright G / \Gamma$ is an affine %volume preserving
 Cartan action. %, where $\R^k \subset G$.
  Let $S$ and $N$ denote the solvradical and nilradical of $G$, respectively, and $G \cong L \ltimes S$ be a Levi decomposition of $G$ for some semisimple group $L$. Then

\begin{enumerate}
\item $L$ has no compact factors.%
%\item  %the connected component $\Gamma ^0$ of $\Gamma$ is normal in $G$.  Thus $\Gamma$ may and will be assumed to be discrete henceforth.
\item If $\sigma : G \to L$ is the canonical projection, then $\sigma(\Gamma)$ is a lattice in $L$. 
\item $S \cap \Gamma$ is a lattice in $S$.
\item The restriction of the Cartan action to $\R^k$ covers a Weyl chamber flow on $L / \sigma(\Gamma)$, and the action of $\Z^l$ factors through a finite group action by automorphisms of $L$ which preserve $\sigma(\Gamma)$.
\item $N \cap \Gamma$ is a lattice in $N$, and if $\chi$ is a coarse Lyapunov exponent which does not restrict to a root of $L$, then $E^\chi \subset \Lie(N)$.
\end{enumerate}
\end{theorem}

\begin{proof}

%Note that the suspension of an affine $\R^k \times \Z^\ell$ action is again affine.
Note that an $\R^k \times \Z^\ell$ action on $X$ is affine if and only if its suspension action is affine, and the homogeneous structure on the suspension space $\tilde{X}$ contains $X$ as a homogeneous submanifold. Thus  it suffices to argue the case of an affine action by $\R ^k$.

(1), (2),(3) and the first part of (5):  These follow from Corollary 8.28 of \cite{raghunathan1972}, provided $L$ has no compact factors. To see that there are no such factors, observe that since the Cartan action is affine, $a \cdot (g\Gamma) = (f(a)\varphi_a(g)\Gamma)$, where $f : \R^k \times \Z^l \to G$ is some homomorphism and $\varphi_a$ is an automorphism of $G$ preserving $\Gamma$. This implies that $da = \operatorname{Ad}(f(a)) \of d\varphi_a$, which is an automorphism of $\mf g = \Lie(G)$. The sum of the eigenspaces of modulus 1 for this automorphism must be $\Lie(f(\R^k))$, since the action is Cartan. Every automorphism of a compact semisimple Lie group has only eigenvalues of modulus 1. But semisimple groups are not abelian, so the Lie algebra cannot only consist of the directions tangent to the $\R^k$-orbit foliation. Therefore, we may apply the result.

 (4): We claim that $a \cdot Sx = S (a \cdot x)$. Since $S$ is a characteristic subgroup, it is invariant under the automorphism $\varphi_a$ as well as conjugation by an element of $G$. Since the action of $\R^k \times \Z^l$ is a composition of such transformations, we get the intertwining property. In particular, the Cartan action descends to an action on $L / \sigma(\Gamma)$. For any semisimple group, the identity component of $\Aut(L)$ is the inner automorphism group. Therefore, the $\R^k$ component of $\R^k \times \Z^l$ must act by translations, since there are no small $g \in L$ such that $g\sigma(\Gamma)g^{-1} = \sigma(\Gamma)$. Furthermore,  the induced action must have some continuous part if $L \not= \set{e}$. Indeed, since the outer automorphism group of $G$ is finite, if the action were by $\Z^l$, some finite index subgroup of the $\Z^l$ action would be a translation action. But the elements of $G$ may be written as $z \cdot a \cdot u$, where $z$ is in the center of $G$, $a$ is in some Cartan subgroup and $u$ is a nilpotent element commuting with $a$. In particular, the one-parameter subgroup generating $a$ would be a 1-eigenspace of the action which is not contained in the orbit, violating the Cartan condition. Therefore, the action must contain at least one one-parameter subgroup which acts by translations. By the Cartan assumption, $da$ is diagonalizable for every $a$ in this subgroup, so the translations must be by semisimple elements. Furthermore, the action must contain all elements that commute with $a$, and therefore an $\R$-split Cartan subgroup. That is, the action of the continuous part of $\R^k \times \Z^l$ is a Weyl chamber flow and the $\Z^l$ factors through a finite group action by automorphisms fixing the corresponding Cartan subgroup.

To see the second part of (4), we note that the arguments of Proposition 3.13 in \cite{GS1} can be applied to show that if $E^\chi \subset \Lie(S)$, then $E^\chi \oplus \Lie(N)$ is a nilpotent ideal of $\Lie(S)$. This is sufficient for our claim, since the nilradical is exactly the elements of the solvradical which are $\ad$-nilpotent. The setting there is that of an affine Anosov diffeomorphism, rather than an action which may have a subalgebra inside the group generating the action. This leads to some exotic examples (see Section \ref{sec:exotic}).
\end{proof}

\section{Exotic examples of Cartan actions}

\label{sec:exotic}

In this section, we review examples of Cartan actions which exhibit unusual behavior. Many of these examples exhibit features which explain the more technical constructions and assumptions needed in Theorem \ref{thm:big-headache}.

\subsection{Homogeneous actions with Starkov component}
\label{subsec:starkov}

The main examples of Cartan $\R^k \times \Z^l$ actions on solvable groups $S$ are Cartan actions by automorphisms and suspensions of such actions. Indeed, (4) of Theorem \ref{thm:homo-classification} implies that the subgroup transverse to $\mf n$, the nilradical of $S$, is a piece of the $\R^k$-orbit. However, the orbit may also intersect $\mf n$. This may happen in a trivial way, by taking any $\R^k \times \Z^l$ action and its direct product with a transitive translation action on a torus. There are other interesting examples as well, and in each such example, the component of the $\R^k$-action intersecting $\mf n$ is always part of the Starkov component. For instance, one may construct an $\R^2$ action in the following way: let $A : \mathbb{T}^2 \to \mathbb{T}^2$ by a hyperbolic toral automorphism.  Then $A$ also induces a diffeomorphism of the nilmanifold $H /\Lambda$, where $H$ is the Heisenberg group $H = \set{\begin{pmatrix}1 & x & z \\ 0 & 1 & y \\0 & 0 & 1 \end{pmatrix} : x,y,z \in \R}$ and $\Lambda$ are the integer points. The automorphism is given by $(x,y,z) \mapsto (A(x,y),z)$. While the automorphism is not Cartan, letting one copy of $\R$ act by translation along the center and the other by a suspension (as in Section \ref{subsubsec:suspensions}) gives an $\R^2$ Cartan action, since the action will still be normally hyperbolic with respect to the $\R^k$ orbits for a dense set of elements. This action has a rank one factor, but similar examples can be constructed without rank one factors by taking a torus extension of $\mathbb{T}^d$, $d \ge 4$, rather than $\mathbb{T}^2$.

In an analogous way, assume that a Cartan representation of a split Lie group $G$ preserves a symplectic form $\omega$ on $\R^N$. Then one may define a Heisenberg extension $1 \to \R \to H \to \R^N \to 1$ of $\R^N$ by letting $H = \R \times \R^N$ topologically, with multiplication $(t_1,v_1) \cdot (t_2,v_2) = (t_1+t_2+\omega(v_1,v_2), v_1+v_2)$. Then let $\bar{G}_\rho = G \ltimes_\rho H$, where $\rho(g)(t,v) = (t,\rho(g)v)$. Since $\rho$ preserves the symplectic form $\omega$, the extension of $\rho$ to $H$ acts by automorphisms. Then by construction, the Cartan subgroup of $G$ commutes with the center of $H$, and the action of $A \times \R$ is a Cartan action on any cocompact quotient of $\bar{G}_\rho$.

What makes these examples special is that they have nontrivial Starkov component, which is not a direct product with a torus action. The examples described here are extensions of totally Cartan actions by isometric actions, as described in Corollary \ref{cor:factor-by-starkov}. In the homogeneous case, the Starkov component is exactly the directions generated by any center of a Heisenberg group appearing in Proposition 
\ref{prop:symplectic-possibilities}, or splits as a direct product.

%Furthermore, in relation to Remark \ref{rem:heis-invariance}, one may construct $\rho$ in a way to produce other interesting behavior. Indeed, fix any Cartan representation $\rho_0$ of $G$ for which every weight $\alpha$ of the representation has $c\alpha$ {\it not} a weight for any $c \not= 1$. Then $\rho_ 0 \oplus \rho_0^*$, the direct sum of this representation with its dual, is a Cartan representation of $G$. Furthermore, the weights of $\rho_0^*$ are the opposites of the weights of $\rho_0$, and there is an invariant symplectic form which pairs a weight with its opposite. We may construct the corresponding Heisenberg extension of the twisted Weyl chamber flow, and in this case, if $\Delta_{\rho_0}$ is the set of weights for $\rho_0$, this is an ideal in the sense of Definition \ref{def:ideals}. However, one easily sees that while every $\chi \in \Delta_{\rho_0}$ generates the Heisenberg group with its corresponding $-\chi \in \Delta_{\rho_0^*}$, $-\chi$ is not in the ideal.

\subsection{Starkov component arising from mixing semisimple and solvable structures}
The following construction is due to Starkov, and first appeared in \cite{KS1}. Consider the geodesic flow of a compact surface, realized as a homogeneous flow on $SL(2,\R) / \Gamma$, and the suspension of a hyperbolic toral automorphism $A : \mathbb{T}^2 \to \mathbb{T}^2$ which lies in a one-parameter subgroup of matrices in $SL(2,\R)$, $t \mapsto A^t$. We first repeat the construction above: take the central extension of the suspension flow to obtain a flow on a homogeneous space of a group $\tilde{H} = \R \ltimes H$, where $H$ is the Heisenberg group, and where $\R$ acts on $H$ by a hyperbolic map ($t \mapsto A^t$) on a subspace transverse to $Z(H)$, and  trivially on $Z(H) = Z(\tilde{H})$. We assume the lattice takes the form $\Lambda = \Z \ltimes H(\Z)$, and notice that $Z_0 := Z(\tilde{H}) \cap H(\Z)$ is isomorphic to $\Z$. Since $\Gamma \subset SL(2,\R)$ is cocompact, it is the extension of some surface group. Therefore, there are many homomorphisms from $\Gamma$ to $Z(H) \cong \R$, fix such a nontrivial homomorphism $f$ such that the image of $f$ is not commensurable with $Z_0 \cong \Z$. We construct a quotient of the space $SL(2,\R) \times \tilde{H}$ by considering the subgroup which is the image of $\Gamma \times \Lambda$ under the map $(\gamma,\lambda) \mapsto (\gamma,\lambda f(\gamma))$. Call this subgroup $\tilde{\Gamma}$.

Notice that $\tilde{\Gamma}$ is still discrete, since if $(\gamma_n,\lambda_nf(\gamma_n)) \to (e,e)$, then $\gamma_n$ is eventually $e$ since $\Gamma$ is a lattice in $SL(2,\R)$. Hence $\lambda_n$ is eventually $e$ as well. Thus, there is a well-defined, homogeneous $\R^3$ action on $X = (SL(2,\R) \times \tilde{H}) / \tilde{\Gamma}$ which flows along the diagonal in $SL(2,\R)$, the $\R$ direction of $\tilde{H} = \R \ltimes H$ and $Z(H)$ as the three generators of the action. Notice also that the action has a rank one factor, since the projection onto the first coordinate is well-defined. However, the action on $\tilde{H} / \Lambda$ is {\it not} a factor, since the projection of the lattice is dense in the center of $H$ (by the non-commensurability assumption on $f$). 

This example highlights the necessity of assuming that the system has trivial Starkov component in Theorem \ref{thm:big-headache}. %, since it has a Starkov component, namely $Z(H)$.
 Indeed, the action of the Starkov component $Z(H)$ factors through a circle action on $X$, and after quotienting by this circle action, one arrives at the direct product of the flow on $SL(2,\R) / \Gamma$ and the suspension of $A$.

\begin{remark}
One can make this example non-homogeneous by using the unit tangent bundle to a surface of negative, non-constant curvature, rather than $SL(2,\R)$ in this construction. This highlights further the need for quotienting by the Starkov component in the Theorem \ref{thm:big-headache}.
\end{remark}

\subsection{Suspensions of Diagonal Actions}
\label{sec:embedded-ex}

Consider a collection of totally Cartan actions by linear automorphisms $A_i : \Z \curvearrowright \mathbb{T}^{2}$, $i = 1,\dots,k$, which may or may not be related in any way. Then let $2 \le \ell < k$, and choose any homomorphism $f : \Z^\ell \to \Z^k$ such that if $\pi_i : \Z^k \to \Z$ is the projection onto the $i\tth$ coordinate, then $\pi_i \of f$ and $\pi_j \of f$ are both nonzero and nonproportional for all $i\not= j$. Such choices are always possible for $\ell \ge 2$.

Then construct the action of $\Z^\ell$ on $\mathbb{T}^{2k}$ by:

\[ a \cdot (x_1,\dots,x_k) = (A_1^{f(a)_1}(x_1),\dots,A_k^{f(a)_k}(x_k)). \]

Notice that we may suspend the $\Z^\ell$ action to an $\R^\ell$ action. Furthermore, the Lyapunov exponents of the action are exactly $\pi_i \of f$, so by construction the action is totally Cartan. The action has trivial Starkov component, but large centralizer. In fact, the centralizer of the $\R^\ell$ action is exactly $\R^\ell \times \Z^{k-\ell}$, since one may act by any of the automorphisms $A_i$ in its corresponding torus.

The key feature of the action we have constructed is that it has infinite index in its centralizer. This highlights the necessity for the technicalities of Theorem \ref{thm:big-headache}. %: in Theorem \ref{thm:main-self-centralizing}, we obtain the action as a finite extension of a product of standard actions. 
The $\R^\ell$ action we have constructed cannot be described as homogeneous extension of a product of Anosov flows, but instead over some invariant subspace for the product of Anosov flows. If one suspends the full $\Z^k$-action, one obtains an $\R^k$-action which is the direct product of $k$ suspension flows.%, as described by Theorem \ref{thm:main-self-centralizing} when the action is assumed to be virtually self-centralizing.
 The suspension space for this action is locally given by $\R^k \times \mathbb{T}^{2k} = \set{(s_1,\dots,s_k; x_1,\dots,x_k) : s_i \in \R, x_i \in \mathbb{T}^2}$, and the flow is given by translations in the $s_i$ coordinates. The embedding of $\R^\ell$ is rational and therefore $f(\R^\ell)$ is the kernel of some rational homomorphism $g : \R^k \to \R^{k-\ell}$. Thus we see that while the $\R^\ell$ action is not itself a product action, it embeds in a product action. It is exactly the restriction of a product action to some rational subspace, which is precisely the content of Theorem \ref{thm:big-headache}(3), which says that $\pi$ is only a submersion {\it onto its image}.

 This construction generalizes to any totally Anosov $\Z^k$-action. By picking a generic $\Z^2$-subgroup and suspending only this subaction, one may obtain an action with large centralizer. However, if the action has no non-Kronecker rank-one factors, even if the action has a large centralizer, if it is still totally Cartan, conjugacy to a homogeneous example can be achieved by Theorem \ref{thm:big-main}.

\subsection{Non-Totally Anosov Actions}
\label{app:anosov-not-totally}

The classical Katok-Spatzier Conjecture \ref{conjecture:Katok-Spatzier} is stated for Anosov actions, and does not require them to be totally Anosov actions. For $\R^k$ actions, a construction carried out by the second author in \cite{vinhage22} shows that the conjecture is false as stated. There, the construction yields an action which is Cartan, but not totally Cartan, and has {\it no rank one factors}, which implies that the totally Cartan assumption is necessary in Theorem \ref{thm:big-main}. 

In most results in the theory (including the results of this paper), the totally Anosov condition assumed. In fact, the only results in which Anosov and not totally Anosov is assumed are \cite{Hertz,RH-W}, which are for $\Z^k$ actions where a topological conjugacy can be deduced. For rank one actions, it is easy to see that Anosov and totally Anosov are equivalent. Here we give a simple construction to show that for higher-rank actions, Anosov does not imply totally Anosov. %Note that this is not a counterexample to the conjecture, since the example we construct has a rank one factor. 
We will produce an example below, assume it is totally Anosov, and get a contradiction.

%\begin{remark}
%The construction outlined here is related to a construction carried out by the second author in \cite{vinhage22}. There, the construction yields an action which is Cartan, but not totally Cartan, and has {\it no rank one factors}, which implies that the totally Cartan assumption is necessary in Theorem \ref{thm:big-main}.
%\end{remark}

Fix some cocompact $\Gamma \subset SL(2,\R)$, let $X = SL(2,\R) / \Gamma$ and $Y = X \times X$. Let $v_i$, $i = 1,2$ be the vector fields generated by $\begin{pmatrix} 1 & 0 \\ 0 & -1 \end{pmatrix} \in \mf{sl}(2,\R)$ in the first and second factor of $Y$, respectively. Let $\varphi : X \to \R$ be a function which has an average of 0, and not cohomologous to a constant. Let $w_1(x_1,x_2) = v_1 + \varphi(x_1)v_2$ and $w_2(x_1,x_2) = v_2$. Notice that:

\[ [w_1,w_2] = [v_1+\varphi(x_1)v_2,v_2] = -[v_2,\varphi(x_1)v_2] = D_{v_2}\varphi(x_1)v_2 = 0 .\]

Therefore $w_1$ and $w_2$ generate an $\R^2$ action. Let $a = w_1 + sw_2$, where $s$ is such that $\varphi(y) \le s+1$ for every $y \in Y$. Now, notice that the derivative of $a$ on $Y$ still preserves the same unstable bundle for $v_1 + v_2$, and uniformly expands it. Similarly, it contracts the corresponding stable bundle, so $a$ is an Anosov element of the $\R^2$ action.

Now, choose $s,t \in \R$ such that $b = tw_1 + sw_2 = tv_1 + (t\varphi(x_1) + s)v_2$ is an element such that $t\varphi(x_1) + s$ has positive integral on some periodic point $x_1$ of $v_1$ on $X$ and negative integral on another periodic point $x_2$. %Then there exist periodic orbits for which the integral becomes arbitrarily close to 0 by weak-* density of periodic measures and connectedness of the space of invariant measures. 
That is, $y_1 = (x_1,x_1)$ is periodic for the $w_1,w_2$-action and the integral of $x \mapsto s + t\varphi(x)$ over the base periodic orbit is $< -\ve$, and $x_2$ and $y_2$, respectively so that the integral of $s + t\varphi$ over the base is $> \ve$. Let $u = \begin{pmatrix} 0 & 1 \\ 0 & 0 \end{pmatrix}$ be the vector field generating the unstable bundle in the second (fiber) factor. Then at $y_1$, $u$ is contracted by, and invariant under $b$, and any sufficiently small perturbation of $b$ within $\R^2$ must contract $u$. Similarly, at $y_2$, $u$ is expanded by, and invariant under $b$, so any sufficiently small perturbation of $b$ within $\R^2$ must expand $u$. Therefore, $u$ cannot be contained in a coarse Lyapunov subspace, since the coarse subspaces are uniformly expanded or contracted for a dense set of elements, by the totally Anosov property. But since it is invariant under the $w_1,w_2$-action, it must be either part of the action or a coarse Lyapunov subspace if the action were totally Anosov. This is a contradiction, so that action is not totally Anosov.

\subsection{Actions with orbifold, but not manifold, rank one factors}
Throughout, we have assumed that each foliations $W^\alpha$ is orientable. Without this assumption, there are examples of $\R^k$ actions in which the factor $X / W^H$ is an orbifold, but not a manifold (for a definition of $X / W^H$, see Section \ref{sec:case1}). Let $S_1$ and $S_2$ be hyperbolic surfaces, each having an order 2 isometry. For $S_1$, we assume there is an isometry $\bar{f}_1 : S_1 \to S_1$ which is a reflection across a systole separating two punctured tori. Notice that $\Fix(\bar{f}_1)$ is exactly the systole. One can modify this construction to be on any surface as long as $\Fix(\bar{f}_1)$ is a union of closed geodesics.

Now let $\bar{f}_2$ be any involutive isometry of a surface $S_2$ which does {\it not} have fixed points. One such example can be obtained by gluing an even number of twice-punctured tori together in a cycle, then permuting them in the circle exactly halfway around.

Now, let $Y_i = T^1S_i$ be the corresponding unit tangent bundles, and $f_i : Y_i \to Y_i$ be the maps induced by the isometries. Notice that $\Fix(f_1)$ is still a union of circles (exactly the orbits of the geodesic flow tangent to the systole), and that while $\bar{f}_1$ may reverse orientation, $f_1$ always preserves it (since if $\bar{f}_1$ reverses the orientation on $S_1$, it reverses the orientation on the circle bundle as well). Finally, if $Z_0 = Y_1 \times Y_2$ and $f(y_1,y_2) = (f_1(y_1),f_2(y_2))$, then $f$ does not have fixed points and commutes with the $\R^2$ action $(t,s) \cdot (y_1,y_2) = (g_t^{(1)}(y_1),g_s^{(2)}(y_2))$, where $g_t^{(1)}$ and $g_s^{(2)}$ are the geodesic flows on $Y_1$ and $Y_2$, respectively. Therefore, there is a totally Cartan $\R^2$ action on $Z = Z_0 / (z \sim f(z))$.

The Lyapunov foliations of the $\R^2$ actions are exactly the horocyclic foliations in $Y_1$ and $Y_2$, respectively. Therefore, if $\alpha$ is the positive exponent for the flow on $Y_1$ and $\beta$ is the positive exponent for the flow on $Y_2$, $\Delta = \set{\alpha,-\alpha,\beta,-\beta}$. Notice that $W^H(y_1,y_2)$ lifts to $\set{y_1} \times Y_2 \cup \set{f_1(y_1)} \times Y_2 \subset Z_0$ (for a discussion on $W^H$, see Sections . Therefore, in this example $Z / W^H = Y_1 / (y \sim f_1(y))$, which is not a manifold since $f_1$ has fixed points.

\subsection{Some non-product examples}
\label{sec:nonproduct-ex}

In this section, we construct a family of $\R^k$-actions which admit rank one factors, but are not the direct product of such factors and a homogeneous action. Let $G \subset GL(N,\R)$ be an $\R$-split semisimple linear Lie group locally isomorphic to $SL(2,\R)$ such that $G(\Z)$ is a cocompact lattice in $G$, and the inclusion of $G$ into $GL(N,\R)$, which we denote by $\rho_0 : G \to GL(N,\R)$, is a representation of $G$ with one-dimensional weight spaces. The construction of such an algebraic group is nonstandard, see \cite[Section 6.1]{Morris-arithmetic}. Examples with non-uniform lattices are more easily constructed (for instance, the standard representation and its symmetric powers).

The most straightforward example of such a group where $G(\Z)$ is cocompact is the group of matrices preserving the indefinite bilinear form $ax^2 + by^2 - z^2$ on $\R^3$, where $a,b \in \Z$ are such that the only integer solution to $ax^2 + by^2 = z^2$ is $(0,0,0)$. The group of real matrices preserving this form is canonically isomorphic to $SO(2,1)$, and the integer points form a cocompact lattice in its identity component, which is isomorphic to $PSL(2,\R)$.  See \cite[Example 6.1.2]{Morris-arithmetic} for further details and discussion.

\subsubsection{Basic example}
\label{sec:basic-counterex}

 Fix some matrix $A \in SL(2,\Z)$ such that $A$ is hyperbolic and $A = \exp(W)$ for some $W \in \mf{sl}(2,\R)$. We will consider the group

\[ H = (G \times \R) \ltimes_\rho \R^{2N}, \]

where the representation $\rho$ determining the semidirect product structure is given as follows. We write a vector $v \in \R^{2N}$ as $\begin{pmatrix} v_1 \\ \vdots \\ v_N \end{pmatrix}$, where each $v_i \in \R^2$. Then let

\[ \rho(g,t)\begin{pmatrix} v_1 \\ \vdots \\ v_N \end{pmatrix} = \tilde{g}\begin{pmatrix} e^{tW}v_1 \\ \vdots \\ e^{tW}v_N \end{pmatrix} = \tilde{g} \begin{pmatrix} e^{tW} \\ & \ddots \\ & & e^{tW} \end{pmatrix} \begin{pmatrix} v_1 \\ \vdots \\ v_N \end{pmatrix} \]
where $\tilde{g} = \begin{pmatrix} g_{11}\id_2 & \dots & g_{1N}\id_2 \\
                                   g_{21}\id_2 & \ddots & \vdots \\
                                   \vdots & & \vdots \\
                                   g_{N1}\id_2 & \dots & g_{NN}\id_2 \end{pmatrix} \in GL(2N,\R)$ is the $2\times 2$-block form of $\rho_0(g) = \begin{pmatrix} g_{11} & \dots & g_{1N} \\
                                   g_{21} & \ddots & \vdots \\
                                   \vdots & & \vdots \\
                                   g_{N1} & \dots & g_{NN} \end{pmatrix}$. 
                                   
\vspace{.4cm}
Pick a Cartan subgroup $A \subset G$ and notice that by construction, $A$ commutes with $\R$, and that the weight spaces of $\rho|_G$ are all two dimensional, since by assumption, the weight spaces of $\rho_0$ are one-dimensional and each will appear twice since $g \mapsto \tilde{g}$ is by construction the direct sum of the inclusion representation. Furthermore, the set $\Gamma = (G(\Z) \times \Z) \ltimes_\rho \Z^{2N}$ is a discrete, cocompact subgroup, since $e^{nW} = A^n \in SL(2,\Z)$, and hence $\rho(g,t) \in SL(N,\Z)$ for every $(g,t) \in G(\Z) \times \Z$.

Now, the action of $\R$ commutes with $G$ and therefore preserves each of its weight spaces. Furthermore, since it is Anosov, it is hyperbolic on each. In particular, if the unstable eigenvalue of $A \in SL(2,\Z)$ is $e^\lambda$, then the Lyapunov exponents of the $A \times \R$ action are exactly:

\begin{eqnarray}
\bar{\alpha}(a,t) & = & \alpha(a), \mbox{ where $\alpha$ is a root of $G$,} \\
\label{eq:beta-up} \hat{\beta}(a,t) & = & \beta(a) + \lambda t, \mbox{ where $\beta$ is a weight of the representation $\rho_0$, and} \\
\label{eq:beta-down} \check{\beta}(a,t) & = & \beta(a) - \lambda t, \mbox{ where $\beta$ is a weight of the representation $\rho_0$.}
\end{eqnarray}

Notice that these functionals are all nonvanishing on $A \times \R$ (even if the representation $\rho_0$ has a zero weight), that no two of them are positively proportional (even if $\rho_0$ is resonant with the adjoint representation), and that for each such Lyapunov exponent, the corresponding joint eigenspace of the $A \times \R$ action is one-dimensional. Hence the left translation action of $A \times \R$ on $H / \Gamma$ is a totally Cartan action.

\subsubsection{Comparing with Theorem \ref{thm:big-headache}}
\label{sec:big-headache-ex}

Notice that $G$ is a factor of $H$, and the kernel of the projection is exactly $\R \ltimes \R^{2N}$. Furthermore, since the lattice also splits in this way, $G / G(\Z)$ is a factor of $\Gamma$, and the action of $A \times \R$ factors through the action of $A$. In the conclusions of Theorem \ref{thm:big-headache}, there is only one rank one factor, $Y_1 = G / G(\Z)$, which carries an Anosov flow which is the left $A$-action.

The complementary Lie subgroup $\R \ltimes \R^{2N}$ is the group $G_{\Rig}$, and notice that the map $\sigma : A \times \R \to A$ has $\R$ as its kernel exactly, which is contained in $G_{\Rig}$. Furthermore, as this example is homogeneous, the automorphisms $\Phi_a$ are exactly the conjugation action of $A \times \R$ on $\R \ltimes \R^{2N}$. Since the action is by conjugation, part (4) of Theorem \ref{thm:big-headache} follows immediately (in fact when the entire action is homogeneous, it follows immediately that $G_{\Rig}$ is a normal subgroup and that the automorphisms $\Phi_a$ are inner).

It is not true in this case that the action of $\ker \sigma = \R$ on $G_{\Rig}$ is totally Cartan (even though it is totally Anosov). Since it is an $\R$-action, there are only two coarse Lyapunov foliations, which are the stable and unstable foliations of the action. The action of $\R$ is by homotheties on each such foliation. They are each $N$-dimensional, and are invariant under the representation of $G$ on $\R^{2N}$. They are each irreducible subrepresentations, which have one dimensional weight spaces. This is why the action becomes totally Cartan when incorporating the Cartan subgroup of $G$.

Finally, it is clear that the homogeneous space $H/\Gamma$ is not a direct product, since the fundamental group of $G / G(\Z)$ acts nontrivially on $G_{\Rig}$. One may also observe this by noting that if the action were a direct product of the $\R$ action on $G_{\Rig}$ and the $A$ action on $G / G(\Z)$, then $A$ would lie in the kernel of every Lyapunov functional corresponding to a coarse Lyapunov subspace of $G_{\Rig}$. By \eqref{eq:beta-up} and \eqref{eq:beta-down}, this is not true, and in fact cannot be true for any subspace complementary to $\R$ in $\R \times A$.

\subsubsection{Arranging a trivial fiber action}
\label{sec:trivial-fiber-action}

In this section, we modify the construction in Section \ref{sec:basic-counterex} in two ways. First, one may notice that the example of Section \ref{sec:basic-counterex} has a Kronecker factor, since we may project onto the circle factor appearing in the suspension. We may alter the construction to have no Kronecker factors, and in particular, to have that the coarse Lyapunov subgroups generate the group $G$.

Second, in the example of Section \ref{sec:basic-counterex}, $\ker \sigma$ acts on each fiber as an Anosov, but not Cartan action. We will now further disrupt the hyperbolicity properties of the $\ker \sigma$-action on the fiber. In particular, we modify the construction in such a way that $\ker \sigma = \set{0}$. Of course, the action of such a group has no hyperbolicity!

Consider a quaternion algebra

\[ \mathbb{H}^{a,b} = \set{ p + qi + rj + sk : p,q,r,s \in \R}, \]
where multiplication of vectors is determined by the relations 

\[i^2 =a, \qquad j^2 = b, \qquad \mbox{and} \qquad ij = k = -ji. \]

Such an algebra comes equipped with a norm defined by $\norm{x} = x \cdot \bar{x} = p^2 - aq^2 - br^2 + abs^2$, where $\bar{p + qi + rj + sk} = p - qi - rj -sk$ is the conjugate. Let $G^{a,b}  = \set{x \in \mathbb{H}^{a,b} : \norm{x} = 1}$ denote the set of unit elements of $\mathbb{H}^{a,b}$. We have the following.

\begin{proposition}[\cite{Morris-arithmetic}, Proposition 6.2.4]
Assume that $a,b > 0$ satisfy that $(0,0,0,0)$ is the only integer solution of $p^2 - aq^2 - br^2 +abs^2 = 0$. Then:

\begin{enumerate}
\item $G^{a,b} \cong SL(2,\R)$, and
\item $G^{a,b}(\Z) = \set{p + qi + rj + sk : p,q,r,s \in \Z} \cap G^{a,b}$ is a cocompact lattice in $G^{a,b}$.
\end{enumerate}
\end{proposition}

We thank Dave Witte Morris for communicating the following construction to us. Consider the representation of $G^{a,b} \times G^{a,b}$ on $\mathbb{H}^{a,b} \cong \R^4$ by:

\[ \rho(g_1,g_2)(h) = g_1h{g_2}^{-1}. \]

Notice that $\rho$ is an embedding of $G^{a,b}$ into $SL(4,\R)$, and that in the obvious basis $\langle 1,i,j,k \rangle$ of $\mathbb{H}^{a,b}$, the preimage of $SL(4,\Z)$ is exactly $G^{a,b}(\Z) \times G^{a,b}(\Z)$. Furthermore, the algebra $\mathbb{H}^{a,b}$ is isomorphic to the algebra of $2 \times 2$ real matrices, and in this form, the representation is left and right multiplication of $2 \times 2$ matrices by elements of $SL(2,\R)$. Therefore, if $\pm \alpha$ and $\pm \beta$ are the roots of $G^{a,b} \times G^{a,b}$, the weights of the representation $\rho$ are exactly:

\[ \frac{1}{2}\set{\alpha + \beta,\alpha - \beta,-\alpha+\beta,-\alpha-\beta}. \]

In particular, no two roots and/or weights are positively proportional. Consider:

\[ H = (G^{a,b} \times G^{a,b}) \ltimes_\rho \mathbb{H}^{a,b} \]
with the cocompact lattice $\Lambda = (G^{a,b}(\Z) \times G^{a,b}(\Z)) \ltimes_\rho \mathbb{H}^{a,b}(\Z)$. Then choose an $\R$-split Cartan subalgebra $A \subset G^{a,b} \times G^{a,b} \cong SL(2,\R) \times SL(2,\R)$, and consider the left multiplication action on $H / \Lambda$. By construction, the action is totally Cartan.

Finally, we again remark on the features of this action. Notice that the action has two rank one factors corresponding to each of the two possible $G^{a,b} / G^{a,b}(\Z)$ factors of $H / \Lambda$. Accordingly, the space $Y$ appearing in Theorem \ref{thm:big-headache} is exactly $Y = (G^{a,b} \times G^{a,b}) / (G^{a,b}(\Z) \times G^{a,b}(\Z))$. Therefore, the map $\sigma$ is simply the identity, and $\ker \sigma$ is trivial. This is therefore as far from a product as possible. Not only is the action not a product of the action on the rank one factors and the homogeneous action, the restriction of the $\R^2$ action to the fibers (ie, the $\ker \sigma$ action) does not even exist! All hyperbolicity on the rigid fibers comes from the motion along the rank one factors.

\subsubsection{Other variations}

We begin in the same way, by constructing the cocompact lattice $G(\Z)$ in $G \cong SL(2,\R)$. Now, consider another $\R$-split semsimple group $G'$. We say that a representation $\rho_1 : G' \to SL(N',\R)$ is a {\it uniform Cartan representation} if it has the following properties:

\begin{itemize}
\item the weights of the $\rho_1$ are all non-zero (Anosov)
\item each weight space of $\rho_1$ is one-dimensional (Cartan)
\item $\Gamma' = \rho_1^{-1}(SL(N',\Z))$ is a cocompact lattice in $G'$ (Uniform)
\end{itemize}

We thank Dave Witte Morris for communicating the following lemma to us.

\begin{lemma}
\label{lem:sltimessl-cartanrep}
There exist uniform Cartan representations of $SL(d,\R) \times SL(d,\R)$ for every $d \ge 2$.
\end{lemma}

The construction of such a representation relies on division algebras. See \cite[Section 6.8]{Morris-arithmetic} for a summary of division algebras and their relationships to the construction of lattices and representations.

\begin{remark}
The fact that the group $G$ in Lemma \ref{lem:sltimessl-cartanrep} is a product of simple groups is no accident. In fact, as explained in personal communication by Dave Witte Morris, if $G$ is simple, there does not exist a uniform Cartan representation of $G$.
\end{remark}

\begin{proof}[Sketch of proof]
Choose a central division algebra $D$ of degree $d$ defined over $\Q$ which splits over $\R$. These are generalizations of the quaternion algebras that appeared in Section \ref{sec:trivial-fiber-action}. The central condition requires that $\langle 1 \rangle$ is the center of $D$. Such algebras exist, and by a theorem of Wedderburn are all isomorphic over $\R$ to $d \times d$ matrices with real entries (in fact, the $d$ appearing here is how the degree is defined).

By choosing an appropriate basis in $D$ (one which generates a subring called an {\it order}), one may express the algebra as $d \times d$ matrices for which the $\Z$-matrices are a subring. As in Section \ref{sec:trivial-fiber-action}, the group of unit-determinant elements of $D$ are isomorphic to $SL(d,\R)$, and the combined left-and-right multiplication actions determine a uniform Cartan represenation of $SL(d,\R) \times SL(d,\R)$ on $\R^{d^2}$.
\end{proof}

Fix a uniform Cartan representation of a semisimple Lie group $G'$, and proceed as in Section \ref{sec:trivial-fiber-action}: let $\bar{G} = G \times G'$, and consider the representation of $H$ which is $\tilde{\rho} = \rho_0 \otimes \rho_1$ on $\R^{N\cdot N'}$. Then $\tilde{\rho}$ has the three properties listed above, since the weights of a tensor product are sums of weights from $\rho_0$ and $\rho_1$. Therefore, since there are $N$ weights of $\rho_0$ and $N'$ weights of $\rho_1$, there are $N \cdot N'$ distinct weights of $\tilde{\rho}$, so they must be one dimensional. Furthermore, even though $\rho_0$ has a zero weight, $\rho_1$ does not, so there are no zero weights of the representation. The third property follows since it is true for each of $\rho_0$ and $\rho_1$.

Now let $H = \bar{G} \ltimes_{\tilde{\rho}} \R^{N \cdot N'}$, so that $(G(\Z) \times G'(\Z)) \ltimes_{\tilde{\rho}} \Z^{N \cdot N'}$ is a cocompact lattice in $H$. Consider the action of a $\R$-split Cartan subgroup of $\bar{G}$ on $H$. Notice that since the weights of $\tilde{\rho}$ are sums of weights of a representation of $G$ and a representation of $G'$, they cannot be proportional to a root of $G$ or $G'$. Therefore, this action is totally Cartan. Furthermore, the root subgroups of $G$ and $G'$ generate the $\R$-split Cartan subgroup, so $H$ is generated by the coarse Lyapunov subgroups. Finally, this action has a rank one factor again, by factoring onto $G / G(\Z)$. Like the example in Section \ref{sec:basic-counterex}, the action of the subgroup preserving the fibers of this rank one factors, namely the $\R$-split Cartan subgroup of $G'$, is not Cartan, even though the action is virtually self-centralizing.

%\subsubsection{Existence of Cartan Representations}
%\label{sec:cartan-reps}

%If $G$ is a semisimple Lie group, say that a representation $\rho : G \to SL(N,\R)$ is a {\it Cartan representation} if the weights of $\rho$ are all nonzero and the corresponding weight spaces are all one-dimensional. We say that $\rho$ is {\it uniform} if $\rho^{-1}(SL(N,\Z))$ is a cocompact lattice in $G$. 

\subsubsection{Making the example non-homogeneous}

In this section, we outline another construction of an action which is Cartan but not totally Cartan. Like the example of Section \ref{sec:DWX}, this example has a non-homogeneous rank one factor which is not a direct factor of the system. We will rely on the existence of nontrivial cocycles over Anosov flows, making this construction unique to the setting of rank one factors.% It illustrates the necessity of the {\it totally} Cartan assumption in the statement of Theorem \ref{thm:big-headache}.

Consider a totally Cartan translation action $\R^k \curvearrowright G / \Gamma$. Assume that there exists a group $H$ which is locally isomorphic to $SL(2,\R)$, and that there exists a surjective homomorphism $\pi : G \to H$ such that $\Lambda = \pi(\Gamma)$ is a lattice in $H$. Further, assume $\sigma : \R^k \to \R$ is a homomorphism such that $\pi(a \cdot x) = \begin{pmatrix}e^{\sigma(a)} \\ & e^{-\sigma(a)}\end{pmatrix}\pi(x)$. Then $\pi$ and $\sigma$ determine a rank one factor of $G / \Gamma$. Notice that this is exactly the situation appearing in the examples of this section.

Since $H$ is locally isomorphic to $SL(2,\R)$, it is semisimple and there exists a splitting homomorphism $\tau : \Lie(H) \to \Lie(G)$ such that $v' = \tau(v) \in \R^k$, where $v' = \tau(v)$ and $v$ is the generating vector field of the diagonal subgroup of $H$. Then $\R^k$ decomposes as a direct sum of $\ker \sigma$ and $\R \cdot v'$.

Now let $K = \ker \pi$, so that $K$ is a normal subgroup of $G$. Assume further that $\ad(v')$ is nontrivial on $\Lie(K)$ (which is true for the examples constructed in this section). Construct a new family of commuting vector fields in the following way. $\ker \sigma$ remains unchanged and gives rise to a $(k-1)$-dimensional action on $G / \Gamma$. Then let $\varphi : SL(2,\R) / \Lambda \to \R$ be any function which is not cohomologous to a constant, considered as a cocycle over the flow generated by $v$, and set $v''(x) = \varphi(\pi(x))v'$.

We claim that $v''$ and $\ker \sigma$ still commute. Indeed, if $w \in \ker \sigma$, then:
\[ [v'',w] = [(\varphi \of \pi) v',w] = (\varphi \of \pi) [v',w] + (w \cdot (\varphi \of \pi))v' = 0, \]

where $w \cdot f$ denotes the directional derivative of $f$ along $w$. Notice that the first term vanishes since $v'$ and $w$ commute, and the second term vanishes since $w$ moves along $\ker \pi$, over which $\varphi \of \pi$ is constant. Therefore we still have an $\R^k$-action.

We claim that this action is no longer totally Cartan. Indeed, consider the original action, and a weight $\beta$ such that $E^\beta \subset \Lie(K)$ and $\beta(v') \not= 0$ (this was assumed to exist and is a feature of the previous examples). Then we may write $\beta(b,t) = \bar{\beta}(b) + \lambda_\beta t$, where $(b,t)$ is the element $b + tv'$, $b \in \ker \sigma$ and $t \in \R$, $\bar{\beta} = \beta|_{\ker \sigma}$ and $\lambda_\beta \in \R \setminus \set{0}$. Notice also that $\bar{\beta} \not \equiv 0$, since otherwise $\beta$ would be proportional to a weight whose corresponding subspace projects to a stable or unstable distribution of the rank one factor.

If the new action were totally Cartan, then there would be a corresponding weight $\beta'$. Note that since the action has a rank one factor, we no longer have that the weights expand and contract at an exact rate. However, we still have that there exists a unique hyperplane $\ker \beta'$ for which pushforwards of vectors in $E^{\beta'}$ have uniformly bounded derivatives by Lemma \ref{lem:uniformly bounded derivative}.

We may find such hyperplanes at each $\R^k$-periodic orbit. At each periodic orbit $p$, the exponent $\beta'_p$ becomes $\beta'_p(b,t) = \bar{\beta}(b) + \lambda_\beta c_p t$, where $c_p$ is ratio of the periods of the vector fields $v$ and $\varphi v$ at $p$. Since $\lambda_\beta \not= 0$, and $\bar{\beta} \not \equiv 0$, $\beta'_p$ has kernel independent of $p$ if and only if $c_p$ is independent of $p$. By the Livsic theorem, this holds if and only if $\varphi$ is a cohomologous to a constant, which we have assumed it is not. Therefore, since the kernel of $\beta'_p$ is not unique, the higher rank time change is not totally Cartan.

\begin{remark}
If the flow along the $SL(2,\R) / \Lambda$ factor could be perturbed in a way other than a time change, it would be possible to construct an example of this form which is still totally Cartan. A totally Cartan example with rank one factors which are not direct appears in the next section.
\end{remark}

\subsection{Constructing non-homogeneous, non-direct factors}
\label{sec:DWX} 
We briefly summarize a recent example of Damjanovic, Wilkinson and Xu, which appeared in \cite{DWX22}. In spirit, the example is similar to that of the previous section, but remains totally Cartan. They consider a 6-dimensional nilpotent Lie group which is isomorphic to $N = \Heis \times \Heis$, where $\Heis$ is the standard 3-dimensional Heisenberg group. Rather than choosing the usual lattice, which would be to take the $\Z$-points of the usual realization of $\Heis$ as a matrix group, they take a lattice $\Lambda$ which does not respect the product structure in  the description of $N$. This procedure was first introduced by Borel and communicated by Smale in \cite{Smale}, in the discussion following Theorem 3.7 of that paper.

In \cite{DWX22}, the authors consider a partially hyperbolic affine automorphism $F : N / \Lambda \to N /\Lambda$, and the centralizer of its perturbations. $F$ descends to an Anosov automorphism on the maximal torus factor of $N /\Lambda$. The induced map is given by $A \times A$ on $\mathbb{T}^4$, where $A \in SL(2,\Z)$. The fiber bundle structure of $N /\Lambda$ has $\mathbb{T}^2$ fibers, and the map from one fiber to another is isometric, and in particular affine. The centralizer of $F$ contains  a totally Cartan $\Z^2$-action.  Theorem \ref{thm:big-headache} applies to any totally Cartan perturbation of this $\Z^2$-action.  

Like the example of the previous section, the action on $N / \Lambda$ is by $\Z^2$, and the action on the $\mathbb{T}^4$ is also a $\Z^2$ action. Hence, the homomorphism $\sigma$ which appears in Theorem \ref{thm:big-headache} is an isomorphism and the $\ker \sigma$ action is trivial, even after passing to the suspension. Furthermore, among the cases they consider, they show the existence of non-homogeneous perturbations which still belong to totally Cartan $\Z^2$-actions. As promised by Theorem \ref{thm:big-headache}, they conclude using other methods that such actions always have an invariant smooth fibration, and are affine with respect to a smooth homogeneous structure along the fibers.

\vspace{1in}

\phantom{phantom}

\part[Rank One Factors and Transitivity of Hyperplane Actions]{\Large Rank One Factors and Transitivity of Hyperplane Actions}

 Recall that $\R^k \curvearrowright X$ is a cone transitive, $C^{1,\theta}$ or $C^\infty$ totally Cartan action with orientable coarse Lyapunov foliations. We also assume throughout this part that we fix some Riemannian metric with which we measure distances and lengths.

\section{Basic properties of Lyapunov central manifolds}
\label{sec:construction}
 Fix a Lyapunov half-space $\alpha$ with Lyapunov hyperplane $H$. In this section, we build a model for a non-Kronecker rank one factor of $\R^k \curvearrowright X$ assuming that the $H$ action is not transitive.

The plan to construct the factor is as follows: the non-existence of dense $H$-orbits provides the existence of $H$-periodic orbits. We then show that the fixed point sets for the elements of $H$ fixing the periodic point are smooth manifolds. Quotients of these manifolds will provide candidates for the rank one factor.

\subsection{Existence of $H$-periodic orbits}

In this section, we describe certain features of $H$-orbits when there does not exist a dense $H$-orbit. Call a subset $S \subset X$ $\ve$-dense if $\bigcup_{x \in S} B(x,\ve) = X$. Recall that a point $p$ is periodic for a group action $\mc H \curvearrowright X$ if the orbit $\mc H \cdot p$ is compact, and call $\Stab_{\mc H}(p) \subset \mc H$ the periods of $p$ (see Definition \ref{def:periodic}).

We first prove the following very basic topological lemma:

\begin{lemma}
\label{lem:dense-condition}
Let $B \curvearrowright X$ be a group action of a group $B$ by homeomorphisms of a compact metric space $X$. Assume that for every $\ve > 0$, the set $\set{x \in X : B \cdot x\mbox{ is }\ve\mbox{-dense}}$ is dense in $X$. Then there exists a dense $B$-orbit.
\end{lemma}

\begin{proof}
Pick a countable dense subset $\set{x_n : n \in \N} \subset X$. Then define 

\[ Y_{n} = \set{ y \in X : \mbox{for all }1\le \ell \le n,\mbox{ there exists }b \in B\mbox{ such that }d(b\cdot y,x_\ell) < 1/n}.\]

By definition, any $1/n$-dense orbit will be contained in $Y_{n}$, so $Y_{n}$ is dense by assumption. It is also open by construction. Therefore, $\bigcap_{n=1}^\infty Y_n$ is residual, and therefore nonempty by the Baire Category Theorem. It is easy to see that any element of this intersection has a dense orbit.
\end{proof}

\begin{corollary}
\label{cor:H-periodic}
%If there does not exist a dense $H$-orbit, then for every $\ve >0$ there exists an $\R^k$-periodic point $p$ that is $\ve$-dense and $H$-periodic. Furthermore, the set of $H$-periodic orbits are dense.
Let $\R^k \curvearrowright X$ be a cone transitive, totally Cartan action. If there does not exist a dense $H$-orbit, then there exists $\ve >0$ such that every point whose $\R^k$-orbit is $\ve$-dense and periodic  is $H$-periodic. Furthermore, the set of $H$-periodic points are dense.

\end{corollary}

\begin{proof}
Assume otherwise. Then for every $\ve > 0$, choose an $\R^k$-periodic point $p$  such that the $\R^k$-orbit of $p$ is $\ve$-dense, but $p$ is not $H$-periodic, which exists by assumption. Since the $\R^k$-orbit of $p$ is closed, there is a canonical rational structure of $\R^k$ coming from identifying $\Stab_{\R^k}(p)$ with $\Z^k$. Using this rational structure, we see that the $H$-orbit of $p$ is dense inside the $\R^k$-orbit of $p$, since $H$ will correspond to an irrational codimension 1 hyperplane. Since this holds for every $\ve$, the union of such orbits is dense. Therefore, the set of $\ve$-dense $H$-orbits are also dense. %, in particular, the set of $\ve$-dense $H$-orbits are dense.
 Hence, by Lemma \ref{lem:dense-condition}, there exists a dense $H$-orbit, contradicting our assumption.

To see that the set of $H$-periodic orbits are dense, notice that once a point is $H$-periodic, so is every point on its $\R^k$-orbit. Since for every $\ve > 0$, there exists an $H$-periodic orbit whose $\R^k$-orbit is $\ve$-dense by  Corollary \ref{cor:eps-dense}, the set of $H$-periodic orbits are dense.
\end{proof}

%Finally, we shall later see that $H$ factors through a torus action, so we recall a criterion for the quotient by a compact group action to be a manifold.

%\begin{proposition}[\cite{lee13}, Theorem 21.10]
%\label{prop:manifold-factor}
%If $K \curvearrowright X$ is a $C^\infty$ free action of a compact group on a smooth manifold $X$, then the quotient space $X / K$ is a $C^\infty$ manifold.
%\end{proposition}

%Let $p$ denote a point which is both $\R^k$- and $H$-periodic. Let $M^\alpha$ be the saturation of $p$ with $\alpha$ and $-\alpha$ coarse Lyapunov leaves, together with an $\R^k$-orbit. Let $a \in H$ be any element of such that $a \cdot p = p$ and $a$ does not belong to any other Lyapunov hyperplane.

\subsection{Fixed point sets as manifolds}

Recall that, in a weak sense, $H =$ ``$\ker \alpha$,'' even though $\alpha$ may itself not correspond to a Lyapunov exponent. That is, $H$ is the boundary of the half-space determined by $\alpha(a) < 0$ which contracts a foliation $W^\alpha$. In general, $-\alpha$ may not be an element of $\Delta$, if it is not, we omit it from the legs used in constructing paths throughout.

If $x \in X$, let $M^\alpha(x)$ be the set of points which can be reached from $x$ by a broken path in the $\alpha$ and $-\alpha$ (if $-\alpha \in \Delta$) coarse Lyapunov foliations and $\R^k$-orbits.

\begin{proposition}
\label{prop:fix-point-set}
%$M^\alpha$ is the connected component of $\Fix(a)$ containing $p$, and is a $C^\infty$ submanifold of $M$.
Let $\R^k \curvearrowright X$ be a $C^r$ totally Cartan action with $r = (1,\theta)$ or $r = \infty$, and let $a \in H$ not belong to any other Lyapunov hyperplane. Then if $\Fix(a)$ is nonempty, it is a compact, embedded $C^r$ submanifold of $X$, which is $(k+2)$- or $(k+1)$-dimensional, depending on whether  or not $-\alpha \in \Delta$. Furthermore, $\Fix(a) = \bigsqcup_{i=1}^n M^\alpha(p_i)$ for some finite collection $\set{p_i} \subset X$.
\end{proposition}

%Thus, for any point $p$ which is both $\R^k$- and $H$-periodic, $M^\alpha(p)$ is a $C^\infty$ manifold, which we call the {\it Lypuanov central manifold} of $H$ through $p$.

\begin{proof}
Suppose that $p \in \Fix(a)$. If $-\alpha \not\in \Delta$, we let $W^{-\alpha}(p) = \set{p}$ and $E^{-\alpha} = \set{0}$ be the 0-dimensional foliation into points and trivial bundle, respectively. We claim that $W^\alpha(p),W^{-\alpha}(p) \subset \Fix(a)$. Indeed, $W^\alpha$ and $W^{-\alpha}$ are each contracting foliations for some elements $b,-b\in \R^k$. Therefore, there exist normal forms coordinates $\phi_y^\pm : \R \to W^{\pm \alpha}$ which intertwine the dynamics with linear dynamics (see Section \ref{sec:normal-forms}). Since $a \in H$, $\norm{da|_{W^{\pm \alpha}}(p)} = 1$.  Indeed, if not, either $a$ or $-a$ would exponentially contract $W^{\pm \alpha}$, and since $a$ fixes $p$, this would violate Lemma \ref{lem:uniformly bounded derivative}. Therefore, since the dynamics is intertwined with linear dynamics and has derivative equal to 1, it must be the identity. That is, $W^{\pm \alpha}(p) \subset \Fix(a)$. One may see this as a simple case of Lemma \ref{lem:identity-return}, with an exact return rather than a sequence of returns. Therefore, saturating $p$ with $\alpha$, $-\alpha$ and $\R^k$ orbits preserves the property of being fixed by $a$, and $M^\alpha(p) \subset \Fix(a)$.

Use the exponential map and coordinates for $T_pM$ subordinate to the splitting $T_pM = E^u_a \oplus E^s_a \oplus E^0_a$ (where $E^0_a = T\R^k \oplus E^\alpha \oplus E^{-\alpha}$) to conjugate the dynamics of $a$ to a map of the form
\begin{equation} \label{eq:normal-hyperbolicity-loc} F(x,y,z) = (L_1x,L_2y,z) + (f_1(x,y,z),f_2(x,y,z),f_3(x,y,z)), \end{equation}
where $L_1$ is a linear transformation such that $\norm{L_1^{-1}} <  1$, $L_2$ is a linear transformation such that $\norm{L_2} < 1$ and each $f_i : \R^{\dim X} \to \R^{\dim E^*_a}$ (where $* = s,u,0$ appropriately) satisfy $df_i(0) = 0$ and $f_i(0) = 0$. Notice that if a point $(x,y,z)$ is fixed by $F$, then $(L_1 - \id)x = -f_1(x,y,z)$, $(L_2 - \id)y = -f_2(x,y,z)$ and $f_3(x,y,z) = 0$. Therefore, if \[g(x,y,z) = ((\id - L_1)^{-1} f_1(x,y,z)-x, (\id -L_2)^{-1} f_2(x,y,z)-y),\] then the fixed point set lies inside $g^{-1}(0)$.

Near 0, $g^{-1}(0)$ is a submanifold by the submersion theorem, since $g'(0) = \begin{pmatrix} -\id &0 & 0\\0 & -\id & 0  \end{pmatrix}$ in block form. Furthermore, its dimension is exactly $\dim E^0_a = k+2$  (or $k+1$ if $-\alpha \not\in \Delta$). We claim that near $0$, $g^{-1}(0) = \Fix(a)$. That $\Fix(a) \subset g^{-1}(0)$ is clear from  construction, so we must show the opposite. We know that $W^{\pm \alpha}$-leaves and $\R^k$-orbits lie inside $\Fix(a)$. Therefore, we can construct a locally injective map from $\R \times \R \times \R^k \to \Fix(a) \subset g^{-1}(0)$ (or from $\R \times \R^k$ if $-\alpha \not\in \Delta$). By invariance of domain, the map is a local homeomorphism, and $\Fix(a) = g^{-1}(0)$ locally. Since this argument can be done at every point of $\Fix(a)$, we get that $M^\alpha(p)$ is a $C^r$ submanifold of $X$. It is embedded since the size of the neighborhood on which the dynamics is described in coordinates \eqref{eq:normal-hyperbolicity-loc} is uniform in $p$, and the fixed point set $\Fix(a)$ intersects this neighborhood in a unique $k+2$- (or $k+1$-)dimensional disc. Hence, the connected components of $\Fix(a)$ are of the form $M^\alpha(p)$ and they are isolated (again, since the size of the neighborhood in which $\Fix(a)$ intersects in a disc is uniform), so there are finitely many. Hence they are compact, embedded submanifolds (they are compact since $\Fix(a)$ is always a closed set).
\end{proof}

\begin{lemma}
\label{lem:TNS}
If $\alpha \in \Delta$, but $-\alpha \not\in \Delta$, then $H$ has a dense orbit.
\end{lemma}

\begin{proof}
Assume otherwise. Then by Corollary \ref{cor:H-periodic}, there is an $\R^k$- and $H$-periodic point, $p$. Since the set of elements of $H$ fixing $p$ must be a lattice in $H$, we may pick some $a \in H$ not contained in any other Lyapunov hyperplane for which $a \cdot p = p$. By Proposition \ref{prop:fix-point-set}, the connected component of $\Fix(a)$ containing $p$ is a  compact, $(k+1)$-dimensional manifold $M^\alpha(p)$ which is saturated by $\R^k$-orbits. Then pick any one-parameter subgroup $\set{tb} \subset \R^k$ such that $\alpha(b) < 0$. This flow must preserve the compact manifold $M^\alpha(p)$, contract  the $W^\alpha$-leaves, and be isometric on $\R^k$-orbits.

We claim that such flows do not exist on compact manifolds.  Recall that $p$ is  $\R^k$- and $H$-periodic, and hence its orbit under each group is compact. Recall that each $x \in M^\alpha(p)$ is reached from $p$ by a path using $\R^k$-legs, $W^\alpha$-legs and $W^{-\alpha}$-legs. As there are no $W^{-\alpha}$-legs by assumption, $x$ is connected to the $\R^k$-orbit of $p$ by a $W^\alpha$-leaf, since the $\R^k$-orbits and $W^\alpha$-leaves are transverse and complementary, and the $\R^k$-action preserves the $W^\alpha$-foliation (which allows one to exchange the order of $W^\alpha$-legs and $\R^k$-legs). Let $\delta : M^\alpha(p) \to \R$ be defined by letting $\delta(x)$ be the shortest distance along the $W^\alpha$-leaf connecting $x$ to the $\R^k$ orbit of $p$. Since such a connection exists, varies continuously, and the intersections are discrete in the leaf topology of $W^\alpha$ by transversality, $\delta$ is a continuous function. Furthermore, if $\delta(x) \not= 0$, $\delta((na) \cdot x) < \delta(x)$ for sufficiently large $n$ whenever $\alpha(a) < 0$. Thus, the maximum of $\delta$ cannot be positive, since applying $-na$ will yield a larger value. Hence $\delta \equiv 0$, and $M^\alpha(p)$ consists of a single $\R^k$-orbit. This is a contradiction.
\end{proof}

\begin{remark}
Lemma \ref{lem:TNS} is reminiscent of arguments used to study {\it totally nonsymplectic (TNS) actions}. The TNS condition is formulated using Lyapunov exponents with respect to an invariant measure, asking them to never be negatively proportional. One can conclude from this that the corresponding  Weyl chamber wall is ergodic. %Because ergodicity is stronger than transitivity,
The proof is more complicated and uses Pinsker partitions (this condition and argument is now widely used and first appeared in \cite{katok_spatzier1996}). Our proof is a spiritual sibling in the topological category.  
\end{remark}

From now on, we assume that $\pm \alpha \in \Delta$, and that $H$ does not have a dense orbit.

\begin{proposition}
\label{prop:factor-by-H}
Let $\R^k \curvearrowright X$ be a $C^r$, totally Cartan action, with $r = (1,\theta)$ or $r=\infty$. Suppose that $p$ is both $H$- and $\R^k$-periodic. Then the $H$ action on $M^\alpha(p)$ factors through a free torus action. If $Y$ denotes the quotient of $M^\alpha(p)$ by $H$-orbits, then $Y$ has a canonical $C^r$ manifold structure, % has an open dense set with a unique smooth structure such 
and the projection from $M^\alpha(p)$ to $Y$ is $C^r$.
\end{proposition}

\begin{proof}
If $p$ is $H$-periodic, then $\Stab_H(p)$ is a lattice in $H$. One may choose a generating set for $\Stab_H(p)$ which consists only of elements which are not contained in any other Lyapunov hyperplane. By Proposition \ref{prop:fix-point-set}, $M^\alpha(p) = \Fix(a)$ for any such generator, so $\Stab_H(p) \subset \Stab_H(q)$ for all $q \in M^\alpha(p)$. Since $\Stab_H(q)$ contains a lattice, it is itself a lattice. Therefore, a symmetric argument shows that $\Stab_H(q) \subset \Stab_H(p)$. That is, the $H$-action factors through $H / \Stab_H(p) \cong \mathbb{T}^{k-1}$ as a free action. That the quotient space of a free, $C^r$, compact group action is a $C^r$-manifold is classical, see e.g., \cite[Theorem 21.10]{lee13}.
\end{proof}

\section{$W^H$-Holonomy Returns to Lyapunov central manifolds}

The following set will be crucial to our construction of a rank one factor, { which is well-defined in the setting of Anosov actions. We therefore work briefly in that generality, as we expect these tools to be useful later.}

\begin{definition}
{  Let $\R^k \curvearrowright X$ be Anosov.} If $x \in X$, let $W^H(x)$ be the set of points $y$ such that there exists a path $\rho$ based at $x$  along the $W^\beta$-foliations for $\beta \not= \pm \alpha$, and $H$-orbit foliations which ends at $y$. That is, $\rho$ is a sequence $x = x_0,x_1,\dots,x_n = y$ such that $x_i \in W^{\beta_i}(x_{i-1})$ or $x_i \in H \cdot x_{i-1}$ for every $i = 1,\dots,n$ (recall Remark \ref{rem:paths-no-action}).  Such a path is called a $W^H$-path (at $x$). Let $e(\rho) = y$ denote the endpoint of $\rho$ and let $\ell(\rho)$  be the sum of the lengths of the legs. We say a $W^H$-path at $x$ is a {\it cycle} if $e(\rho) = x$.
\end{definition}

{ 
Consider a totally Anosov action $\R^k \curvearrowright X$. As indicated in Section \ref{sec:summary}, if $\alpha \in \Delta$ and $H = \ker \alpha$, the existence of a dense $H$-orbit is intricately related to the existence of a rank one factor, which is important in Parts III and IV. We therefore wish to establish an easier way to detect whether $H$ has a dense orbit. Fix some $a \in H$ which is regular (so that $\beta(a) \not= 0$ for all $\beta \not= \pm \alpha$), and let $R_0$ denote the set of points for which there exist sequences $n_k, m_k \to +\infty$ such that $a^{n_k} \cdot x, a^{-m_k} \cdot x \to x$. Notice that by Lemma \ref{lem:residual-recurrence}, $R_0$ is residual. Then let $R \subset X$ denote the set of points for which $R_0 \cap W^s_a(x)$ and $R_0 \cap W^u_a(x)$ is residual. Notice that by Lemma \ref{lem:kur-ulam}, $R$ is also residual.

\begin{lemma}
\label{lem:WH-closures}
If $\R^k \curvearrowright X$ is a totally Anosov, cone transitive action, with the associated set $R$ described above, then for every $x \in R$, $W^H(x) \subset \overline{H \cdot x}$.
\end{lemma}

\begin{proof}
%First, note that since $H \cdot x \subset W^H(x)$, the inclusion $\overline{H \cdot x} \subset \overline{W^H(x)}$ follows immediately. So we must show that $W^H(x) \subset \overline{H \cdot x}$.

Fix $x \in R$, and let $y \in W^s_a(x) \cap R$. Then there exists $n_k,m_k$ such that $a^{n_k} \cdot x \to x$ and $a^{m_k} \cdot y \to y$. Notice that since $x$ and $y$ are stably related, $a^{n_k} \cdot y \to x$ and $a^{m_k} \cdot x \to y$. Therefore, $y \in \overline{H \cdot x}$. Since $R \cap W^s_a(x)$ is residual in $W^s_a(x)$ (by definition of $R$), $W^s_a(x) \subset \overline{H \cdot x}$. Similarly, $W^u_a(x) \subset \overline{H \cdot x}$. By induction, for any broken path whose break points all belong to $R$, the endpoint of such a path is in $\overline{H \cdot x}$. Since any path can be accumulated by such a path, the claim follows.
\end{proof}

\begin{corollary}
\label{cor:WH-sufficient}
 Under the assumptions of Lemma \ref{lem:WH-closures}, if there exists a point $x$ such that $W^H(x)$ is dense, then there exists a point $x$ such that $H \cdot x$ is dense.
\end{corollary}
}

\subsection{Local Integrability of $M^\alpha$}

{  We return to working with a cone transitive, totally Cartan action $\R^k \curvearrowright X$ with oriented coarse Lyapunov foliations, and assume that a fixed coarse Lyapunov hyperplane $H = \ker \alpha$ does not have a dense orbit  (and hence $-\alpha$ is also a weight by Lemma \ref{lem:TNS}).  We also fix an arbitrary smooth Riemannian metric on $X$ such that the action of $\R^k$ is isometric along its orbits. Choose a line $L$ transverse to $H \subset \R^k$, and an $H$- and $\R^k$-periodic point $p \in X$ which exists by Lemma \ref{cor:H-periodic}. Notice that $L$ preserves $M^\alpha(p)$, where $M^\alpha(p)$ is as in Proposition \ref{prop:fix-point-set}. Let $Y$ be as defined in Proposition \ref{prop:factor-by-H}.}

\label{sec:malpha-integrability}
\begin{lemma}
\label{lem:anosov-factor}
The action of $L$ descends to $Y$ as an Anosov flow.
\end{lemma}

\begin{proof}
That the $L$-action descends to $Y$ follows from the fact that $L$-commutes with $H$, and that $Y$ is the factor of $M^\alpha(p)$ by $H$-orbits. The Anosov property is also clear: $Y$ is a 3-dimensional manifolds with complementary foliations given by $L$-orbits, $W^\alpha$ and $W^{-\alpha}$. Since $\alpha$ and $-\alpha$ are expanded and contracted, respectively, by elements of $L$, the flow is Anosov.
\end{proof}

%\begin{proposition}
%\label{prop:M-moves}
%If $\rho$ is a $W^H$-path based at $y_0$ and $e(\rho) \in M^\alpha(y_0)$, then there is a homeomorphism $h_\rho : Y \to Y$ that commutes with $(a_t)$ such that $e(\rho)$ is a lift of $h_\rho(\bar{y_0})$. Furthermore, for any $\bar{y} \in Y$, if $y$ is a lift of $y$, and $y'$ is a lift of $h_\rho(\bar{y})$, then $y' \in W^H(y)$.
%\end{proposition}

%We prove Proposition \ref{prop:M-moves} using a number of lemmas. Recall the definition of $M^\alpha(x)$ preceding Proposition \ref{prop:fix-point-set}. Fix some Riemannian metric on $X$. Define a map from $\R^{k+2} = \set{(t,s,u) : t,s\in \R, u \in \R^k}$ to $M^\alpha(x)$ by letting $y_t = \psi^\alpha_x(t)$ be the point at (signed) distance $t$ along the $W^\alpha$-foliation from $x$ in the normal forms coordinates of Section \ref{sec:normal-forms}, $z_{t,s} = \psi^{-\alpha}_{y_t}(s)$ be the point at signed distance $s$ along the $W^{-\alpha}$-foliation from $y_t$ (using normal forms coordinates), and $\psi(t,s,u) = u \cdot z_{t,s}$.

Let $a_t$ denote the Anosov flow on $Y$ as in Lemma \ref{lem:anosov-factor}. 
Fix some point $\bar{y_0} \in Y = M^\alpha(p) / H$, let $y_0 \in M^\alpha(p)$ be any lift of $\bar{y_0}$, and consider $W^H(y_0)$. %We let $\rho$ denote some $W^H$-path based at $y_0$.  
Let $\rho$ denote a $W^H$-path based at $y_0$, and $\eta$ denote a path in the $\alpha$,$-\alpha$ and $\R^k$-orbit foliations (with $e(\eta)$ and $\ell(\eta)$ denoting its endpoint and length, respectively).

%Recall that we have passed to a finite cover on which the foliations $W^\beta$ are orientable for every $\beta \in \Delta$.
Recall the normal forms coordinates described in Theorem \ref{thm:normal-forms-foliation}, and for $\pm \alpha$, let $\psi_x^{\pm \alpha} : \R \to W^{\pm \alpha}(x)$ denote the corresponding family of charts.  Fix $x \in X$, and construct a function $\psi_x : \R \times \R \times \R^k \to X$ in the following way: given $t,s \in \R$ and $u \in \R^k$, let $x_1 = \psi_x^\alpha(t)$, the point at (signed) distance $t$ from $x$ along $W^\alpha$ in the normal forms coordinates, $x_2 = \psi_{x_1}^{-\alpha}(s)$, the point at (signed) distance $s$ from $x_1$ along $W^{-\alpha}$ in the normal forms coordinates and $\psi_x(t,s,u) = u \cdot x_2$.  This corresponds to the evaluation map of the path group along a fixed combinatorial pattern (see Remark \ref{rem:paths-no-action}). The following lemma can be thought of as proving $M^\alpha$ to be ``locally topologically integrable.'' That is, local $C^0$ Euclidean structures can be constructed using $\psi$.

\begin{lemma}
\label{lem:psi-coords}
There exists $\ve_0 > 0$ such that for every $x \in X$, if $y \in M^\alpha(x)$ is reached from $x$ by an $(\alpha,-\alpha,\R^k)$-path $\eta$ with $\ell(\eta) < \ve_0$ (with an arbitrary number of legs), then $y = \psi_x(t,s,u)$ for some $(t,s,u) \in \R^{k+2}$ with $\abs{t} + \abs{s} + \norm{u} \le C\ell(\eta)$ for some constant $C$.%, with $C(\ell) \to 0$ as $\ell \to 0$.
\end{lemma}

We introduce the following notation: $M_\ve^\alpha(x)$ denotes $\psi(B(0,\ve))$, the local $M^\alpha$-leaf through $x$.

\begin{proof}
Fix some $a \in H$ which does not belong to any other Lyapunov hyperplane, and let $A$ be as in Lemma \ref{lem:uniformly bounded derivative}. Let $\ve_0$ be less than the injectivity radius of $X$ divided by $A$ and be small enough such that if two points on an unstable or stable leaf have distance  $\delta < 2\ve_0\cdot A$ in the metric on $X$, their distance on the manifold is between $\frac{9}{10}\delta$ and $\frac{11}{10}\delta$. Let $\eta$ be an arbitrary path in the $W^{\alpha}$, $W^{-\alpha}$ and $\R^k$-orbit foliations with $\ell(\eta) < \ve_0$. Note that there may be an arbitrary number of switches between each of the foliations. Put the stable and unstable coarse exponents for $a$ in a circular order, $\set{\beta_1,\dots,\beta_r}$ and $\set{\gamma_1,\dots,\gamma_s}$. Since each of the coarse Lyapunov foliations are transverse and complementary when put together with the $\R^k$-orbit foliations, we may choose a path beginning at $x$ and ending at $e({ \eta})$ which first moves along $\alpha$, then $-\alpha$, then $\R^k$, then $\beta_1$ through $\beta_r$ and finally $\gamma_1$ through $\gamma_s$ (see, eg, \cite[Lemma 3.2]{Schmidt-thesis}), % and compare with Lemma \ref{lem:local-transitivity}) 
shrinking $\ve_0$ as needed.

Assume, for a contradiction, that one of the legs in the $\beta_i$ foliations or $\gamma_i$ foliations is nontrivial. Suppose without loss of generality, that one of the $\set{\gamma_i}$ is nontrivial. Since the $\gamma_i$ are listed in a circular $\eta \subset H$ if anf only if , moving along each in sequence parameterizes the expanding foliation by iteratively applying Lemma \ref{lem:extending-charts}. In particular, once one of them is nontrivial, the points at the start and end of the $\gamma_i$-legs will be distinct points on the same unstable leaf. Then iterating $a$ the unstable leaf will expand, and the stable leaf will contract. Since $a \in H$, the endpoint after moving along the $\alpha$-legs and $\R^k$-orbit will stay within the ball of radius ${ \ell(\eta)}\cdot A$. But a pair of distinct points in the unstable manifold of $a$ will eventually grow to a distance larger than ${ \ell(\eta)}\cdot A$, a contradiction. This implies that only the $\alpha$ and $-\alpha$-legs, and $\R^k$-orbits have nontrivial contribution. That is,that $y$ is in the image of $\psi_x$.

Finally, we show that $(t,s,u)$ is unique and depends in a Lipschitz way on $\ell(\eta)$. Perturb $a$ to a regular element which now expands $\alpha$ and contracts $-\alpha$. Observe that any pair of points can be locally reached uniquely by first moving along the the unstable foliation, then the stable foliation, then the $\R^k$-orbit foliation. This immediately implies uniqueness. The lengths of each leg are {  uniformly Lipschitz} controlled by the distance on the manifold { (see, eg, \cite[Lemma 2]{brin03})}, so we get the bound on the length of the path as described.
\end{proof}

{ 
Fix $x \in X$, and let $\eta$ be an $(\alpha,-\alpha,\R^k)$-path based at $x$. % in the $W^\alpha$, $W^{-\alpha}$ and $\R^k$-orbit foliations. 
Recall that this means that $\eta$ is a sequence of points $x = x_0, x_1,\dots, x_n$ such that $x_{i+1} \in W^{\pm \alpha}(x_i)$ or $x_{i+1} = a_i \cdot x_i$ for some $a_i \in \R^k$. If each $a_i \in L$, we say that it is an $(\alpha,-\alpha,L)$-path.

\begin{lemma}
\label{lem:eta-LPS}
    Let $\eta_1,\eta_2$ denote sufficiently short $(\alpha,-\alpha,L)$-paths based at some $x \in M^\alpha(p)$. Assume that each $\eta_i$ has a single $\alpha$, a single $-\alpha$ and a single $L$-leg, appearing in the same order. Then if $e(\eta_1) \in H e(\eta_2)$, $\eta_1 = \eta_2$. In particular, if $e(\eta_1) = x$, $\eta_1$ is the trivial path. 
\end{lemma}

\begin{proof}
Consider the projection of $\eta_1,\eta_2$ to $M^\alpha(p)/H$, which has the structure of an Anosov flow by Lemma \ref{lem:anosov-factor}. Since the $\alpha$-leg, $-\alpha$-leg and $L$-leg project to the unstable, stable, and orbit leg for the Anosov flow on $M^\alpha(p)/H$, the local product structure implies that the projected paths coincide. But since we restrict the orbit to $L$, the legs must coincide on $M^\alpha(p)$ as well. 
\end{proof}
}

\subsection{The $W^H$-Holonomy Action}
\label{sec:holonomy-action}

Recall Section \ref{sec:slow-foliations} and the notion of a slow foliation within a  foliation $\mc F$ tangent to some sum of bundles $E^\beta$, where $\beta$ ranges over some $\Omega \subset \Delta^+(a)$ for some regular $a \in \R^k$ (each such foliation 
 %has at least one 
{  is the} slow foliation for {  at least one} Anosov element by  Lemma \ref{lem:slow-foliations}).

\begin{lemma}
\label{lem:smooth-holonomy}
%For every $\alpha,\beta \in \Delta$, there exists a regular $a \in \R^k$ such that $\alpha,\beta \in \Delta^+(a)$ and $\alpha$ is the slow foliation of $a$. %If $W^\alpha$ is the slow foliation of $a$, then $\widehat{E^\alpha}$ is integrable to a foliation $\widehat{\mc W^\alpha}$, and 
Let $a \in \R^k$ be an Anosov element, $\Omega \subset \Delta^+(a)$ be a collection of weights whose sum integrates to a foliation $\mc F$, and assume that $W^\alpha$ is the slow foliation for $a$ in $\mc F$. If $y \in \widehat{\mc W^\alpha}(x)$, then the induced holonomy map $\mc H^\alpha_{x,y} :  W^\alpha(x) \to  W^\alpha(y)$ is $C^r$, where $r = (1,\theta)$ or $\infty$ if the dynamics is $C^r$.
\end{lemma}

\begin{proof}  The holonomy is well-defined by the transverse intersection property of Lemma \ref{lem:slow-foliations}. $\widehat{\mc W^\alpha}$ is $C^r$ along $\mc F$ by the $C^r$-section theorem of \cite{HPS1970} and the fact that the growth of the slow foliation can be made arbitrarily slow by Lemma \ref{lem:uniformly bounded derivative} (its application to this setting can be found, for instance, in \cite[Proposition 3.9]{KalSad07}). %[{  Kalinin-Sadovskaya, '07, Prop 3.9}]). 
The regularity of  $\widehat{\mc W^\alpha}$  within $\mc F$ implies that the holonomies have the desired differentiability. %{   Reference?} In general the regularity of the holonomy depends on how close $\norm{a_*|_{E^\alpha}}$ is to 1, but by choosing a sufficiently adapted norm, and sufficiently small perturbation of $a_0$, this closeness can be made arbitrarily small. In particular, the holonomy is $C^r$ for every $r$ and therefore $C^\infty$.
\end{proof}

We know even more about the structure of the holonomy maps, namely that they are linear in the system of local forms coordinates $\psi^\alpha_x$ for the foliation $W^\alpha$ defined in Section \ref{sec:normal-forms}. Let $m_\lambda : \R \to \R$ denote multiplication by $\lambda$.

\begin{lemma}
\label{lem:linear-holonomy}
Under the same assumptions as Lemma \ref{lem:smooth-holonomy}, $(\psi_y^\alpha)^{-1} \of \mc H^\alpha_{x,y} \of \psi_x^{\alpha} = m_\lambda$ for some $\lambda \in [A^{-2},A^2]$, where $A$ is as in Lemma \ref{lem:uniformly bounded derivative}.
\end{lemma}

\begin{proof}
We establish some basic identities relating the dynamics of $\R^k$, normal forms and holonomy maps $\mc H^\alpha_{x,y}$.  First, the normal forms coordinates give that for any $a \in \R^k,z \in X$, $a\psi_z^\alpha = \psi_{a\cdot z}^\alpha \of L_a(z)$, where $L_a(z) : \R \to \R$ is multiplication by the (norm of the) derivative of $a$ restricted to the $\alpha$-leaves at $z$, since { $\norm{(\psi_z^\alpha)'(0)} = 1$ and $\psi_z^\alpha$ preserves orientation}. Notice also that since the stable holonomies are defined through invariant foliations, for any $a \in \R^k$ and $x_1,x_2 \in X$, $a \mc H_{x_1,x_2} a^{-1} = \mc H_{a\cdot x_1,a \cdot x_2}$. Furthermore, if $x_n, y_n \in X$ satisfy $y_n \in \widehat{W^\alpha}(x_n)$, and $d(x_n,y_n) \to 0$, then $\mc H_{x_n,y_n} \to \id$.

Let $a \in \R^k$ be as in Lemma \ref{lem:smooth-holonomy}, and perturb $a$ to $a_0 \in \ker \alpha$ which still contracts every $\gamma \in \Delta^+(a) \setminus \alpha$ { (this $a_0$ is as in proof of Lemma \ref{lem:slow-foliations}, and we may assume $a$ came from this construction)}.
Notice that the sequences $L_{a_0^n}(x)$ and $L_{a_0^n}(y)$ are multiplications by constants bounded below and above by $A^{-1}$ and $A$ by Lemma \ref{lem:uniformly bounded derivative}.  By choosing subsequences, we may assume that $L_{a_0^{n_k}}(x) \to m_{\lambda_1}$, $L_{a_0^{n_k}}(y) \to m_{\lambda_2}$ and that ${a_0}^{n_k}\cdot x,{a_0}^{n_k} \cdot y \to z$ for some $z \in X$ (notice that the iterates of $x$ and $y$ must converge to the same point since they are in the same leaf of $\widehat{W^\alpha}(x)$). Notice that $A^{-1} \le \lambda_1,\lambda_2 \le A$. Putting this all together, we get that:

\begin{eqnarray*}
(\psi^\alpha_y)^{-1} \of \mc H_{x,y} \of \psi^\alpha_x & = & (\psi^\alpha_y)^{-1} \of a_0^{-n_k}\mc H_{a_0^{n_k} \cdot x,a_0^{n_k} \cdot y} a_0^{n_k}\of \psi^\alpha_x \\
& = & L_{{a_0}^{n_k}}(y)^{-1} \of (\psi_{{a_0}^{n_k}\cdot y}^\alpha)^{-1} \of \mc H_{a_0^{n_k}\cdot x, a_0^{n_k} \cdot y} \of \psi_{{a_0}^{n_k} \cdot x}^\alpha \of L_{{a_0}^{n_k}}(x) \\
 & \to & m_{\lambda_1/\lambda_2}
\end{eqnarray*}

\noindent since the holonomy maps $\mc H_{a_0^{n_k}\cdot x, a_0^{n_k} \cdot y}$ converge to $\id$ and the normal forms charts $\psi^\alpha_{{a_0}^{n_k} \cdot x}$ and $\psi^\alpha_{{a_0}^{n_k} \cdot y}$ both converge to $\psi^\alpha_z$ by continuous dependence. The bounds on $\lambda_1$ and $\lambda_2$ imply the bounds in the statement of the Lemma.
\end{proof}

\begin{remark}
Lemma \ref{lem:linear-holonomy} is not a consequence of the usual centralizer theorem for the normal forms of dynamical systems since  there is usually no guarantee that the holonomy maps can be extended globally to a fibered extension commuting with the dynamics. It is also crucial to the remaining arguments that the bound on its derivative does not depend on the closeness of $x$ and $y$ along $\widehat{W^\alpha}$.
\end{remark}

Fix $x \in X$, and let $\eta$ be an $(\alpha,-\alpha,\R^k)$-path based at $x$. % in the $W^\alpha$, $W^{-\alpha}$ and $\R^k$-orbit foliations. 
%Recall that this means that $\eta$ is a sequence of points $x = x_0, x_1,\dots, x_n$ such that $x_{i+1} \in W^{\pm \alpha}(x_i)$ or $x_{i+1} = a_i \cdot x_i$ for some $a_i \in \R^k$. 
If $\beta \in \Delta$ is not equal to $\pm \alpha$, let $\tilde{D}(\alpha,\beta) = D(\alpha,\beta) \cup D(-\alpha,\beta) \setminus \set{\pm \alpha} \subset \Delta$ be the set of weights which can be expressed in the form $t\beta + s\alpha$, for $t > 0$, $s \in \R$. Notice that by picking $a \in \ker \alpha$ such that $\beta(a) < 0$, we may perturb $a$ to $a_1$ so that $\tilde{D}(\alpha,\beta) \cup \set{-\alpha}$ is stable for $a_1$ or to $a_2$ so that $\tilde{D}(\alpha,\beta) \cup \set{\alpha}$ is stable for $a_2$. Therefore, $W^{\pm \alpha}$ can be made the slow foliation for $a_1$ or $a_2$, respectively, implying that $\tilde{D}(\alpha,\beta)$ integrates to a H\"older foliation with $C^r$ leaves which we denote $\mc F^{\alpha,\beta}$, and that $\mc F^{\alpha,\beta}$-holonomy between two $\alpha$-leaves or two $-\alpha$-leaves is $C^r$  (by Lemmas \ref{lem:slow-foliations} and \ref{lem:smooth-holonomy}).

% so the sum of their corresponding distributions integrates to a foliation $\mc F^{\alpha,\beta}$ with $C^r$ leaves by Lemma \ref{lem:extending-charts}.% (we may also make any weights linearly independent with $\alpha,\beta$ equal to 0, so it is an intersection of stable manifolds). %Let $\mc F^\beta$ be the foliation which is the integration of the bundle $TH \oplus \bigoplus_{\gamma \in \Omega} TW^\gamma$. Note that since the action of $H$ leaves each $W^\gamma$-foliation invariant, this is also integrable.

\begin{lemma}
\label{lem:path-holonomy}
If $\beta \in \Delta$ satisfies $\beta \not= \pm \alpha$ and $y \in W^\beta(x)$ and $\eta$ is an $(\alpha,-\alpha,\R^k)$-path based at $x$ as described above, then there exists an $(\alpha,-\alpha,\R^k)$-path $\eta'$ based at $y$ with the same number of legs $y = y_0,\dots,y_n$ such that $y_i \in \mc F^{\alpha,\beta}(x_i)$.
\end{lemma}

\begin{remark}
Lemma \ref{lem:path-holonomy} can be almost immediately deduced from Lemma \ref{lem:smooth-holonomy} and the preceding discussion. However we provide a proof in which we connect two points of the $\mc F^{\alpha,\beta}$ with some path in the 1-dimensional foliations $W^\gamma_i$, $\gamma_i \in \tilde{D}(\alpha,\beta)$, as the ideas in this proof will be used again later (for instance, in Corollary \ref{cor:holonomy-action}).
\end{remark}

\begin{proof}
We build the legs of $\eta'$ one at a time, proceeding by induction. Assume we have found the points $x_0,\dots,x_i$ to $y_0,\dots,y_i$. We now show that we can continue to $x_{i+1}$ and $y_{i+1}$. By Lemma \ref{lem:geom-commutator1}, the points have the following structure: at $x_i$, we have a distinguished path in the foliations $W^\gamma$, $\gamma \in \tilde{D}(\alpha,\beta)$ such that 

\begin{equation}\label{eq:connections}
 x_i = x_{i,0} \xrightarrow{W^{\gamma_1}} x_{i,1} \xrightarrow{W^{\gamma_2}} \dots \xrightarrow{W^{\gamma_m}} x_{i,m} = y_i.
\end{equation}

\noindent where $x \xrightarrow{\mc F} y$ means that $y \in \mc F(x)$ for the foliation $\mc F$. First, suppose that $x_{i+1} = a \cdot x_i$ for some $a \in \R^k$. Then for each leg connecting $x_i$ to $y_i$, the $a$-holonomy is clear:

\[ x_{i+1} = a\cdot x_i = a \cdot x_{i,0} \xrightarrow{W^{\gamma_1}} (a\cdot x_{i,1}) \xrightarrow{W^{\gamma_2}} \dots \xrightarrow{W^{\gamma_m}} (a \cdot x_{i,m}) = a \cdot y_i = y_{i+1}.\] 

Now, suppose $x_{i+1} \in W^\alpha(x_i)$ (the proof for $-\alpha$ is identical). One may make $\alpha$ the slow foliation in a stable leaf containing $\gamma_i$ for each $\gamma_i$ in the connection \eqref{eq:connections}  (as discussed before the statement of the lemma). So by Lemma \ref{lem:smooth-holonomy}, one may use stable holonomy to transport the connections in \eqref{eq:connections}. In particular, we get:

\begin{equation}\label{eq:switching-num}
x_{i+1} = \tilde{x}_{i,0} \xrightarrow{W^{[\alpha,\gamma_1}]} \tilde{x}_{i,1} \xrightarrow{W^{[\alpha,\gamma_2}]} \dots \xrightarrow{W^{[\alpha,\gamma_m]}} \tilde{x}_{i,m} =: y_{i+1}.
\end{equation}

Here, each $\tilde{x}_{i,j}$ is obtained as the image of stable holonomy from $\tilde{x}_{i,j-1}$. Notice that $[\alpha,\gamma_i] \subset \tilde{D}(\alpha,\beta)$, by Lemma \ref{lem:extending-charts}, each connection of the form $W^{[\alpha,\gamma_m]}$ above can be placed in a circular ordering, and therefore a new connection (possibly with more legs) can be made between $x_{i+1}$ and $y_{i+1}$.
 We then iterate the procedure and get the result by induction.
\end{proof}

\begin{lemma}
Let $\eta$ be any $(\alpha,-\alpha,\R^k)$-path based at $x$, and $a \in \R^k$. Define $a * \eta$ to be the path which starts at $x$, moves along $\eta$, then moves along $a$. Then if $\eta'$ and $(a * \eta)'$ are the paths obtained through the holonomy defined in Lemma \ref{lem:path-holonomy}, $e((a * \eta)') = a \cdot e(\eta')$.
%Under the same assumptions as Lemma \ref{lem:path-holonomy}, if the path $\eta$ satisfies $e(\eta)= a \cdot x$, then $e(\eta') \in H\cdot y$.
\end{lemma}

\begin{proof}
Notice that in the proof of Lemma \ref{lem:path-holonomy}, any connection $x \xrightarrow{\R^k} y$ corresponds to $y = a \cdot x$. Since this connection is the last type, it is lifted to a connection of the form $y' = a \cdot x'$. Since this is the last connection of the path $\eta$, the last connection of $\eta'$ also takes this form.
\end{proof}

The following definition is crucial in the subsequent analysis.  Notice that when $W^\alpha$ is the slow foliation in a foliation $\mc F$, Lemma \ref{lem:smooth-holonomy} shows that the holonomy along the fast foliation is smooth. In general, there exists an $\alpha$-holonomy between two different fast leaves, but the holonomy may not be smooth. We call these $\alpha$-holonomies, which can be extended to paths as in Lemma \ref{lem:path-holonomy}. In general, the number of legs for a given $W^H$-path may increase after applying a $\pm \alpha$-holonomy. We shall see that another notion, the {\it $a$-switching number} remains constant.

\begin{definition}
Fix $a \in H$ such that $a \not\in \ker \beta$ for $\beta \not= \pm \alpha$, and let $\rho$ be a $W^H$-path in the foliations $W^\beta$. Then $\rho$ has some combinatorial pattern of weights $(\beta_1,\dots,\beta_n)$ (where we omit appearances of $H$ in the pattern). Define the {\em $a$-switching number} of $\rho$, $s(\rho)$, to be 1 plus the number of weights $\beta_i$ such that $\beta_i(a) \beta_{i+1}(a) < 0$. We say that the switching number is 0 if $\rho$ is only an $H$-leg.
\end{definition}

\begin{corollary}
\label{cor:holonomy-action}
Let $\eta$ be an % broken path in the $W^\alpha$, $W^{-\alpha}$ and $L$-orbit foliations  
$(\alpha,-\alpha,\R^k)$-path based at $x$, and $\rho$ be a path based at $x$ in $W^H$. Then there exists an $(\alpha,-\alpha,\R^k)$-path $\pi_{\rho}( \eta)$ based at $e(\rho)$ and a $W^H$-path $\pi_\eta(\rho)$ based at $e(\eta)$ such that $e(\pi_{\rho}(\eta)) = e(\pi_{\eta}(\rho))$ and $s(\rho) = s(\pi_{\eta}(\rho))$. The functions $\pi_\eta$ and $\pi_\rho$ are continuous. {  Furthermore, $\pi_\eta$ is bijective on the space of $W^H$-paths, and $\pi_\rho$ is bijective on the space of $(\alpha,-\alpha,\R^k)$-paths and $(\alpha,-\alpha,L)-path$.}
\end{corollary}

\begin{proof}
Apply Lemma \ref{lem:path-holonomy} inductively along each leg of the $W^H$-path. The path $\pi_\rho(\eta)$ is the corresponding $(\alpha,-\alpha,\R^k)$-path at the end of the process, and the path $\pi_\eta(\rho)$ is the connection made along the $\beta$-foliations. Finally, observe the switching number is preserved by \eqref{eq:switching-num}, since $\alpha(a) = 0$ and if $\beta(a) < 0$, every $\gamma \in [\alpha,\beta]$ has $\gamma(a) < 0$ (similarly for $\beta(a) > 0$).

{ Invertibility follows because each has an inverse function, namely $\pi_{\rho^{-1}}$ and $\pi_{\eta^{-1}}$, where the inverse of a path denotes its reversal.}
\end{proof}

\begin{lemma}
\label{lem:holonomy-bounded-der}
If $\rho$ is a $W^H$-path with $s(\rho) \le n$ and $\eta$ is an $(\alpha,-\alpha,\R^k)$-path, then $A^{-(2n+1)}\ell(\eta) \le \ell(\pi_\rho(\eta)) \le A^{2n+1} \ell(\eta)$, where $A$ is as in Lemma \ref{lem:uniformly bounded derivative}.
\end{lemma}

\begin{proof}
By induction, it suffices to show that $\pi_\rho(\eta) \le A^2 \ell(\eta)$ when $s(\rho) = 1$. That is, when there exists $a \in H$ such that $a$ contracts every weight in $\rho$. Recall that $\pi_\rho(\eta)$ is defined through stable holonomy, so the paths $\pi_\rho(\eta)$ and $\eta$ converge to one another under iterations of $a$. But since the derivative of $a$ is bounded above by $A$ and below by $A^{-1}$ when restricted to each $\alpha$ and $-\alpha$ leaf, the lengths of $a\cdot \pi_\rho(\eta)$ and $a\cdot \eta$ are each distorted by at most $A$. Thus, since they converge to one another, the lengths of $\eta$ and $\pi_\rho(\eta)$ differ by at most $A^2$. Since we may assume that there is only one $H$-leg appearing at the beginning of the $W^H$-path, this contributes one more factor of $A$ in the estimate.
\end{proof}

\subsection{Controlling $W^H$-returns to $M^\alpha$}

Lemma \ref{lem:holonomy-bounded-der} gives us Lipschitz control on $W^H$-holonomies, which we will use to build an equicontinuous action of holonomies. The difficulty in doing this directly is that the derivative bound depends on $N$, the $a$-switching number. Therefore, if $x \in X$, we define $W^H_N(x)$ to be the set of endpoints of $W^H$-paths $\rho$ with $s(\rho) \le N$ starting from $x$.

\hspace{1cm}

\hypertarget{cases}
\noindent{{\bf The Return Dichotomy.}} Let $\R^k \curvearrowright X$ be a totally Cartan, cone transitive action such that there does not exist a dense $H$-orbit, with associated invariant set $M^\alpha(p)$ constructed above for some well-chosen $H$-periodic point $p$. Then %For every $H$-periodic {  point} $p$, and every $N \in \N$,
 we have the following dichotomy:

\begin{enumerate}[label=(D\arabic*)]
\item \label{dich:1}  For every $N \in \N$ and $x \in M^\alpha(p)$, $W^H_N(x) \cap M^\alpha(p)$ is a finite {  union of closed $H$-orbits}.
\item  \label{dich:2} There exists $x \in M^\alpha(p)$ and $N \in \N$ such that $W^H_N(x) \cap M^\alpha(p)$ is an infinite union of $H$-orbits.
\end{enumerate}

%If every $N$ and $p$ fall into
In case (D1) of the dichotomy, we will construct a rank one factor of the action (see Section \ref{sec:case1}). Otherwise, assuming case (D2) for some $N$ and periodic point $p$, we will show that $W^H(x)$ is dense in $X$ for some $x \in X$ (and therefore $H$ has a dense orbit by Corollary \ref{cor:WH-sufficient}, see Section \ref{sec:case2}).

Before going into the two cases of the dichotomy, we establish a lemma which gives a bound on the switching number required to saturate $X$ by starting from all points of $M^\alpha(p)$.

\begin{lemma}
\label{lem:WHN-transitivity}
Let $p$ be $H$-periodic. There exists $N \in \N$ such that $W^H_N(M^\alpha(p))$ (the saturation of $M^\alpha(p)$ by $W^H_N$) is all of $X$.
\end{lemma}

\begin{proof}
In fact, we show that $W^H_4(M^\alpha(p)) = X$. Fix $a_0 \in H$ which is not in any other Lyapunov hyperplane, and let $a_1$ be a perturbation of $a_0$ which is {  Anosov}. First, notice that $W^H_2(M^\alpha(p))$ contains an open neighborhood of $M^\alpha(p)$. Indeed, using local product structure there exists $\ve > 0$ such that any $y \in X$ $\ve$-close to $x \in M^\alpha(p)$ is reached by applying some element $b \in \R^k$, moving along $W^u_{a_1}(b\cdot x)$ to some $z$, then arriving at $y \in W^s_{a_1}(z)$. Now, by Lemma \ref{lem:smooth-holonomy}, one may connect $b\cdot x$ to $y$ by first moving along the $\alpha$-leaf, then along the weights in $\Delta^+(a_1) \setminus \alpha$. { Let $\rho^+ * \eta^+$ denote the path starting from $b \cdot x$ and ending at $y$, where $\rho^+$ is a $W^H$-path and $\eta^+$ is a single $\alpha$-leg.} Similarly, one may isolate the $-\alpha$-leg in the move from $y$ to $z$, { denote the path by $\rho^- * \eta^-$}. %Notice that the weights of $\Delta^+(a_1) \setminus \alpha$ are all in $W^H$. Therefore, one may project the $\pm \alpha$-leaves down to $x$ along these $W^H$-paths 
{ Using the projections in Corollary \ref{cor:holonomy-action}, we note that the original path given by $\rho_- * \eta_- * \rho_+ * \eta_+ * b$, can be rearranged to 
\begin{equation}\label{eq:rho-eta-swap}\rho_- * \pi_{\eta_-}({\rho_+}^{-1})^{-1} * \pi_{{\rho_+}^{-1}}(\eta_-) * \eta_+ * b,\end{equation}} where $\rho^{-1}$ denotes the reversal of the path $\rho$. This corresponds to pushing the path $\eta_-$ along ${\rho_+}^{-1}$ via holonomies. We consider the nearby point $x' = e(\pi_{{\rho_+}^{-1}}(\eta_-) * \eta_+ * b) \in M^\alpha(p)$. %, then along $W^s_{a_0}$ and finally along $W^u_{a_0}$. 
 From this, {  since the last two connections of \eqref{eq:rho-eta-swap} are in $W^H$ and the switching number is preserved by $\pi_{\eta^-}$}, it is clear that $y \in W^H_2(M^\alpha(p))$.

Now, $W^H_2(M^\alpha(p))$ is invariant under the $\R^k$-action (since $M^\alpha(p)$ is invariant under the $\R^k$-action), and contains an open set, so is therefore dense by transitivity of the $\R^k$-action. By density, $B(W^H_2(M^\alpha(p)),\ve) = X$ for every $\ve > 0$. Repeating the argument above proves that after saturating again by local $(\alpha,-\alpha,\R^k)$-paths and local $W^{s/u}_{a_0}$ leaves, we must get all of $X$ since such a saturation at a point $x$ contains $B(x,\ve)$ for some $\ve > 0$ independent of $x$. Since we may again push the $(\alpha,-\alpha,\R^k)$-path to the start of the path using holonomy projections, we conclude that $W^H_4(M^\alpha(p)) = X$.
%Fix $a \in H$ which is not in any other Lyapunov hyperpanWe first show the following: let $W_n$ be the saturation of $p$ by 
\end{proof}

\section{Finite returns: Constructing a rank one factor}
\label{sec:case1}

In this section, we assume throughout that $p$ is a fixed, $H$-periodic point for which \ref{dich:1} holds for every $N$. Under this assumption,  we will construct a $C^r$ rank one factor of the transitive, totally Cartan, $C^r$ action. We recall that $W^H_N$ and the number of switches is defined by fixing some $a \in H$ such that $a \not\in \ker \beta$ for every $\beta \not= \pm \alpha$, so that $a$ acts partially hyperbolically on $X$. We may perturb $a$ to $a' \in \R^k$ which is regular, so that $a'$ acts normally hyperbolically with respect to the orbits of the $\R^k$-actions.

We recall some facts from the theory of partially hyperbolic transformations. Recall that $a'$ acts normally hyperbolically  with respect to the $\R^k$-orbit foliation, and let $W^*_{a',\delta}$ denote the local manifolds of $a'$ where $* = s,u,cs,cu$ stands for stable, unstable, center-stable, center-unstable, respectively and $\delta$ denotes the size of the local neighborhood. Since $a'$ acts normally hyperbolically, there exists $\delta > 0$ such that the map $g {  = g_x} : W^u_{a',\delta}(x) \times W^{cs}_{a',\delta}(x) \to X$ defined by $g(y,z) = W^{cs}_{a',\delta}(y) \cap W^u_{a',\delta}(z)$ is well defined and a homeomorphism onto a neighborhood of $x$. Choose $\ve_0$ to be such that $B(x,\ve_0)$ is contained in this neighborhood {  for every $x \in X$}.

\begin{lemma}[Local integrability of $W^H$]
\label{lem:locally-integrable}
Assume \ref{dich:1} holds. % for every $N$ at some $H$-periodic point $p$. 
There exists $\ve_0$ with the following property:
suppose that { $x \in X$ and $\rho \in W^H_2(x)$ begins with its $a$-stable legs, then moves along $a$-unstable legs, each written in a circular ordering, and possibly $H$-legs}. Then for every $\ve_0 > \ve > 0$, there exists $\delta > 0$ such that if the legs of $\rho$ all have lengths less than $\delta$, there exists a $\rho' \in W^H_2(x)$ with the same base and endpoint as $\rho$, which begins with its $a$-unstable legs, then moves along $a$-stable legs, written in a circular ordering, and possibly $H$-legs. Every leg of $\rho'$ has length at most $\ve$ (for both paths, the position and lengths of the $H$-legs do not matter).
\end{lemma}

\begin{proof}
Let $w \in B(x,\ve)$ be reached by first moving along some $a$-stable legs, then some $a$-unstable legs (we may without loss of generality assume that there are no $H$-legs). Note that anything which is $a$-(un)stable is automatically $a'$-(un)stable. Then we may reverse the order of the stable and unstable legs by writing $w = g(y,z)$ for some $y \in W^u_{a',\delta}(x)$ and $z \in W^{cs}_{a',\delta}(x)$. Notice also that since the center foliation of $a'$ is given by an abelian group action of $\R^k$, we may reach $w$ from $y$ by first moving along the $W^s_{a'}(y)$ to get to $y_1$, then applying an element $b$ of the $\R^k$ action to get to $w = b \cdot y_1$. By choosing a subgroup $L \subset \R^k$ transverse to $H$, we may further write $b = h + \ell$, where $h \in H$ and $\ell \in L$.
%we may assume that $z_1 = b\cdot x$ and $z \in W^s_{a',\delta}(z_1)$ for some $b \in \R^k$. 

The plan is as follows: We wish to show that the $\ell$-component, and the $W^{\pm \alpha}$-components of the $W^{s/u}_{a'}$-paths are trivial. So we turn them into continuous functions depending on the points $y$ and $z$. If the components are not trivial at some point, they can be made arbitrarily small by continuity and the fact that they are trivial for the trivial path. Together, these components can be pushed to appear as the last legs of the path using Corollary \ref{cor:holonomy-action}. We may further control their distances using Lemma \ref{lem:holonomy-bounded-der}, which gives rise to a new intersection point of $W^H(x)$ with $M^\alpha(x)$ if nontrivial. By the discreteness assumption,  we will conclude that the components must be trivial at all points.

We now provide the details of the proof, which is depicted in Figure \ref{fig:su-comm}. Assume that $a'$ is chosen sufficiently close to $a$ so that $\alpha$ and $-\alpha$ are the slow foliations in $W^u_{a',\delta}$ and $W^s_{a',\delta}$, respectively. Thus, by Lemmas \ref{lem:extending-charts} and \ref{lem:smooth-holonomy}, we may break off the $\alpha$ and $-\alpha$ pieces to get points $y' \in W^\alpha(x)$ and $w' \in W^{s}_{a,\delta}(y_1)$ such that $y \in W^u_{a,\delta}(y')$, and $y_1 \in W^s_{a,\delta}(w')$. { We write the paths as follows: $\rho_+ * \eta_+$ is the path starting at $x$ and ending at $y$, where $\rho_+$ is an unstable path in $W^H$ and $\eta_+$ is a single $\alpha$-leg. $\rho _- * \eta_-$ is the path starting at $y$ and ending at $y_1$, where $\rho_-$ is a stable path in $W^H$ and $\eta_-$ is a single $-\alpha$-leg. Finally, $b = h+ \ell$ is the $\R^k$-element connecting $y_1$ and $w$, so that $w$ is the endpoint of
\[  (h+\ell) * \rho_- * \eta_- * \rho_+ * \eta_+  .\]

   We abuse notation to freely move the element $b$ around the product, noting that it preserves the combinatorics of the legs appearing as well as whether the leg is trivial or not. Since it is short it also does not expand or contract the legs a significant amount In particular, we  continue to write $\rho_-$ for $\pi_\ell(\rho_-)$ after moving $\ell$ to the right. } Notice that there are now only three legs which are not permitted in $W^H$: the $\alpha$ leg connecting $x$ and $y'$ { (ie, $\eta_+$)}, the $-\alpha$ leg connecting $y$ and $w'$ { (ie, $\eta_-)$}, and the $L$-component of the $\R^k$-connection between $w$ and $y$. One may (in a way which assigns a unique path) use the holonomy projections to arrange all three of these legs to appear at $x$ to arrive at a point $x_1 \in M^\alpha(x)$ (by commuting the $W^{-\alpha}$-leg with the $W^u_a$-leg, and using the fact that $L$ preserves all foliations), then move along $W^u_a$, then along $W^s_a$, then along the $H$-orbit { (this is exactly the rearrangement done previously to achieve \eqref{eq:rho-eta-swap}).}
 
 Since the holonomies are continuous, and the transverse intersection must vary continuously, the point $x_1$ is a continuous function of $y$. Call $\eta$ the $(\alpha,-\alpha,L)$-path used to get from $x$ to $x_1$ { (ie, $\eta = \pi_{{\rho_+}^{-1}}(\eta_-) * \eta_+ * \ell$)}. Let $q(w) = d(x,x_1)$ (where the distance is measured in $X$). Then $q$ is a continuous function whose domain is the set of points of the form $w$ (those reached from $x$ by an $a$-stable leaf, then an $a$-unstable leaf, whose lengths are each at most $\ve$), and $q(x) = 0$.

\begin{figure}[!ht]
\scalebox{.75}{
	\begin{tikzpicture}
	\draw[blue, thick] (0,0) .. controls (3,1) .. (4,4);
	\node at (3,1) {$W^s_a$};
	\draw[red, thick] (4,4) .. controls (5,1) .. (8,0);
	\node at (5,1) {$W^u_a$};
	\draw[magenta, thick] (0,0) .. controls (.25,-.5) .. (1,-1);
	\node at (0,-.65) {$W^\alpha$};
	\draw[red,thick] (1,-1) .. controls (2,-2.5) .. (3,-3);
	\node at (1.5,-2.5) {$W^u_a$};
	\draw[cyan,thick] (3,-3) .. controls (3.5,-2.75) .. (4,-2);
	\node at (4,-2.75) {$W^{-\alpha}$};
	\draw[blue,thick] (4,-2) .. controls (5.5,-2.25) .. (7,-1);
	\node at (6,-2.25) {$W^s_a$};
	\draw[green,thick] (7,-1) -- (8,0);
	\node at (7.75,-.75) {$\R^k$};

	\draw[ultra thick, ->] (9,0) -- (11,0);
	\draw[blue, thick] (12,0) .. controls (15,1) .. (16,4);
	\node at (15,1) {$W^s_a$};
	\draw[red, thick] (16,4) .. controls (17,1) .. (20,0);
	\node at (17,1) {$W^u_a$};
	\draw[green, thick] (12,0) -- (13,-1);
	\node at (12.25,-.65) {$L$};
	\draw[magenta, thick] (13,-1) .. controls (13.25,-1.5) .. (14,-2);
	\node at (13,-1.65) {$W^\alpha$};
	\draw[cyan,thick] (14,-2) .. controls (14.5,-1.75) .. (15,-1);
	
	\node at (14.25,-1.25) {$W^{-\alpha}$};
	\draw[red,thick] (15,-1) .. controls (16,-2.5) .. (17,-3);
	\node at (15.5,-2.5) {$W^u_a$};
	\draw[blue,thick] (17,-3) .. controls (18,-2.5) .. (19,-1);
	\node at (18.5,-2.5) {$W^s_a$};
	\draw[green,thick] (19,-1) -- (20,0);
	\node at (19.5,-.75) {$H$};
	
	\filldraw[black] (0,0) circle (2pt) node[anchor=east]{$x$};
	\filldraw[black] (12,0) circle (2pt) node[anchor=east]{$x$};
	\filldraw[black] (8,0) circle (2pt) node[anchor=west]{$w$};
	\filldraw[black] (20,0) circle (2pt) node[anchor=west]{$w$};
	\filldraw[black] (15,-1) circle (2pt) node[anchor=west]{$x_1$};
	\filldraw[black] (3,-3) circle (2pt) node[anchor=north]{$y$};
	\filldraw[black] (7,-1) circle (2pt) node[anchor=north]{$y_1$};
	\filldraw[black] (1,-1) circle (2pt) node[anchor=west]{$y'$};
	\filldraw[black] (4,-2) circle (2pt) node[anchor=south]{$w'$};
	\end{tikzpicture}
}
\caption{Constructing a $W^s_a$, $W^u_a$-commutator}
\label{fig:su-comm}
\end{figure}

The lemma, stated in this language, exactly claims that $q \equiv 0$ where defined (ie, on the points reached by a local $a$-stable leaf, then a local $a$-unstable leaf). {  Indeed, if $q > 0$ at some point, then $x$ and $x_1$ do not lie in the same $H$-orbit by Lemma \ref{lem:eta-LPS}.}%, since $L$ is transverse to $H$, and together with $W^\alpha$ and $W^{-\alpha}$ generates coordinates on the 3-manifold factor.}  

Suppose that $q \not\equiv 0$, ie that there exists $w$ such that $q(w) > 0$. Choose any path $\gamma : [0,1] \to B(x,\ve)$ such that $\gamma(0) = x$ and $\gamma(1) = w$, and $q$ is defined on $\gamma(t)$ for every $t \in [0,1]$. This is possible simply by retracting the $W^{s/u}_a$ connections from $a$ to $w$. Then $q \of \gamma : [0,1] \to \R_{\ge 0}$ is continuous, so by the intermediate value theorem, for sufficiently large $n$, there exists $t_n \in [0,1]$ such that $q(\gamma(t_n)) = \frac{1}{n}$. 

For each $n$, we use the notations developed above, replacing $w$ by $w_n = \gamma(t_n)$.  For the distinguished $H$-periodic orbit $p$, choose any $W^H$-path $\rho$ from $x$ to $M^\alpha(p)$, and let $x' = e(\rho)$ denote its endpoint (such a path exists by Lemma \ref{lem:WHN-transitivity}). Let $N = s(\rho)$. See Figure \ref{fig:project-return}, in which the blue curves represent collections of $W^H$-legs and the red curves are collections of $M^\alpha$-legs. Let $\eta_n$ denote the $(\alpha,-\alpha,L)$-path connecting $x$ and $x_n$, which corresponds to the point $x_1$ in the construction above.

\begin{figure}[!ht]
	\begin{tikzpicture}
	\draw[blue,thick] (-.5,-2) -- (-.5,0);
	\draw[blue,thick] (.5,-2) -- (.5,0);
	\draw[blue,thick] (-.5,0) .. controls (-1,1) .. (0,1.5);
	\draw[blue,thick] (.5,0) .. controls (1,1) .. (0,1.5);
	\draw[red,thick] (-.5,0) -- (.5,0);
	\draw[red,thick] (-.5,-2) -- (.5,-2);
	
	\filldraw[black] (-.5,0) circle (2pt) node[anchor=east]{$x$};
	\filldraw[black] (.5,0) circle (2pt) node[anchor=west]{$x_n$};
	\filldraw[black] (0,1.5) circle (2pt) node[anchor=south]{$w_n = \gamma(t_n)$};
	\filldraw[black] (-.5,-2) circle (2pt) node[anchor=east]{$x'$};
	\filldraw[black] (.5,-2) circle (2pt) node[anchor=west]{$z_n$};
	
	\node at (-.75,-1) {$\rho$};
	\node at (0,.25) {$\eta_n$};
	\node at (1.1,-1) {$\pi_{\eta_n}(\rho)$};
	\end{tikzpicture}
	\caption{Projecting a close $W^H$-return to $M^\alpha(p)$}
	\label{fig:project-return}
\end{figure}

Then one may construct paths of length $2N+4$ starting from $x'$  and ending at a point  $z_n \in M^\alpha(p)$ close to $x'$ in the following way: first follow the reverse of $\rho$ to end at $x$. Then follow the ``commutator path''  which is defined in the following way: first, following the $a$-stable/unstable path to get  from $x$ to $w_n = \gamma(t_n)$ which the constructions above began with, then the $a$-stable/unstable path to get to $x_n$. Recall that the point $x_n$ is connected to $x$ by a short $(\alpha,-\alpha,L)$-path $\eta_n$. Project the long path $\rho$ along $\eta_n$ to obtain $\pi_{\eta_n}(\rho)$ (see Corollary \ref{cor:holonomy-action}), which will start at $x_n$ and end at a point $z_n$ whose distance is at most $A^{  2N+1}q(\gamma(t_n)) = A^{  2N+1}/n$ from $x'$ (by Lemma \ref{lem:holonomy-bounded-der}). { Finally, if $\eta_n$ is nontrivial, so is its projection, and $z_n$ cannot lie in the same $H$-orbit as $x'$ (by Lemma \ref{lem:eta-LPS}).} Since this holds for every $n$, we have contradicted that we are in  Case \ref{dich:1} of the Return Dichotomy, so we conclude $q \equiv 0$.
\end{proof}

Recall the definition of $M^\alpha_\ve(x)$ at the end of Section \ref{sec:malpha-integrability}.

\begin{lemma}
\label{lem:Malpha-intersections}
For every $N \in \N$ and sufficiently small $\ve > 0$, there exists $\delta = \delta(N,\ve)$ such that if $\rho$ is a $W^H_N$-path based at a point $x \in X$ with $e(\rho) \in M_\ve^\alpha(x)$, and every leg in $\rho$ has length less than $\delta$, then $e(\rho) \in H x$.
\end{lemma}

\begin{proof}
We first prove the lemma for $N  = 2$. Define the following map $f : W^u_a(x) \times W^s_a(x) \times H \times L \times \R^2 \to X$. Pick $(y,z,h,b,t,s)$ in the domain. Let $w = W^s_{a'}(y) \cap W^{cu}_{a'}(z)$. Then define $w' = (h+b)w$. Let $w_1$ be the point of $W^\alpha(w')$ at signed distance $t$ from $w'$ and $w_2$ be the point of $W^{-\alpha}(w_1)$ at signed distance $s$ from $w_1$. Setting $f(y,z,h,b,t,s) = w_2$, we get that the map $f$ is a local homeomorphism onto its image (one can rearrange the $\alpha$ and $-\alpha$ legs as in the proof of Lemma \ref{lem:locally-integrable} and use the standard local product structure of Anosov actions). In particular, one cannot have $f(y,z,h,0,0,0) = f(x,x,0,b,t,s)$ unless all components of $(y,z,h,b,t,s)$ are trivial. %That is, we cannot locally return to the $M^\alpha$ leaf by short $W^H_2$-moves, which we have shown is sufficient for the lemma.

Now we proceed in general. Let $0 < \ve < \ve_0$, where $\ve_0$ is as in Lemma \ref{lem:locally-integrable}. We may apply Lemma \ref{lem:locally-integrable} to arrive at some $\delta = \ve_1$ for which any $W^H_2$-path which begins with $W^H$-stable legs and ends with $W^H$-unstable legs can have its stable/unstable order reversed, and after reversal, the new path has its legs less than $\ve$. Then again apply Lemma \ref{lem:locally-integrable} to $\ve_1$ to arrive at a new $\delta = \ve_2 < \ve_1/2$ for which rearrangements of paths of length at most $\ve_2$ has its lengths at most $\ve_1/2$ after rearrangement. Repeat this process $N$ times to arrive at a sequence of $\ve_i$ such that $\ve_i < \ve_{i-1}/2$ and after rearranging  a stable/unstable path to an unstable/stable path with legs of length at most $\ve_i$, the new legs have length at most $\ve_{i-1}/2$. Then, if every leg of $\rho$ has length at most $\ve_N$, we may apply Lemma \ref{lem:locally-integrable} $N$ times to rearrange the path to begin with an unstable leg, then move along a stable leg (or vice-versa). Since the lemma holds for $N = 2$, we have finished the proof.
\end{proof}

\begin{lemma}
\label{lem:WH-localmanifold}
Fix $x \in X$. For every $y \in W^H_N(x)$, there exists $\delta_0 = \delta_0(y) > 0$ such that if $z$ is reached from $y$ by a $W^H_2$-path whose legs have length at most $\delta_0$, then $z \in W^H_N(x)$. The constant $\delta_0$ depends only on the length of the final group of $a$-stable or $a$-unstable legs of a $W^H$-path connecting $x$ and $y$, and is bounded below in a local $W^H_2$-leaf of $x$. % Furthermore, if $z$ is sufficiently close to $y$, $\delta_0(z)$ can be chosen to be larger than $\delta_0(y)/2$.
\end{lemma}

\begin{proof}
Fix any $W^H_N$-path $\rho$ connecting $x$ and $y$, and suppose that $\rho$ ends with $a$-stable legs, so that $\rho = \rho_2 * \rho_1$, where $s(\rho_1) \le N-1$, $\rho_1$ ends with an $a$-unstable leaf and $\rho_2$ is contained in a single $a$-stable leaf. Then let $\ve_0$ denote the length of $\rho_2$. Choose $t > 0$ such that the length of $(ta) \cdot \rho_2$ is smaller than $\delta$ as in Lemma \ref{lem:locally-integrable} applied to some fixed $\ve$. Let $\delta_0 = \delta/ \norm{(ta)_*}$. Then if $\rho_3$ is any $W^H_2$-path starting from $y$ whose legs have length at most $\delta_0$, $(ta) \cdot (\rho_3 * \rho_2)$ is a $W^H_4$-path whose legs all have length at most $\delta$. By Lemma \ref{lem:locally-integrable}, we may find a $W^H_2$ path which shares the same start and endpoint, begins with an $a$-unstable leaf and ends with an $a$-stable leaf. Applying $-ta$ to this path yields the desired path and absorbing the $a$-unstable leaf into the last leaf of $\rho_1$ gives the desired path.
\end{proof}

\begin{lemma}
\label{lem:foliation1}
Let $\R^k \curvearrowright X$ be a $C^r$ action, with $r = (1,\theta)$ or $r = \infty$. For sufficiently large $N$, the collection $\set{W^H_N(x) : x \in M^\alpha(p)}$ is a $C^r$ foliation of $X$.%, where $r' = (1,\theta)$ or $\infty$, respectively.
\end{lemma}

\begin{proof}
Fix $N$ as in Lemma \ref{lem:WHN-transitivity}, and choose $x \in X$. Then there exists a $W^H$-path $\rho$ connecting $x$ to some $x' \in M^\alpha(p)$ with $s(\rho) \le N$. 

Using that we are in case \ref{dich:1} of the return dichotomy and  applying it to $W^H_{2N+6}(x')$, there exists $\ve > 0$ such that if $y' \in M^\alpha(p) \cap W^H_{2N+6}(x')$ and $d(x',y') < \ve$, then $x' { \in H}y'$. Choose $\delta'$ such that if $d_X(x,y) < \delta'$ and $y \in M^\alpha_\ve(x)$, then $y$ is connected to $x$ by an $(\alpha,-\alpha,\R^k)$-path whose length is at most $\ve$ (this is possible by the continuity of the map from Lemma \ref{lem:psi-coords}). Let $0 < \delta < A^{-N}\delta'$ be chosen so that Lemma \ref{lem:Malpha-intersections} applies with $N = 6$. Then define the map $f : B(x,\delta) \to M^\alpha(p) / H$ as follows:

 Recall that $a \in H$ does not belong to any other Lyapunov hyperplane. If $z \in B(x,\delta)$, one may use local product structure to find a unique path connecting $x$ and $z$ which moves along $W^\alpha(x)$ to a point $y$, then along $W^{-\alpha}(y)$ to a point $y'$, then along the $L$-orbit to a point $w$ (recall that $L$ is some fixed subgroup of $\R^k$ complementary to $H$), then along a short $W^H_2$-path to arrive at $z$ { (this path looks like \eqref{eq:rho-eta-swap}). Let $\eta$ denote the $(\alpha,-\alpha,L)$-component of this connection which connects $x$ and $w$, and  $z'  \in M^\alpha(p)$ denote the endpoint of $\pi_\rho(\eta)$}. Then define $f(z)$ to be $Hz' \in M^\alpha(p) / H$.

Let $F_y$ be the set of points in $B(x,\delta)$ reached from $y$ by $W^H_2$-paths whose legs have length at most $\ve$. We claim that if $y \in B(x,\delta)$, then $F_y = f^{-1}(f(y))$. Indeed, if $y,z \in B(x,\delta)$ and there is a $W^H_2$-path $\hat{\rho}$ connecting $y$ to $z$, then when one connects $x$ to $y$ and $x$ to $z$, one obtains intermediate points $\hat{y}, \hat{z} \in M^\alpha(x)$  which are reached from $x$ by short $(\alpha,-\alpha,L)$-paths $\eta_1$ and $\eta_2$, respectively, and $W^H_2$-paths $\rho_1$ and $\rho_2$ starting from  $\hat{y}$ and $\hat{z}$ and ending at  $y$ and $z$, respectively.  By definition, $f(y) = e(\pi_\rho(\eta_1))$ and $f(z) = e(\pi_\rho(\eta_2))$. Notice that if $\hat{y} \not= \hat{z}$, then we may concatenate the paths $\rho_1$, then $\hat{\rho}$, then $\rho_2^{-1}$ and get a $W^H_6$-path connecting $\hat{y}$ and $\hat{z}$. By choice of $\delta$ and Lemma \ref{lem:Malpha-intersections}, we conclude that $\hat{y} = \hat{z}$ {  by Lemma \ref{lem:eta-LPS}. %{ (since $L$ and the $\pm\alpha$ have local product structure on the factor of $M^\alpha$ by $H$ by Lemma \ref{lem:anosov-factor}, $\hat{y} \in H \hat{z}$ if and only if $\eta_1 = \eta_2$)}.
Therefore, $F_y\subset f^{-1}(f(y))$.}

Now suppose that $f(y) = f(z)$.  Again let the paths $\eta_1$ and $\eta_2$ denote the $(\alpha,-\alpha,L)$-paths from $x$ to $\hat{y}$ and $\hat{z}$, as in the proof of the previous inclusion. $\eta_1$ and $\eta_2$ project to the same path on $M^\alpha(p)$ by Lemma \ref{lem:psi-coords}, so $\eta_1 = \eta_2$. Therefore, the path $\rho_2\rho_1^{-1}$ is a $W^H_4$-path connecting $y$ and $z$. Applying Lemma \ref{lem:locally-integrable} to the middle switch implies that $z \in W^H_2(y)$, so $f^{-1}(f(y)) \subset F_y$.

To see that $W^H_N$ has the structure of a $C^{r}$ foliation, we claim it suffices to show that the map $f$ is a $C^{r}$ submersion. Indeed, in this case one can see that since the chosen path connecting points of $B(x,\delta)$ to $M^\alpha(p)$ varies continuously, the $\delta_0$ of Lemma \ref{lem:WH-localmanifold} can be chosen uniformly in a neighborhood of $x$. Therefore, in a sufficiently small neighborhood of $x$, if $f$ is a $C^{r
'}$ submersion, the preimages $f^{-1}(y')$ form a foliation, which we have just shown are the local leaves of $W^H_N(y)$ by Lemma \ref{lem:WH-localmanifold}.

We now prove that $f$ is a $C^{r}$ submersion. Fix $y \in B(x,\delta)$, and consider some $z \in W^\alpha(y)$. Then $f(z)$ is defined by first moving from $x$ to some point $\hat{z}$ along a local $(\alpha,-\alpha,L)$-path $\eta$, then along a local $W^H_2$-path $\rho_1$. $f(z)$ is defined by projecting the local $M^\alpha$ path $\eta$ along the distinguished connection $\rho$ from $x$ to $x' = f(x)$. That is, $f(z)$ is the endpoint of $\pi_\rho(\eta)$.

{  For notational convenience, let $[z,y]$ denote the path with a single leg along the foliation $W^\alpha$.} Notice that since $z \in W^\alpha(y)$,  the local paths defining $f(y)$ are $\pi_{\rho_1}([z,y])* \eta$ and $\pi_{[y,z]}(\rho_1)$. Therefore,  $f(y) = \pi_\rho( \pi_{\rho_1}([z,y])* \eta)$, which differs from $f(z)$ exactly by $\pi_\rho(\pi_{\rho_1}([z,y]))$, so the projection is defined by applying the holonomies of $\rho_1$ and $\rho$. Therefore, we may iterate Lemma \ref{lem:smooth-holonomy} to obtain that $f$ is $C^{r}$ along $W^\alpha(y)$ and $W^{-\alpha}(y)$ for every $y \in B(x,\delta)$. Fix a perturbation $b\in \R^k$ of $a$ which is not contained in any  Lyapunov hyperplane and for which $\pm \alpha$ are both slow foliations (recall the start of Section \ref{sec:holonomy-action}). This determines foliations $\widehat{W^{\pm \alpha}} \subset W^{u/s}_b$ by Lemma \ref{lem:smooth-holonomy}. Notice that $\widehat{W^{\alpha}}(x)$ and $\widehat{W^{-\alpha}}(x)$ are both H\"older foliations with smooth leaves, and that $f$ collapses each to a point. Since $W^\alpha$ and $\widehat{W^\alpha}$ are complementary foliations inside $W^u_b(x)$, $f$ is smooth along $W^u_b(x)$. Similarly, $f$ is smooth along $W^s_b(x)$. It is also clear that $f$ is $C^{r}$ along $\R^k$-orbits. Applying Journ\'{e}'s theorem (Theorem \ref{thm:journe}) to the foliations $W^s_b$ and $\R^k$ gives that $f$ is $C^r$ along $W^{cs}_b$. Then apply Journ\'{e}'s theorem again to the foliations $W^{cs}_b$ and $W^u_b$ to see that $f$ is $C^{r}$ on $B(x,\delta)$.
\end{proof}

%\begin{remark}
%The reason for the loss of regularity from $C^2$ to $C^{1,\theta}$ occurs in the application of Journ\'{e}'s theorem. Under the assumption of $C^{2,\theta}$, there is no such loss of regularity. A $C^r$ version of our results will also hold, $r \ge 2$.
%\end{remark}

\begin{lemma}
\label{lem:WH-WHN}
If $N$ is as in Lemma \ref{lem:WHN-transitivity} $W^H_N(x) = W^H(x)$ for every $x \in X$.
\end{lemma}

\begin{proof}
By Lemma \ref{lem:foliation1}, $W^H_N$ is a $C^{r}$ foliation containing local $W^\beta$-leaves and $H$-orbits, {  $\beta \not= \pm \alpha$}. Since its leaves intersect $M^\alpha(p)$ in finitely many $H$-orbits, it must be codimension 3. Therefore, $T(W^H_N) = TH \oplus \bigoplus_{\beta \in \Delta \setminus \set{\pm \alpha}} E^\beta$. Therefore, any $W^H$-path, regardless of its structure, is tangent to $W^H_N$, so $W^H_N(x) = W^H(x)$.
\end{proof}

\begin{lemma}
\label{lem:compact-leaves}
Each leaf $W^H_N(x)$, $x \in M^\alpha(p)$ is compact. %Given any two leaves $W^H_N(x)$ and $W^H_N(y)$, there exists $\ve > 0$ such that $B(W^H_N(x),\ve) \cap B(W^H_N(y),\ve) = \emptyset$.
\end{lemma}

\begin{proof}
%Since every $W^H_N$-leaf intersects $M^\alpha(p)$, to prove compactness we may without loss of generality assume that $x \in M^\alpha(p)$. 
Suppose that $x_k \in W^H_N(x)$ converges to some $y \in X$, and let $\rho_k$ be a $W^H_N$-path connecting $x_k$ and $x$. If $x_k$ enters a fixed, small neighborhood of $y$ (where local product structure applies) we may connect $x_k$ to $y$ by first moving along a short $W^H_2$-path $\rho_k'$, then a short $(\alpha,-\alpha,L)$-path $\eta_k$, so that $\eta_k * \rho_k'$ is a path connecting $x_k$ and $y$. Hence $\eta_k * \rho_k' * \rho_k$ and $\eta_\ell * \rho_\ell' * \rho_\ell$ are paths connecting $x$ and $y$. So $\rho_\ell^{-1} * {\rho_\ell'}^{-1} * \eta_\ell^{-1} * \eta_k * \rho_k' *\rho_k$ is a path beginning and ending at $x$. Applying the holonomies of $\rho_k'$ and $\rho_k$ to $%\eta_{k,\ell} =
\eta_\ell^{-1} * \eta_k$ yields a $W^H_{2N+4}$-path $\rho_{k,\ell}$ and a short $(\alpha,-\alpha,L)$-path $\eta_{k,\ell}$, both based at $x$, with the same endpoints. { That is,

\begin{eqnarray*}
    \eta_{k,\ell} & = & \pi_{(\rho_k'*\rho_k)^{-1}}(\eta_{\ell}^{-1} * \eta_k), \mbox{ and} \\
    \rho_{k,\ell} & = & \pi_{{\eta_\ell}^{-1} * \eta_k}((\rho_k' * \rho_k)^{-1}) * \rho_\ell' * \rho_{\ell}.
\end{eqnarray*}
}

Since the length of $\eta_{k,\ell}$ is at most $A^{2(N+2)+1}$ times the length of $\eta_\ell^{-1} * \eta_k$, and the lengths of $\eta_k$ and $\eta_\ell$ both tend to 0, from the assumption of Case \ref{dich:1} of the return dichotomy applied to $W^H_{2N+4}$, we conclude that $\eta_k = \eta_\ell$ for sufficiently large $k,\ell$ { by Lemma \ref{lem:eta-LPS}.} %{  (they are equal since their endpoints must lie in the same $H$-orbit, and $\alpha$, $-\alpha$ and $L$ induce a local product structure in the manifold $M^\alpha(p)/H$)}.
Since the sequence $\eta_k$ also converges to the trivial path, we conclude that $\eta_k$ is eventually the trivial path. Therefore, by Lemma \ref{lem:WH-WHN}, $y \in W^H_{N+2}(x) \subset W^H_N(x)$. {  Now, since $y$ was an arbitrary point of $\overline{W^H_N(x)}$, we conclude that $W^H_N(x) = \overline{W^H_N(x)}$. So $W^H_N(x)$ is closed and hence compact.}
\end{proof}

Recall that $M^\alpha_\ve(x)$ is the image of $B(0,\ve)$ under the charts $\psi$ defined in Lemma \ref{lem:psi-coords} at the point $x$. Given $x \in X$, let $r(x)$ denote the cardinality of $(W^H(x) \cap M^\alpha(p))/H$.

\begin{lemma}
There exists $r_0$ such that $r(x) \le r_0$ for all $x \in X$, and there exists an open, dense set $U \subset M^\alpha(p)$ (in the topology of $M^\alpha(p)$) such that $r(x) = r_0$ on $U$.
\end{lemma}

\begin{proof}
First, notice that $r$ is always finite, by Lemma \ref{lem:WH-WHN} and the assumption that we are in case \ref{dich:1} of the return dichotomy for every $N$. We claim that $r$ is lower semicontinuous. Indeed, suppose that $r(x) = c$. Then one may choose $W^H_N$-paths $\rho_1,\dots,\rho_c$ starting from $x$ and ending at $M^\alpha(p)$ and ending at distinct points $y_1,\dots,y_c \in M^\alpha(p)$.  Let $\ve$ be such that $M^\alpha_{\ve}(y_i) \cap M^\alpha_{\ve}(y_j) = \emptyset$ for all $i\not= j$, and consider the saturation of $M^\alpha_{A^{-N}\ve}(x)$ by local $W^H$-leaves, call this set $V$. Given $x' \in V$, there exists an $(\alpha,-\alpha,\R^k)$-path $\eta$ starting from $x$ and ending at some $x_1$ such that $x' \in W^H(x_1)$. Then all $\pi_\eta(\rho_i)$ are connections from $x_1$ to $M^\alpha(p)$ which end at a point in $M^\alpha_{\ve}(y_i)$ and are therefore distinct. Concatenating $\pi_\eta(\rho_i)$ with the $W^H$-path connecting $x'$ and $x_1$ yields a $W^H$-path connecting $x'$ and a point of $M^\alpha_{\ve}(y_i)$. Therefore, $r(x') \ge r(x)$ whenever $x' \in V$, and $r$ is lower semicontinuous.

 Consider the set $A_c = r^{-1}([c,\infty)) = r^{-1}((c-1/2,\infty))$ for $c \in \N$. Notice that since $r$ is lower semicontinuous, $A_c$ is open. Furthermore, if $\rho$ is any $W^H$-path starting from $x$ and ending at $y \in M^\alpha(p)$ and $a \in \R^k$, $a \cdot \rho$ connects $a \cdot x$ and $a \cdot e(\rho) \in M^\alpha(p)$. Therefore, $A_c$ is $\R^k$-invariant. Since there is a dense $\R^k$-orbit, if $A_c$ is nonempty it is open and dense. Therefore, if $A_c \not= \emptyset$ for all $c \in \N$, $A_\infty = \bigcap_{c \in \N} A_c$ is nonempty by the Baire category theorem. This contradicts that we are in case \ref{dich:1} of the return dichotomy, so there exists some $r_0 \in \N$ such that $r(x) \le r_0$ for all $x \in X$. Choosing the maximal such $r_0$ yields that $A_{r_0}$ is an open dense subset of $X$.

Since $A_{r_0}$ is saturated by $W^H$-leaves (by concatenating any connecting path), $r = r_0$ on an open and dense subset of $M^\alpha(p)$, as well.
\end{proof}

Recall that a $W^H$-cycle is a $W^H$-path whose end and start points coincide (see Definition \ref{def:cycles} and Remark \ref{rem:paths-no-action}).

\begin{lemma}
\label{lem:WH-cycles-trivial}
Let $\rho$ be a $W^H$-cycle based at a point $x \in M^\alpha(p)$, and $\eta$ be an $(\alpha,-\alpha,\R^k)$-path. Then $\pi_\eta(\rho)$ is also a cycle.
\end{lemma}

\begin{proof}
Notice that it suffices to prove the lemma when $\eta$ consists of a single leg, since by definition, $\pi_\eta$ is the composition of the projections of the legs of $\eta$. If $\eta$ is only a piece of an $\R^k$-orbit, the lemma follows from the fact that the $\R^k$-action takes $W^H$-cycles to $W^H$-cycles. We prove it here when $\eta$ is a single $\alpha$ leg (the proof for $-\alpha$ legs is identical). Notice that $\pi_\eta(\rho)$ is a cycle if and only if $e(\pi_\rho(\eta)) = e(\eta)$. That is, it suffices to show that $\pi_\rho : W^\alpha(x) \to W^\alpha(x)$ is the identity map. Notice that from Lemma \ref{lem:linear-holonomy}, $\pi_\rho$ is an orientation-preserving affine map  in normal forms coordinates  of Theorem \ref{thm:normal-forms-foliation} which fixes $x$. If $\pi_\rho \not= \id$, then the orbits $\set{\pi_\rho^n(y) : n \in \Z}$ are infinite. Since every $\pi_\rho^n(y) \in W^H(y)$, and $W^H(y) = W^H_N(y)$ by Lemma \ref{lem:WH-WHN}, we contradict that we are in case \ref{dich:1} of the return dichotomy. This finishes the proof.
\end{proof}

\begin{lemma}
\label{lem:r-constant}
The function $r$ is constant on $M^\alpha(p)$.
\end{lemma}

\begin{proof}
Suppose that $r(x) < r_0$ for some $x \in M^\alpha(p)$. Then since $r = r_0$ on an open dense subset of $M^\alpha(p)$, there is a nearby point $x'$ such that $r(x') = r_0$. Choose a short $(\alpha,-\alpha,\R^k)$-path $\eta$ connecting $x$ and $x'$. Then project the set of paths giving the intersection point of $W^H(x') \cap M^\alpha(p)$ to $x$ along $\eta$. Since the cardinality is strictly greater, there must be some $W^H$-path $\rho'$ starting from $x'$ and returning to some $y' \not= x'$ such that $\rho = \pi_\eta(\rho')$ is a cycle at $x$. Then $\pi_{\eta^{-1}}$ maps a $W^H$-cycle $\rho$ to a non-cycle, contradicting Lemma \ref{lem:WH-cycles-trivial}. So $r$ is constant on $M^\alpha(p)$.
\end{proof}

\begin{proposition}
\label{prop:case1-finished}
If $\R^k \curvearrowright X$ is totally Cartan, cone transitive and $C^r$, $r = (1,\theta)$ or $r = \infty$ and satisfies case \ref{dich:1} of the return dichotomy, then the space $X / W^H$ is a $C^{r}$ manifold, and the projection $\pi : X \to X/ W^H$ determines a $C^{r}$, non-Kronecker rank one factor of $\R^k \curvearrowright X$.%, where $r' = (1,\theta)$ or $\infty$, respectively.
\end{proposition}

\begin{proof}
We model the space $X/W^H$ locally on $M^\alpha(p)/H$. Notice that by Lemma \ref{lem:WHN-transitivity}, every $W^H$-leaf intersects $M^\alpha(p)$, and by Lemma \ref{lem:r-constant} the number of intersections is constant. We construct manifold charts as follows: given $x := x_1 \in M^\alpha(p)$, let $\set{x_1,\dots,x_r} = W^H(x) \cap M^\alpha(p)$. Choose $\ve > 0$ such that $M^\alpha_\ve(x_i) \cap M^\alpha_\ve(x_j) = \emptyset$ if $i \not= j$. Let $\delta = A^{-N}\ve$ and let $U = B_\delta(x) \cap M^\alpha(p)/H$. Then $U$ is an open set in a 3-manifold, and has a chart $\psi : V \to U$, where $V \subset \R^3$. Define the chart for $X / W^H$ by $\tilde{\psi}(t) = W^H(\psi(t))$.

We first prove that $\tilde{\psi}$ is injective. Any $W^H$-path connecting $x$ to $x_i$ can be achieved in $W^H_N$ by Lemma \ref{lem:WH-WHN}. Notice that by choice of $\delta$, each point $y := y_1 \in U$ must intersect $M^\alpha(p)$ in points $\set{y_1,\dots,y_r}$ near $\set{x_1,\dots,x_r}$, so $y_1$ is the only $y_i$ in $U$. Therefore, the map is injective.

 The transition maps between charts are $C^{r}$, since $W^H$ is a $C^{r}$ foliation. Therefore, $X /W^H$ has an induced $C^{r}$ manifold structure, and the projection $X \to X / W^H$ is $C^{r}$ by construction. Finally, since the $\R^k$-action takes $W^H$-leaves to $W^H$-leaves, and $H$ fixes the $W^H$-leaves, the action descends to an Anosov flow on $X / W^H$.
\end{proof}

\section{Infinite returns: Finding a dense $H$-orbit}
\label{sec:case2}

Throughout this section, we assume that we are in  case \ref{dich:2} of the return dichotomy, giving some $N \in \N$ and $x \in M^\alpha(p)$ such that $W^H_N(x) \cap M^\alpha(p)$ is an infinite union of $H$-orbits. While in case \ref{dich:1}, we showed the existence of a rank one factor  (and in particular, that no $H$-orbit is dense). In this section we show that we have one of the other two conclusions of Theorem \ref{thm:main-anosov}: that there is a dense $H$-orbit or that $H$-orbit closures are dense in the fiber of some circle factor (Proposition \ref{prop:case2-finished}). In particular, this is a proof by contradiction, as we have assumed in establishing the return dichotomy that no $H$-orbit is dense.

Recall the holonomy action of $W^H$-paths on $M^\alpha$ paths (and vice versa) given by Corollary \ref{cor:holonomy-action}: if $\rho$ is a $W^H$-path and $\eta$ is an $(\alpha,-\alpha,\R^k)$-path which share the same base point, $\pi_\rho(\eta)$ is the path $\eta$ ``slid along'' $\rho$, beginning at $e(\rho)$, and similarly $\pi_\eta(\rho)$ is the path $\rho$ ``slid along'' $\eta$, beginning at $e(\eta)$. We first prove the following:

\begin{lemma}
\label{lem:return-dichotomy}
For any $x \in M^\alpha(p)$  such that $W^H_N(x) \cap M^\alpha(p)$ is an infinite union of $H$-orbits, $x \in \overline{W^H_{2N}(x) \cap M^\alpha(p) \setminus (H\cdot x)}$.
\end{lemma}

\begin{proof}
By assumption, there are infinitely many $W^H_N$-paths $\rho_1,\rho_2,\dots$ based at $x$ such that $x_i = e(\rho_i)$ all belong to distinct $H$-orbits. %Notice that $x_i = e(\pi_{\sigma_i}(\rho))$ as well, by Corollary \ref{cor:holonomy-action}, and that $s(\pi_{\sigma_i}(\rho)) = s(\rho)$. 
Since $\{x_i\}$ are infinitely many distinct points in a compact space, they accumulate somewhere, and for every $\ve > 0$, there exist $i$ and $j$ such that $0 < d(x_i,x_j) < \ve$. By adding an $H$-leg as necessary, we may assume that each $x_i$ can be reached from $x$ by a short $W^\alpha$, $W^{-\alpha}$, $L$-path (where $L$ is some fixed subgroup transverse to $H$).  %Then by following the path ${\rho_i}^{-1}$ from $x_i$ to $x$, then $\rho_j(\rho)$ from $x$ to $x_j$, there exists a $W^H$-path $\rho'$ from $x_i$ to $x_j$ with $s(\rho') \le 2N$.

Choose a very short $(\alpha,-\alpha,\R^k)$-path $\eta_{ij}$ connecting $x_i$ and $x_j$, and let $\ell_{ij} = \ell_{ij}(\ve)$ denote its length. Notice that $\ell_{ij} \to 0$ as $\ve \to 0$. Recall $\rho_i^{-1}$ is the opposite $W^H$-path connecting $x_i$ with $x$. Then $\rho_i' = \pi_{\eta_{ij}}(\rho_i^{-1})$ is a $W^H$-path connecting $x_j$ with some point whose distance from $x$ is at most $A^N \ell_{ij}$ by Lemma \ref{lem:holonomy-bounded-der}. Furthermore, $s(\rho_i') = s(\rho_i) \le N$, so the concatenation $\rho_i' * \rho_j$ connects $x$ to a point at distance at most $A^N \ell_{ij}$ from $x$, and has $s(\rho_i' * \rho_j^{-1}) \le 2N$ and $e(\rho_i' * \rho_j^{-1}) \not \in H \cdot x$.
%Fix an arbitrary $y \in M^\alpha(p)$. Since $M^\alpha(p)$ is compact, we may connect $x_i$ to $y$ with an $\alpha,-\alpha,\R^k$-path $\eta$ of uniformly bounded length and the number of legs of $\eta$ also bounded. Now, $\rho'' = \pi_\eta(\rho')$ is a path based at $y$ with $s(\rho'') = 2s(\rho)$. Furthermore, $e(\pi_\eta(\rho')) = e(\pi_{\rho'}(\eta)$, and $\pi_{\rho'}(\eta)$ is a path whose lengths $C^0$-close to $\rho'$, so its endpoint is close to the endpoint of $\eta$, which is $y$. such that $d(y, e(\rho'')) < C\ve$ for some $c$. Since $\ve$ was arbitrarily small, we get (2).
\end{proof}

 Fix $\ve_0$ as in Lemma \ref{lem:psi-coords}. We say that a $W^H$-path $\rho$ is a {\it local $M^\alpha$-return at $x$} if it is a $W^H$-path based at $x$, and there exists an $(\alpha,-\alpha,\R^k)$-path $\eta$ of length at most $\ve_0$ such that $e(\rho) = e(\eta)$. 
 
 \begin{definition}
 Let $S_N$ be the set of points $x\in X$ such that there exists a sequence of $W^H_N$-paths $\rho_n$, which begin at $x$ and end at $M^\alpha_{\ve_0}(x)$ such that $e(\rho_n) \to x$ and $e(\rho_n) \not\in B_H(0,\ve_0) \cdot x$, where $B_H(0,\ve_0)$ is the ball of radius $\ve_0$ in $H$ (ie, points which are accumulated by local $M^\alpha$-returns with at most $N$ switches). We call $S_N$ the {\rm  $N$th self-accumulating set.}
 \end{definition}

%$S_N = \set{ x \in M^{\alpha}(p) : \overline{W^H_N(x) \cap M^{\alpha}(p) \setminus \set{x}} \ni x}$ and notice that $S_N$ is $\R^k$-invariant, and that Case (2) of Lemma \ref{lem:return-dichotomy} is exactly the condition that $S_N \not= \emptyset$ for some $N \in \N$.

%{ 
%\begin{lemma}
%\label{lem:SN-closed}
%$S_N$ is closed.
%Let $k = 2 \abs{\Delta}$. Then $\overline{S_N} \subset S_{N+k}$.
%\end{lemma}

%\begin{proof}
%...something something transversality...
%Let $x$ be an accumulation point of $S_N$, so that there exists $x_i \in S_N$ such that $x_i \to x$. Fix $\ve > 0$. For sufficiently large $i$, $x_i$ is uniquely reached from $x$ by a path with a short $\alpha$ leg, a short $-\alpha$ leg and a short $\R^k$ leg (since $M^\alpha(p)$ is compact). Call this path $\eta_i$, which begins at $x_i$ and ends at $x$. Choose $i$ large enough so that each of the legs of $\eta_i$ has length at most $\ve / 9A^{N}$. Since $x_i \in S_N$, there exists a $W^H$-path $\rho_i$ such that $d(x_i,e(\pi_{\rho_i}(x_i))) < \ve / 3$ and $e(\rho_i) \in M^\alpha(p)$. Then let $\rho_i' = \pi_{\eta}(\rho_i)$, so that $\rho_i'$ is a path that begins from $x$ (since the endpoint of $\eta$ is $x$). Notice that:

%\[ d(x,e(\rho_i')) \le d(x,x_i) + d(x_i,e(\rho_i)) + d(e(\rho_i),e(\rho_i')) \le \frac{\ve}{3A^N} + \frac{\ve}{3} + 3\cdot A^N \cdot \frac{\ve}{9A^N} < \ve, \]

%since $e(\rho_i)$ and $e(\rho_i')$ are exactly the endpoints of the path $\pi_{\rho_i}(\eta_i)$. Since $\ve$ was arbitrary and $s(\rho_i') = s(\rho_i)$, $x \in S_N$.
%\end{proof}}

Lemma \ref{lem:return-dichotomy} exactly implies that there exists $x \in S_N \cap M^\alpha(p)$ for some $\R^k$- and $H$-periodic orbit $p$ and some $N \in \N$. Notice that for any point $y \in W^H(x) \cap M^{\alpha}(p)$ sufficiently close to $x$, $y$ is in the image of $\psi_x$, where $\psi_x$ is as in Lemma \ref{lem:psi-coords}. Let $A_x \subset \R^{k+2}$ be defined by $A_x = \psi_x^{-1}(W^H(x) \cap M^{\alpha}(p) \cap B(x,\ve))$. Notice that since we assume that $x \in \overline{W^H(x) \cap M^{\alpha}(p)}$, $0 \in \overline{A_x \setminus \set{0}}$. That is there exists a sequence $(t_n,s_n,v_n) \to 0$ such that $t_n,s_n \in \R$, $v_n \in \R^k$ and ${t_n}^2 + {s_n}^2 + \norm{v_n}^2 \to 0$. We describe types of accumulation for elements of $S_N$, which we call Type I-V (in each case, we allow the ability to take a subsequence):

\vspace{2.7cm}

\begin{center}
{\hypertarget{acc-modes}{\bf Table of Accumulation Modes}} 
\vspace{.8em}

\begin{tabular}{|c|r|l|}
\hline
I. & Generic accumulation & $t_k,s_k \not= 0$ and $v_k \not\in H$ for all $k$ \\
\hline
II. & Center-stable accumulation & For every $k$, $s_k = 0$, $v_k \not\in H$ and $t_k \not= 0$ \\
& (or center-unstable accumulation) & \qquad (or $t_k = 0$, $v_k \not\in H$ and $s_k \not= 0$) \\
\hline
III. & Orbit accumulation & For every $k$, $s_k = 0$ and $t_k = 0$ \\
\hline
IV. & Non-$L$ accumulation & For every $k$, $v_k \in H$ and $s_k,t_k \not= 0$ \\
\hline
V. & Strong-stable accumulation & For every $k$, $s_k = 0$ and $v_k \in H$ \\
 & (or strong-unstable accumulation) &  \qquad (or $t_k = 0$ and $v_k \in H$)\\
\hline
\end{tabular}
\end{center}

\vspace{.2cm}

%\begin{lemma}
%There exists $\ve > 0$ such that

%\begin{enumerate}
%\item[(a)] $S_N$ is $\R^k$-invariant
%\item[(b)] If $x \in S_N$, then $W^\alpha_\ve(x) \subset S_{2N}$
%\item[(c)] If $x \in S_N$, then $W^{-\alpha}_\ve(x) \subset S_{2N}$
%\end{enumerate}
%\end{lemma}

%\begin{corollary}
%If $M^\alpha(p)$ is compact, either $S_N = \emptyset$ for every $N$ or there exists $N$ such that $S_N = M^\alpha(p)$.
%\end{corollary}

In each accumulation type, we think of $x_n$ as accumulating nontrivially in the quotient of $M^\alpha(p)$ by $H$ (see Proposition \ref{prop:factor-by-H}). This is natural since we may add any $H$-leg to the end of a $W^H$-path without increasing the switching number. Thus, any $H$-component can be ignored from an arbitrary self-accumulation. We will prove the following Lemma for each of the cases above:

\begin{lemma}
\label{lem:distinct-alpha}
If $x \in S_N \cap M^\alpha(p)$, and there exists sequence $x_n \to x$ such that $x_n \in W^H_N(x) \cap M^\alpha(p)$ and  $x_n \not\in B_H(0,\ve_0) \cdot W^\alpha_{\ve_0}(x)$ for sufficiently large $n$, then for every $y \in W^{\alpha}(x)$, $y \in S_{2N}$.
\end{lemma}

\begin{proof}
Suppose $x_n \to x$ is a sequence converging to $x$ as described in the statement of the Lemma, so that $x_n \not\in B_H(0,\ve_0)\cdot W^\alpha_{\ve_0}(x)$ for every $n$. We claim that we may assume that $x_n \not\in B_H(0,\ve_0)\cdot W^\alpha_{\ve_0}(x_\ell)$ for all $n \not= \ell$.  Indeed, if the local $B_H(0,\ve) \cdot W^\alpha$-leaf did not change infinitely often, it would eventually stabilize. But then the limit of the $x_n$ would lie in that local $H\cdot W^\alpha$-leaf, and the accumulation would be of type V. Let $\eta$ denote the path in $W^{\alpha}(x)$ from $x$ to $y$ with only one leg.  Let $\rho_n$ be the $W^H_N$-path which connects $x$ to $x_n$, and set $\rho_n' = \pi_\eta(\rho_n)$. Define $y_n = e(\pi_\eta(\rho_n))$.

Notice that by Lemma \ref{lem:holonomy-bounded-der}, $y_n \in W^\alpha_{A^N\ve_0}(x_n)$ since $y \in W^\alpha_{\ve_0}(x)$  and $x_n \in W^H_N(x)$. Hence, since $x_n \not\in B_H(0,\ve_0) \cdot W^\alpha_{\ve_0}(x_\ell)$, $y_n \not\in B_H(0,\ve_0) \cdot W^\alpha_{A^{-N}\ve_0}(y_\ell)$ for every $\ell \not= n$. In particular, $y_n \not= y_\ell$. Notice also that if $y$ is sufficiently close to $x$, then $y_n$ stays within a bounded distance of $y$, and therefore has a convergent subsequence, $y_k \to z$. As in the proof of Lemma \ref{lem:return-dichotomy}, we now choose $y_n$ and $y_\ell$ such that $d(y_n,y_\ell) < \delta$. Then connect them by a short $(\alpha,-\alpha,\R^k)$-path $\eta'_{n,\ell}$ which begins at $y_n$ and ends at $y_\ell$. Set $\tilde{\rho}_{n,\ell} = \pi_{\eta'_{n,\ell}}({\rho_n'}^{-1}) * \rho_n'$. Then $\tilde{\rho}_{n,\ell}$ is a path which starts at $y$, and ends at $e(\pi_{\eta'_{n,\ell}}((\rho_n')^{-1})) = e(\pi_{{\rho_n'}^{-1}}(\eta'_{\ell}))$, which has length less than $A^N\delta$. Furthermore, the endpoint cannot be equal to $y$ since it lies on a distinct $\alpha$ leaf by construction. Thus, since $\delta$ was arbitrarily small, $y \in S_{2N}$.
\end{proof}

Since $S_N$ is $\R^k$-invariant, and the assumptions of Lemma \ref{lem:distinct-alpha} fail only for type V. accumulation,  the following is immediate from Lemma \ref{lem:distinct-alpha}:

\begin{corollary}
\label{cor:extending-SN}
For every $N \in \N$, there exists $\ve_1 > 0$ such that if $x \in S_N \cap M^\alpha(p)$, $x$ has an accumulation of type I,II,III or IV, and $y \in W_{\ve_1}^\alpha(x) \cup W^{-\alpha}_{\ve_1}(x) \cup \R^kx$, then $y \in S_{2N}$.
\end{corollary}

We now turn our attention to Type V accumulation. Let $\Aff(\R) = \set{ax +b : a,b \in \R, a\not= 0}$ be the group of affine transformations of $\R$. We think of $\Aff(\R)$ as the matrix group $\set{\begin{pmatrix} a & b \\ 0 & 1 \end{pmatrix} : a,b \in \R, a\not= 0}$, whose Lie algebra is $\Lie(\Aff(\R)) = \set{\begin{pmatrix} a & b \\ 0 & 0 \end{pmatrix} : a,b \in \R}$. The following lemma is a standard result.

\begin{lemma}\label{lem:aff-conjugacy}
For every $X \in \Lie(\Aff(\R))$, there exists $g \in \Aff(\R)$ such that $\Ad(g)X$ is a multiple of either $\begin{pmatrix} 1 & 0 \\ 0 & 0 \end{pmatrix}$ or $\begin{pmatrix}0 &1 \\ 0 & 0 \end{pmatrix}$.
Any one-parameter subgroup of $\Aff(\R)$ is either the additive group $\set{x+b : b \in \R}$, or conjugate to the multiplicative group $\set{ax : a \in \R_+}$.
\end{lemma}

\begin{lemma}
\label{lem:typeV}
Let $x \in S_N \cap M^\alpha(p)$ and have a type V accumulation (so that there exists $x_k \in W^H_N(x) \cap W^\alpha(x)$ such that $x_k \to x$, $x_k \not= x$). There exists $\ve_2>0$, $M \in \N$ and $z \in W^\alpha(x) \setminus \set{x}$ such that if $y \in W^{-\alpha}_{\ve_2}(x) \cup (W^\alpha_{\ve_2}(x) \setminus \set{z}) \cup \R^kx$, $y \in S_M$. Furthermore, each $y \in W^\alpha_{\ve_2}(x)\setminus \set{z}$ also has a {self-accumulation} of type V along the $\alpha$-leaf (ie, there exists $y_n \in W^H_{2N}(y) \cap W^\alpha(y)$ such that $y_n \to y$).
\end{lemma}

\begin{proof}
Since $W^H$ is saturated by $H$-orbits, we may without loss of generality assume that if $x_k = \psi_x(t_n,s_n,v_n)$, then $v_k = 0$. Notice that accumulations of Type V correspond to accumulating along an $\alpha$ or $-\alpha$ leaf. We without loss of generality assume that $x_k$ accumulates to $x$ along an $\alpha$-leaf. In this case, notice that each $x_k$ belongs to a distinct local $-\alpha$ leaf, so we may apply {Lemma \ref{lem:distinct-alpha} (replacing $\alpha$ by $-\alpha$)} to get $y \in S_{2N}$ for all $y \in W^{-\alpha}_{\ve_2}(x)$.

So we must show $y \in S_{2N}$ for all but one $y \in W^\alpha(x)$. The previous argument fails exactly because the $\alpha$ leaf fails to change. Notice that since $x \in S_N$, there exists a sequence of paths $\rho_n$ in $W^H$ with $s(\rho_n) \le N$ such that $e(\rho_n) \in W^\alpha(x)$ and $e(\rho_n) \to x$. We may extend each $\rho_n$ to a holonomy {(from $W^\alpha(x)$ to $W^\alpha(x_n) = W^\alpha(x)$)} as described in Lemmas \ref{lem:smooth-holonomy} and \ref{lem:linear-holonomy}. Since $e(\rho_n) \in W^\alpha(x)$, we may use the normal forms charts of Theorem \ref{thm:normal-forms-foliation} (and part \ref{enum:nf-consistency} of that theorem to  center the charts at $x$) to get that the endpoints of the path translated along $\alpha$ are given by affine maps $\varphi_n \in \Aff(\R)$ such that $\varphi_n(t) = a_nt + b_n$, $A^{-2N} \le a_n \le A^{2N}$, $b_n \not= 0$ and $b_n \to 0$. 

Notice that $\varphi_n$ is a precompact family of affine maps. Therefore, the collection $\varphi_n \of \varphi_m^{-1}$ accumulates at $\id \in \Aff(\R)$. Furthermore, $\varphi_n \of \varphi_m^{-1}$ is a map which associates to a point $y$, the endpoint of a path in $W^H$ based at $y$ with at most $2N$ switches. Since these maps accumulate nontrivially at the identity, evaluation of these maps at $y \in W^\alpha(x)$ gives a sequence of points converging to $y$ which lie in $W^H_{2N}(y)$. If this accumulation occurs along a sequence without a common fixed point (ie, they do not lie in a multiplicative one-parameter subgroup as described in Lemma \ref{lem:aff-conjugacy}), one can use the corresponding returns to get that $y \in S_{2N}$ for all $y \in W^\alpha(x)$. Otherwise, there is a common fixed point of the transformations converging to $\id$ (this is the point $z$ as in the statement), and all $y \in W^\alpha(x) \setminus \set{z}$, $y \in S_{2N}$ and has accumulation of type V along the $\alpha$-leaf. %Since $S_{2N}$ is closed by Lemma \ref{lem:SN-closed}, $W^\alpha(x) \subset S_{2N}$.
%So we must show $y \in S_{2N}$ for all $y \in W^\alpha(x)$. The previous argument fails exactly because the $\alpha$ leaf fails to change. Use a Riemannian metric to identify $W^\alpha(x)$ as a copy of $\R$, with $x$ corresponding to 0. Without loss of generality, we may assume that $x$ is accumulated along the positive reals. We claim that $\overline{W^H(x) \cap \set{ t > 0} } = \set{ t \ge 0}$. Suppose otherwise. Then $\set{ t \ge 0} \setminus \overline{W^H(x) \cap \set{ t > 0} } = \set{ t \ge 0} = \bigcup_{i=1}^\infty (a_i,b_i)$ is a (possibly finite) union of open intervals, and since it accumulates at 0 $a_i \not= 0$.
\end{proof}

\begin{lemma}
\label{lem:periodic-accum}
If $S_N \cap M^\alpha(p) \not= \emptyset$, then $S_{N'}$ contains an open set in $X$ for some $N' \in \N$.%there exists some $\R^k$-periodic point $q \in X$ and $W^H$-paths $\rho_n$ such that $e(\rho_n) \to q$ and $q \not= e(\rho_n) \in M^\alpha(q)$ (ie, $q$ belongs to its own self accumulation set).
\end{lemma}

\begin{proof}
We prove the Lemma in two distinct cases: since $S_N \cap M^\alpha(p)$ is nonempty, we have either that every $x \in S_N \cap M^\alpha(p)$ has an accumulation of type I-IV, or there exists $x \in S_N \cap M^\alpha(p)$ which has a type V accumulation. If every point has an accumulation of type I-IV, then for some $N'$, one can build an open subset of $S_{N'-2} \cap M^\alpha(p)$ by iteratively saturating with local $W^\alpha$-leaves, $W^{-\alpha}$-leaves and $\R^k$-orbits. % to get an set in $S_{N'-2} \cap M^\alpha(p)$ which is open in $M^\alpha(p)$ for some $N' \in \N$.
 Notice that we can also saturate with $\beta$-leaves by applying $\beta$-holonomies to get that for some open set $U \subset X$ every point in $U$ is self-accumulated by a $W^H$-path with at most $N'$ switches whose endpoints are connected by arbitrarily short, nontrivial, $\alpha$, $-\alpha$, $\R^k$-paths. This proves the Lemma in this case. %By the Anosov closing lemma, we may choose a periodic orbit $q$ from $U$. This proves the Lemma in this case.

Now assume that $x \in S_N \cap M^\alpha(p)$ has an accumulation of type V, and without loss of generality, assume it is along the $W^\alpha$-leaf. Then by Lemma \ref{lem:typeV}, there is a neighborhood $U_1 \subset W^\alpha(x)$ containing $x$ such that every $y \in U_1$ also has a type V accumulation. Therefore, again by Lemma \ref{lem:typeV}, we may saturate $U_1$ with $-\alpha$ leaves and still lie in $S_{N'-2}$. Since $S_{N'-2}$ is saturated by $\R^k$-orbits, we have constructed an open set in $M^\alpha(p) \cap S_M$. Repeating the argument above obtains an open set of points in $X$ which has the desired self-accumulation property.%, and we may choose a periodic orbit from it.
%Begin with some $x_0 \in S_N$. Then by Lemma \ref{lem:extending-SN}, a $W^\alpha_{\ve_0}(x_0) \subset S_{M'}$ for some $\ve_0,M'$. For each $y \in W^\alpha_{\ve_0}(x_0) \subset S_{M'}$, reapplying Lemma \ref{lem:extending-SN}, $W^{-\alpha}_{\ve_1}(y) \subset S_M$ for some $\ve_1>0$ and $M\ge M'$. $S_M$ is also saturated by $\R^k$-orbits. Since $S_{M'} \subset S_M$, we have therefore produced an open set inside $S_M$.
\end{proof}

As an immediate corollary of Lemma \ref{lem:periodic-accum}, we get the following:

\begin{lemma}
\label{lem:eps-dense-accum}
There exists $N > 0$ and $\ve_3 > 0$ such that if $p$ is an $\R^k$-periodic point whose $\R^k$-orbit is $\ve_3$-dense, then %periodic orbit, 
$p \in S_N$. %In particular, there is a dense set of periodic orbits in $S_N$. 
\end{lemma}

%{  This implies that the saturation of $S_N$ by local $W^H$-moves is in $S_{N + \abs{\Delta}}$ and open. Using transitivity, one gets that $S_{N + \abs{\Delta}}$ is dense, so for all periodic orbits $q$ in an open dense set, we have $q \in S_N$.}

Let $p$ be an $\R^k$- and $H$-periodic point in $S_N$, which is guaranteed to exist by Lemma \ref{lem:eps-dense-accum} and Corollary \ref{cor:eps-dense}. %, so that we may construct the set $S_N$ for $p$ such that $p \in S_N$ for some $N \in \N$. %Since $S_N$ contains an open set for some $N$, we may choose some $z \in S_N$ such that $z$ is periodic.
 Notice that for any point $y \in W^H(p) \cap M^{\alpha}(p)$ sufficiently close to $p$, $y$ is in the image of $\psi$, where $\psi$ is as in Lemma \ref{lem:psi-coords}. Let $A_p \subset \R^{k+2}$ be defined by $A_p = \psi_p^{-1}(W^H(p) \cap M^{\alpha}(p) \cap B(p,\ve))$. Notice that since we assume that $p \in \overline{W^H_N(p) \cap M^{\alpha}(p)  \setminus \set{p}}$, $0 \in \overline{A_p  \setminus \set{0}}$.

\begin{lemma}
\label{lem:closure-types}
For every $\R^k$-periodic orbit $p$ belonging to $S_N$, at least one of the following holds:

\begin{enumerate}
\item[(a)] $\overline{W^H(p)} \supset \R^k \cdot p$,
\item[(b)] $\overline{W^H(p)} \supset W^\alpha(p)$, or
\item[(c)] $\overline{W^H(p)} \supset W^{-\alpha}(p)$.
\end{enumerate}
\end{lemma}

Since $p \in S_N$, we know that $p$ has accumulation of some type I-V described above. We prove that at least one of (a), (b) or (c) holds in each of the cases. %The numbering of types I-V was made so that the proof of Lemma \ref{lem:extending-SN} proceeded in a linear fashion. 
We prove Lemma \ref{lem:closure-types} for each type of accumulation, in the following order: III, V, II, I, IV.

\begin{proof}[Proof of Lemma \ref{lem:closure-types}, Type III]
In type III, we know that $W^H(p)$ accumulates to $p$ only along the $\R^k$-orbit of $p$. We claim that $\Gamma = \set{ a \in \R^k : a\cdot p \in W^H(p)}$ is a subgroup. Indeed, if $a \in \Gamma$, there exists a $W^H$-path $\rho$ starting from $p$ such that $e(\rho) = a\cdot p$. Then $(-a) \cdot \rho$ is a $W^H$-path starting from $a^{-1} \cdot p$ and ending at $p$. Reversing the path results in a path from $p$ to $(-a) \cdot p$. That is, $(-a) \in \Gamma$.

Now suppose that $a,b \in \Gamma$. Then there exist paths $\rho_1$ and $\rho_2$ beginning at $p$ and ending at $a\cdot p$ and $b \cdot p$, respectively. Then $a \cdot \rho_2$ is a $W^H$-path that begins at $a \cdot p$ and ends at $(a+b) \cdot p$. Concatenating $\rho_1$ and $a \cdot \rho_2$ gives a path from $p$ to $(a+b) \cdot p$, so $a+b \in \Gamma$.

Now, $\Gamma$ is a subgroup of $\R^k$ containing the codimension one subgroup $H$ and which accumulates on 0 transversally to $H$. Therefore, it must be dense. So we are in case (a) of Lemma \ref{lem:closure-types}.
\end{proof}

\begin{proof}[Proof of Lemma \ref{lem:closure-types}, Type V]
In this case, $W^H_N(p)$ accumulates to $p$ only along an $\alpha$ or $-\alpha$ leaf. Assume without loss of generality that it is along the $\alpha$ leaf, and set $C = W^H(p) \cap W^\alpha(p)$. Fix the normal forms parameterization of $W^\alpha(p)$ as described in Section \ref{sec:normal-forms}. Let $\rho$ be a $W^H$-path with endpoint $e(\rho) \in W^\alpha(p)$. Recall from the proof of  Lemma \ref{lem:typeV} that $\rho$ induces a map $\varphi_\rho : W^\alpha(p) \to W^\alpha(p)$ such that the endpoint of $\pi_\rho([p,x])$ is $\varphi_\rho(x)$, where $[p,x]$ is the path with one leg along $W^\alpha$ connecting $p$ and $x$. Furthermore, the map $\varphi_\rho$ is affine in the normal forms coordinates at $p$ (using Theorem \ref{thm:normal-forms-foliation}). We have that there exist paths $\rho_n$ such that $\varphi_n(x) = \varphi_{\rho_n}(x) = a_nx + b_n$ is a sequence of affine maps such that $A^{-2N} \le a_n \le A^{2N}$ and $b_n \to 0$.

Let $\Gamma$ be the subgroup generated by $\varphi_n$. Notice that $\overline{\Gamma}$ is a Lie group and it cannot be discrete since it accumulates at 0. If $\overline{\Gamma}$ is 2-dimensional, it must be the identity component of $\Aff(\R)$, and therefore acts transitively. In particular, we know we are in case (b) of the Lemma.

If $\overline{\Gamma}$ is 1-dimensional, its identity component is a one-parameter subgroup, and by Lemma \ref{lem:aff-conjugacy} is therefore either the additive group $\set{x+ b : b \in \R}$ or conjugate to the multiplicative group $\set{ax : a \in \R_+}$. If it is the additive group, 0 has a dense orbit and we are again in case (b) of the Lemma. If it is the muliplicative group, $p$ cannot be the common fixed point since $\varphi_n(p) \not= p$ for every $n$. Therefore, the orbit closure of $p$ under $\overline{\Gamma}$ is a ray $(t,\infty)$, where $t < 0$ if the coordinates are centered at $p$  (or, symmetrically, a ray $(-\infty,t)$ for some $t > 0$). Choose an element $a \in \R^k$ such that $a \cdot p = p$ and $\alpha(a) > 0$. Then the orbit $\Gamma \cdot p$ is also invariant under multiplication by $\norm{da|_{W^\alpha(p)}}$. Therefore, the $\Gamma$-orbit of $p$ cannot be a ray, so we know that $\overline{\Gamma}$ contains the additive group. So we are in case (b) of the lemma.
\end{proof}

\begin{proof}[Proof of Lemma \ref{lem:closure-types}, Type II]
The proof is similar to Type III, but with an additional twist. Let $\Gamma = \set{ a \in \R^k : W^\alpha(a\cdot p) \cap W^H(p) \not= \emptyset}$. Notice that $\Gamma$ accumulates at 0 by construction, so it is not discrete. We claim that $\Gamma$ is a subgroup. Indeed, if $\rho$ is a $W^H$-path that connects $p$ to some $x \in W^\alpha(ap)$, then $a^{-1}\rho$ is a path that connects $a^{-1}\cdot p$ to some $y \in W^\alpha(p)$. Applying an $\alpha$ holonomy to this path connects some point $z \in W^\alpha(a^{-1}p)$ to $p$. Reversing the path gives that $a^{-1} \in \Gamma$.

Now, assume that $a,b \in \Gamma$. Then there exist $W^H$-paths $\rho_1$ and $\rho_2$ starting at $p$ and terminating at some points $x_1 \in W^\alpha(ap)$ and $x_2 \in W^\alpha(bp)$, respectively. Then $a \cdot \rho_2$ is a path that begins at $ap$ and ends at some $y \in W^\alpha((a+b)p)$, and again applying an appropriate $\alpha$-holonomy, we find a path $\rho_1'$ beginning at $x_2$ and ending at some $z \in W^\alpha((a+b)p)$. Concatenating $\rho_2$ and $\rho_1'$ gives a desired path.

We claim that for every $a \in \Gamma$, $ap \in \overline{W^H(p)}$. Indeed, notice that if $a \in \Gamma$, there exists a $W^H$-path $\rho$ such that $e(\rho) \in W^\alpha(ap)$. Choose some $a_0$ such that $\alpha(a_0) < 0$ and $a_0 p = p$. Then ${a_0}^n\rho$ is still a path that begins at $p$ and ends at some point in $W^\alpha(ap)$. But $\lim_{n\to\infty} e({a_0}^n\rho) \to ap$. Therefore, $ap \in \overline{W^H(p)}$, and since $\Gamma$ is a dense subgroup, we are in case (a) of Lemma \ref{lem:closure-types}.
\end{proof}

The last two cases are the most difficult, and we adapt the argument used in Type II to the case when $\overline{W^H(p)}$, but not $W^H(p)$, accumulates to $p$ along $W^\alpha(p)$. We develop a lemma which aids us in the proof. Recall that if $\rho$ is any $W^H$-path based at $p$, $\rho$ determines a holonomy map $\pi_{\rho} : W^\alpha(p) \to W^\alpha(e(\rho))$ (see Lemmas \ref{lem:smooth-holonomy} and \ref{lem:linear-holonomy}). Say that a sequence of $W^H$-paths $(\rho_k)$ based at $p$ is a {\it controlled return} if:

\begin{itemize}
\item $s(\rho_k)$ is uniformly bounded in $k$,
\item $e(\rho_k)$ converges to some point $x \in W^\alpha(p)$, and
\item the derivatives of the holonomy maps $\pi_{\rho_k}$ on $W^\alpha$ converge.
\end{itemize}

 Then let $\Gamma \subset \Homeo(W^\alpha(p))$ be the set of transformations which are limits of the maps $\pi_{\rho_k}$, where $(\rho_k)$ is a controlled return. Note that $\Gamma$ has a grading in the following sense: we let $\Gamma_j = \set{ f\in \Gamma : f = \lim \pi_{\rho_k}, s(\rho_k) \le j \mbox{ for all }k}$. Crucially, we allow the bound $j$ to go to $\infty$ in the definition of $\Gamma$, as we will not be able to show that each $\Gamma_j$ is a subgroup.

\begin{lemma}
\label{lem:controlled-return}
If $f \in \Gamma$, then $f$ is affine in the normal forms coordinates at $p$. Furthermore, $\Gamma$ is a subgroup of $\Aff(\R)$.
\end{lemma}

\begin{proof}
That each $f$ is affine follows from the continuity of the normal forms coordinates and convergence of derivatives (see Section \ref{sec:normal-forms}). So we wish to show that $\Gamma$ is a group. Suppose that $(\rho_k)$ and $(\sigma_k)$ are two controlled returns, and that $f$ and $g$ are their corresponding affine maps. Using the local product structure of the coarse Lyapunov foliations and $\alpha$-holonomies, we may assume without loss of generality that $e(\rho_k), e(\sigma_k) \in M^\alpha(p)$. Since $e(\rho_k) \to x = f(p) \in W^\alpha(p)$, we may construct an $(\alpha,-\alpha,\R^k)$-path $\eta_k = \eta_k' * [p,x]$ as follows: begin with an $\alpha$-leg connecting $p$ to $x$ (this is the leg $[p,x]$), then connect $x$ to $e(\rho_k)$ by a very short $(\alpha,-\alpha,\R^k)$-path (the path $\eta_k'$). Since $e(\rho_k)$ converges to $x$ in $M^\alpha(p)$, we may do this using the local chart $\psi$ at $x$ as described in Lemma \ref{lem:psi-coords}.

Let $\tilde{\sigma}_k = \pi_{\eta_k}(\sigma_k)$ be the projection of the path $\sigma_k$ along $\eta_k$ as built in Corollary \ref{cor:holonomy-action}. Then $\tilde{\sigma}_k$ is a path which begins at $e(\eta_k) = e(\rho_k)$. Notice that if $\sigma_{k}' = \pi_{[p,x]}(\sigma_k)$, then $e(\sigma_k') = e(\pi_{[p,x]}(\sigma_k)) = e(\pi_{\sigma_k}([p,x])) \to g(x) = g(f(p))$ as $k \to \infty$ by definition. Furthermore, since $\eta_k'$ can be made arbitrarily short and $s(\sigma_k) \le N$ the length of $\pi_{\sigma_k}(\eta_k')$ can also be made arbitrarily short by Lemma \ref{lem:holonomy-bounded-der}. Therefore, $e(\tilde{\sigma}_k * \rho_{k}) \to g(f(p))$ so $\tilde{\sigma}_k * \rho_k$ is a controlled return. One can easily check that the derivative of the corresponding limiting holonomy map is $f'(p)\cdot g'(p)$. Combining these facts exactly implies that the limiting holonomy map is $f \of g$, so $\Gamma$ is closed under composition.

We now check that if $f \in \Gamma$, then $f^{-1} \in \Gamma$. Indeed, suppose that $\rho_k$ is a controlled return whose limiting holonomy map is $f$ with $f(p) = x$. Build a path $\eta_k = \eta_k'*[p,x]$ as before, which begins at $p$, moves along the $\alpha$-leaf to $x$, then moves from $x$ to $e(\rho_k)$ with a very short $(\alpha,-\alpha,\R^k)$-path. Let $\bar{\rho}_k$ and $\bar{\eta}_k$ denote the reversal of the paths $\rho_k$ and $\eta_k$, respectively, so they are both paths connecting $e(\rho_k)$ to $p$. Let $\sigma_k = \pi_{\bar{\eta}_k}(\bar{\rho}_k)$, so that $\sigma_k$ is a path based at $p$. Again notice that since $\rho_k$ has $s(\rho_k)$ uniformly bounded, the length of $\pi_{\rho_k}(\eta_k')$ can be made arbitrarily small by choosing $k$ very large. In particular, $\sigma_k$ is a controlled return. We claim that the limiting map of $\sigma_k$ is $f^{-1}$. Indeed, we have shown above that $\tilde{\sigma}_k * \rho_k$ is the controlled return which determines the composition. But by construction,  $\tilde{\sigma}_k$ is the reversal of $\rho_k$. Therefore, their corresponding holonomy maps are inverses, and the sequence of maps induced by $\tilde{\sigma}_k * \rho_k$ is constantly equal to the identity.
\end{proof}

\begin{proof}[Proof Lemma \ref{lem:closure-types}, Types I and IV]
In this case we will show that we are in case (b) (in fact, we will be able to show both (b) and (c)). We use the local chart near $p$ determined by $\psi$ used to describe \hyperlink{acc-modes}{cases I-V} (this was established in Lemma \ref{lem:psi-coords}). We assume that our accumulating sequence $x_n = (t_n,s_n,v_n)$ of endpoints of paths in $W^H_N$ has $v_n$ chosen in some fixed one-parameter subgroup $L$ transverse to $H$, so that we may think of $(t_n,s_n,v_n)$ as a sequence in $\R^3$. Recall that by construction, $t_n$ and $s_n$ are the normal forms parameterizations of the leaves $W^\alpha(p)$ and $W^{-\alpha}(p)$. Define the following cones:

\[ C_K = \set{ (t,s,v) : \abs{t} > K(\abs{s} + \abs{v})} \]

We claim that for every $K > 0$,  we may construct a new sequence $x_n^K = (t_n^K,s_n^K,v_n^K) \in C_K \cap W^H_N(p)$ such that $x_n^K \to 0$. Indeed, fix some $a_0$ such that $\alpha(a_0) > 0$ and $a_0 p = p$. Then apply $a_0$ $k$ times to a path $\rho_n$ starting at $p$ and ending at $x_n = (t_n,s_n,v_n)$ in $W^H_N$ so that the new path ends at $(\lambda^k t_n,\mu^{(k)}s_n,v_n)$, where $\lambda > 1$ is the derivative of $a_0$ at $p$ restricted to $W^\alpha(p)$, and $\mu^{(k)} < 1$ is the derivative of $a_0^k$ at $\psi_{t_n}^\alpha(p)$ (notice that $\lambda^k$ is a power since $p$ is fixed, but $\mu^{(k)}$ is only a cocycle). Let $x_n^K$ be $(\lambda^kt_n,\mu^{(k)}s_n,v_n)$, where $k$ is the smallest integer such that $x_n^K \in C_K$.

We claim that $x_n^K$ still converges to 0 as $n \to \infty$. Indeed, notice that for $(\lambda^kt_n,\mu^{(k)}s_n,v_n) \in C_K$, we must have that $\lambda^k t_n > K(\mu^{(k)}s_n + v_n)$. The smallest such $k$ will therefore result in $\lambda^kt_n \le \lambda K(\mu^{(k)} s_n+ v_n)$. Since $\mu^{(k)} < 1$, and since $s_n,v_n \to 0$, we conclude that $\lambda^k t_n \to 0$ as well.

Define a new set $B$, which is the set of all points $x \in W^\alpha(p)$ such that there exists $N \in \N$ and $W^H$-paths $\rho_k$ starting at $p$ such that $s(\rho_k) \le N$ with $e(\rho_k) \in M^\alpha(p)$, $e(\rho_k) \to x$ (ie, the set of points reached by controlled returns). 
We claim that $p \in \overline{B \setminus \set{p}}$. Indeed, using the normal forms coordinates for $W^\alpha(p)$ centered at $p$, decompose the segment $(0,1]$ as $(0,1] = \bigcup_{k=0}^\infty (\lambda^{-(k+1)},\lambda^{-k}]$. By the existence of the points $x_n^K$, there exists infinitely many $k$ with some $x_k = (t_k,s_k,v_k) \in C_K$ and $t_k \in (\lambda^{-(k+1)},\lambda^{-k}]$. We claim it is true for all $k$. Fix $k_0$, and notice that if $k > k_0$, $a_0^{k-k_0} x_k$ is a point of $C_K$ with $t_k \in (\lambda^{-(k+1)},\lambda^{-k}]$. Letting $K \to \infty$ shows that there is some point of $B \cap (\lambda^{-(k_0+1)},\lambda^{-k_0}]$. Therefore, $B$ accumulates at $p$.

Notice that $B$ is contained in the orbit of the group $\Gamma$ defined in Lemma \ref{lem:controlled-return}, so $\Gamma$ is not discrete. Proceeding exactly as in Type V shows that $\bar{\Gamma}$ must be either all of $\Aff(\R)$ or the additive group. In either case, $W^\alpha(p)$ is contained in its orbit, and we are in case (b).
\end{proof}

%{  Need to address not EVERY periodic orbit belonging... open dense enough?}

%The $\ve_1$-dense $\R^k$-periodic orbits are dense by the Anosov closing lemma (\ref{}). 

Choose a sequence of $H$- and $\R^k$-periodic points $p_n$ such that $p_n$ is $\frac{1}{n}$-dense (which exists by Corollary \ref{cor:eps-dense}). Notice that $p_n \in S_N$ by Lemma \ref{lem:eps-dense-accum}. Then at least one of the cases (a), (b) or (c) of Lemma \ref{lem:closure-types} occur infinitely often. Clearly, if a case occurs infinitely often, it must occur on a dense set of periodic points  since by construction, it will occur on each point of a $\frac{1}{n}$-dense orbit for arbitrarily large $n$. % Therefore, for at least one of (a), (b) or (c), there are a dense set a of periodic orbits satisfying that condition. 
Since (b) and (c) are symmetric we treat only cases (a) and (b).

\begin{lemma}
If (a) occurs on a dense set of periodic orbits, there is a point $x \in X$ such that $\overline{W^H(x)} = X$.
\end{lemma}

\begin{proof}
Let $Y \subset X$ denote the set of points $y$ such that $\overline{W^H(y)} \supset \R^k \cdot y$. Our assumption implies that $Y$ is dense. We claim that it is also a $G_\delta$ set, in which case it must contain a point such that $\R^k \cdot y$ is dense. This will imply the result.

So we must show that it is a $G_\delta$ set. Fix $a \in \R^k$, and let $U_m(a)$ be the set of all points $x \in X$ such that there exists a $W^H$-path $\rho$ based at $x$ such that $d(e(\rho),a\cdot x) < \frac{1}{m}$. We claim that $U_m(a)$ is open. Indeed, if $x'$ is sufficiently close to $x$, we may follow a path $\rho'$ based at $x'$ with the same lengths of the legs of the paths of $\rho$ for some fixed Riemannian metric. Since the foliations are continuous, $e(\rho')$ can be made arbitrarily close to $e(\rho)$. Similarly, by choosing $x'$ sufficiently close to $x$, $a\cdot x'$ can be made arbitrarily close to $a \cdot x$. Therefore, the set $U_m(a)$ is open.

Finally, observe that the set of points such that $a \cdot x \in \overline{W^H(x)}$ is exactly $\bigcap_{m \in \N} U_m(a)$. Since $\R^k$ has a countable dense subset, the set of points such that the $\R^k$-orbit is contained in $\overline{W^H(x)}$ is also a $G_\delta$ set.
\end{proof}

We now assume that case (b) holds for the remainder of the section. Given $p,q \in X$ periodic, let $\Lambda_p = \Stab_{\R^k}(p)$ denote the periods of $p$ (the elements $a \in \R^k$ which fix $p$). %We say that $p$ and $q$ are {\it $H$-noncommensurable} if $H+\Gamma_p + \Gamma_q$ is a dense subgroup of $\R^k$. 
Given a periodic point $p$, let $B_p \subset \R^k$ denote a  compact fundamental domain for $\R^k / \Lambda_p$. By this we mean a compact subset $B_p \subset \R^k$ whose $\Lambda_p$-translates cover $\R^k$ and such that $(\gamma \cdot B_p) \cap B_p$ is either empty, or on the boundary of $B_p$. More explicitly, we assume $B_p$ is the image of $[0,1]^k$ under some isomorphism, between $\Z^k$ and $\Lambda_p$, extended to $\R^k$.

\begin{lemma}
\label{lem:caseb-density}
If (b) holds at an $\R^k$-periodic point $p$ and $q$ is another $\R^k$-periodic point, then there exists $b \in B_p$ such that $b\cdot q \in \overline{W^H(p)}$. In particular, $B_p \cdot \overline{W^H(p)} = X$.
\end{lemma}

\begin{proof}
Fix a regular element $a \in \R^k$ such that $-\alpha(a) < 0 < \alpha(a)$. We may choose $a$ such that there exists $x \in X$ with $\set{(ta) \cdot x : t \in \R_+}$ dense in $X$ (this is possible by Lemma \ref{lem:cone-transitive}(5)).  Choose some intersection $z \in W^u_a(p) \cap W^{cs}_a(x)$. We may connect $p$ to $z$ by first moving along an $\alpha$ leg to arrive at a point $y$, then along the remaining $\beta$-legs expanded by $a$. We then connect $z$ to $x$ by first moving along the $\beta$-legs contracted by $a$ (other than $-\alpha$) to a point $y'$, then a $-\alpha$ leg, then the $\R^k$-orbit. We may assume that the $\R^k$-orbit piece is trivial, since any other point on the $\R^k$-orbit of $x$ also has a dense forward orbit under $a$. Call $\rho$ the $W^H$-path connecting $y$ and $y'$.

%We may choose $ \in \Gamma_p$ such that $\alpha(v) > 0$ and such that $\set{(ta) \cdot x : t \in \R_+}$ is dense in $X$.  Choose some intersection $z \in W^u(p) \cap W^{cs}(x)$. We may connect $p$ to $z$ by first moving along $\alpha$ to arrive at a point $y$, then along the remaining $\beta$-legs expanded by $a$. We then connect $z$ to $x$ by first moving along the $\beta$-legs contracted by $a$ to a point $y'$, then $-\alpha$, then the $\R^k$-orbit. Call $\rho$ the $W^H$-path connecting $y$ and $y'$. See Figure \ref{}.

Now apply $ta$ to this picture, $t > 0$. By taking a subsequence, we may assume $(t_na) \cdot x \to q$ and $(t_na) \cdot p \to b \cdot p$ for some $b \in B_p$.  Also, since (b) holds for $p$, it also holds for $(t_na) \cdot p$, so $(t_n a) \cdot y \in \overline{W^H(t_na \cdot p)}$. Note that $(t_na) \cdot y$ and $(t_na) \cdot y'$ are connected by the path $(t_na) \cdot \rho$, and that $s((t_na) \cdot \rho) = s(\rho)$. We may approximate $(t_na) \cdot y$ arbitrarily well by endpoints of $W^H$-paths based at $(t_n a) \cdot p$ by the assumption (b). Choose such a path, whose closeness is to be determined later, and let $z_n$ denote its endpoint. Choose a very short $(\alpha,-\alpha,\R^k)$-path $\eta_n$ connecting $z_n$ and $(t_na) \cdot y$. Then we may push the path $(t_na) \cdot \rho$ along $\eta_n$ using holonomies, and since the number of switches of $(t_na) \cdot \rho$ remains constant, the distance between the new and old endpoints is Lipschitz controlled by the length of $\eta_n$. Finally apply some $s_n \in \R^k$ so that $(s_n+t_na)\cdot p = b\cdot p$ to get a path based at $b\cdot p$ (since $(t_na)\cdot p \to b\cdot  p$, we may assume$s_n \to 0$). Notice that by picking all such choices sufficiently small, we get that $y'' = \lim_{n \to \infty} (t_nv) \cdot y' = q$ belongs to $\overline{W^H(b \cdot p)}$. Since the $W^H$-sets are $\R^k$-equivariant, $b^{-1} q \in \overline{W^H(p)}$.
\end{proof}

%The following is clear from Lemma \ref{lem:caseb-density}, since the periodic orbits are dense by Theorem \ref{thm:dense-periodic}:

%\begin{lemma}
%\label{lem:good-nonempty}
%If (b) holds at an $\R^k$-periodic point $p$, then $B_p \cdot \overline{W^H(p)} = X$.%then there exists $x \in X$ such that $B_p \cdot \overline{W^H(x)} = X$.
%\end{lemma}

Fix a periodic point $p$ for which Lemma \ref{lem:caseb-density} holds, and let $B = B_p$ be the corresponding compact fundamental domain for the $\R^k$-orbit through $p$. Call a point $x$ {\it good} if $B \cdot \overline{W^H(x)} = X$, and call $G \subset X$ the set of good points. Notice that the set of good points is nonempty since $p \in G$, and therefore dense, since if $y \in B \cdot W^H(p)$, then $y$ is also good.

If $U \subset X$ is any set, let $U_1= \bigcup_{x \in U}\left( H\cdot x \cup \bigcup_{\beta \not=\pm \alpha} W^\beta(x)\right)$ be the {\it first $W^H$-saturation} of $U$. We inductively define the $n$th $W^H$-saturation of $U$, $U_n$, to be the first $W^H$ saturation of $U_{n-1}$. The {\it full $W^H$-saturation} of $U$ is defined to be $U_\infty = \bigcup_{N=1}^\infty U_N$.

\begin{lemma}
$G$ is a $G_\delta$ set.
\end{lemma}

\begin{proof}
Recall that $W^H(x) = \bigcup_{N =1}^\infty W^H_N(x)$ (it is exactly the full saturation of $\set{x}$). Fix a countable dense subset $\set{x_n}_{n\in\N} \subset X$. Then the set of points such that $B \cdot \overline{W^H(x)} = X$ is exactly the set of points such that for every $m,n \in \N$ there exists a path $\rho$ with legs in $W^H$ based at $x$ and $b \in B$ such that $d(b \cdot e(\rho),x_n) < 1/m$. Notice that this is equivalent to saying that there exists $b \in B$ such that $b^{-1} \cdot x$ is in the full $W^H$-saturation of $B(x_n,1/m)$. Notice also that the first $W^H$-saturation of any open set is open, so the full $W^H$-saturation is as well. Let $U_{m,n}$ be the full $W^H$-saturation of $B(x_n,1/m)$, so $U_{m,n}$ is open. Then the set of good points is:

\[ \bigcap_{m,n =1}^\infty B^{-1}\cdot U_{m,n} .\]

\noindent In particular, it is a $G_\delta$ set.
\end{proof}

We know from Lemma \ref{lem:caseb-density} that $G$ is nonempty and dense, and therefore residual. By Lemma \ref{lem:WH-closures}, the set $R$ of points $x$ such that $\overline{H \cdot x} = \overline{W^H(x)}$ is also residual, as is the set $D$ of points with a dense $\R^k$ orbit. Therefore, we may choose $x_0 \in G \cap R \cap D$. If $x \in X$, let $A_x \subset \R^k$ be defined by:

\[ A_x = \set{ a \in \R^k : a \cdot x \in \overline{H \cdot x}} .\]

\begin{lemma}
\label{lem:gamma-subgroup}
$A_{x_0}$ is a closed subgroup containing $H$ as a proper subgroup.
\end{lemma}

\begin{proof}
We first show that if $a \in A_{x_0}$, then $a^{-1} \in A_{x_0}$. Since $a \in A_{x_0}$, there exists $h_n \in H$ such that $h_n \cdot x_0 \to a \cdot x_0$. Using local coordinates we can connect $h_n \cdot x_0$ to $a \cdot x_0$ using a $W^H$-path with at most 2 switches, then an $(\alpha,-\alpha,\R^k)$-path such that the length of every leg tends to 0. Using $\beta$-holonomies and recalling Lemma \ref{lem:uniformly bounded derivative}, we can rearrange this connection as a path from $x_0$ to $a\cdot x_0$ as follows: the path begins with an arbitrarily short $(\alpha,-\alpha,\R^k)$-path, then moves along a possibly very long $h_n \in H$, then along a short $W^H$-path with two switches. Apply $a^{-1}$ to this path. The derivative of $a$ is fixed, so we get a path from $a^{-1} \cdot x_0$ to $x_0$ which begins with an arbitrarily short $(\alpha,-\alpha,\R^k)$-path, then moves along a possibly very long $h_n \in H$, then along some $W^H$-path. Reversing the start and end point of this path shows that $a^{-1} \cdot x_0 \in \overline{W^H(x_0)} = \overline{H \cdot x_0}$, as claimed.

Now, let $a,b \in A_{x_0}$. To see that $a + b \in A_{x_0}$, let $h_n \cdot x_0 \to a \cdot x_0$ and $k_n \cdot x_0 \to b \cdot x_0$. Since $b$ is uniformly continuous, $(b + h_n) \cdot x_0 \to (a+b)\cdot x_0$. Fixing $n$, we get that $(k_m + h_n) \cdot x_0 \to (b + h_n) \cdot x_0$. Since $(b+h_n) \cdot x_0$ can be made arbitrarily close to $(a+b) x_0$, one may choose $m$ and $n$ so that $k_m + h_n$ is arbitrarily close to $(a+b) \cdot x_0$. Therefore, $(a+b) \cdot x_0 \in A_{x_0}$.

Finally, notice that it obviously contains $H$ since $H \cdot x_0 \subset \overline{H \cdot x_0}.$ Furthermore, by choice of $x_0$, $B \cdot \overline{W^H(x_0)} = B \cdot \overline{H \cdot x_0} = X$, so $\R^k \cdot x_0$ is contained in $B \cdot \overline{H \cdot x_0}$. Therefore, $\bigcup_{b \in B} bA_{x_0} = \R^k$. This implies that $A_{x_0}$ must be cocompact so it must contain some element transverse to $H$.
\end{proof}

We shall see the following case of $A_{x_0}$ is important in the proof. Recall that our totally Cartan action is $C^r$, where $r = (1,\theta)$ or $r = \infty$.

\begin{lemma}
\label{lem:circle-factor}
If $A_{x_0} = H \oplus \Z \ell$ for some $\ell \not\in H$, then there is a $C^{r}$ factor $\pi : X \to \R^k / A_{x_0} \cong \mathbb{T}$ of the $\R^k$-action such that $\overline{H x_0} = \pi^{-1}(0)$. %where $r' = (1,\theta)$ or $\infty$, depending on whether the action is $C^{1,\theta}$ or $C^\infty$, respectively.
\end{lemma}

\begin{proof}
Fix $y \in X$, and let $B_y \subset \R^k$ be the set of all $a \in \R^k$ such that $\overline{(a + H)x_0} \ni y$. We claim that $B_y$ is a coset of $A_{x_0}$. Suppose that $a_1,a_2 \in B_y$, then $y \in \overline{(a_1 + H)x_0} \cap \overline{(a_2 + H)x_0}$. Then $a_1^{-1}y \in \overline{Hx_0} \cap \overline{((a_2 -a_1) + H) x_0}$. As in the proof of Lemma \ref{lem:gamma-subgroup}, we may build a path from $x_0$ to $a_1^{-1}y$ by choosing a long $H$-orbit segment, then a short path whose legs lie in $W^H$ with a fixed number of switches and a short $(\alpha,-\alpha,\R^k)$-path. We may find a similar path from $(a_2-a_1)x_0$ to $y$. Concatenating these paths gives a path from $x_0$ to $(a_2-a_1)x_0$. Since the number of switches is uniformly bounded, we may push the $\alpha,-\alpha,\R^k$-legs to the end of the path keeping them very short using control on the number of switches. This shows that $(a_2-a_1)x_0 \in \overline{W^H(x_0)} = \overline{H x_0}$, so that $a_2 - a_1 \in A_{x_0}$. Hence, $B_y$ is a coset.

Now define $\pi : X \to \R^k / A_{x_0} \cong \mathbb{T}$ by $\pi(y) = B_y$. It is clear that the map determines a factor of the $\R^k$-action, and by construction, $\overline{H x_0} = \pi^{-1}(0)$. We claim that $\pi$ is continuous. Indeed, it suffices to show that the preimage of a closed set is closed. Let $C \subset \mathbb{T}$ be closed. Then $\pi^{-1}(C)$ is homeomorphic to $C \times \pi^{-1}(0)$ which is compact and hence closed. The homeomorphism is exactly given by $(c,x) \mapsto c \cdot x$.

Choose a regular element $a \in \R^k$. Since $\pi$ is continuous and determines a factor of the $\R^k$ action, it is constant on the leaves of $W^s_a$ and $W^u_a$. Then $\pi$ is constant on the leaves of $W^s_a$ and smooth along the $\R^k$-orbits (since it is a factor), so by Theorem \ref{thm:journe}, $\pi$ is smooth along the leaves of $W^{cs}_a$. Notice that it is also constant along $W^u_a$, so again applying Theorem \ref{thm:journe} to the complementary foliations $W^{cs}_a$ and $W^u_a$, we get that $\pi$ is $C^{r}$ globally.
\end{proof}

%Lemma \ref{lem:circle-factor} and the observation preceding it immediately show the following, based on the assumptions we made at the start of this subsection:

\begin{proposition}
\label{prop:case2-finished}
If there is a periodic orbit $p$ for which Case \ref{dich:2} of the return dichotomy  holds, Theorem \ref{thm:main-anosov} holds.
\end{proposition}

\begin{proof}
$A_{x_0}$ contains $H$ and $H$ has codimension one, so either $A_{x_0} = \R^k$ or $A_{x_0} = H \oplus \Z \ell$. Since $x_0$ has a dense $\R^k$-orbit, if $A_{x_0} = \R^k$, $\overline{H \cdot x_0} = X$, we arrive at the last conclusion of Theorem \ref{thm:main-anosov}. Lemma \ref{lem:circle-factor} shows that if $A_{x_0} = H \oplus \Z\ell$, we have the second conclusion of Theorem \ref{thm:main-anosov}. 
\end{proof}

\section{Proofs of Theorems \ref{thm:main-anosov} and \ref{thm:main-cartan}}
\label{sec:part1-proofs}

First, we note that Theorem \ref{thm:main-anosov} follows immediately from Propositions \ref{prop:case1-finished} and \ref{prop:case2-finished} and the \hyperlink{cases}{return dichotomy}. In this section, we use Theorem \ref{thm:main-anosov} to prove Theorem \ref{thm:main-cartan}.% and Theorem \ref{thm:big-main}.

\begin{proof}[Proof of Theorem \ref{thm:main-cartan}]
If there is a non-Kronecker rank one factor the theorem holds, so assume that we either have the second or last conclusion of Theorem \ref{thm:main-anosov}. By Lemma \ref{lem:susp}, we may without loss of generality assume that we have an $\R^k$ action (rather than an $\R^k \times \Z^\ell$ action). When $H$ has a dense orbit, we may use Proposition \ref{lem:kalinin-lem} and the Livsic argument, to obtain the desired metric (as was done in \cite{KaSp04}, see Section \ref{subsec:metrics-prelim}). 

Next, suppose we have a circle factor $\pi : X \to \mathbb{T}$ of $\R^k \curvearrowright X$, and that there exists $x \in X_0 = \pi^{-1}(0)$ such that $H\cdot x$ is dense in $X_0$. Define a metric on $E^\alpha_{h\cdot x}$ by pushing forward any metric on $E^\alpha_x$. Then using Proposition \ref{lem:kalinin-lem}, one may again show that the metric defined on the $H$-orbit of $x$ extends continuously to an $H$-invariant metric on $X_0$. This argument also shows the metric is unique up to muliplicative constant. Choose any $a \in \R^k$ such that $\pi(a \cdot x) = 0$, but $\pi(ta \cdot x) \not= 0$ for $t \in [0,1)$. Then pushing the metric forward by $a$ yields another $H$-invariant metric on $X_0$, which must be a scalar multiple of the original one, so that $\norm{a_*v} = \lambda\norm{v}$ for some $\lambda > 0$ and every $v \in E^\alpha$ based at some point in $X_0$. Notice that $a\not\in H$ since if this were the case, we would have $\pi(ta \cdot x) \equiv 0$. Therefore, $\R^k = H \oplus \R a$. So we may define a functional $\alpha : \R^k \to \R$ by $\alpha(h,t) = t\lambda$.
Finally, we define a metric on $X$ in the following way: choose $x \in X$, so that $x = ta \cdot y$ for some $y \in X_0$ and $t \in \R$. Then if $v \in E^\alpha_x$, define $\norm{v} = \lambda^t\norm{(ta)^{-1}_*v}$. One may easily check that this metric and the functional $\alpha$ are well-defined and satisfy the conclusions of Theorem \ref{thm:main-cartan}. Uniqueness follows immediately from the construction.
\end{proof}

\part{\Large Constructing a Homogeneous Structure}

Throughout Part IV, unless stated otherwise, we assume that $\R ^k \curvearrowright X$ is a $C^r$ ($r = (1,\theta)$ or $r = \infty$), cone transitive, totally Cartan action for which every coarse Lyapunov foliation is orientable and has a distinguished orientation, and that no finite cover of the action has a non-Kronecker rank one factor (we reduce rigidity of $\R^k \times \Z^\ell$ actions to rigidity of $\R^{k+\ell}$ action case in Section \ref{sec:proofs} by passing to a suspension).  We can always pass to a finite cover of $X$ on which the coarse Lyapunov foliations are oriented.  By Theorem \ref{thm:main-anosov} and \ref{thm:main-cartan}, we may and therefore do assume the following throughout Part IV:

\vspace{.2cm}

\noindent{\hypertarget{h-r-a}{\bf Higher Rank Assumptions.}}

\begin{enumerate}[label=(HR\arabic*)]
\item \label{HR:a}Each functional $\beta \in \Delta$ is chosen together with metrics $\norm{\cdot}_\beta$ such that $\norm{a_*v}_\beta = e^{\beta(a)}\norm{v}_\beta$ for every $a \in \R^k$ and $v \in W^\beta$. The functional $\beta$ is unique, and the $\norm{\cdot}_\beta$ is unique up to global scalar.
\item \label{HR:b} For every $\beta \in \Delta$, $\ker \beta$ either has a dense orbit, or has orbits dense in every fiber of some circle factor of the $\R^k$-action.
\end{enumerate}

\begin{partremark}
The second case of \ref{HR:b} is mysterious, but we will still  be able to prove that systems of this form are algebraic. However, we do not know if any such examples exist, we expect that they do not. Their existence is related to whether, on a nilmanifold, having no rank one factors implies the same feature for the actions on its center. It will be an annoyance, but not a real obstruction throughout the proofs in this part.
\end{partremark}

%\begin{partremark}
%In fact, the second case of b) never happens, but cannot be ruled out from the techniques of Part I. We will later see that this is incompatible with homogeneous actions (Lemma \ref{lem:no-S1-factor}), and that we may always break off rank one factors and end up at a homogeneous action (this procedure is outlined in Part III).
%\end{partremark}

\begin{partremark}
\label{rem:lyap-coefficients2}
Throughout this section, each functional $\beta$ is now defined uniquely, whereas in Part III (and again in Part V), it is only defined up to multiplicative constant. Therefore, unlike in those sections, we will take special care in considering the coefficients of the weights (so rather than saying the case of a pair of weights $\alpha$ and $-\alpha$, we will need to consider the case of the pair $\alpha$ and $-c\alpha$). Compare with Remark \ref{rem:lyap-coefficients1}.

We also note that \ref{HR:a} implies that the Lyapunov exponent for $E^\alpha$ with respect to any invariant measure is equal to $\alpha$ (and not just some scalar multiple of $\alpha$).
\end{partremark}

\begin{partremark}
\label{rem:homogeneous-weights}
Assumptions \ref{HR:a} and \ref{HR:b} will need to be relaxed in Part V. In particular, we will identify a factor of the action on which the weights \[\Delta_{\Rig} := \set{ \beta \in \Delta : \ker \beta \mbox{ has a dense orbit or is dense in the fiber of some circle factor}}\] are collapsed, leaving only the weights corresponding to rank one factors $\Delta_1 = \Delta \setminus \Delta_{\Rig}$. By the results of Part III, every weight in $\Delta_{\Rig}$ satisfies the conclusions of \ref{HR:a} and \ref{HR:b}. The arguments here will be written for the case when $\Delta_{\Rig} = \Delta$ (which is the case of Theorem \ref{thm:big-main}). We encourage readers interested in the general case to look at Section \ref{sec:fiber-homogeneous} before proceeding.
\end{partremark}

%{\cb OUTLINE
%\begin{proposition}
%If $E = \set{\chi_1,\dots,\chi_k} \subset \Delta$ are the weights contracted by a family of elements $\set{a_1,\dots,a_n}$ listed in the circular ordering, and $\mc W^E$ is the foliation whose leaves are the intersections of the corresponding stable manifolds, then:

%\[ (t_1,\dots,t_k) \mapsto \eta^{\chi_1}_{t_1}\dots \eta^{\chi_k}_{t_k}(x) \]

%is a homeomorphism from $\R^k$ to $W^E_x$
%\end{proposition}

%\begin{definition}
%Geometric Bracket
%\end{definition}
%}

\section{Topological Cartan actions and geometric brackets}

For $\chi \in \Delta$, we first define flows $\eta^\chi$ along the  oriented foliations $W^\chi$ which will be critical to our analysis:

\begin{definition}
For each $\chi \in \Delta$, let $\eta^\chi_t$ denote the positively oriented translation flow along $W^\chi$ (with respect to the norm $\norm{\cdot}_\chi$), which satisfies

\begin{equation}
\label{eq:renormalization}
 a\cdot \eta^\chi_t(x) = \eta^\chi_{e^{\chi(a)}t}(a \cdot x).
\end{equation}

%Let $\mc P = \R^{*\abs{\Delta}}$, and $\eta : \mc P \curvearrowright \tilde{M}$ be the action defined by

%\[ \eta(t_1^{(\chi_1)} * \dots * t_n^{(\chi_n)})x = \eta_{t_1}^{\chi_1} \of \dots \of \eta_{t_n}^{\chi_n}(x) \]
\end{definition}

For homogeneous examples, these are exactly the unipotent one parameter subgroups which parameterize the coarse Lyapunov foliations. In our setting, we do not know that they have good regularity properties (in fact, for now, we only know that they are $C^1$ along their orbits and H\"older transversally). Therefore, they are locally H\"older, and each evaluation map $t \mapsto \eta^\chi_t(x)$ is locally bi-Lipschitz onto its image, according to the following

\begin{definition}
We say that a flow $\psi_t$ on a metric space $Y$ is a {\normalfont H\"older flow} if the evaluation map $(t,y) \mapsto \psi_t(y)$ is locally H\"older. A function $f : X \to Y$ between metric spaces $X$ and $Y$ is called {\normalfont locally bi-Lipschitz} if for every $x \in X$, there exists a neighborhood $U$ such that $f : U \to f(U)$ is invertible, and both $f$ and $f^{-1}$ are Lipschitz on $U$ and $f(U)$, respectively.
\end{definition}

Because the flows $\eta^\chi_t$ are only locally H\"older, it is natural for us to ignore smooth structures and work in the setting of metric spaces, rather than smooth manifolds. We therefore introduce a notion of a {\it topological Cartan action}, which distills the output of the smooth dynamical assumptions into coarser topological data. Remarkably, we will still obtain a classification result for such actions, of which Theorem \ref{thm:big-main} will be a corollary, see Theorem \ref{thm:technical}. %The local product structure and resonance assumptions will always hold for { $C^{1,\theta}$} actions. % While interesting in its own right, this is absolutely needed for proving such a result, even for smooth actions.  

We introduce this notion for two reasons:

%One of the main steps of the proof of Theorem \ref{thm:big-main} is to
%define a factor of such an action with much better transitivity properties for the coarse Lyapunov foliations.  A priori, this factor is a compact metric space that may fail to be a manifold, and hence the action may not be smooth. While such a factor may not be a topological Cartan action at first, it inherits enough good properties of the action to produce a homogeneous structure on it.  We then slowly show that all spaces and actions are indeed homogeneous.  

First, it illustrates the versatility of the method for working with actions that may a priori fail to have any smoothness properties. Indeed, even for $C^\infty$ actions, the coarse Lyapunov foliations  transversely are only H\"older, so it is not clear how to take Lie brackets of vector fields tangent to these foliations. We therefore replace standard tools of differential geometry with ``geometric brackets''.  These are motivated by the usual geometric interpretation of the Lie bracket but explicitly use the rigid structure of the coarse Lyapunov foliations coming from a higher-rank action. 

{Second}, it clarifies and emphasizes which structures of the smooth Cartan actions we will be using and is a useful reference for definitions in subsequent arguments.

\begin{remark}
The next definition is very technical and imitates the structures one naturally gets from a smooth totally Cartan actions. %In particular, we get one dimensional laminations generalizing the coarse Lyapunov foliations (which will be orbits of some flows $\eta^\chi$, which we will assume are rescaled by the Cartan actions from Theorem \ref{thm:cocycle rigidity}). We will see that several of the smooth structures and properties can be reconstructed on the topological level, and importantly for us, we will still have rates which play the role of Lyapunov exponents.
 We intentionally use the same notation for objects in the topological category that are constructed in the smooth one to emphasize their roles and properties, but would like to emphasize that no smooth structures are assumed. In particular, stable, unstable, center-stable and center-unstable manifolds will be introduced, which have similar properties and play similar roles to their smooth counterparts.
\end{remark}

If $Y$ is a metric space, let $B_Y(y,\ve)$ denote the ball of radius $\ve$ along $y$. 

\begin{definition}
\label{def:top-cartan}
{Consider a  continuous, transitive} action of $\R^k$ on a compact %finite-dimensional 
metric space $(X,d)$ with the following structures: a set $\Delta \subset (\R^k)^*$ (called the weights of the action) and a collection of H\"older flows $\set{\eta^\chi : \chi \in \Delta}$. { We will call the elements of $\R^k \setminus \bigcup_{\chi \in \Delta} \ker \chi$ the {\normalfont Anosov elements} of the action, and the orbit foliations of the $\eta^\chi$ the {\normalfont coarse Lyapunov foliations}.} We say that the action is  a {\normalfont topological Cartan action} if satisfies the following additional properties:

\begin{enumerate}[label=(TC-\arabic*)]
\item \label{tc0} {For every $x \in X$, $a \mapsto a \cdot x$ is locally bi-Lipschitz from $\R^k$ onto its image.}
\item \label{tc1} For any $\beta \in \Delta$, $c\beta \not\in \Delta$ for any $c > 0$, $c \not= 1$.
\item \label{tc2} For every $x \in X$ and $\chi \in \Delta$, $t \mapsto \eta^\chi_t(x)$ is locally bi-Lipschitz {from $\R$ onto its image}.
\item \label{tc3} For every $a \in \R^k$, $t \in \R$ and $x \in X$, $a \cdot \eta^\chi_t(x) = \eta^\chi_{e^{\chi(a)}t}(a \cdot x)$.
\item \label{tc4} %for each $a \in \R^k$, there exists a lamination $W^s_a(x)$ for each $x \in X$ 
If $a_1,\dots,a_m \in \R^k$ is a list of Anosov elements, and %, $\Phi \subset \Delta$ is the subset of weights such that $\chi(a_i) < 0$ for all $i = 1,\dots,m$, and 
$\set{\chi_1,\dots,\chi_r}$ is a circular ordering of $\Delta^-(\set{a_1,\dots,a_m})$, then for every $x \in X$, the evaluation map for $\tilde{\eta}|_{C_{(\chi_1,\dots,\chi_r)}}$ at $x$ is injective (recall Definition \ref{def:comb-cells}). If $W^s_{(a_i)}(x) := \tilde{\eta}(C_{(\chi_1,\dots,\chi_r)})x$, then for any combinatorial pattern $\bar{\beta}$ whose letters are all from {$\Delta^-(\set{a_1,\dots,a_m})$}, $\tilde{\eta}(C_{\bar{\beta}})x \subset W^s_{(a_i)}(x)$ for every $x \in X$. 
\end{enumerate}

{
\noindent Keeping the notation from \ref{tc4}, we introduce local stable and unstable, center-stable and center-unstable, and local center-stable and local center-unstable manifolds, in the following fashion (we define them here for the stable versions, the unstable versions are defined, for instance by setting $W^{cu}_{(a_i)} := W^{cs}_{(-a_i)}$, etc.): 

\begin{eqnarray*}
W^s_{(a_i)}(x,\delta) & := & \tilde{\eta}(C_{(\chi_1,\dots,\chi_r)} \cap B_{\R^r}(0,\delta)) \\
W^{cs}_{(a_i)}(x) & := & \R^k \cdot W^s_{(a_i)}(x) \\
W^{cs}_{(a_i)}(x,\delta) & := & B_{\R^k}(0,\delta) \cdot W^s_{(a_i)}(x,\delta).
\end{eqnarray*}
}
\begin{enumerate}[label=(TC-\arabic*)]
 \setcounter{enumi}{5}
\item \label{tc5} If $\hat{\eta}$ is the corresponding action of $\hat{\mc P}$ (see Definition \ref{def:P-groups}), then $\hat{\eta}$ is transitive in the sense that for every $x,y \in X$, there exists $\rho \in \hat{\mc P}$ such that $\hat{\eta}(\rho) x = y$.
\item \label{tc6}  There exists $\delta,\ve > 0$, such that if $a$ is Anosov, $y \in W^s_a(x,\delta)$ and $z \in W^{cu}_a(x,\delta)$, then the map $(y,z) \mapsto W^{cu}_a(y,\ve) \cap W^s_a(z,\ve)$ is well-defined from $W^s_a(x,\delta) \times W^{cu}_a(x,\delta)$ and maps homeomorphically onto a neighborhood of $x$. \\

\noindent We call the action {\normalfont cone transitive} if there exists an open cone  whose closure is contained in the set of Anosov elements which has a dense orbit.
\end{enumerate}

%We say that the topological Cartan action has {\normalfont locally transverse laminations} if it also satisfies

%\begin{enumerate}
%\setcounter{enumi}{6}
%\item for any ordering $\bar{\beta} = (\beta_1,\dots,\beta_n)$ of $\Delta$ which lists every weight exactly once, there exists $U \subset C_{\overline{\beta}} \times \R^k$ open containing 0 (the trivial path) such that for every $x \in X$, the restriction of the evaluation map at $x$ from $U \to X$ is onto a neighborhood of $x$.
%\item the action preserves a probability measure $\mu$ of full support such that the conditional measures $\mu ^{\eta} _x$ of $\mu$ for  $\mu$-a.e. orbit $\eta ^{\chi} _t x$ have full support on this orbit.
%\end{enumerate}

\end{definition}

{
\begin{remark}
Notice that \ref{tc4} immediately gives the stable and unstable manifolds the structure of topological manifolds. It is then not difficult to see that the center-stable and center-unstable manifolds are also topological manifolds using \ref{tc2}.
\end{remark}
}

\begin{remark}
Several of the properties of cone transitive totally Cartan actions can be deduced for cone transitive topological Cartan actions, including the Anosov closing lemma (Theorem \ref{thm:anosov-closing})  and density of periodic orbits (Lemma \ref{lem:cone-transitive}(2)). We will use them freely for in the topological setting.  Notice also that the \hyperlink{h-r-a}{higher rank assumptions} still make sense, after replacing \ref{HR:a} by \ref{tc2} and \ref{tc3}. See Remark \ref{rem:top-cart-ok}.
\end{remark}

\begin{remark}
One may introduce the notion of a topological Cartan action of $\R^k \times \Z^\ell$ in a similar way, and prove similar results about them. They are related to $\R^{k+\ell}$-actions by suspensions. 
\end{remark}

\begin{proposition}
\label{prop:smooth-is-topological}
Let $\R^k \curvearrowright X$ be a $C^{1,\theta}$ cone transitive, totally Cartan action on a compact manifold without a rank one factor. % preserving a volume $\mu$.   
%Further assume that every Lyapunov hyperplane has a dense orbit.  
Then if every coarse Lyapunov foliation is {oriented}, this data defines a topological Cartan action.% with locally transverse laminations.
\end{proposition}

%As discussed at the start of this section, we may always assume that the coarse Lyapunov foliations are orientable by passing to a finite cover.

\begin{proof}
The special H\"older metrics from Theorem \ref{thm:cocycle rigidity} determine unit vector fields, positively oriented, and thus flows which satisfy properties \ref{tc1}-\ref{tc3}.  %The product structure of the measure follows from the disintegration properties of the measure in Section \ref{subset:disintegration}.

We now show \ref{tc4}. Indeed, by Lemma \ref{lem:coarse-lyapunov}, %$W^s_{(a_i)}$ is the intersection $W^s_{(a_i)} = \bigcap_{i=1}^m W^s_{a_i}$ shows that
 $W^s_{\set{a_i}}$ foliates the space and by Lemma \ref{lem:extending-charts}, we have the corresponding parameterizations. %By Lemma \ref{lem:local-transitivity}, we get that $\tilde{\eta}(C_{(\chi_1,\dots,\chi_r)})x \subset W^s_{(a_i)}(x)$ contains a neighborhood of $x$ whose size is uniformly large. 
%By intertwining with some large iterate of one of the $a_i$, we get $\tilde{\eta}(C_{(\chi_1,\dots,\chi_r)})x = W^s_{(a_i)}(x)$.
\ref{tc6} is immediate from local product structure of Anosov actions. \ref{tc5} follows from \ref{tc6} and connectedness of $X$.
\end{proof}

{
One may prove Theorem \ref{thm:big-main} for topological Cartan actions, which we now formulate. For the rest of Part IV, we will not rely on the smoothness properties, only the structures obtained in this section, { and one additional structure provided from regularity (see Lemma \ref{lem:at-least-one}).}  so we aim to prove the following:

\begin{theorem}
\label{thm:technical}
Let $\R^k \curvearrowright M$ be a cone transitive, topological Cartan action satisfying \ref{HR:b}. { Further assume that if $u\alpha + v\beta \in [\alpha,\beta]$, then either $u$ or $v \ge 1$ (see Definition \ref{def:weight-bracket})}.  Then there exists a homogeneous Cartan action $\R^k \curvearrowright G /\Gamma$ which is topologically conjugate to $\R^k \curvearrowright M$. %If the topological Cartan action is $C^{1,\theta}$ or $C^\infty$, then the conjugacy is $C^{1,\theta}$ or $C^\infty$, respectively.
\end{theorem}
}

\subsection{The geometric commutator}
\label{subsec:geom-comm}
We can now establish the main result in this section, the principal technical tool in our analysis of topological Cartan actions: the {\it geometric commutator}. Fix $\alpha,\beta \in \Delta$, and pick some regular $a \in \R^k$ such that $D(\alpha,\beta) \subset \Delta^-(a)$. Recall that $D(\alpha,\beta)$ comes equipped with a canonical circular ordering $\set{\alpha,\gamma_1,\dots,\gamma_n,\beta}$ (see Definition \ref{def:canonical-order}). We use the following convention for group commutators: if $\mc G$ is a group and $g,h \in \mc G$, then

\begin{equation}
\label{eq:comm-convention}
[g,h] = h^{-1}g^{-1}hg.
\end{equation}

\begin{lemma}
\label{lem:comm-relation}
If $t,s\in \R$, $\alpha,\beta \in \Delta$ are non-proportional% and $\set{\chi_1,\dots,\chi_n}$ is a circular ordering of $D(\alpha,\beta)$
, there exists a unique $\rho^{\alpha,\beta} : \R \times \R \times X \to \mc P_{D(\alpha,\beta) \setminus \set{\alpha,\beta}}$ such that:

\begin{equation}\label{eq:rho-characterization}
 \rho^{\alpha,\beta}(s,t,x)* [s^{(\alpha)},t^{(\beta)}]  \in \mc C(x), {\mbox{ and}}
\end{equation}

\[ \rho^{\alpha,\beta}(s,t,x) = \rho^{\alpha,\beta}_{\chi_n}(s,t,x)^{(\chi_n)} * \dots *  \rho^{\alpha,\beta}_{\chi_1}(s,t,x)^{(\chi_1)}. \]

\noindent for a unique collection of functions $\rho^{\alpha,\beta}_{\chi_i} : \R \times \R \times X \to \R$, $\chi_i \in D(\alpha,\beta) \setminus \set{\alpha,\beta}$.
Furthermore, $\rho^{\alpha,\beta}$ and $\rho^{\alpha,\beta}_{\chi_i}$ are continuous functions in $t,s,x$, and

\begin{equation}
\label{eq:rho-equivariance}
e^{\chi_i(a)}\rho^{\alpha,\beta}_{\chi_i}(s,t,x) = \rho^{\alpha,\beta}_{\chi_i}(e^{\alpha(a)}s,e^{\beta(a)}t,a\cdot x).
\end{equation}
\end{lemma}

\begin{figure}[!ht]
\begin{tabular}{cc}
\includegraphics[width=3in]{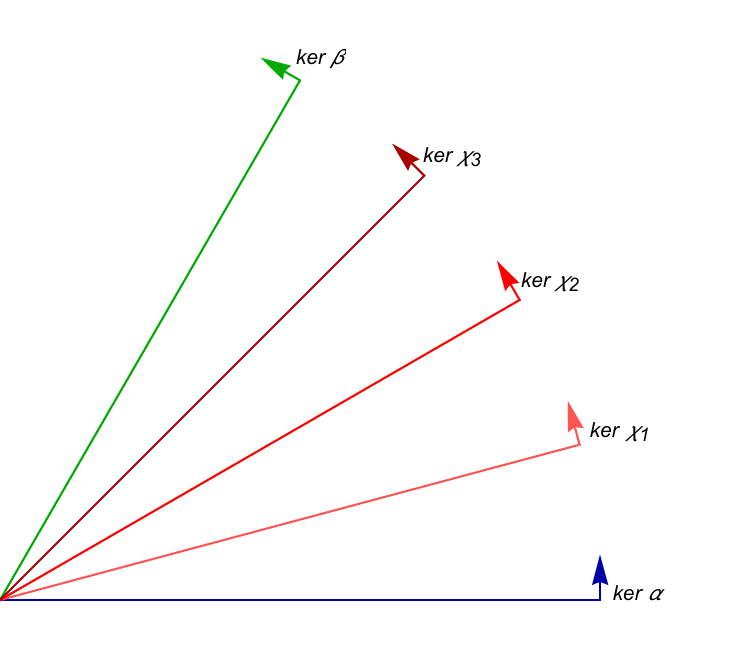} & \includegraphics[width=3in]{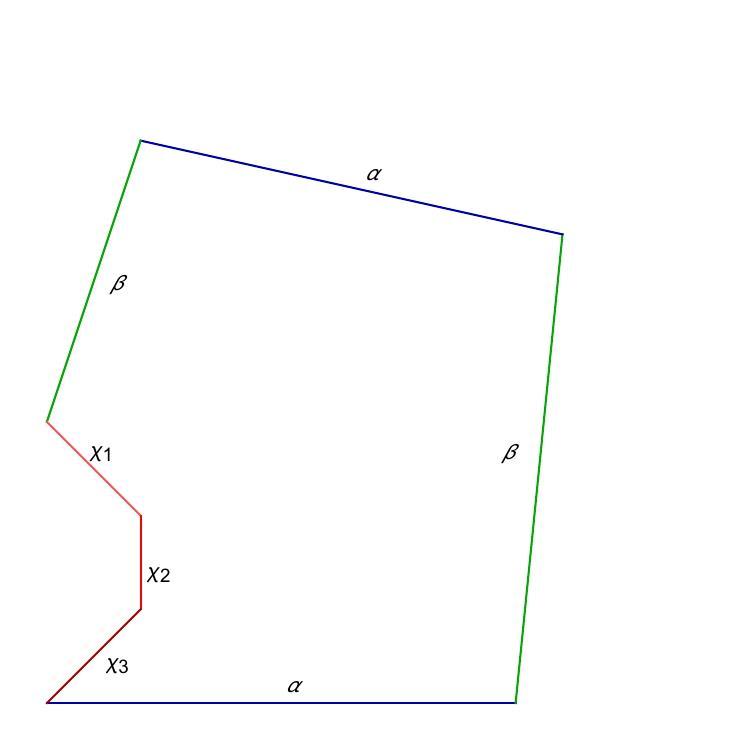}
\end{tabular}
\caption{Canonical Commutator Relations}
\label{fig:rho-def}
\end{figure}

\begin{proof}
Since $\alpha$ and $\beta$ are not proportional, if follows from \ref{tc4} in the definition of a topological Cartan action that %(Lemma \ref{lem:root-comb} and iterated application of Lemma \ref{lem:extending-charts} that 
there exists some unique path $\rho_x =  u_2^{(\beta)} * v_n^{(\chi_n)} * \dots * v_1^{(\chi_1)} * u_1^{(\alpha)} \in \mc P_{D(\alpha,\beta)}$ satisfying $\rho_x * [t^{(\beta)},s^{(\alpha)}] \in \mc C(x)$. {Indeed, the endpoint of $[t^{(\beta)},s^{(\alpha)}]$ belongs to $W^s_{(a_i)}$ for some collection of $a_i$ with $\Delta^-(\set{a_i}) = D(\alpha,\beta)$ (see Lemma \ref{lem:root-comb}). It is therefore reached by a path of the described form.}  We wish to show that $\rho_x$ has trivial $\alpha$ and $\beta$ components.  Notice that by linear independence of $\alpha,\beta$, we may choose $a$ such that $\alpha(a) = 0$ and $\beta(a) < 0$. Then notice that  by \eqref{eq:renormalization}, for any $x \in X$:

\begin{equation}
\label{eq:comm-contraction}
(na)\left[s^{(\alpha)},t^{(\beta)}\right](-na)\cdot x = \left[s^{(\alpha)},e^{n\beta(a)}t^{(\beta)}\right]x \xrightarrow{n\to\infty} x.
\end{equation}

Since the action of $\tilde{\eta}$ is H\"older, the convergence is exponential since $\beta(a) < 0$. But if $\rho_x$ has a nontrivial $\alpha$ component, that component does not decay, since it is isometric. Therefore it must be trivial. By a completely symmetric argument, the $\beta$ component is trivial.
\end{proof}

\begin{definition}
\label{def:weight-bracket}
If $\alpha,\beta \in \Delta$ are linearly independent weights, let \[[\alpha,\beta] = \set{\chi \in \Delta : \rho^{\alpha,\beta}_\chi(s,t,x)
 \not= 0 \mbox{ for some }t,s \in \R, x \in X}.\]
\end{definition}

It is clear that $[\alpha,\beta] = [\beta,\alpha]$, since $\rho^{\beta,\alpha}(t,s,\rho^{\alpha,\beta}(s,t,x) \cdot x) = \left(\rho^{\alpha,\beta}(s,t,x)\right)^{-1}$. Notice also that, a priori, each $\rho^{\alpha,\beta}_\chi$ may vanish at some $x \in X$, but not every $x$. We will see later that the no rank-one factor assumption implies that $\rho^{\alpha,\beta}_\chi$ is actually a polynomial {in $s$ and $t$} independent of $x$.

{ 

\subsection{Coefficients of nontrivial commutators}
In this section, we prove a lemma which will be crucial in the next section. We begin with a technical observation. If $\mc F_i,\dots,\mc F_m$ are continuous foliations of a manifold $Y$, we say that they {\it induce continuous local coordinates} if for every $x \in X$, there is a map $\phi_x$ defined on an open subset $U \subset \R^{d_1} \times \dots \R^{d_m}$, such that

\[\phi_x(u_1,\dots,u_\ell,u_{\ell+1},0,\dots,0) \in \mc F_{\ell+1}(\phi_x(u_1,\dots,u_\ell,0,0,\dots,0))\]
for every $u_1,\dots,u_{\ell+1} \in \R^{d_1} \times \dots \times \R^{d_m}$, and $\phi_x$ is onto a neighborhood of $x$. We think of this as using the foliations to produce coordinates around $x$ (even though we do not ask that it be injective).

\begin{lemma}
\label{lem:local-transitivity}
Let $Y$ be a smooth manifold, and $\mc F_1,\dots,\mc F_m$ be uniquely integrable continuous foliations with smooth leaves such that $T_yY = \bigoplus_{i=1}^m T_y\mc F_i$ for every $y \in Y$. Then the foliations $\mc F_i$ induce continuous local coordinates.
\end{lemma}

Lemma \ref{lem:local-transitivity} can be deduced in two ways: one may use a standard trick of approximating the vector fields tangent to the foliations $\mc F_i$ by smooth ones, proving the continuous local coodinates property using the inverse function theorem, and showing persistence of the transitivity using a degree argument. Alternatively, this is a more general phenomenon which does not require unique integrability, see \cite[Lemma 3.2]{Schmidt-thesis}.

\begin{lemma}
\label{lem:at-least-one}
If $\R^k \curvearrowright X$ is a $C^2$ totally Cartan action on a smooth manifold with no non-Kronecker rank one factors, $\alpha,\beta \in \Delta$ are linearly independent, and $\chi = u \alpha + v \beta$ with $u < 1$ and $v < 1$, then $\chi \not\in [\alpha,\beta]$.
\end{lemma}

\begin{proof}
 Divide $D(\alpha,\beta) = A \cup B$, where $A = \set{ u\alpha + v \beta : u \ge 1} \cap \Delta$ and $B = \set{ u\alpha + v\beta : u < 1} \cap \Delta$. Let $B' = \set{ u\alpha + v\beta :  u< 1,v \ge 1} \cap \Delta$. Give the weights of $D(\alpha,\beta)$ the canonical circular ordering $\set{\alpha = \chi_1,\dots,\chi_n = \beta}$. We will show that if $\chi_i,\chi_j \in A \cup B'$, then $[\chi_i,\chi_j] \subset A \cup B'$ by induction on $\abs{i-j}$. Then the result follows since $\alpha = \chi_1$ and $\beta = \chi_n$.

The base case is trivial: if $\abs{i-j} = 1$, then $\chi_i$ and $\chi_j$ are adjacent in the circular ordering, and therefore commute. We now try to commute $\chi_i,\chi_j \in A \cup B'$, $\abs{i-j} > 1$. Notice that choosing $a \in \ker \beta$, and perturbing by a very small amount will yield an element $a'$ close to $a$ for which $\left(\bigoplus_{\chi \in A} T\mc W^\chi\right) \oplus \left(\bigoplus_{\chi \in B} T\mc W^\chi \right)$ is  a dominated splitting (cf. slow foliations, Section \ref{sec:holonomy-action}). Therefore, $\bigoplus_{\chi \in A} T\mc W^\chi$ is tangent to a foliation $\mc W^A$. Notice that each leaf $\mc W^A$ has local $C^0$ charts given by motion along each $W^\chi$, $\chi \in A$ (the proof goes as in Lemma \ref{lem:extending-charts}). In particular, if both $\chi_i$ and $\chi_j \in A$, $[\chi_i,\chi_j] \subset A$. Similarly, if both $\chi_i$ and $\chi_j$ belong to $B'$, then they both belong to $C = \set{u\alpha + v\beta : v \ge 1} \cap \Delta$. So for similar reasons, choosing a perturbation of $b \in \ker \alpha$, gives $[B',B'] \subset [C,C] \subset C \subset A \cup B'$.

We now consider the case when $\chi_i \in B'$ and $\chi_j \in A$ with $\abs{i-j} > 1$  (the last case follows from a symmetric argument). Let $x \in X$, $y = t^{(\chi_i)} \cdot x$, $x' = s^{(\chi_j)} \cdot x$, $y' = s^{(\chi_j)} \cdot y$ and $w = t^{(\chi_i)} \cdot x'$. Notice that $y' = [(-s)^{(\chi_j)},(-t)^{(\chi_i)}] \cdot w$.

Let $D_{ij} = \set{\chi_{i+1},\dots,\chi_{j-1}}$ be the set of weights strictly between $\chi_i$ and $\chi_j$. Decompose $D_{ij}$ into $\set{\gamma_1,\dots,\gamma_{m_1},\delta_1,\dots,\delta_{m_2},\epsilon_1,\dots,\epsilon_{m_3}}$, where $\set{\gamma_k}$, $\set{\delta_k}$ and $\set{\epsilon_k}$ are the weights of $A\cap D_{ij}$, $(B \setminus B') \cap D_{ij}$ and $B'\cap D_{ij}$, respectively, each listed with the induced circular ordering. We assume $y$ and $w$ are sufficiently close, to be determined later (if we show it for sufficiently small $s,t$, we may use the dynamics of $\R^k \curvearrowright X$ and \eqref{eq:renormalization} to conclude it for arbitrary $s,t$). 

Notice that the distribution $\bigoplus_{k=i+1}^{j-1} T\mc W^{\chi_k}$ is uniquely integrable to a foliation $\mc W^{ij}$ with $C^2$ leaves since it is the intersection of stable manifolds for the action. Since $y' = [(-s)^{(\chi_j)},(-t)^{(\chi_i)}] \cdot w$, $y' \in \mc W^{ij}(w)$ by Lemma \ref{lem:comm-relation}. Therefore, by Lemma \ref{lem:local-transitivity} applied to the splitting $T\mc W^{ij}$ into the bundles $E^{\gamma_k}$, $E^{\delta_k}$ and $E^{\epsilon_k}$, there exists a path moving from $w$ to $y'$ which first moves along the leaves of the 1-dimensional foliations corresponding to $B' \cap D_{ij}$, then the 1-dimensional foliations corresponding to $(B \setminus B')\cap D_{ij}$, then the 1-dimensional foliations corresponding to $A \cap D_{ij}$, in the given orderings. Let $p$ denote the point obtained after moving along the $B'$ foliations from $w$, and $q$ denote the point obtained by moving from $p$ along the $B \setminus B'$ foliations. Then $q$ is also connected to $y'$ via the $A$ foliations. See Figure \ref{fig:sub-res}.

\begin{figure}[!ht]
\begin{center}
\begin{tikzpicture}[scale=.75]
\node [below left] at (0,0) {$x$};
\draw [thick] (0,0) -- (12,0);
\node [below] at (6,0) {\tiny $t^{(\chi_i)}$};
\node [below right] at (12,0) {$y$};
\draw [thick] (0,0) -- (0,8);
\node [left] at (0,4) {\tiny $s^{(\chi_j)}$};
\node [above left] at (0,8) {$x'$};
\draw (12,0) to [out=90,in=240] (14,7);
\node [right] at (12.5,4) {\tiny $s^{(\chi_j)}$};
\draw (0,8) to [out=0, in=220] (11,10);
\node [above] at (5.5,8) {\tiny $t^{(\chi_i)}$};
\node [above left] at (11,10) {$w$}; 
\draw [red] (11,10) -- (13,10);
\node [above right] at (13,10) {$p$};
\node [below] at (12,10) {\tiny $B'$};
\draw [blue] (13,10) -- (14,9);
\node [right] at (14,9) {$q$};
\node [below right] at (14,7) {$y'$};
\draw [green] (14,9) -- (14,7);
\node [rotate=-45,below] at (13.5,9.5) {\tiny $B \setminus B'$};
\node [left] at (14,8) {\tiny $A$};
\draw[fill] (0,0) circle [radius=0.075];
\draw[fill] (12,0) circle [radius=0.075];
\draw[fill] (0,8) circle [radius=0.075];
\draw[fill] (14,9) circle [radius=0.075];
\draw[fill] (14,7) circle [radius=0.075];
\draw[fill] (11,10) circle [radius=0.075];
\draw[fill] (13,10) circle [radius=0.075];
\end{tikzpicture}
\caption{A geometric commutator}
\label{fig:sub-res}
\end{center}
\end{figure}

Choose any pair of $C^{2}$, $\abs{B}$-dimensional embedded discs $D_1 \ni x,y$, $D_2 \ni x',q$ transverse to $\mc W^A$ inside of $W^{\abs{\alpha,\beta}}(x)$. This is possible since $x$ and $y$ are connected via $\chi_i \in B'$ and $x'$ and $q$ are connected via only weights coming from $B$. Therefore, $x'$ and $q$ are the images of $x$ and $y$ under the $\mc W^A$ holonomy from $D_1$ to $D_2$. By \cite[Section 8.3, Lemma 8.3.1]{entropyII-annals1985}, there exist bi-Lipschitz coordinates for which the leaves of the foliation $\mc W^A$ are parallel Euclidean hyperspaces. In particular, the holonomy along $\mc W^A$ is uniformly Lipschitz independent of the choice of $D_1$ and $D_2$, given sufficiently good transversality conditions, so $d(x,y)/d(x',q)$ is bounded above and below by a constant.

We claim that $p = q$ (ie, that no weights of $(B \setminus B') \cap D_{ij}$ appear). Roughly, the reason is that such weights contract too slowly. Indeed, pick an element $a'\in \ker \alpha$ such that $\beta(a') = -1$. We may perturb $a'$ to an element $a$ which is regular and such that $\alpha(a) < 0$, and such that if $\delta \in B\setminus B'$, $ \chi_i(a) < \delta(a) < 0$. This is possible because if $\delta = u\alpha + v\beta \in B \setminus B'$, then $v< 1$, so $\delta(a') = -v > -1 = \beta(a') \ge \chi_i(a')$, since when $\chi_i \in B'$, the $\beta$ coefficient is at least 1. This is clearly an open condition for each $\delta$, so we may choose $a$ as indicated. We may also assume, {  by rescaling $a$ as necessary}, that $\beta(a) = -1$.

Since $y \in W^{\chi_i}(x)$, we can estimate distance between iterates of $x$ and $y$ using the H\"older metric along the coarse Lyapunov leaf $W^{\chi_i}(x)$. Recall that since $\chi_i \in B'$, $\chi_i(a) < \beta(a) = -1$. Therefore $d_{W^{\chi_i}}(a^k\cdot x, a^k \cdot y) = e^{k\chi_i(a)}d_{W^{\chi_i}}(x,y) < e^{-k}d_{W^{\chi_i}}(x,y)$ using the H\"older metric on the manifold. Now, suppose $p \not= q$. Recall that $p$ and $q$ are connected by legs in $B \setminus B' = \set{\delta_1,\dots,\delta_m}$, so that there exist $p = x_0,x_1,\dots,x_m = q$ such that $x_l \in W^{\delta_l}(x_{l-1})$. Since the foliations $W^{\delta_l}$ are transverse, if $p \not= q$, there exists some $l$ for which $x_l \not= x_{l-1}$. Without loss of generality, we assume that $\set{\delta_l}$ are ordered such that $0 > \delta_1(a) > \delta_2(a) > \dots > \delta_m(a) > \chi_i(a)$. Then let $l_0$ be the minimal $l$ for which $x_l \not= x_{l-1}$, and $c_1 = \delta_{l_0}(a)$, $c_2 = \delta_{l_0+1}(a)$. Notice that $0 > c_1  > c_2 > -1$. By minimality, we get $x_{l_0-1} = p$.

Let $d$ denote the Riemannian distance on the manifold. Since for any $\gamma \in \Delta$, the distance along each $W^{\gamma}$ leaf is locally Lipschitz equivalent to the distance on the manifold, there exists $L > 0$ such that for all $\gamma \in \Delta$ and sufficiently close points $z \in W^{\gamma}(z')$, we have $L^{-1}d_{W^{\gamma}}(z,z') \le d(z,z') \le Ld_{W^{\gamma}}(z,z')$. Then after applying the triangle inequality, we get:

\begin{eqnarray*} d(a^k \cdot x',a^k \cdot q) & \ge & d(a^k \cdot x_{l_0},a^k \cdot p) - d(a^k \cdot x',a^k \cdot p) - d(a^k\cdot x_{l_0},a^k q)  \\
 & \ge & L^{-1}d_{W^{\delta_{l_0}}}(a^k \cdot x_{l_0},a^k \cdot p) - d(a^k \cdot x',a^k \cdot p) - d(a^k\cdot x_{l_0},a^k q) \\
& \ge & L^{-2}e^{c_1k}d(x_{l_0},p) - L^2Ce^{c_2k} \ge C'e^{c_1k}
\end{eqnarray*}

\noindent since by construction, we may iteratively apply the triangle inequality to all legs connecting $x_{l_0}$ and $q$ and $x'$ and $p$, which contract faster than $e^{c_2k}$ and $e^{-k}$, respectively, since $c_1 > c_2, -1$. We may construct new disks $D_{1,k}$ and $D_{2,k}$ with sufficient transversality conditions to $\mc W^A$ connecting $a^k\cdot x$ and $a^k \cdot y$, and $a^k \cdot x'$ and $a^k \cdot q$ (note that we may not simply iterate the disks $D_1$ and $D_2$ forward, {  since $\mc W^A$ is not the fast foliation for $a$}, and the transversality may degenerate). Therefore, since each of the foliations along weights of $B$ are uniformly transverse to those of $A$, and the holonomies are Lipschitz with a uniform Lipschitz constant on sufficiently transverse discs, we arrive at a contradiction, so $p = q$.

Therefore, the connection between $w$ and $y'$ only involves weights of $A \cup B'$. By the induction hypothesis, the commutator of two weights $ \chi_k,\chi_\ell \in (A \cup B') \cap D_{ij}$ produces only weights in $ (A \cup B') \cap D_{k\ell}$. {  Hence, for any $x$, 

\[ s^{(\chi_k}) * t^{(\chi_\ell)} \cdot x = \rho_1 * t^{(\chi_\ell)} * s^{(\chi_k)} \cdot x = t^{(\chi_\ell)} * s^{(\chi_k)} * \rho_2 \cdot x\] for some paths $\rho_1$ and $\rho_2$ which only involve weights in $A \cup B'$, but may depend on $x$, $s$ and $t$. Using such relations, for any word in the weights $A \cup B'$,} we may put it in a desired circular ordering without weights in $B \setminus B'$. This proves the inductive step, and hence the lemma, since $B \setminus B' = \set{u \alpha + v\beta : u,v < 1} \cap \Delta$. 
% Notice that choosing $a \in \ker \beta$, and perturbing to give a contracting element  gives that $\bigoplus_{\chi \in A} T\mc W^\chi$ is tangent to a foliation $\mc W^A$ (since $A$ will dominate $B$). Furthermore, by \cite{brown??}, the holonomy along $\mc W^A$ is $C^{1,\theta}$.
%Fix $x \in X$, and consider $y = (t^{(\alpha)} * s^{(\beta)}) \cdot x$, and $z = t^{(\alpha)} \cdot x$. By Lemma \ref{lem:local-transitivity}, there exists  $y' \in \mc W^A(y)$ and $z_i \in \mc W^{\chi_i}(z_{i-1})$, with $\chi_i \in B$ such that $z_n = y$, for $n = \abs{B}$.
%We show the claim if $u < 1$, the argument that $v < 1$ is completely symmetric. Choose any $a_0 \in \ker \alpha$. Then $\chi(a_0) < 0$ for all $\chi \in D(\alpha,\beta) \setminus \set{\alpha}$, so perturbing $a$ to a regular element $a$ such that $\alpha(a),\beta(a) < 0$ and $\abs{\alpha(a)} = \min\set{\abs{\chi(a)} : \chi \in D(\alpha,\beta)}$. Consider the foliation $\mc W$ with tangent distribution $\bigoplus_{\chi \in D(\alpha,\beta)} T\mc W^\chi$, which is H\"older with smooth leaves because it can be written as an intersection of stable manifolds. This contains the foliation $\mc W'$ with tangent distribution $\bigoplus_{\chi \in D(\alpha,\beta) \setminus\set{\alpha}} T\mc W^\chi$, which is invariant under $a$. Furthermore, the splitting into $T\mc W^\alpha \oplus T\mc W'$ is a dominated splitting, therefore, by \cite{brown????}, the holonomy along $\mc W'$ is $C^{1,\theta}$.
\end{proof}
}

\subsection{Further Technical Tools}

In the remainder of this section we establish some further technical results related to the geometric commutators which will be useful in future sections.

\begin{lemma}
\label{lem:group-integrality}
Let $\Omega \subset \Delta^-(a)$ for some Anosov $a \in \R^k$. Suppose that the action of $\mc P_\Omega$ factors through a Lie group action $H$. Then

\begin{enumerate}
\item $H$ is nilpotent {and simply connected,}
\item if $\alpha,\beta \in \Omega$ and $\chi \in [\alpha,\beta]$, then $\chi = u\alpha + v \beta$ with $u,v \in \Z_+$, and
\item if $\alpha,\beta \in \Omega$, then $\rho^{\alpha,\beta}_{u\alpha + v\beta}(s,t,x) = cs^ut^v$ for some $c \in \R$.
\end{enumerate}
\end{lemma}

\begin{proof}
By Proposition \ref{prop:integrality}, the automorphism $\psi_a$ descends to an automorphism $\bar{\psi}_a$ of $H$. Since the eigenvalues $d\bar{\psi}_a$ are all less than 1, $\bar{\psi}_a$ is a contracting automorphism of $H$, and $H$ is nilpotent and simply connected. If $\rho^{\alpha,\beta}_\chi(s,t,x) \not\equiv 0$, then the subgroups of $H$ corresponding to $\alpha$ and $\beta$ do not commute. Furthermore, the Baker-Campbell-Hausdorff formula, together with the end of Proposition \ref{prop:integrality}(3), implies that the sum of the 1-dimensional subalgebras corresponding to $\set{ u \alpha + v \beta : u, v \ge 0, u,v \in \Z}$ is a closed subalgebra. In particular, if a weight $\chi$ satisfies $\rho^{\alpha,\beta}_{\chi}(s,t,x) \not\equiv 0$, then $\chi$ must have integral coefficients, as claimed. Finally, the Baker-Campbell-Hausdorff formula also implies that each $\rho^{\alpha,\beta}_\chi$ is a polynomial as $H$ is nilpotent, and \eqref{eq:rho-equivariance} implies that this polynomial is $u$-homogeneous in $s$ and $v$-homogeneous in $t$.
\end{proof}

\section{Polynomial forms of commutators}
\label{sec:fibers}

In the remainder of Part IV, we will prove that certain canonical cycles are all constant, and that  constancy of such cycles are sufficient for homogeneity.  For this we first prove that geometric brackets are constant.

\begin{proposition}
\label{prop:base base relations}
Let $\R^k \curvearrowright X$ be a topological Cartan action satisfying the \hyperlink{h-r-a}{higher rank assumptions} { and assumptions of Theorem \ref{thm:technical}}. If $\alpha,\beta \in \Delta$ are not proportional, then $\rho^{\alpha,\beta}(s,t,x)$ is independent of $x$.
\end{proposition}

\subsection{A simple setting}
{The proof of Proposition \ref{prop:base base relations} is long and complicated, and will occupy the rest of this section. Here, we first provide a proof under a significant strengthening of assumptions \ref{HR:a} and \ref{HR:b}, namely that we have \ref{HR:a} and 
\begin{itemize}
\item[(HR2+)] \hypertarget{HR:c}{} for every pair of linearly independent $\alpha,\beta \in \Delta$, $\ker \alpha \cap \ker \beta$ has a dense orbit on $X$.
\end{itemize} 
  These assumptions are very restrictive. Indeed, it is impossible for them to hold for $\R^2$-actions and for $\R^3$-actions is similar to the assumption that every one-parameter subgroup has a dense orbit. Among the homogeneous  $\R^k$-actions, it does not allow for any of the standard irreducible rank two actions as factors for any $\R^k$-action. One may compare this assumption with the assumptions of \cite{KaSp04}. We provide this alternate proof in the special case because it is remarkably simpler.

\begin{proof}[Proof under \ref{HR:a} and \hyperlink{HR:c}{(HR2+)}]
%Since for each $\alpha,\beta \in \Delta$, we may choose $x$ such that $(\ker \alpha \cap \ker \beta)x$ is dense, since the set of dense orbits is either empty or a dense $G_\delta$ set, we may choose some $x_0$ such that its orbit is dense for every such intersection. 
Notice that if $a \in \ker \alpha \cap \ker \beta$, then $a$ commutes with $\eta^\gamma$ for any $\gamma = s\alpha + t\beta$, $s,t\in\R$. In particular:

\[ [t^{(\alpha)},s^{(\beta)}] \rho^{\alpha,\beta}(s,t,a\cdot x)\cdot x =  a^{-1} [t^{(\alpha)},s^{(\beta)}] \rho^{\alpha,\beta}(s,t,a\cdot x) a\cdot x = x. \] 

\noindent Therefore, $\rho^{\alpha,\beta}(s,t,a\cdot x) = \rho^{\alpha,\beta}(s,t,x)$ for all $a \in \ker \alpha \cap \ker \beta$. So $\rho^{\alpha,\beta}(s,t,\cdot)$ is constant on the orbits of $\ker \alpha \cap \ker \beta$, and hence constant everywhere by our dense orbit assumption.
\end{proof}

\subsection{Polynomial forms in the general setting}
The remainder of this section is the proof of Proposition \ref{prop:base base relations} under the weaker assumptions \ref{HR:a} and \ref{HR:b}.} We will show that for every $\alpha,\beta \in \Delta$ and $\chi \in D(\alpha,\beta)$, $\rho^{\alpha,\beta}_\chi(s,t,x)$ is independent of $x$. The proof of this case involves two inductions. In each step of the inductions we will show that

\begin{eqnarray}
&& \label{74out-induction1}\mbox{if $\chi \in [\alpha,\beta]$, then $\chi = k \alpha + l \beta$ for some  $k,l \in \Z_+$}, \mbox{ and} \\
&& \label{74out-induction2} \mbox{if $\chi = k\alpha + l \beta \in [\alpha,\beta]$, then $\rho^{\alpha,\beta}_\chi(s,t,x) = c_\chi s^kt^l$}.
\end{eqnarray}

{\color{olive}
%Every weight $\chi \in [\alpha,\beta] \cap \Delta_f$ is of the form $u \alpha + v \beta$, so we may introduce a lexicographical ordering on $[\alpha,\beta] \cap \Delta_f$ by saying $u \alpha + v \beta \prec u' \alpha + v'\beta$ if $u < u'$ or $u = u'$ and $v < v'$. Without loss of generality assume $\Omega = \set{\chi_0,\dots,\chi_m}$ is ordered so that $\chi_0 \prec \dots \prec \chi_m$.

%\begin{corollary}[Base case]
%If $[\alpha,\beta] \cap \Delta_b = \emptyset$, for each $\chi_i \in \Omega$, $\rho^{\alpha,\beta}_{\chi_i}(s,t,x)$ is a cocycle over $\eta^\alpha$ in $s$ when fixing $t$, and is a cocycle over $\eta^\beta$ in $t$ when fixing $s$.
%\end{corollary}

%\begin{proof}
%This is essentially the same as the proof of Lemma \ref{lem:cocycle-property}. Indeed, the $\chi \in \Omega$ cannot be obtained as commutators with either $\alpha$ or $\beta$ and hence  ``slide'' past $\alpha$ or $\beta$ legs unchanged possibly introducing higher level legs however.
%\end{proof}

}
{
The outer induction is on $\#D(\alpha,\beta)$, in which case we will show \eqref{74out-induction1} and \eqref{74out-induction2} for all $\chi \in D(\alpha,\beta)$ (at the base case when $\#D(\alpha,\beta) = 0$, $\eta^\alpha$ and $\eta^\beta$ commute and there is nothing to prove). The inner induction varies $\chi$ with fixed $\alpha,\beta$: Let $\Omega_l = \set{ u \alpha + v\beta : u + v = l, u,v  \in \R_+} \cap \Delta$. Then there are finitely many values $l_0 < l_1 < \dots < l_m$ such that \[ [\alpha,\beta] \cap \Delta \subset D(\alpha,\beta) \cap \Delta = \bigcup_{i=0}^m \Omega_{l_i}.\]

\noindent With $\alpha$ and $\beta$ fixed, we will show properties \eqref{74out-induction1} and \eqref{74out-induction2} by induction on $i$ (first proving the statements for $\chi \in \Omega_{l_0}$, then $\chi \in \Omega_{l_1}$, etc). 
}
Given a subset $S$ of weights, and a weight $\chi \in \Delta$, let $[\chi,S] = \bigcup_{\chi' \in S} [\chi,\chi']$.
%$\Delta_0 = \Omega$ and $\Delta_{i+1} = \big((\Delta_i + \alpha) \cup (\Delta_i + \beta)\big) \cap ([\alpha,\Delta_i] \cup [\beta,\Delta_i]) \cap [\alpha,\beta]$.

%\begin{corollary}
%\label{cor:delt-increasing}
%$\bigcup \Delta_i = [\alpha,\beta]$ %, % $\Delta_i \cap \Delta_j = \emptyset$ if $i \not= j$, 
%and if $i \le j$, $[\alpha,\Delta_j] \cap \Delta_i = [\beta,\Delta_j] \cap \Delta_i = \emptyset$. 
%\end{corollary}
%\begin{proof}
%The first claim follow from Lemma \ref{lem:uppertriangular}. {  Second claim?????}
%\end{proof}

\begin{lemma}
\label{lem:comm-omegas}
If $\chi \in [\alpha,\Omega_l]$ or $[\beta,\Omega_l]$, then $\chi \in \Omega_{l+t}$ for some $t \ge 1$.
\end{lemma}

\begin{proof}
If $\gamma \in \Omega_l$ and $\chi \in [\alpha,\gamma]$, we know that $\#D(\alpha,\gamma) < \#D(\alpha,\beta)$, since $D(\alpha,\gamma)$ is strictly contained in $D(\alpha,\beta)$. Therefore, this follows from the induction hypothesis  \eqref{74out-induction1} applied to $\alpha$ and $\gamma$.
%Without loss of generality consider $\chi \in [\alpha,\Omega_l]$. %If $\chi \in \Delta_b$, let $\mc Q$ be the group freely generated by copies of $\R$ corresponding to $\alpha$, $\chi$, any base weight which is a positive combination of $\alpha$ and $\chi$, as well as $[\alpha,\beta] \cap \Delta_f$. Let $Q$ be the quotient of $Q$ by the normal subgroup generated by the commutator relations between each weight appearing (by Section \ref{subsec:bf} and the outer induction hypothesis, we know the corresponding $\rho$-functions are independent of the basepoint). Using a circular ordering to present elements gives a unique presentation as in Lemma \ref{lem:group-presentation}, so $Q$ is a Lie group. By Proposition \ref{prop:integrality}, the coefficients of any weight in $\chi \in [\alpha,\Omega_l]$ appear with linear integral terms, so the $l$-value increases.
%Notice that $\chi \in \Delta_f$, since $[\alpha,\Delta_f] \subset \Delta_f$, and each $\Omega_l \subset \Delta_f$. Thus, the lemma follows from Claim \ref{out-induction1} of Section \ref{subsec:bf}.
\end{proof}

Since $[\alpha,\beta] \cap \Delta \subset \bigcup \Omega_{l_i}$, it suffices to show \eqref{74out-induction1} and \eqref{74out-induction2} for each weight $\chi \in \Omega_{l_i}$. %We will do this by an induction on $i$, starting from $i = -1$, setting $\Omega_{l_{-1}} = \emptyset$. % The base of the induction will be $i = -1$, with $\Omega_{l_{-1}} = \emptyset$. 
Assume \eqref{74out-induction1} and \eqref{74out-induction2} hold for $\chi \in \Omega_{l_j}$, $j < i$. We use this induction hypothesis together with the induction hypothesis on the forms of commutators of $\alpha$ and $\beta$ with $\gamma \in D(\alpha,\beta)$ (the outer induction hypothesis) and $\chi \in \Omega_{l_j}$ with $j < i$ (the inner induction hypothesis).

For simplicity of notation, with fixed $\alpha$ and $\beta$, we defined $\varphi_\chi(s,x) = \rho^{\alpha,\beta}_\chi(s,1,x)$. Notice that if $\chi = u\alpha + v\beta$ and $a \in \ker \beta$, by \eqref{eq:rho-equivariance}:

\begin{equation}\label{eq:phi-pullout}
\varphi_\chi(s,x) =  e^{-u\alpha(a)}\varphi_\chi(e^{\alpha(a)}s,a \cdot x).
\end{equation}
 While the proof of the following requires checking some complicated details, the following lemma follows from two simple ideas: splitting a commutator into a sum of two commutators requires a conjugation and reordering, and with careful bookkeeping, the reordering and conjugation can be shown to contribute polynomial terms only. {Recall that we assume that $\chi \in \Omega_{l_i}$ and that \eqref{74out-induction1} and \eqref{74out-induction2} hold for every weight of $\Omega_{l_j}$, $j < i$.}

\begin{lemma}[{The cocycle-like property}]
\label{eq:cocycle-like-bb}
$\varphi_{\chi}(s_1+s_2,x) = \varphi_{\chi}(s_1,x) + \varphi_{\chi}(s_2,\eta^{\alpha}_{s_1}x) + s_2p(s_1,s_2)$  for some polynomial $p$ whose coefficients are independent of $x$.% Furthermore, $\varphi_{\chi}(s,x) = cs^k$ for some $c \in \R$.
\end{lemma}

\begin{proof}

We assume that $\chi \in \Omega_{l_i}$. Recall that $\varphi_\chi(s,x) = \rho^{\alpha,\beta}_\chi(s,1,x)$ is the length of the $\chi$-component of the unique path $\rho^{\alpha,\beta}(s,1,x)$, written in circular ordering of the weights in $D(\alpha,\beta)$, which connects $[s^{(\alpha)}, 1^{(\beta)}] \cdot x$ and $x$. Notice that using only the free group relations, we get that:

\begin{eqnarray*}
[(s_1+s_2)^{(\alpha)},t^{(\beta)}] & = & (-t)^{(\beta)} * (-s_1-s_2)^{(\alpha)} * t^{(\beta)} * (s_1+s_2)^{(\alpha)} \\
 & = &  (-t)^{(\beta)}*  ( -s_1-s_2)^{(\alpha)} *  t^{(\beta)} * s_2^{(\alpha)}   \\
 & & \qquad * \; \big ( ( (-t)^{(\beta)} *s_1^{(\alpha)} *  t^{(\beta)} ) *  ( (-t)^{(\beta)} *  (-s_1)^{(\alpha)} * t^{(\beta)})\big) *  s_1^{(\alpha)} \\
 & = & (-t)^{(\beta)} * (-s_1)^{(\alpha)} *\big( t^{(\beta)} *  (-t)^{(\beta)} \big) * (-s_2)^{(\alpha)} * t^{(\beta)} * s_2^{(\alpha)} *   \\
 & & \qquad * \;  \big((-t)^{(\beta)}  * s_1^{(\alpha)} * t^{(\beta)}\big) * [s_1^{(\alpha)},t^{(\beta)}] \\
 & = &  \big( (-t)^{(\beta)}  * (-s_1)^{(\alpha)} * t^{(\beta)} \big) * [s_2^{(\alpha)},t^{(\beta)}] *  \big( (-t)^{(\beta)} * s_1^{(\alpha)} *  t^{(\beta)}\big) * [s_1^{(\alpha)},t^{(\beta)}] .
\end{eqnarray*}

\begin{figure}[!ht]
\includegraphics[width=5in]{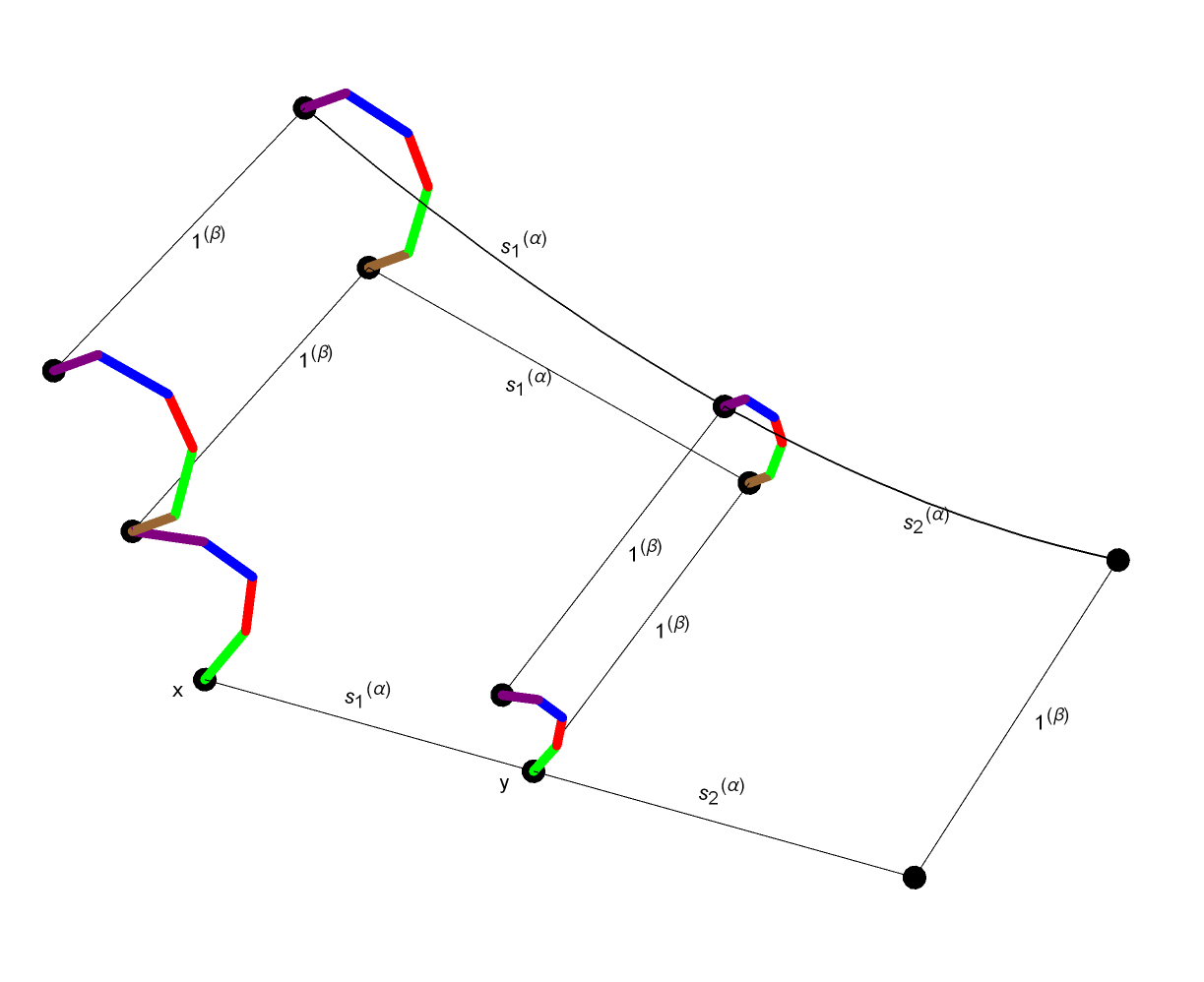}
\caption{The cocycle-like property}
\label{fig:cocycle-like}
\end{figure}

Figure \ref{fig:cocycle-like} illustrates the equality of $[(s_1+s_2)^{(\alpha)},t^{(\beta)}]$ and the last term of the string of equalities above. We may replace the last line with

\[
 \Big( \big((-t)^{(\beta)} * s_1^{(\alpha)} * t^{(\beta)}\big)^{-1} * \rho^{\alpha,\beta}(s_2,t,y)^{-1} * \big((-t)^{(\beta)} * s_1^{(\alpha)} * t^{(\beta)}\big) \Big) *  \rho^{\alpha,\beta}(s_1,t,x)^{-1}
\]

where $y = \eta^{\alpha}_{s_1}(x)$. Therefore, as group elements acting on $[(s_1+s_2)^{(\alpha)},t^{(\beta)}] \cdot x$,

\begin{equation}
\label{eq:fig-explain} 
\rho^{\alpha,\beta}(s_1+s_2,t,x) =  \rho^{\alpha,\beta}(s_1,t,x) *  \Big( \big((-t)^{(\beta)} * s_1^{(\alpha)} * t^{(\beta)}\big)^{-1} * \rho^{\alpha,\beta}(s_2,t,y) * \big((-t)^{(\beta)} * s_1^{(\alpha)} * t^{(\beta)}\big) \Big).
\end{equation}

We now use the induction hypothesis. Notice that we would like to compute the $\chi$ term, so we must try to use the group relations known to put the expression into its canonical circular ordering. In our computations, we will carefully reorder the weights appearing in $\rho^{\alpha,\beta}(s_1+s_2,t,x)$, which in each case may introduce polynomial expressions in the $\chi$ term. This does not affect our result because we only wish to obtain the cocycle property up to a polynomial term. We first make certain geometric arguments which correspond to the formal proof manipulating expression \eqref{eq:fig-explain} which follows.

 Let us summarize the figure. Notice that in Figure \ref{fig:cocycle-like}, there are five curves consisting of up to 5 colors each representing the (conjugates of the) $\rho^{\alpha,\beta}$-terms in \eqref{eq:fig-explain}. % which terminate at $x$ and $y$, respectively.
The blue legs consist of weights from the collection $\gamma \in \Omega_{l_j}$ with $j < i$, the red legs correspond to the weight $\chi \in \Omega_{l_i}$ (which we isolate in the current inductive step), the green legs consist of the other weights in $\Omega_{l_i}$, together with any weights of $\Omega_{l_j}$, $j > i$. The brown legs consist of the weights in $[\alpha,[\alpha,\beta]] \setminus [\alpha,\beta]$ (while this set will be empty by the Baker-Campbell-Hausdorff formula once the we conclude that the action is homogeneous, we cannot assume this now).

%The purple legs in each curve represent the weights in $[\alpha,\beta] \cap \Delta_b$. Since we know the base-fiber relations from Section \ref{subsec:bf} take polynomial forms independent of $x$, we may assume that we have reordered the usual circular ordering to have all base weights appearing first. The blue curve represents the legs from the collection of weights $\gamma \in \Omega_{l_j}$ with $j < i$, for which we know that $\rho^{\alpha,\beta}_\gamma$ take the form of Claim \ref{74out-induction2}.\footnote{Note that in the case of $\Omega_{l_0}$, there are no blue legs, so one may consider this to be the base case, as we do not require any inductive hypothesis. We elect to write the proof this way as the proof of the base and inductive steps are virtually identical.} The red curve represents the leg $\chi \in \Omega_{l_i}$ which we now consider. Finally, the green legs represent the other legs in $\Omega_{l_i}$, together with any legs of $\Omega_{l_j}$, $j > i$.

We make some observations to justify the picture. % For each blue leg, we know that $\rho^{\alpha,\beta}_\gamma$ take the form of Claim \ref{74out-induction2} by the induction hypothesis. Note that in the case of $\Omega_{l_0}$, there are no blue legs, so one may consider this to be the base case, as we do not require any inductive hypothesis. We elect to write the proof this way as the proof of the base and inductive steps are virtually identical. 
Conjugating the  $\rho^{\alpha,\beta}(s_2,t,y)$ in \eqref{eq:fig-explain} corresponds to ``sliding'' it along the $t^{(\beta)}$, $s_1^{(\alpha)}$ and in the opposite direction $t^{(\beta)}$ legs.  If $\gamma$ is a green leg, then $\rho^{\alpha,\gamma}_\chi = \rho^{\beta,\gamma}_\chi \equiv 0$ by Lemma \ref{lem:comm-omegas}, since $\chi \in \Omega_{l_i}$ and the green legs are in $\Omega_{l_j}$ with $j \ge i$. Therefore, the nonlinear parts of the $\chi$ terms come from commuting the blue legs $(\Omega_{l_j}, \; j < i)$ and brown legs. The $\chi$-contributions from commuting with the blue  legs are known to have polynomial form by the outer induction, respectively. The new brown color curves correspond to the possible weights $\lambda \in [\alpha,\gamma]$ or $[\beta,\gamma]$, where $\gamma \in [\alpha,\beta]$. If $\gamma \in \Delta_f$, $\gamma$  appears in some $\Omega_{l_j}$, and we may use Lemma \ref{lem:comm-omegas} if $j \ge i$ or the induction hypothesis if $j < i$ to give a polynomial form of $\lambda$, which guarantees a polynomial contribution to $\chi$. If $\gamma \in [\alpha,\beta]$, the $\lambda$ terms take polynomial form by induction. In summary, at each conjugating step when conjugating $\rho^{\alpha,\beta}(s_2,t,y)$, we obtain polynomial forms for all legs which feed into $\chi$, yielding the desired cocycle-like property.

We also provide a formal, algebraic argument, using \eqref{eq:fig-explain}. We must commute the $\rho^{\alpha,\beta}(s_2,t,x)$ term with the $\big(t^{(\beta)} * s_1^{(\alpha)} * (-t)^{(\beta)}\big)$ term. By the outer induction, if $\gamma \in [\alpha,\beta]$, since $\#D(\alpha,\gamma) < \#D(\alpha,\beta)$, we know that commuting $(-t)^{(\beta)}$ past any such $\gamma$ gives polynomials, possibly in other base weights or fiber weights, but all of which lie strictly between $\alpha$ and $\beta$. 

Now suppose that $\chi' \in D(\alpha,\beta)$ satisfies $\chi \in [\alpha,\chi']$ or $\chi \in [\beta,\chi']$, so that $\chi' \in \Omega_{l_j}$, $j < i$ by Lemma \ref{lem:comm-omegas}. Rearranging with the fiber weights, we know also by the inner induction that %and Section \ref{subsec:bf} that %any fiber weight $\chi'$ for which $\chi \in [\alpha,\chi']$ or $[\beta,\chi']$ (or any such chain of commutators) 
%the $\chi'$ legs, and hence 
$\chi$ contributions from such rearrangements are polynomial form. Thus, commuting $(-t)^{(\beta)}$ past each leg of $\rho^{\alpha,\beta}(s_2,t,x)$ in the fiber also has $\chi$-term which is a polynomial in $s_2$ and $t$. We may similarly pass $s_1^{(\alpha)}$ and $t^{(\beta)}$ through to arrive at a product of terms such that every contribution from some $\chi' \in \Omega_{l_j}$ with $j < i$ is a polynomial. Any $\chi$ legs obtained by commuting polynomial terms with $\alpha$ or $\beta$ are therefore also polynomial by induction, and these are the only possible weights that can feed into $\chi$ by Lemma \ref{lem:comm-omegas}. The $\chi$ terms directly from $\rho^{\alpha,\beta}(s_i,t,x)$ are therefore the only nonpolynomial terms which could appear after putting the weights in a circular ordering, and must coincide with the original $\chi$ terms $\rho^{\alpha,\beta}(s_1,t,x)$ and $\rho^{\alpha,\beta}(s_2,t,y)$.

Thus we have shown that $\varphi_\chi(s_1+s_2,x) = \varphi_\chi(s_1,x) + \varphi_\chi(s_2,\eta^{\alpha}_{s_1}x) + P(s_1,s_2)$ for some polynomial $P$. Divisibility of $P$ by $s_2$ follows from taking $s_2 = 0$ in this equality, giving $\varphi_\chi(s_1,x) = \varphi_\chi(s_1,x) + P(s_1,0)$, so $P(s_1,0) = 0$ for every $s_1$. This occurs if and only if $P$ is divisible by $s_2$.
%We proceed an induction on $k+l$, $k,l \ge 1$. The base of our induction is $k+l = 2$, which is $\alpha + \beta$, for which we have shown bilinearity. 
%Taking the $\chi$ component of the term above yields the first two sum terms as the commutators. Then we must conjugate the second commutator by $t^{(\beta)} * s_1^{(\alpha)} * (-t)^{(\beta)}$. By Section \ref{subsec:bf}, we know that $\alpha$ and $\beta$ act on $[\alpha,\beta] \cap \Delta_f$ by polynomials. Furthermore, each $\gamma \in [\alpha,\beta] \cap \Delta_b$, $\# D(\alpha,\gamma), \# D(\gamma,\beta) < \# D(\alpha,\beta)$. Therefore, the $\chi$ contribution from commuting any terms in $[\alpha,\beta] \cap \Delta_b$ with $t^{(\beta)} * s_1^{(\alpha)} * (-t)^{(\beta)}$ is a polynomial.
%As in the end of Section \ref{subsec:bf}, we may compute the derivative of $\varphi_{\chi}(s,x)$ in the $s$ variable to be independent of $x$ and a polynomial in $s$. An identical argument shows a polynomial form in $t$. Therefore, $\rho^{\alpha,\beta}_{\chi}(s,t,x) = c_{\chi}s^kt^l$ by the renormalization equation \eqref{eq:rho-equivariance}.
\end{proof}

%\begin{proposition}
%\label{prop:oneatleast1}
%{\color{olive} SHOULD HOLD IN GENERAL, NOT SURE HOW TO PHRASE GENERALITY. PROOF IS NORMAL FORMS?} If $\chi = u\alpha + v\beta$ with $u,v < 1$, then $\rho^{\alpha,\beta}_\chi \equiv 0$.
%\end{proposition}

%By Lemma \ref{lem:at-least-one}, for every $\chi \in [\alpha,\beta]$, either $u \ge 1$ or $v \ge 1$ (or both). Without loss of generality assume that $u \ge 1$.
{ 
By Lemma \ref{lem:at-least-one}, if $\rho^{\alpha,\beta}_\chi \not\equiv 0$, either $u$ or $v$ is at least 1. Without loss of generality, assume that $u \ge 1$. 
}

\begin{corollary}
\label{cor:phi-polynomial}
$\varphi_{\chi}(s,x) = c_xs^u$ for some $c_x \in \R$ which is either independent of $x$, or constant along the fibers of a circle factor on which $\ker \beta$ has dense orbits.
\end{corollary}

\begin{proof}
We claim that for every $x \in X$ and Lebesgue almost every $s_1 \in \R$, $\left.\frac{\partial}{\partial s}\right|_{s = s_1} \varphi_\chi(s,x)$ exists. It suffices to show that $\varphi_\chi(s,x)$ is locally Lipschitz in $s$. By Lemma \ref{eq:cocycle-like-bb}, 

\begin{eqnarray*}
\abs{\varphi_\chi(s,x) - \varphi_\chi(s_1,x)} & = & \abs{\varphi_\chi(s-s_1,\eta^{\alpha}_{s_1}x) + (s-s_1) \cdot p(s,s-s_1)} \\
 & \le & \abs{s-s_1}^u\abs{\varphi_\chi(1,a \cdot \eta^{\alpha}_{s_1}x)} + \abs{s-s_1}\cdot \abs{p(s,s-s_1)} \\
\end{eqnarray*}

\noindent for a suitable choice of $a \in \ker\beta$ (using \eqref{eq:phi-pullout}). Since $p$ is a polynomial and $u \ge 1$, $\varphi_\chi(s,x)$ has a Lipschitz constant in any neighborhood of $s$. Therefore, for almost every $s_1 \in \R$, $\varphi$ is differentiable in $s$ at $s_1$. By Lemma \ref{eq:cocycle-like-bb}, for every $x \in X$, Lebesgue almost every $y \in W^\alpha(x)$ has $\varphi_\chi(s,y)$ differentiable at 0. Therefore, $f(x) = \left.\frac{\partial}{\partial s}\right|_{s = 0} \varphi_\chi(s,x)$ exists on a dense subset of $X$, since it holds at Lebesgue almost every point of every leaf $W^\alpha(x)$. The set on which $f$ is defined is also $\ker \beta$-saturated, since by \eqref{eq:phi-pullout}, if $a \in \ker \beta$,

\begin{equation}
\label{eq:f-equivariance}
 f(a \cdot x) = \lim_{s \to 0} \frac{1}{s}\rho^{\alpha,\beta}_\chi(s,1,a \cdot x) = \lim_{s \to 0} \frac{1}{s} \cdot e^{u\alpha(a)} \rho^{\alpha,\beta}_\chi(e^{-\alpha(a)}s,x) = e^{(u-1)\alpha(a)}f(x) .
\end{equation}

 We claim that $\abs{f(x)} \le B$ for some $B \in \R$ whenever it exists. Indeed:

\[ \abs{f(x)} = \lim_{s \to 0} \frac{1}{s}\abs{\varphi(s,x)} = \lim_{s \to 0} \frac{1}{s} s^u \abs{\varphi_\chi(1,a_s \cdot x)} \le \sup_{y \in X} \abs{\varphi_\chi(1,y)} \]

\noindent where $a_s \in \ker\beta$ is chosen appropriately (using 
\eqref{eq:phi-pullout}), since we assumed $u \ge 1$.
%Notice that if $a \in \ker \beta$, then

%$f$ is a measurable eigenfunction for every $a \in \ker \beta$, with eigenvalue $e^{(u-1)\alpha(a)}$. 

Therefore, either $u = 1$ or $f \equiv 0$, since otherwise one may apply an element $a$ with $\alpha(a)$ arbitrarily large to \eqref{eq:f-equivariance} to contradict the boundedness of $f$. Since $u = 1$ or $f \equiv 0$, $f$ is constant along $\ker \beta$ orbits. We claim that $f$ is also constant along $W^\alpha$ leaves whenever it exists. { Assume it is not identically 0, otherwise the claim follows immediately. In this case $u = 1$, and $\varphi$ must be a cocycle over the $\eta^\alpha_s$-flow by comparing Lemma \ref{eq:cocycle-like-bb} and \eqref{eq:phi-pullout}. Then pick a Weyl chamber $\mc W$ for which $\ker \beta$ bounds $\mc W$ and $\alpha(\mc W) > 0$. We may then build the SRB measure $\mu_{\mc W}$ as in Section \ref{sec:SRB}, and disintegrate it into ergodic components for the $\ker \beta$ action. By Proposition \ref{prop:SRB}, it follows that for $\mu_{\mc W}$-almost every point $x$, $f$ is constant at Lebesgue-almost every point of $W^\alpha$ (since it is a $\ker \beta$-invariant function, and defined at Lebesgue almost every point of {\it every} $W^\alpha$ leaf). Since $f$ is the derivative of $\varphi$, and $\varphi$ is a cocycle over the $\eta^\alpha_s$-flow, it follows that

\[ \varphi(s,x) = \int_0^s f(\eta^\alpha_t(x)) \, dt = ms\]
where $m_x$ is the commmon value of $f$ at almost every point of $W^\alpha(x)$. Then $\varphi$ is linear on a dense set of $W^\alpha$ leaves, and has slope invariant under $\ker \beta$. Since $\varphi$ is continuous, it follows that $\varphi(s,x) = m_xs$, where $m_x$ is a slope which depends only on a circle factor. This concludes the case when $u =1$.}   %Indeed, if $x_2 \in W^\alpha(x_1)$ and $f$ exists at both $x_1$ and $x_2$, choose $a_k \in \ker \beta$ such that $\alpha(a_k) \to -\infty$ and $a_k \cdot x_1 \to z$ for some $z \in X$ (first choose any sequence satisfying the first property, then choose a convergent subsequence). Since $\alpha$ is contracted by $a_k$, $a_k \cdot x_2 \to z$ as well. Then:

%\[ f(x_i) = \lim_{s \to 0} \frac{1}{s} \varphi(s,x_i) = \lim_{k\to \infty} e^{-\alpha(a_k)} \varphi(e^{\alpha(a_k)},x_i) = \lim_{k\to\infty} \varphi(1,a_k \cdot x_i) = \varphi(1,z). \]

% Since in each leaf, $\varphi$ is a continuous function whose derivative exists almost everywhere with respect to Lebesgue measure on the leaf.
 Now assume that $f \equiv 0$. Then

\begin{multline}
\left.\frac{\partial}{\partial s}\right|_{s = s_1} \varphi(s,x) = \left.\frac{\partial}{\partial s}\right|_{s = 0} \varphi(s+s_1,x) \\ = \left.\frac{\partial}{\partial s}\right|_{s = 0} \left[\varphi(s_1,x) + \varphi(s,\eta^{\alpha}_{s_1}x) + (s-s_1)p(s_1,s) \right] = q(s_1)
\end{multline}

\noindent where $q$ is some polynomial independent of $x$ (or depends only on the value of $x$ is some circle factor). Since $\varphi(0,x) = 0$, one may integrate $ q(s_1)$ to a get a polynomial form for $\varphi$ at each $x$. 
%Since $f$ is invariant under $\ker \beta$, if $\ker \beta$ has a dense orbit, $f$ is constant, and hence $m_x$ is independent of $x$. Otherwise, $\ker \beta$ has an orbit dense in the fiber of some circle factor. So by assumption \ref{HR:b}, the coefficients of the polynomial are either constant or depend only on some circle factor. 
Finally, $\varphi(s,x) = cs^u$, since by \eqref{eq:phi-pullout}, it must be $u$-homogeneous.
%We break into two cases: if $u > 1$, then we may use Lemma \ref{eq:cocycle-like-bb} to proceed as in the end of Section \ref{subsec:bf} and conclude the polynomial form by computing the derivative of $\varphi_\chi$ in the $s$ direction.
%In the case of $u = 1$, we claim that the polynomial term in Lemma \ref{eq:cocycle-like-bb} is identically 0. Indeed, observe that iterated applications of Lemma \ref{eq:cocycle-like-bb} gives:
%\[ \varphi(ns,x) = \sum_{k=0}^{n-1} \varphi(s,\eta^\alpha_kx) + ksp(s,ks) \]
%but $\sum_{k=0}^{n-1}ksp(s,sk)$ grows at least quadratically unless $p(s,sk) = 0$ for all $s$. However, the left-hand side satisfies $\varphi(ns,x) = n\varphi(s,a\cdot x)$ where $a$ is such that $\alpha(a) = -\log n$ and $\beta(a) = 0$. Therefore, $\varphi$ grows at most linearly, and $p \equiv 0$.
%Therefore, in the case of $u = 1$, $\varphi$ satisfies the cocycle equation, and one may apply Proposition \ref{prop:cocycle-trivialize} to get linearity in $s$, as desired.
\end{proof}

{ 
Now, we analyze the second variable, $t$. Notice that if $\psi(t,x) = \rho^{\alpha,\beta}_\chi(1,t,x)$ and $a \in \R^k$, then

\[ e^{\chi(a)}\psi(t,a\cdot x) = e^{\chi(a)} \rho^{\alpha,\beta}(1,t,a\cdot x) = \rho^{\alpha,\beta}(e^{\alpha(a)},e^{\beta(a)}t, x) = e^{v\beta(a)}\psi(e^{\beta(a)}t,x).\]

Here, we may use an arbitrary element of $\R^k$ instead of just $\ker \alpha$ since we have already established the polynomial form in Corollary \ref{cor:phi-polynomial}. Hence, for fixed $t$, $\psi(t,x)$ is a continuous function invariant under $\ker \alpha$, which has a dense orbit (or dense in the fiber of some circle factor). With $\psi$ independent of $x$, the following finishes the proof:

\begin{lemma}
\label{lem:cocycle-like-implies-polynomial}
Let $d \in \R_+$ and $f : \R \to \R$ be a continuous function such that $f(\lambda t) = \lambda^df(t)$ for all $t, \lambda \in \R$ and $f(t_1+t_2) = f(t_1) + f(t_2) + q(t_1,t_2)$ for some polynomial $q$. Then $d \in \Z_+$ and $f$ is a polynomial.
\end{lemma}

\begin{proof}
If $d = 1$, then $f(t) = t\cdot f(1)$, so $f$ is linear. Now assume that $d \not= 1$ . Then  $f(2t) = 2^df(t) = f(t) + f(t) + q(t,t)$, so $(2^d-2)f(t) = q(t,t)$, so $f(t) = (2^d-2)^{-1}q(t,t)$ and $f$ is a polynomial. Since $f$ is a polynomial, $d \in \Z$.
\end{proof}
}

\begin{proof}[Proof of  \eqref{74out-induction2}]
{We have shown that $\varphi_{\chi}(s,x) := \rho^{\alpha,\beta}_\chi(s,1,x)$ is a monic polynomial of degree $u$, which is either independent of $x$, or depends only on the fiber over the circle factor over which $x$ lies. Define $f(x) := \rho^{\alpha,\beta}_\chi(1,1,x) = \varphi_\chi(1,x)$. Then if $a \in \ker \beta$, by \eqref{eq:rho-equivariance}, Corollary \ref{cor:phi-polynomial}, and since $\chi = u\alpha + v\beta$:

\begin{multline*} f(a \cdot x) = \rho^{\alpha,\beta}_\chi(1,1,a\cdot x) = e^{\chi(a)} \rho^{\alpha,\beta}_\chi(e^{-\alpha(a)},1,x) = e^{u\alpha(a)}\varphi_\chi(e^{-\alpha(a)},x) \\ = e^{u\alpha(a)} \cdot c_x(e^{-\alpha(a)})^u = c_x = \varphi_\chi(1,x) = f(x).
\end{multline*}

{  We have also shown that $\psi : t \mapsto \rho^{\alpha,\beta}_\chi(1,t,x)$ is a monomial with coefficient depending only on the fiber of a circle factor} (except the circle factor is now determined by $\ker \alpha$-orbit closures rather than $\ker \beta$-orbit closures), so that the map $t \mapsto \rho^{\alpha,\beta}_\chi(1,t,x)$ is a monic polynomial of degree $v$. From this, we may conclude that $f(a\cdot x) = f(x)$ for all $a \in \ker \alpha$. Since $\ker \alpha$ and $\ker \beta$ span $\R^k$, $f$ is $\R^k$-invariant, and hence $C := \rho^{\alpha,\beta}_\chi(1,1,x)$ is independent of $x$. Therefore, we get \eqref{74out-induction2}, since for any $s,t \in \R_+$, we may choose $a \in \R^k$ such that $e^{\alpha(a)} = s$ and $e^{\beta(a)} = t$. Then by \eqref{eq:rho-equivariance},

\begin{equation}\label{eq:poly-from-const} \rho^{\alpha,\beta}_\chi(s,t,x) = e^{\chi(a)} \rho^{\alpha,\beta}_\chi(1,1,-a \cdot x) = Ce^{u\alpha(a) + v\beta(a)} = Cs^ut^v. 
\end{equation}

%To get this for negative values of $s,t$, notice that we still have Lemma \ref{eq:cocycle-like-bb} for all $s_1,s_2 \in \R$, and consider the values $s_1,s_2 > 0$ allows us to conclude that $s_1p(s_1,s_2) = C(s_1+s_2)^u - Cs_1^u - Cs_2^u$, and this must hold for all $s_1,s_2 \in \R$. Then taking $s_1 = s$ and $s_2 = -s$ in Lemma \ref{eq:cocycle-like-bb} shows that 
%\[ \varphi_\chi(-s,x) = \varphi_\chi(s + -s,x) - \varphi_\chi(s,x) - \big(C(s + -s)^u - Cs^u - C(-s)^u\big) = C(-s)^u,\]
%since $\varphi(0,x) = 0$. This shows that $\rho^{\alpha,\beta}_\chi(-1,1,x) = (-1)^uC$. 
Now, from Corollary \ref{cor:phi-polynomial} (and its counterpart for $t$), we get that $\rho^{\alpha,\beta}_\chi((-1)^{m_1},(-1)^{m_2},x) = (-1)^{m_1u+m_2v}C$. Therefore, we may repeat \eqref{eq:poly-from-const} to get the polynomial forms for negative values of $s,t$. This completes the proof.}
\end{proof}

\section{Symplectic Cycles}
\label{sec:symplectic}

\begin{definition}
\label{def:sub-paths-spaces}
Fix a topological Cartan action $\R^k \curvearrowright X$. Given a subset $\Phi = \set{\beta_1,\dots,\beta_n} \subset \Delta$, let $\tilde{\eta}_\Phi$ denote the induced action of $\mc P_\Phi = \R^{*\abs{\Phi}}$ on $X$, and $\mc C_\Phi(x) \subset \mc C(x)$ denote the stabilizer of $x$ for $\tilde{\eta}_\Phi$, and $\mc C(x) = \mc C_\Delta(x)$. Let  $\hat{\mc P}_\Phi = \R^k \ltimes \mc P_\Phi$, considered as a subgroup of $\hat{\mc P} = \R^k \ltimes \mc P$ (Definition \ref{def:P-groups}).
\end{definition}

In this section, we prove another class of cycles are independent of $x$ by studying the stabilizer of the action of $\mc P_{\set{\alpha,-c\alpha}}$ for $\alpha \in \Delta$.
 Since our endgoal is understanding the structure of $\Stab_{\hat{\mc P}}(x)$ for some $x$, the following lemmas provide important structures for paths coming from symplectic pairs.

\begin{lemma}
\label{lem:symplectic2}
Let $\R^k \curvearrowright X$ is a topological Cartan action, %satisfying the \hyperlink{h-r-a}{higher rank assumptions}, 
and suppose that $\alpha, -c\alpha \in \Delta$ are negatively proportional weights. Then if {$x_0 \in X$ is such that $\overline{(\ker \alpha) \cdot x_0} \supset \mc P_{\set{\alpha,-c\alpha}}\cdot x_0$}, $\mc C_{\set{\alpha,-c\alpha}}(x_0) \subset \mc C_{\set{\alpha,-c\alpha}}(x)$ for all {$x \in \overline{(\ker \alpha) \cdot x_0}$} and $\mc P_{\set{\alpha,-c\alpha}} / \mc C_{\set{\alpha,-c\alpha}}(x_0)$ is a Lie group.
\end{lemma}

\begin{proof}
Notice that in the group $\hat{\mc P}_{\set{\alpha,-c\alpha}}$, if $a \in \R^k$, $a \mc C_{\set{\alpha,-c\alpha}}(x)a^{-1} = \mc C_{\set{\alpha,-c\alpha}}(a \cdot x)$, and if $a \in \ker \alpha$, $a \mc C_{\set{\alpha,-c\alpha}}(x)a^{-1} = \mc C_{\set{\alpha,-c\alpha}}(x)$ by \eqref{eq:renormalization}. In particular, $\mc C_{\set{\alpha,-c\alpha}}(\cdot)$ is constant on $\ker \alpha$ orbits, and hence if $x \in \overline{(\ker \alpha)\cdot x_0}$, $\mc C_{\set{\alpha,-c\alpha}}(x_0) \subset \mc C_{\set{\alpha,-c\alpha}}(x)$. {Since the $\ker \alpha$-orbit closure of $x_0$ contains the $\mc P_{\set{\alpha,-c\alpha}}$-orbit of $x_0$, $\mc C_{\set{\alpha,-c\alpha}}(x_0) \subset \mc C_{\set{\alpha,-c\alpha}}(x)$ for every $x \in \mc P_{\set{\alpha,-c\alpha}}$.
Hence, the quotient $\mc P_{\set{\alpha,-c\alpha}} / \mc C_{\set{\alpha,-c\alpha}}(x_0)$ is a Lie group by Corollary \ref{cor:lie-from-const}. }
%If $\ker \alpha$ has a dense orbit in the fiber of some circle factor, then the $\mc C_{\set{\alpha,-c\alpha}}(x_0)$ is still normal. However, the corresponding group may vary as one moves along the circle factor. Fix any $b$ transverse to $\ker \alpha$. %topological group, which is locally path-connected since $\mc P_{\set{\alpha,-c\alpha}}$ is locally path-connected. Furthermore, the evaluation map $\rho \mapsto \rho \cdot x_0$ is injective and continuous%on a neighborhood of $x_0$
%, so the quotient group is a Lie group by Lemma \ref{lem:gleason-palais}.
\end{proof}

\begin{remark}
{Lemma \ref{lem:symplectic2} applies immediately when $\ker \alpha$ has a dense orbit, since this clearly implies the conditions. If $\ker \alpha$ only has orbits dense in the fiber of some circle factor, see Lemma \ref{lem:symplectic}.}
\end{remark}

\begin{definition}
\label{def:stand-gens}
Let $\mf g = \mf{sl}(2,\R)$. We call $\begin{pmatrix} 0 & 1 \\ 0 & 0 \end{pmatrix}$ and $\begin{pmatrix}0 & 0 \\ 1 & 0 \end{pmatrix}$ the {\normalfont standard unipotent generators}, and $\begin{pmatrix} 1 & 0 \\ 0 & -1 \end{pmatrix}$ the {\normalfont corresponding neutral element.}
Let $\mf h = \Lie(\Heis)$, where $\Heis$ is the standard 3-dimensional Heiseberg group, so that $\mf h = \set{\begin{pmatrix}0 & x & z \\ 0 & 0 & y \\0 & 0 & 0 \end{pmatrix} : x,y,z \in \R}.$ We call $\begin{pmatrix} 0 & 1 & 0 \\ 0 & 0 & 0 \\ 0 & 0  & 0 \end{pmatrix}$ and $\begin{pmatrix} 0 & 0 & 0 \\ 0 & 0 & 1 \\ 0 & 0 & 0 \end{pmatrix}$ the {\normalfont standard unipotent generators} and $\begin{pmatrix} 0 & 0 & 1 \\ 0  & 0 & 0 \\ 0 & 0 & 0 \end{pmatrix}$ the {\normalfont corresponding neutral element. }
\end{definition}

Notice that in both $\mf{sl}(2,\R)$, and $\mf h$, the map which multiplies one of the standard unipotent generators by $\lambda$, the other by $\lambda^{-1}$ and fixes the corresponding neutral element is an automorphism of the Lie algebra. Let $\mc G_{\set{\alpha,-c\alpha}}$ be the group $\mc P_{\set{\alpha,-c\alpha}} / \mc C_{\set{\alpha,-c\alpha}}$.

{
The following proposition classifies which groups can appear as $\mc G_{\set{\alpha,-c\alpha}}$. It works extremely generally, requiring only the topological Cartan condition (together with the assumption that $\mc G_{\set{\alpha,-c\alpha}}$ is a Lie group).
}

\begin{proposition}
\label{prop:symplectic-possibilities}
{If $\R^k \curvearrowright X$ is a topological Cartan action and $\chi,-c\chi \in \Delta$ satisfy that $\mc G_{\set{\chi,-c\chi}} := \mc P_{\set{\chi,-c\chi}} / \mc C_{\set{\chi,-c\chi}}(x_0)$ is a Lie group for some $x_0 \in X$, then $\mc G_{\set{\chi,-c\chi}}$} is either locally isomorphic to $\R^2$, $PSL(2,\R)$ or the Heisenberg group. Furthermore, in the last two cases, the flows $\eta^\chi$ and $\eta^{-c\chi}$ are the one-parameter subgroups of the standard unipotent generators and the corresponding neutral element generates a one-parameter subgroup of the topological Cartan action. 
\end{proposition}

\begin{proof}
$\mc G_{\set{\chi,-c\chi}}$ is a Lie group by Lemma \ref{lem:symplectic2}, which is generated by two one-parameter subgroups corresponding to the flows $\eta^\chi$ and $\eta^{-c\chi}$. Let $v_\pm$ denote the elements of $\Lie(\mc G_{\set{\chi,-c\chi}})$ which generate $\eta^{\chi}$ and $\eta^{-c\chi}$, respectively, and assume they don't commute (if they commute we are done). By Proposition \ref{prop:integrality}, the action of $a \in \R^k$ intertwines the action of $\mc G_{\set{\chi,-c\chi}}$ by automorphisms, and $v_0 = [v_+,v_-]$ is an eigenspace with eigenvalue $e^{(1-c)\chi(a)}$. % in the distance given by the Lie group. 
By Lemma \ref{lem:lip-to-holder}, the action of the one-parameter subgroup generated by $v_0$ is H\"older, so the distance in the Lie group is distorted in only a H\"older way. So for each element $a \in \R^k$, if $c > 1$ and $\chi(a) > 0$, each point in the orbit of the one parameter subgroup generated by $v_0$ will diverge at an exponential rate. Similarly, one may see exponential expansion or decay properties for different choices of $a$ and different possibilities for $c$. Therefore, if $c\not= 1$, the orbit of this one-parameter subgroup is contained in the coarse Lyapunov submanifold for either $\chi$ or $-c\chi$. Since Cartan implies that each coarse Lyapunov foliation is one-dimensional, $c = 1$ and $v_0$ is neither expanded nor contracted. Similarly, since $[v_0,v_\pm]$ must be an eigenvalue of the automorphism corresponding to $a$ with eigenvalue $e^{\pm \chi(a)}$, $[v_0,v_{\pm}] \subset \R v_\pm$. In particular, as a vector space $\Lie(\mc G_{\set{\chi,-c\chi}})$ is generated by $\set{v_-,v_0,v_+}$. One quickly sees that the only possibilities for the 3-dimensional Lie algebras are the ones listed.

%Assume that $[v_-,v_+] \not= 0$, and l
Assume $c = 1$ and let $f^\chi$ denote the flow generated by $v_0$. We claim there exists $b \in \R^k$ such that $f^\chi_t(x) = (tb) \cdot x$.  We first show this for a sufficiently small, fixed $t$. Therefore, let $f = f^\chi_t$ and $t$ be small enough so that $d(f(x),x) < \delta$ for every $x \in X$.  Notice that $f(a \cdot x) = a \cdot f(x)$, since by Proposition \ref{prop:integrality}, $v_0$ is an eigenvector of eigenvalue 0 for the resulting automorphism of $\mc G_{\set{\chi,-\chi}}$. Fix a regular element $a_1 \in \R^k$. For each $x$, there exists a path of the form $\rho_x = t_1^{(\alpha_1)} * \dots * t_m^{(\alpha_m)} * s_1^{(\beta_1)} * \dots * s_l^{(\beta_l)} * b$ such that $f(x) = \rho_x * x$, and $\alpha_1,\dots,\alpha_m, \beta_1,\dots,\beta_l$ is a circular ordering on the negative weights and positive weights, respectively (see Definition \ref{def:circular-ordering}). %., with $\alpha_i(a) < 0$ and $\beta_j(a) > 0$.
 We claim that $t_i = 0$ for every $i$. Indeed, by applying $a_1$ to $\rho$, we know that the $\alpha_i$ components will all expand, and the $b$ and $s_j$ components remain small. Thus, if $\delta$ is sufficiently small (relative to the injectivity radius of $X$), we may conclude that $d(f(x),x) > \delta$ for some $x \in X$, a contradiction. A similar argument shows that $s_j = 0$ for every $j$.

Therefore, $f(x) = b_x \cdot x$ for some $b_x \in \R^k$. Notice that if $\Stab_{\R^k}(x) = \set{0}$, $b_x$ is unique and continuously varying. Since $f$ commutes with the $\R^k$ action, $f(a \cdot x) = b_{a\cdot x}a \cdot x = ab_x \cdot x$, thus $b_x$ is constant on free orbits. Since there is a dense $\R^k$ orbit, $b_x$ is constant, and $f(x) = bx$ for all $x \in X$, as claimed.

Thus we have shown that each $f^\chi_t$ has an associated $b(t) \in \R^k$ such that $f^\chi_t(x) = b(t) \cdot x$ for sufficiently small $t$. It is easy to see that $b(t)$ varies continuously with $t$ from the construction above, and that the map $t \mapsto b(t)$ is a homomorphism from $\R$ to $\R^k$. Hence, $f^\chi_t(x) = (tb) \cdot x$ for some unique $b \in \R^k$.
\end{proof}

\begin{lemma}
\label{lem:symplectic}
Let $\R^k \curvearrowright X$ be a {topological} Cartan action satisfying the \hyperlink{h-r-a}{higher rank assumptions}, and suppose that $\alpha, -c\alpha \in \Delta$ are negatively proportional weights. Then the action of $\mc P_{\set{\alpha,-c\alpha}}$ factors through the action of a Lie group locally isomorphic to $\R^2$, some cover of $PSL(2,\R)$ or $\Heis$. %Then if the orbit of $x_0$ has a dense $\ker \alpha$-orbit, $\mc C_{\set{\alpha,-c\alpha}}(x_0) \subset \mc C_{\set{\alpha,-c\alpha}}(x)$ for all $x \in X$ and $\mc P_{\set{\alpha,-c\alpha}} / \mc C_{\set{\alpha,-c\alpha}}(x_0)$ is a Lie group.
\end{lemma}

\begin{proof}
We combine the techniques of the previous arguments. Assumption \ref{HR:b} of the \hyperlink{h-r-a}{higher rank assumptions} implies that either $\ker \alpha$ has a dense orbit (in which case we allude to Lemma \ref{lem:symplectic2}), or $\ker \alpha$ has a dense orbit on a fiber of a circle factor. In the latter case, it must have a dense orbit on every fiber, since $\R^k$ can take one fiber to any other, and the action intertwines the action of $\ker \alpha$. {Since the fibers of the circle factor are saturated by all coarse Lyapunov exponents (and hence contain $\mc P_{\alpha,-c\alpha}$-orbits), we may use Lemma \ref{lem:symplectic2} and Proposition \ref{prop:symplectic-possibilities}} to get that each fiber has an action of a group locally isomorphic to $\R^2$, $\Heis$ or a cover of $PSL(2,\R)$, but that the group and action may, a priori, depend on the basepoint.

Now, simply notice again that $\mc C_{\alpha,-c\alpha}(ax) = \psi_a(\mc C_{\alpha,-c\alpha}(x))$ by \eqref{eq:renormalization} and definition of $\psi_a$ (see Definition \ref{def:P-groups}) for every $a$. But since $\psi_a$ descends to automorphisms of $\R^2$, $\Heis$ and any cover of $SL(2,\R)$, we conclude that $\psi_a(\mc C_{\alpha,-c\alpha}(x)) = \mc C_{\alpha,-c\alpha}(x)$, and therefore the result.
\end{proof}

\section{Homogeneity from Pairwise Cycle Structures}
\label{sec:pairwise-sufficient}
Assuming constant commutator relations,  we now construct a homogeneous structure of a Lie group for totally Cartan actions.  %By the Gleason-Palais Theorem \ref{lem:gleason-palais}, we will  show constant cycle relations  for certain canonical subgroups corresponding to stable leaves for elements of the action. This will allow us to rearrange a general cycle to one in cyclical order, using the commutator relations and special ``symplectic'' relations for negatively proportional weights.  % When all weights are stable for a fixed element, this is easy.  Passing through symplectic relations which generate $SL(2, \R)$,
Propositions \ref{prop:base base relations} and \ref{prop:symplectic-possibilities}  give partial algebraic structures corresponding to certain dynamically defined objects (commutators in stable manifolds and symplectic sub-accessibility classes, respectively). We carefully piece together such structures to build one on $X$. In particular, we  use specific, computable relations between the flows of negatively proportional weights provided by the classification in Proposition \ref{prop:symplectic-possibilities}. We will see that such relations yield a canonical presentation of paths in an open neighborhood of the identity in a group $\hat{\mc P}$ (recall Definition \ref{def:P-groups}).  This procedure is analagous to a K-theoretic argument that has appeared in works on local rigidity \cite{DamjanovicKatok2011}, and was generalized to other settings in \cite{MR2672298,vinhageJMD2015,Vinhage:2015aa}.

Our approach is to find canonical presentations for words in $\hat{\mc P}$ using only commutator relations and symplectic relations which we assume are constant and well-defined. By fixing a regular element, we will be able to use such relations to rearrange the terms in an open set of words to write them using only stable legs, then only unstable legs, then the action (Proposition \ref{stable unstable cycle decomposition}). This will imply that the quotient group of $\hat{\mc P}$ by the commutator  is locally Euclidean, which allows for the application of classical Lie criteria. The core of the approach is Lemma \ref{open dense commutation}, which gives the ability to commute stable and unstable paths.

\begin{definition}
\label{def:const-pairwise}
Fix a topological Cartan action $\R^k \curvearrowright X$. We say that the action has {\normalfont constant pairwise cycle structure} if 

\begin{enumerate}[label=(CPCS\arabic*)]
\item \label{CPCS1} for each nonproportional pair $\alpha,\beta \in \Delta$, {  $\gamma \in D(\alpha,\beta)$,} and fixed $s,t \in \R$, $\rho^{\alpha,\beta}_\gamma(s,t,x)$ (as defined in Lemma \ref{lem:comm-relation}) is independent of $x$ and 
\item \label{CPCS2} for each $\alpha \in \Delta$ such that $-c\alpha \in \Delta$ for some $c > 0$, the action of $\mc P_{\set{\alpha,-c\alpha}}$ factors through a Lie group action.%$\mc C_{\set{\alpha,-c\alpha}}(x)$ is independent of $x$ and the quotient $\mc P_{\set{\alpha,-c\alpha}} / \mc C_{\set{\alpha,-c\alpha}}(x)$ is a Lie group.
\end{enumerate}
\end{definition}

The main goal of this section is to prove the following.

\begin{theorem}
\label{thm:pairwise-to-homo}
If $\R^k \curvearrowright X$ is a topological Cartan action with constant pairwise cycle structure, then the action is topologically conjugate to a homogeneous action.
\end{theorem}

{
\begin{remark}
Theorem \ref{thm:pairwise-to-homo} does {\it not} require the higher-rank assumptions \ref{HR:a} and \ref{HR:b}. In particular, it still holds in the presence of rank-one factors. Notice that \ref{CPCS1} and \ref{CPCS2} are implied by \ref{HR:a} and \ref{HR:b}, by  Lemma \ref{lem:symplectic} and Proposition \ref{prop:base base relations}.
\end{remark}
}

\subsection{Stable-unstable-neutral presentations}
\label{subsec:sun-presentations}

%We do this by first showing that the geometric brackets are constant, ie, that $\rho^{\alpha,\beta}(s,t,x)$ is independent of $x$, and applying arguments from generators and relations of Lie groups adapted to this dynamical setting to show that this is sufficient. Throughout this section, we assume that we work on the space $X^E$. We will sometimes emphasize this by referring to such factor actions as {\it maximal factor actions}.

For the remainder of Section \ref{sec:pairwise-sufficient}, we will assume that $\R^k \curvearrowright X$ is a topological Cartan action with constant pairwise cycle structures (ie, that conditions \ref{CPCS1} and \ref{CPCS2} hold). We let $\rho^{\alpha,\beta}_\gamma(s,t)$ denote the common value of $\rho^{\alpha,\beta}_\gamma(s,t,x)$ for $\alpha,\beta \in \Delta$ and $\gamma \in D(\alpha,\beta)$ (which is guaranteed to be constant by constant pairwise cycle structures). %Throughout this section, we work on the space $X^E$ for a fixed $E$, and 
Recall Definition \ref{def:sub-paths-spaces}, and let $\mc P = \mc P_{\Delta}$. 

\begin{definition}\label{def:scriptG} Let $\mc C'$ be the smallest closed normal subgroup containing all cycles of the form $[s^{(\alpha)},t^{(\beta)}] * \rho^{\alpha,\beta}(s,t)$ as described in Lemma \ref{lem:comm-relation} and any element of $\mc P_{\set{\alpha,-c\alpha}}$ which factors through the identity of the Lie group action provided by  \ref{CPCS2}. Since such cycles are cycles at every point by assumption, $\mc C' \subset \mc C(x)$ for every $x \in X$ and $\mc C'$ is normal. Let $\mc G = \mc P / \mc C'$ be the quotient of $\mc P$ by $\mc C'$.
\end{definition}

%{\color{olive}
%Our goal will be to show that $\mc G$ is a Lie group, so that $X$ is a homogeneous space of a Lie group. Then since the Cartan action acts by automorphisms on the generating subgroups of $\mc G$, we will conclude that the action is affine. We will show that $\mc G$ is a Lie group by showing it is locally path-connected by showing that the relations in $\mc C'$ allow group elements of $\mc P$ to be put in a normal form. The normal form will give a Lie factor $H$ of $\mc P$ with the same structure in which the normal form is locally unique. Then we will conclude that $H \cong \mc G$, and $\mc G$ is Lie, as desired.

%}

Fix $a_0 \in \R^k$, a regular element. The goal of this subsection is to show that any $\rho \in \mc G$ can be reduced (via the relations in $\mc C'$) to some $\rho_+ * \rho_- * \rho_0$, with $\rho_+$ having only terms $t^{(\chi)}$ with $\chi(a_0) > 0$, $\rho_-$ having only terms $t^{(\chi)}$ with $\chi(a_0) < 0$, and $\rho_0$ being products of neutral elements generated by symplectic pairs (see Definition \ref{def:stand-gens}). {We will do so by first embedding $\mc G$ in a larger group that incorporates the Cartan action, and proving this decomposition for the larger group (see Proposition \ref{stable unstable cycle decomposition}).} We begin by identifying well-behaved subgroups of $\mc G$. Given a subset $\Omega \subset \Delta$, let $\mc G_\Omega$ denote the subgroup of $\mc G$ generated by the flows $\eta^\chi$, $\chi \in \Omega$. {Call $\Omega$ a {\it closed stable subset} if there exists $a_0 \in \R^k$ such that $\alpha(a_0) < 0$ for all $\alpha \in \Omega$, and if $\alpha_1,\alpha_2 \in \Omega$, then $D(\alpha_1,\alpha_2) \subset \Omega$.}

\begin{lemma}
\label{lem:stable-lie-factor}
If {$\Omega$ is a closed stable subset of weights}, then $\mc G_{\Omega}$ is a nilpotent Lie group.
\end{lemma}

\begin{proof}
Write $\Omega = \set{\beta_1,\beta_2,\dots,\beta_r}$ in a circular ordering, so that $[\beta_i,\beta_j] \subset \set{\beta_{i+1},\dots,\beta_{j-1}}$. %Let $\mc G_{D(\alpha,\beta)}$ denote the factor of the group $\mc P_{D(\alpha,\beta)}$ modulo the commutator relations \eqref{eq:rho-characterization} (which are independent of $x$ since we have assumed pairwise constant cycle structures). 
We first claim that every $\rho \in \mc G_{\Omega}$ can be written as

\begin{equation}
\label{eq:nil-presentation}u_1^{(\beta_1)} * \dots * u_r^{(\beta_r)}
\end{equation}

\noindent and that the elements $u_i$ are unique. Indeed, any $\rho \in \mc G_{\Omega}$ can be written as $\rho = v_1^{(\beta_{i_1})} * \dots * v_k^{(\beta_{i_k})}$. We may begin by pushing all of the terms from the $\beta_1$ component to the left. {We do this by looking at the first term to appear with $\beta_1$, call it $v_j^{(\beta_{i_j})} = v_j^{(\beta_1)}$. We may commute it with $v_{j-1}^{(\beta_{i_{j-1}})}$ using the commutator relations which are constant from \ref{CPCS1}. In doing so, we may accumulate some $\rho(v_j^{\beta_1},v_{j-1}^{(\beta_{i_{j-1}})})$ which consists of terms without $\beta_1$, since we have quotiented by the commutator relations \eqref{eq:rho-characterization} (see Definition \ref{def:scriptG}). We may continue to push each $\beta_1$-leg to the left of the expression, so that in $\mc G$, $\rho$ is equal to $u_1^{(\beta_1)} * \rho'$, where $\rho'$ consists only of terms without $\beta_1$.}

We now proceed inductively. We may in the same way push all $\beta_2$ terms to the left. Notice now that each time we pass through, the ``commutator'' $\rho(u^{\beta_2},v^{\beta_j})$, $j \ge 3$ has no $\beta_1$ or $\beta_2$ terms.  Iterating this process yields the desired presentation of $\rho$.

Uniqueness will follow from an argument similar to Lemma \ref{lem:extending-charts}. %{\color{olive} and Remark \ref{rem:rho-ambiguity}}.
Suppose that $u_1^{(\beta_1)} * \dots * u_r^{(\beta_r)} = v_1^{(\beta_1)} * \dots * v_r^{(\beta_r)}$, so that

\[ u_1^{(\beta_1)} * \dots * u_r^{(\beta_r)} * (-v_r)^{(\beta_r)} * \dots * (-v_1)^{(\beta_1)} \]

\noindent stabilizes every point of $X$. Picking some $a \in \ker \beta_r$ such that $\beta_i(a) < 0$ for all $i = 1,\dots,r-1$ implies that

\[ (e^{\beta_1(a)}u_1)^{(\beta_1)} * \dots * u_r^{(\beta_r)} * (-v_r)^{(\beta_r)} * \dots * (-e^{\beta_1}(a)v_1)^{(\beta_1)} \]

\noindent also stabilizes every point of $X$. Letting $a \to \infty$ implies that $(u_r-v_r)^{(\beta_r)}$ stabilizes every point of $X$. If $u_r \not= v_r$, applying an element which contracts and expands $\beta_r$ yields an open set of $t$ for which $t^{(\beta_r)}$ stabilizes every point of $X$. But this implies that $\eta^{\beta_r}$ is a trivial flow, a contradiction, so we conclude that $u_r = v_r$. This allows cancellation of the innermost term, and one may inductively conclude that $u_i = v_i$ for every $i= 1,2,\dots,r$.

Thus, every element of $\mc G_{\Omega}$ has a unique presentation of the form \eqref{eq:nil-presentation}. The map which assigns an element $\rho$ to such a presentation gives a an injective map from $\mc G$ to $\R^r$. By Lemma \ref{lem:continuity-criterion}, it will be continuous once its lift to $\mc P_{\Omega}$ is continuous. In each combinatorial cell $C_{\overline{\beta}}$, the map is given by composition of addition of the cell coordinates and functions $\rho^{\alpha,\beta}(\cdot,\cdot)$ evaluated on cell coordinates, which are continuous. Therefore, the lift is continuous, so the map from $\mc G_{\Omega}$ is continuous. 

So there is an injective continuous map from $\mc G_{\Omega}$ to a finite-dimensional space, and $\mc G_{\Omega}$ is a Lie group by Theorem \ref{lem:gleason-palais}. Fix $a$ which contracts every $\beta_i$. The fact that $\mc G_\Omega$ is nilpotent follows from Lemma \ref{lem:group-integrality}.
\end{proof}

Let $\chi \in \Delta$ be a weight such that $-c\chi \in \Delta$ for some $c$, and $\beta \in \Delta$ be any linearly independent weight. Let $\Omega = \set{ t\beta + s \chi : t \ge 0, s \in \R} \cap \Delta$, and $\Omega' = \set{ t\beta + s\chi : t > 0, s\in \R} \cap \Delta = \Omega \setminus \set{\chi,-c\chi}$.

\begin{proposition}
\label{prop:semi-structure}
If $\Omega$ is as above and $\rho \in \mc G_\Omega$ is any element, then $\rho = \rho_\chi * \rho_{\Omega'}$, where $\rho_\chi \in \mc G_{\set{\chi,-c\chi}}$, $\rho_{\Omega'} \in \mc G_{\Omega'}$. Furthermore, such a decomposition is unique.
\end{proposition}

\begin{proof}
The proof technique is the same as that of Lemma \ref{lem:stable-lie-factor}. Using constancy of commutator relations, we may push any elements $u^{(\chi)}$ or $u^{(-c\chi)}$ to the left, accumulating elements of $\mc G_{\Omega'}$ as the commutator, as well as cycles on the right (on the right because they are cycles at every point, so we may conjugate them by whatever appears to their right). If $\rho_\chi * \rho_{\Omega'} = \rho_\chi' * \rho_{\Omega'}'$, then $(\rho_\chi')^{-1} * \rho_\chi = \rho_{\Omega'}'  * \rho_{\Omega'}^{-1}$. But $\mc G_{\Omega'}$ is a a subgroup of $\mc G_\Omega$ and it is clear that $\mc G_{\Omega'} \cap \mc G_{\set{\chi,-c\chi}} = \set{e}$. Therefore, $\rho_\chi' = \rho_\chi$ and $\rho_{\Omega'} = \rho_{\Omega'}'$, and the decomposition is unique.
\end{proof}

\begin{corollary}
\label{cor:semi-stabilizer}
If $\Omega$ is as above, $\mc G_\Omega$ is a Lie group. Furthermore, $\mc G_\Omega$ has the semidirect product structure $\mc G_{\set{\chi,-c\chi}} \ltimes \mc G_{\Omega'}$, with $\mc G_{\set{\chi,-c\chi}}$ locally isomorphic to $\R^2$, the Heisenberg group, or (some cover of) $PSL(2,\R)$, and $\mc G_{\Omega'}$ a nilpotent group.
\end{corollary}

\begin{proof}
Notice that in the proof of Proposition \ref{prop:semi-structure}, we get a unique expression by moving the elements of $\mc G_{\set{\chi,-c\chi}}$ to the left, and doing so changes only the $\mc G_{\Omega'}$ element. Therefore, the decomposition gives $\mc G_\Omega$ the structure of a semidirect product of $\mc G_{\set{\chi,-c\chi}}$ and $\mc G_{\Omega'}$. The action of $\mc G_{\set{\chi,-c\chi}}$ on $\mc G_{\Omega'}$ is continuous since the action of its generating subgroups corresponding to $\chi$ and $-c\chi$ are given by commutators, which are continuous. Therefore, $\mc G_\Omega$ is the semidirect product of the Lie group $\mc G_{\set{\chi,-\chi}}$ with the Lie group $\mc G_{\Omega'}$, with a continuous representation, and is hence a Lie group. The possibilities for $\mc G_{\set{\chi,-c\chi}}$ follow from Proposition \ref{prop:symplectic-possibilities}.
%We show that if $\sigma \in \mc G_\Omega$ fixes $x$, then $\sigma = \sigma_\chi * \sigma_{\Omega'}$, where $\sigma_\chi$ is a cycle in $\mc G_{\set{\chi,-c\chi}}$ and $\sigma_{\Omega'} \in \mc G_{\Omega'}$. By Proposition \ref{prop:semi-structure}, it can always be written as such with $\sigma_\chi$, $\sigma_{\Omega'}$ not necessarily being cycles, instead only paths. Assume for contradiction that $\sigma_\chi$ is not a cycle. Pick $a$ such that $a \in \ker \chi$ and $\beta(a) < 0$, so that $\lambda(a) < 0$ for all $\lambda \in \Omega'$. Then on the one hand, $a^n \sigma$ is a cycle at $a^n x$ for every $n$, but it becomes closer to $\sigma_\chi \cdot (a^n x)$ since $a \in \ker \chi$. By choosing a convergent subsquence we obtain that $\sigma_\chi$ is a cycle at some point and hence also at $x$ (since the cycles in $\set{\chi,-c\chi}$ are independent of $x$. Hence $\sigma_\chi$ is a cycle and so it $\sigma_{\Omega'}$. Therefore, if $\sigma$ fixes $x$, $\sigma = e \in \mc G_\Omega$. Then the map which associates some $\rho \in \mc G_\omega$ an element of the form $\rho_\chi * \rho_{\Omega'}$ shows that $\mc G_\Omega$ is a Lie group by Theorem \ref{lem:gleason-palais}. The semidirect product structure is clear from the proof of Proposition \ref{prop:semi-structure}.
\end{proof}

The crucial tool in showing that $\mc P / \mc C'$ is Lie is to show that it is locally Euclidean. To that end, the crucial result is Lemma \ref{open dense commutation}.
{Fix an Anosov element $a_0 \in \R^k$, and recall that}

\[ \Delta^+(a_0) = \set{\chi \in \Delta : \chi(a_0) > 0}, \qquad \Delta^-(a_0) = \set{ \chi \in \Delta : \chi(a_0) < 0},\] 

\noindent and set $\mc G_+ = \mc G_{\Delta^+(a_0)}$ and $\mc G_- = \mc G_{\Delta^-(a_0)}$.
Consider a $\mc{G}_{\set{\chi, -c\chi}}$ locally isomorphic to $\Heis$ or $SL(2, \R)$.  Let $\mc{D}_{\chi} $ be the subgroup of $\mc{G}_{\{\chi, -c\chi\}}$ generated by the neutral element corresponding to the generators $\eta^\chi$ and $\eta^{-\chi}$ (see Definition \ref{def:stand-gens}). %which normalizes both $\eta ^{\chi}$ and 
%$\eta^{-c \chi}$. Thus, $\mc D_\chi$ corresponds to either the diagonal subgroup of $SL(2,\R)$ (together with any possible central elements), or the center of $\Heis$.
Let $\mc{D} \subset \mc{G}$ be the group generated by all such $\mc{D}_{\chi} $ as a subgroup of $\mc G$. Note that the action of each element of $\mc D$ coincides with the action of some element of the $\R^k$ action by Proposition \ref{prop:symplectic-possibilities}. However, the action of $\mc D$ may not be faithful, as is the case for the Weyl chamber flow on $SL(3,\R)$ where there are 3 symplectic pairs of weights, each generating one-parameter subgroups of the diagonal subgroup, which is isomorphic to  $\R^2$.

\begin{lemma}
\label{lem:CD-normal}
Suppose the $\R^k$ orbit of $x_0$ is dense.  Then 
if $d \in \mc{D}$ is a cycle at $x_0$, then $d$ is a cycle everywhere. 
\end {lemma} 

\begin{proof}
{  The action of each generator of $\mc D$ coincides with that of some element of the Cartan action by definition of $\mc D_\chi$ and $\mc D$}, therefore the action of $\mc D$ commutes with the Cartan action. Then if $d$ is a cycle at  $x_0$, $d$ is a cycle at any point in $\R ^k (x_0)$, hence everywhere as  the $\R^k$ orbit of $x_0$ is dense.  
%Let $g \in \mc{G}_{\{\chi, -c\chi\}}$ and $a \in \R^k$.  Then $a g = \phi _a (g)$ for some automorphism $\phi _a$ of $\mc{G}_{\{\chi, -c\chi\}}$ in the connected component of this automorphism group. Hence $\phi _a$ is an inner automorphism. Note that $\R^k$ and hence  $\phi _a$ normalize $\eta ^{\chi}$, and $\eta ^{-c \chi}$.  
% Hence $\phi _a (g) =d_0gd_0^{-1}$ for some $d_0 \in \mc{D}_{\chi}$.   As $\mc{D}_{\chi}$ is abelian,  $\phi _a (d) =d$ for $d \in \mc{D}_{\chi}$ and in fact for all $d \in \mc{D}$.  Hence  $a d = d a$.  Hence if $d$ is  a cycle at $x_0$, $d$ is a cycle at any point in $\R ^k (x_0)$, hence everywhere as $x_0$ is generic for the $\R ^k$ action.  
 \end{proof}

 Denote by $\mc{C}_D$ the group of cycles in $\mc{D}$ at a point $x_0$ with a dense $\R^k$-orbit,  and set $G = \mc{G}/\mc{C}_D$ and $D= \mc{D}/\mc{C}_D$. Notice that $G$ and $D$ are topological groups since $\mc C_D$ is a closed normal subgroup by Lemma \ref{lem:CD-normal}. Furthermore, let $G_\pm$ denote the projections of the groups $\mc G_\pm$ to $G$. The following lemma is immediate from the fact that the action of $D$ coincides with {some subgroup} of the Cartan action and Corollary \ref{cor:semi-stabilizer}:
 
% \begin{lemma}
%$D$ is a subgroup of $\R ^k$.
%\end {lemma} 

%\begin{proof}
%For $d \in D$, $d x \in \R ^k (x)$.  
%\end{proof}

 \begin{lemma}
 \label{lem:D-commutators}
 Let $\rho ^{\pm} \in G _{\pm}$ and $\rho ^0 \in D$.  Then 
$$\rho ^0 \rho ^{\pm} ( \rho ^0) ^{-1}\in G_{\pm}.$$
 \end{lemma} 
 
 {We require the following computational lemma in the subsequent result:

\begin{lemma} \label{horocycle commutators}
Let $H$ cover $PSL(2, \R)$ and  denote by $h^+ _s$ and $h^- _s$ the unipotent flows in $H$ covering $\begin{pmatrix} 1 & s \\ 0 & 1 \end{pmatrix}$ and $\begin{pmatrix} 1 & 0 \\t & 1 \end{pmatrix} $ respectively.   Also let $a_t = \begin{pmatrix} t & 0 \\ 0 & \frac{1}{t} \end{pmatrix}$.  

If $s  t \neq -1$ then  
$$ h^+ _s   h ^-  _t = h^- _{\frac{t}{1+st}} h^+ _{s(1+st)} a_{1+st} z(s,t)$$
where $z(s,t) $ belongs to the center of $H$. If $s$ and $t$ are sufficiently small, $z(s,t) = e$.
\end{lemma}

\begin{proof} If $s  t \neq -1$ then 
\[\begin{pmatrix} 1 & s \\ 0 & 1 \end{pmatrix} \begin{pmatrix} 1 & 0 \\t & 1 \end{pmatrix} = \begin{pmatrix}1 & 0 \\ t/(1+st) & 1 \end{pmatrix} \begin{pmatrix} 1 & s(1+st) \\0 & 1 \end{pmatrix} \begin{pmatrix} 1 + st & 0 \\ 0 & 1/(1+st) \end{pmatrix} .\]
Therefore the claim holds in $PSL(2,\R)$, and thus in $H$ up to some central element $z(s,t)$.  If $s$ and $t$ are sufficiently small, all calculations take place in a neighborhood of $e$, so the central element is $z(s,t) = e$.
\end{proof} 

}

  \begin{lemma} \label{open dense commutation}
 For an open set of elements  $\rho ^+ \in G _+$, $\rho ^-  \in G _{-}$ containing $\set{e} \times \set{e}$ there exist 
 $(\rho ^+) ' \in G_+$, $(\rho ^- )' \in G_-$
 and $\rho ^0 \in D$ such that 
 $$ \rho ^+ * \rho ^-  = (\rho ^-) ' * (\rho ^+ )'  * \rho ^0 .$$

\noindent Furthermore, $(\rho^+)'$, $(\rho^-)'$ and $\rho^0$ depend continuously on $\rho^+$ and $\rho^-$.
  \end{lemma}

 \begin{proof}
Order the weights of $\Delta^+(a_0)$ and $\Delta^-(a_0)$ using a fixed circular ordering as $\Delta^+(a_0) = \set{\alpha_1,\dots,\alpha_n}$ and $\Delta^-(a_0) = \set{\beta_1,\dots,\beta_m}$. Since $\Delta^+(a_0)$ is a {closed} stable subset, $G_+$ is a nilpotent group. Therefore, we may write $\rho_+ = t_n^{(\alpha_n)} * \dots * t_1^{(\alpha_1)}$ for some $t_1,\dots,t_n \in \R$ {by Lemma \ref{lem:stable-lie-factor}}. 
We will inductively show that we may write the product $\rho^+ * \rho^-$ as $t_n^{(\alpha_n)} * \dots * t_k^{(\alpha_k)} * (\rho^-)' * s_{k-1}^{(\alpha_{k-1})} * \dots * s_1^{(\alpha_1)} * \rho^0$ for some $t_i,s_i \in \R$, $(\rho^-)' \in \mc G_-$ and $\rho^0 \in D$ (all of which depend on $k$). Our given expression is the base case $k = 1$.
 
 Suppose we have this for $k$. If $-c_k\alpha_k \in \Delta$, then it must be in $\Delta_-$. Let $l(k)$  denote the index for which $\beta_{l(k)} = -c_k \alpha_k$  if $- c_k \alpha _k$  is a weight.  Otherwise, since there is no weight negatively proportional to $\alpha_k$, we set $\beta_{l(k)} = -\alpha_k$ with $l(k)$ a  half integer 
  so that $-\alpha_k$ appears between $\beta_{l(k) - 1/2}$ and $\beta_{l(k) + 1/2}$. Then decompose $\Delta$ into six (possibly empty) subsets: $\set{\alpha_k}$, $\set{-c_k\alpha_k}$, $\Delta_1 = \set{ \alpha_l : l < k}$, $\Delta_2 = \set{\alpha_l : l > k}$, $\Delta_3 =  \set{ \beta_l : l < l(k)}$ and $\Delta_4 = \set{\beta_l : l > l(k)}$. See Figure \ref{fig:quadrants}.

\begin{figure}[!ht]
\begin{center}
\begin{tikzpicture}[scale=.75]
\draw [thick,<->] (-4,-2) --(4,2);
\draw [ultra thick,blue,dashed] (0,4) --(0,-4);
\node [right] at (4,2) {$\alpha_k$}; 
\node [left] at (-4,-2) {$-c_k\alpha_k$};
\node [above right] at (0,4) {$\set{\chi : \chi(a_0) = 0}$};
\node at (2,-1) {$\Delta_1$};
\node at (2,2) {$\Delta_2$};
\node at (-2,1) {$\Delta_3$};
\node at (-2,-2) {$\Delta_4$};
\end{tikzpicture}
\caption{Decomposing $(\R^k)^*$ into quadrants}
\label{fig:quadrants}
\end{center}
\end{figure}

We let $G_{\Delta_i}$ denote the subgroup of $G$ generated by the $\Delta_i$. Notice that $\Delta_- = \Delta_3 \cup \set{-c_k\alpha_k} \cup \Delta_4$ (with $\set{-c_k\alpha_k}$ omitted if there is no weight of this form) is stable, so again, since $G_-$ is nilpotent, $(\rho^-)' $ may be expressed uniquely as $q_3 * u^{(-c_k\alpha_k)} * q_4$ with $q_3 \in G _{\Delta _3}$ and $q_4 \in G _{\Delta _4}$
    (if $-c_k\alpha_k$ is not a weight, we omit this term). Now, $\set{\alpha_k} \cup \Delta_2 \cup \Delta_3$ is a closed stable subset, {since we may choose an open half-space which contains $\alpha_k$, no element of $\Delta_1$ and all elements of $\Delta_2,\Delta_3$ after perturbing the open half space cut from $(\R^2)^* \setminus \R\alpha$ containing $\Delta_2$ and $\Delta_3$}. Therefore, its associated group is nilpotent. So $t_k^{(\alpha_k)} * q_3 = q_2 * (q_3)' * s^{(\alpha_k)}$ for some $q_2 \in G_{\Delta_2}$, $(q_3)' \in G_{\Delta_3}$ and $s \in \R$. Notice that by iterating some $a \in \R^k$ for which $\alpha_k(a) = 0$, and $\beta(a) < 0$ for all $\beta \in \Delta_2 \cup \Delta_3$, we actually know that $s = t_k$. %The term $s_k'$ is determined by a polynomial in $t_k$ and the lengths of the legs of $q_3$ by Lemma \ref{lem:group-integrality}(3). In particular, properties which hold for an open dense set of the $s_k'$ hold for an open dense set of the $t_k$. 
Thus, we have put our expression in the form:
 
\begin{multline*}
t_n^{(\alpha_n)} * \dots * t_k^{(\alpha_k)} * (\rho^-)' * s_{k-1}^{(\alpha_{k-1})} * \dots * s_1^{(\alpha_1)} * \rho^0  \\=  t_n^{(\alpha_n)} * \dots * t_k^{(\alpha_k)} * (q_3 * u^{(-c_k\alpha_k)} * q_4) * s_{k-1}^{(\alpha_{k-1})} * \dots * s_1^{(\alpha_1)} * \rho^0\\
  =  t_n^{(\alpha_n)} * \dots * t_{k+1}^{(\alpha_{k+1})} * q_2 * (q_3)' * t_k^{(\alpha_k)} * u^{(-c_k\alpha_k)} * q_4 * s_{k-1}^{(\alpha_{k-1})} * \dots * s_1^{(\alpha_1)} * \rho^0.
\end{multline*}
 
 Now, there are three cases: there is no weight of the form $- c_k \alpha_k$ in which case $u ^{(- c_k \alpha_k)}$ does not exist.  Or this weight exists but commutes with $t_k^{(\alpha_k)}$ and we simply switch the order.  Otherwise,  $\alpha_k$ and $- c_k \alpha_k$ generate a cover of $PSL(2, \R)$ or $\Heis$.  In the case of $PSL(2,\R)$, 
  if $t_k$ and $u$ are sufficiently small, which is an open property, we may use Lemma \ref{horocycle commutators} to commute $t_k^{(\alpha_k)}$ and $u^{(-c_k\alpha_k)}$ up to an element of $D$. In the case of $\Heis$, we use the standard commutator relations in $\Heis$ to commute the elements up to an element of $D$. Furthermore, notice that $q_2$ and the terms appearing before $q_2$ all belong to $\Delta_2$, so we may combine them to reduce the expression to:
 
 \[ (t_n')^{(\alpha_n)} * \dots * (t_{k+1}')^{(\alpha_{k+1})} * (q_3)' * (u')^{(-c_k\alpha_k)} * (s_k')^{(\alpha_k)} * d * q_4 * s_{k-1}^{(\alpha_{k-1})} * \dots * s_1^{(\alpha_1)} * \rho^0 \]
  for some collection of $t_i',u',s_k' \in \R$, and $d \in D$. But by Lemma \ref{lem:D-commutators}, $d$ may be pushed to the right preserving the form of the expression and being absorbed into $\rho^0$. {We abuse notation,} do not change these terms and drop $d$ from the expression.
 
 Now, we do the final commutation by commuting $(s_k')^{(\alpha_k)}$ and $q_4$. Notice that $\set{\alpha_k} \cup \Delta_1 \cup \Delta_4$ is a closed stable subset {(for the same reason that $\set{\alpha_k} \cup \Delta_2 \cup \Delta_3$ is a closed stable subset)}. Therefore, we may write $(s_k')^{(\alpha_k)} * q_4$ as $(q_4)' * (s_k'')^{(\alpha_k)} * q_1$ with $q_1 \in G_{\Delta _1}, (q_4)' \in G _{\Delta _4}$ and $s_k '' \in \R$. Inserting this into the previous expression, we see that the $q_1$ term can be absorbed into the remaining product of the $s_i^{(\alpha_i)}$ terms. This yields the desired form.
 \end{proof}

Recall each the action of each element $D$ coincides with an element of the Cartan action. { Then if $\tilde{D}$ is the universal cover of $D$, there exists a homomorphism  $f : \tilde{D} \to \R^k$ which associates an element of $\tilde{D}$ with the corresponding element of the $\R^k$ action.} Then let $\hat{G}$ be the quotient of $\hat{\mc P}$ (see Definition \ref{def:P-groups}) by the group generated by $\ker(\mc P \to G)$ and elements of the form $f(d)\pi_{\tilde{D}}(d)^{-1}$, { where $d \in \tilde{D}$ and $\pi_{\tilde{D}} : \tilde{D} \to D$ is the covering homomorphism}. {Note that $\R^k$  naturally embeds as a subgroup of $\hat{G}$, and that the restriction of the $\hat{G}$-action to this copy of $\R^k$ is exactly the totally Cartan action}.

Recall that $\mc P$ is the free product of copies of $\R$, and has a canonical CW-complex structure as described in Section \ref{subsec:free-prods}. The cell structure can be seen by considering subcomplexes corresponding to a sequence of weights $\bar{\chi}= (\chi_1,\dots,\chi_n)$ and letting $C_{\bar{\chi}} = \set{ t_1^{(\chi_1)} * \dots * t_n^{(\chi_n)} : t_i \in \R} \cong \R^n$. Then a neighborhood of the identity is a union of neighborhoods in each cell $C_{\bar{\chi}}$ containing 0.

%Let $D^\perp$ be a subspace of $\R^k$ such that $D^\perp \cap D = \emptyset$ and $D$ and $D^\perp$ generate the entire $\R^k$ action. Notice that since $D^\perp$ also normalizes each of the flows $\eta^\chi$, and also maps cycles to cycles, the group generated by the actions of $G$ and $D^\perp$ has the structure that any element can be written in the form $gd$ with $d \in D^\perp$. The main result of this section is the following  proposition. Let $\hat{G}$ be this group.

\begin{proposition} \label{stable unstable cycle decomposition}
There exists an open neighborhood $U$ of $e \in \hat{\mc P}$ and a continuous map $\Phi : U \to G_+ \times G_- \times \R^k$ such that $\pi(\Phi(u)) = \pi(u)$, where $\pi : \hat{\mc P} \to \hat{G}$ is the canonical projection. % (ie, any cycle $\rho$ with contractible homotopy type),% there exist $\rho _k \rightarrow \rho$ in $G$ such that 
%where $\rho ^{\pm} \in G _{\pm}$ and $ \rho^0 \in \R^k$ are unique.  Moreover,  $\rho  ^{\pm} $ and $\rho ^0 $ vary continuously on the open dense set and extend continuously to any $\rho \in \hat{G}$ such that $\eta (\rho) (x_0) = x_0$ on the universal cover . 
\end{proposition}

\begin{proof}
We describe the map $\Phi$, whose domain will become clear from the definition.  Let $\Delta^+(a_0) = \set{\alpha_1,\dots,\alpha_n}$ and $\Delta^-(a_0) = \set{\beta_1,\dots,\beta_m}$ be the weights as described in the proof of Lemma \ref{open dense commutation}. Given a word $\rho = t_1^{(\chi_1)} * \dots * t_n^{(\chi_n)}$, we begin by taking all occurrences of $\alpha_n$ in $\rho$ and pushing them to the left, starting with the leftmost term.  When we commute it past another $\alpha_i$, we accumulate only other $\alpha_j$, $i+1 \le j \le n-1$, in $\rho^{\alpha_i,\alpha_n}$, which we may canonically present in increasing order on the right of the commutation. A similar statement holds for the commutation of $\alpha_n$ with $\beta_i$. We iterate this procedure as in the proof of Lemma \ref{open dense commutation} to obtain the desired presentation. Since the commutation operations involved are determined by the combinatorial type, the resulting presentation is continuous from the cell $C_{\bar{\chi}}$. Furthermore, if one of the terms $t_i^{(\chi_i)}$ happens to be 0, the procedure yields the same result whether it is considered there or not. Thus, it is a well-defined continuous map from $\mc P \to G_+ \times G_- \times \R^k$ (it is continuous from $\hat{\mc P}$ because it is continuous from each $C_{\bar{\chi}}$).

Notice that in the application of Lemma \ref{open dense commutation}, we require that $st \not= 1$ whenever we try to pass $t^{(\chi)}$ by $s^{(-\chi)}$. This is possible if $s$ and $t$ are sufficiently small. Thus, in each combinatorial pattern, since the algorithm is guaranteed to have a finite number of steps and swaps appearing, and each term appearing will depend continuously on the initial values of $t_i^{(\chi_i)}$, we know that for each $\bar{\chi}$, some neighborhood of 0 will be in the domain of $\Phi$, by the neighborhood structure described above.

Notice that the reduction of a word $u$ to a word of the form $u_+ * u_- * a \in G_+ \times G_- \times \R^k$ uses only relations in $\hat{G}$. Therefore, if after the reductions, the same form is obtained, the original words must represent the same element of $\hat G$.
%The existence follows from Lemma \ref{open dense commutation}.  Uniqueness is argued by the dynamics: indeed,   if there was a second presentation with different  $\rho^-$ value, then moving them by $a _0 ^n$,  these would diverge while the other components contract or do not expand.  A similar argument gives uniqueness of $\rho^+$.  Uniqueness of the $\rho^0$ follows. Notice also that the map $(g_+,g_-,a) \mapsto g_+g_-a$ from $G_+ \times G_- \times \R^k$ locally surjective. In particular, by uniqueness, if a elements of the form $\rho^+ * \rho^- * \rho^0$ converge to a cycle, each individual term making up the product converges to a cycle.
 \end{proof}
 
{
\begin{corollary}
\label{cor:GhatLie}
$\hat{G}$ is a Lie group, and acts locally freely on $X$.
\end{corollary}

\begin{proof}
We claim that if $V_+$, $V_-$ and $V_0$ are sufficiently small neighborhoods of $G_+$, $G_-$ and $\R^k$, respectively, then the map $\sigma : U_+ \times V_- \times V_0 \to \hat{G}$ is a homeomorphism onto its image, and that its image is a neighborhood of $e \in \hat{G}$. To see that it is onto a neighborhood, notice that by Proposition \ref{stable unstable cycle decomposition}, there exists a neighborhood $U$ of $\hat{\mc P}$ whose elements can be reduced to this form. Then $\pi(U)$ is contained in the image of $\sigma$, and open in $\hat{G}$ by definition of the quotient topology. 

To see that $\sigma$ is injective, notice that by condition \ref{tc6}, if $v_+v_-v_0 \cdot x = v_+'v_-'v_0' \cdot x$, then $v_+ = v_+'$, $v_- = v_-'$ and $v_0 = v_0'$. Therefore, $\sigma$ is injective, since if $\sigma(v_+,v_-,v_0) = \sigma(v_+,v_-,v_0)$, then the images in $\hat{G}$ act in the same way on $X$.

Once $\sigma$ is injective, it is a homeomorphism onto its image when restricted to compact sets. Choosing such a compact set with interior in $V_+ \times V_- \times V_0$ shows that $\hat{G}$ is locally Euclidean, and hence Lie (by, for instance, Theorem \ref{lem:gleason-palais}). The previous paragraph further shows that the action of $\hat{G}$ is locally free.
\end{proof}
 }
{\begin{proof}[Proof of Theorem \ref{thm:pairwise-to-homo}]
The topological Cartan action $\R^k \curvearrowright X$ is embedded in the action of $\hat{G}$, and the $\hat{G}$-action on $X$ is locally free by Corollary \ref{cor:GhatLie}. The action of $\hat{G}$ is transitive on $X$ by \ref{tc5}, giving $X$ the structure of a $\hat{G}$-homogeneous space, and the action of $\R^k$ is by translations.
\end{proof}
}

\section{Proofs of Theorems \ref{thm:big-main} and \ref{thm:technical}}
\label{sec:proofs}
%In this section, we illustrate how to produce a conjugacy provided the $\eta$ actions form a groups structure.

We prove all results for $\R^k$-actions, since given an $\R^k \times \Z^\ell$ action, one may build a suspension space $\tilde{X}$ with an $\R^{k+\ell}$ action into which the $\R^k \times \Z^\ell$ action embeds, with a $\mathbb{T}^\ell$ factor. This action has a non-Kronecker rank one factor if and only if the original action has a non-Kronecker rank one factor by Lemma \ref{lem:susp}. The resulting homogeneous action will have a $\mathbb{T}^\ell$-factor, and the fibers of this factor will be conjugate to the original $\R^k \times \Z^\ell$ action.

\begin{proof}[Proof of Theorems \ref{thm:big-main} and \ref{thm:technical}]
Let $\R^k \curvearrowright X$ be a Cartan action satisfying the assumptions of Theorem \ref{thm:big-main} or \ref{thm:technical}. {Then by Theorem \ref{thm:main-anosov} and \ref{thm:main-cartan}, it follows that $\R^k \curvearrowright X$ satsifes the higher-rank assumptions \ref{HR:a} and \ref{HR:b} (recall that these were proven at the end of part III in section \ref{sec:part1-proofs}). Hence, by Lemma \ref{lem:symplectic} and Theorem \ref{prop:base base relations}, the action satisfies conditions \ref{CPCS1} and \ref{CPCS2}. Thus, by Theorem \ref{thm:pairwise-to-homo}, the action is topologically conjugate to a homogeneous action. This completes the proof of Theorem \ref{thm:technical}.}

In the case of Theorem \ref{thm:big-main}, we show that the conjugating map $h : X \to G / \Gamma$ is $C^{1,\theta}$. % for {\it some} $\theta \in (0,1)$.
  For any $\alpha \in \Delta$, $h$ intertwines the actions of $\eta^\alpha_t$ with some the one parameter groups of $G / \Gamma$ by {construction}.
%For $\beta \notin E$, every orbit of $\eta ^{\beta} _t$ gets mapped to a point.  Hence  $\pi$ restricted to any $\eta ^{\lambda} _t$ is $C^{1,\theta}$ on those one-dimensional directions.  We also claim that the derivative of $\pi$ along the $\eta ^{\beta} _t$ is transversally \holder.  This is clear for  $\beta \notin E$, as the derivative of $\pi$ along such directions is 0.  For $\beta \in E$, 
{ Thus if $V(x) = \left.\frac{d}{dt}\right|_{t=0} \eta^\alpha_t(x)$, $dh(V)$ exists and is equal to some element of  $\Lie(G)$ (thought of as right-invariant vector fields). Hence, $dh(V)$ is smooth and  $V$} is {$\beta$-H\"older since $W^\alpha$ is a $C^{\beta}$ foliation  with $C^{1,\theta}$-leaves for some $\beta$}, and the generator is chosen to be unit length according to the  $\beta$-H\"older metric from Section \ref{subsec:metrics-prelim} {(without loss of generality we may choose a H\"older constant $\beta$ which applies to both the metric and the foliation, see Proposition \ref{lem:kalinin-lem})}.  
In consequence, $h$  is $C^{1,\beta}$ along all $\eta ^{\alpha} _t$ orbits.  Similarly, $h$ is $C^{1,\theta}$ along the $\R ^k$ orbit foliation.

We in fact claim that $h$ {  has the desired regularity} ($C^{1,\theta}$ or $C^\infty$) along $\eta^{\beta}_t$ orbits. We let $r = {(1,\theta)}$ or $\infty$ depending on the version of the theorem. Fix a Lyapunov exponent $\alpha$ and some smooth Riemannian metric on $X$ (not necessarily the one of Section \ref{subsec:metrics-prelim}). Let $\varphi_a : X \to \R$ be defined to be the norm derivative of $a$ restricted to $E^\alpha$ with respect to the chosen Riemannian metric. That is, $\varphi_a(x) = \norm{da|_{E^{\alpha}}(x)}$, {so that} $\varphi_a$ is H\"older. Notice that if the ratio of the Riemannian norm to the dynamical norm of Section \ref{subsec:metrics-prelim} is given by $\psi(x)$, then $\psi(x)$ is H\"older (since the dynamical metric is H\"older) and 

\begin{equation}
\label{eq:psi-transfer}
 \varphi_a(x) = e^{\alpha(a)}  \psi(a \cdot x) \psi(x)^{-1}.
\end{equation}

Take $\mc L = X \times \R$ to be the trivial line bundle over the manifold $X$, and let $\pi(x,t)$ be the point which has (signed) distance $t$ from $x$ in the smooth Riemannian metric on $X$. For each $a \in \R^k$, there exists a unique lift $\tilde{a}$ to $\mc L$ such that $\pi (\tilde a(x,t)) = a \cdot \pi(x,t)$. Notice that since each fiber corresponds to a $C^r$ submanifold, $\tilde{a}$ is a $C^r$ extension as in Section \ref{sec:normal-forms}. Let $H : \mc L \to \mc L$ be the map $H(x,t) = (x,H_x(t))$, with $H_x$ uniquely satisfying $\eta^\alpha_{H_x(t)}(x) = \pi(x,t)$. Then observe that $H\tilde{a}H^{-1}(x,t) = (a \cdot x, e^{\alpha(a)}t)$ by \eqref{eq:renormalization}. 

Notice that the derivative of $\pi(x,t)$ with respect to $t$ is the unit vector of $E^\alpha$ with respect to the smooth Riemannian metric, and the derivative of $\eta^\alpha_{H_x(t)}(x)$ is the derivative of $H_x$ times the unit vector of $E^\alpha$ with respect to the dynamical norm. Therefore, $H_x'(t) = \psi(\pi(x,t))$, so $H_x \in C^1(\R,\R)$ varies continuously with $x$. Now, finally we modify $H$ slightly by defining $G(x,t) = H(x,\psi(x)^{-1}t)$. Then $G_x'(0) = 1$ and $G$ is still a linearization by \eqref{eq:psi-transfer}. Therefore, $G$ is a system of $C^r$ normal form coordinates {by the uniqueness of the normal forms provided by} Theorem \ref{thm:normal-forms}, and $H_x(t)$ is $C^r$ in $t$. This implies that $h$ is $C^r$ along $W^\alpha$-leaves.

%We recall the main result of Journ\'{e} in \cite{journe88}:  Given two continuous transverse foliations $F_1$ and $F_2$ with uniformly smooth leaves on a  manifold $M$. Suppose $f$ is a function uniformly $C^{1,\theta}$ along $F_1$ and $F_2$.  Then $f$ is $C^{1,\theta}$.  Note that the theorem in \cite{journe88} states a version for $\smooth$ functions along  $F_1$ and $F_2$.  The proof however works for uniformly $C^{1,\theta}$ functions as the author explicitly states. 

Choose some Anosov element $a_0 \in \R ^k$ which determines the positive and negative weights $\Delta _+$ and $\Delta _-$.
Order the weights $\lambda \in \Delta _-$ cyclically: $ \lambda _1, \ldots \lambda _l$.  Then inductively we see from Theorem \ref{thm:journe} that $h$ is $C^{1,\theta}$ along the foliations tangent to $E^{\lambda _1}, E^{\lambda _1} \oplus E^{\lambda _2}, \ldots, \oplus _{I=1, \ldots, l} E^{\lambda _i}$.
Clearly, the last foliation is nothing but the stable foliation of $a_0$, ${W}^s _{a_0}$.  Hence, $h$ is $C^{1,\theta}$ along ${W}^s _{a_0}$, and similarly along the unstable foliation ${W}^u _{a_0}$  of $a_0$.  
Using Theorem \ref{thm:journe} again, we get that $\pi$ is $C^{1,\theta}$ along the weak stable foliation, by combining stable and orbit foliations, and then on $X$ combining weak stable and unstable foliations. 

In the $C^\infty$ setting, one may repeat the arguments with the $C^\infty$ version of Theorem \ref{thm:journe} to obtain that $h$ is $C^\infty$.%, provided that the H\"older metrics of Section \ref{subsec:metrics-prelim} are $C^\infty$ when restricted to each coarse Lyapunov leaf, which we show below.
\end{proof}

\part{\Large Structure of Rank One Factors and Applications}
\label{part:3}

In Part V, we return to {the setting of $C^r$ totally Cartan $\R^k$-actions, $r = (1,\theta)$ or $r = \infty$}, and prove a structure theorem for them. {As usual, we deduce theorems about $\R^k \times \Z^\ell$ actions by passing to suspensions.} In particular, since we allow rank one factors, we no longer have the H\"older metrics from b) of the \hyperlink{h-r-a}{higher rank assumptions}. Instead, we will again rely on Lemma \ref{lem:uniformly bounded derivative}, which only requires the totally Cartan assumption to hold. 

From Parts III and IV, we have established a dichotomy: either there exists a non-Kronecker rank one factor for the action which is an Anosov 3-flow, or the action is homogeneous. {In this part, we establish structure for totally Cartan actions which have rank one factors. We begin by identifying the a factor of the action by the 0-entropy elements, the Starkov component. After passing to such a quotient, we build a homogeneous structure on the common fiber of every rank one factor in Section \ref{sec:main-structure}.

We also some observations about the centralizer of a totally Anosov $\R^k$-action (Section \ref{sec:centralizer}), and using this to identify when a rank one factor splits as a direct product (Section \ref{sec:prod-struct}). Actions in which rank one factors exist but do not form a direct product with a homogeneous action were not known to exist until recently, see Sections \ref{sec:nonproduct-ex} and \ref{sec:DWX}. We conclude the paper by proving some applications of the main theorems (Section \ref{sec:corollaries}).}  %The key insight in obtaining Theorem \ref{thm:full-classify} is the following: any rank one factor of the action actually splits as a direct product { (up to finite cover)}, and the $\R^{k-1}$ action on the complementary manifold is still totally Cartan (Proposition \ref{prop:self-centralizing-extension}). This allows us to ``peel off'' rank one factors until there are none left, and we may apply the rigidity results of Part II to get that the action must (covered by) the direct product of its rank one factors, and some homogeneous action. 

\section{The Starkov component}
\label{sec:starkov}

Fix a $C^r$, totally Cartan $\R^k$ action on a manifold $X$, and let $S = \bigcap_{\beta \in \Delta} \ker \beta$ denote its Starkov component (recall the definition made directly before Remark \ref{rem:starkov}). 

\begin{proposition}
\label{prop:starkov-torus}
The action of $S$ factors through a torus action.
\end{proposition}

\begin{proof}
Fix some Riemannian metric on $X$. 
Recall that by Lemma \ref{lem:uniformly bounded derivative}, if $a \in \ker \beta$, then $\norm{da|_{E^\beta}}$ is uniformly bounded above and below by constants $L$ and $1/L$, respectively. Since $a \in S$, we have this for every coarse exponent $\beta$ and since $da|_{T\mc O}$ is isometric (where $\mc O$ is the orbit foliation of the $\R^k$ action), the action of $S$ is equicontinuous. By the Arzel\`{a}-Ascoli theorem, the closure of $S$ in the group of homeomorphisms of $X$ is a compact group, call it $H$.

Observe that if $h \in H$, then $h$ is an equicontinuous homeomorphism that commutes with the $\R^k$-action. {Now, for every $a \in \R^k$, $d(a^n \cdot h(x),a^n \cdot x) = d(h(a^n \cdot x),a^n \cdot x)$ is uniformly bounded above and below. Let $h$ be close to the identity in $H$, and $h(x)$ be reached from $x$ by an $a$-stable leg, an  $a$-unstable leg, and then an $\R^k$-orbit leg. If either the unstable or stable leg was nontrivial, $h(x)$ and $x$ would separate at an intermediate distance for some iterate of $a$ or $-a$, respectively. Thus, it follows that  $h(x) \in \R^k \cdot x$}. Therefore, {if $\ve$ is sufficiently small and  $d_{C^0}(h,\id) < \ve$}, $h$ determines a continuous function $\tau : X \to \R^k$ by $h(x) = \tau(x) x$ {and $\norm{\tau(x)} \le \ve$}. A simple computation using the fact that $h$ commutes with the $\R^k$-action shows that $\tau$ is constant on $\R^k$-orbits and since there is a dense $\R^k$-orbit, $\tau$ is constant everywhere.

Therefore, the action of a neighborhood of the identity in $H$ coincides with a subgroup of the $\R^k$-action. Since $H$ is the closure of a connected group, $H$ is connected, and the action of $H$ and the action of a subspace of $\R^k$ coincide. But any $a \not\in S$ cannot act equicontinuously (since then $\beta(a) \not=0$ for some $\beta$), so we know that $S$ factors through $H$. Since $H$ is a compact factor of $S$, $H = \mathbb{T}^\ell$ for some $\ell$ and we obtain the result.
\end{proof}

{ 
We now characterize the Starkov component in terms of various dynamical properties, thereby justifying Remark \ref{rem:starkov}.

\begin{lemma}
\label{lem:starkov-conditions}
Let $\R^k \curvearrowright X$ be a $C^r$ totally Cartan action. Then the following are equivalent:

\begin{enumerate}
\item $a \in S$.
\item $a$ has zero topological entropy.
\item $\set{a^k : k \in \Z}$ is an equicontinuous family.
\item $a$ does not have sensitive dependence on initial conditions.
\item $a$ is not partially hyperbolic (with nontrivial stable and unstable bundles).
\end{enumerate}
\end{lemma}

\begin{proof}
It is clear that condition (1) implies all others by Proposition \ref{prop:starkov-torus}. Therefore, we need only show that the others imply (1). In each it is easier to prove the contrapositive. 

{\it (5) implies (1)}: Assume that $a \not\in S$. Then the bundles $E^s_a =\displaystyle \bigoplus_{\alpha(a) < 0} E^\alpha$, $E^u_a = \displaystyle\bigoplus_{\alpha(a) < 0} E^{\alpha}$, and $E^c_a = T\mc O \oplus \displaystyle\bigoplus_{\alpha(a) = 0} E^\alpha$ are all nontrivial. Therefore $a$ is partially hyperbolic (with nontrivial stable and unstable bundles), so (5) is not satisfied. Hence (5) implies (1).

{\it (4) implies (1):} Again, assume that $a \not\in S$. Recall that $a$ has sensitive dependence on initial conditions if there exists $\ve >0$ such that for every $x \in X$ and $\delta > 0$, there exists $y \in X$ and $n \in \N$ such that $d(x,y) < \delta$, but $d(T^n(x),T^n(y)) > \ve$. To find such a $y$, we use that $a \not\in S$, so there exists some $\alpha$ such that $\alpha(a)>0$. If $y \in W^\alpha_{\operatorname{loc}}(x) \setminus \set{x}$, then some sufficiently large iterate of $a$ will separate $x$ and $y$ by a fixed amount. Hence (4) implies (1).

{\it (3) implies (1):} Assume that $a \not\in S$. Then along some coarse Lyapunov leaf, the derivatives grow exponentially. This implies that $a^k$ cannot be equicontinuous.

{\it (2) implies (1):} This follows from the standard construction. Assume $a \not \in S$. Fix $x \in X$, and consider $\alpha$ such that $\alpha(a) > 0$. Then let $\lambda = \displaystyle \inf_{x \in W^\alpha_{\operatorname{loc}}} \norm{a_*|_{E^\alpha}(x)}$. Then we may choose a metric such that $\lambda > 1$. We claim that the topological entropy of $a$ is at least $\lambda$. Indeed, choose any subset of $W^\alpha$ which is spaced at distance $\ve/\lambda^n$. By 1-dimensionality, up to a constant depending on the curvature of the manifolds $W^\alpha(x)$, these points are $(n,\ve)$-separated. Since we may fit $C\ve \cdot \lambda^n$ such points, the topological entropy is at least $\lambda$.
\end{proof}
}

\begin{lemma}
\label{lem:starkov-factor}
The quotient $X / S$ is a manifold.
\end{lemma}

Recall the $X / S$ is called the {\it Starkov factor} of $X$. It is clear that {the action of any  $(k-\dim(S))$-dimensional subspace of $\R^k$ transverse to $S$} descends to a locally free action on ${X / S}$ and the coarse Lyapunov foliations $W^\beta$ descend for every $\beta \in \Delta$, giving $X / S$ a transitive, totally Cartan action.

\begin{proof}
{  Using Proposition \ref{prop:starkov-torus},} the proof is identical to Proposition \ref{prop:factor-by-H}, we briefly summarize the method. We must show that $\Stab_S(x)$ is independent of $x$. To this end, we show that for every $x,y \in X$, $\Stab_S(x) \subset \Stab_S(y)$. Then reversing the roles of $x$ and $y$ shows equality. It suffices to show that $\Stab_S(x) \subset \Stab_S(y)$ for every $y \in \R^k \cdot x$ and $y \in W^\beta(x)$ for any $\beta \in \Delta$. The proof for $\R^k$ is trivial by commutativity. One may use the preservation of the $\beta$ leaves together with normal forms {(Theorem \ref{thm:normal-forms-foliation})} to show that if $a \in S$ and $a \cdot x = x$, then $a|_{W^{\beta}(x)} = \id$. {Indeed, since $x$ is fixed by $a$, if $\norm{da|_{TW^\beta}} \not =1$, then it would grow or shrink exponentially, violating Lemma \ref{lem:uniformly bounded derivative}.}
\end{proof}

{
\begin{corollary}
\label{cor:factor-by-starkov}
Let $\R^k \curvearrowright X$ be a cone transitive $C^r$ totally Cartan action with $r = (1,\theta)$ or $r = \infty$, and Starkov component $S \subset \R^k$. Then $X$ is a principal $\mathbb{T}^\ell$ bundle over a $C^r$ manifold $Y$, where $\ell = \dim S$. Moreover, $Y$ carries a canonical cone transitive $C^r$ totally Cartan action of $\R^k / S$ with trivial Starkov component, and if $\pi : X \to Y$ is the corresponding factor map, then $\pi(a \cdot x ) = (aS) \cdot \pi(x)$. 
\end{corollary}}

\section{The main structure theorem}
\label{sec:main-structure}
{Corollary \ref{cor:factor-by-starkov} implies that given an action with a Starkov component, there is always a factor of the action with trivial Starkov component from which the original action is obtained by a principal bundle extension. These actions are described in the following theorem. { Recall the definitions of $\Delta_{\operatorname{rig}}$ and $\Delta_1$ in Remark \ref{rem:homogeneous-weights}.}

\begin{theorem}
\label{thm:big-headache}
Let $\R^k \curvearrowright X$ be a $C^r$, cone transitive, totally Cartan action, with { $r = 2$} or $r = \infty$ and trivial Starkov component. { We let $r' = (1,\theta)$ or $\infty$}, respectively. Then (after passing to a finite cover of $X$) there exists

\begin{itemize}
\item a collection of transitive $C^{r'}$ Anosov flows on closed 3-manifolds $\psi_i : \R \curvearrowright Y_i$, $i = 1,\dots,\ell$, 
\item a homomorphism $\sigma_{\Rig} : \R^k \to \R^\ell$ which is onto a rational subspace of $\R^\ell$ (in the usual $\Q$-structure),
\item a $C^{r'}$ map $\pi : X \to Y = Y_1 \times \dots \times Y_\ell$,
\item a Lie group $G_{\Rig}$ with a locally free $C^{r'}$  action $G_{\Rig} \curvearrowright X$,
\item a homomorphism from $ \R^k \to \Aut(G_{\Rig})$ (for which we denote the image of $a$ by $\Phi_a$), and
\item a subgroup $A \subset G_{\Rig}$ with an isomorphism $i_{\Rig} : \ker \sigma_{\Rig} \to A$ (which identifies $\ker \sigma$ as a subgroup of $G_{\Rig}$)
\end{itemize}

such that

\begin{enumerate}
\item the product action $\R^\ell \curvearrowright Y = Y_1 \times \dots \times Y_\ell$ defined by 

\begin{equation}\label{eq:sigma-intertwinig}(t_1,\dots,t_\ell) \cdot (y_1,\dots,y_\ell) = (\psi_1^{t_1}(y_1),\dots,\psi_\ell^{t_\ell}(y_\ell))\end{equation} is a totally Cartan action,
\item  for every $x \in X$, $G_{\Rig} \cdot x = \pi^{-1}(\pi(x))$,
\item $\pi$ is a $C^{r'}$ submersion onto its image, which is a $C^{r'}$ submanifold whose tangent space is given by the sum of the stable and unstable bundles on each $Y_i$, and the tangent space to the orbit foliation of $\sigma(\R^k)$,
\item for every $x \in X$, $g \in G_{\Rig}$, and $a \in \R^k$, $a \cdot (gx) = \Phi_a(g) \cdot ax$,
\item for every $x \in X$ and $a \in \ker \sigma_{\Rig}$, $a \cdot x = i_{\Rig}(a) \cdot x$, 
\item for every $x \in X$, $\pi(a \cdot x) = \sigma_{\Rig}(a) \cdot \pi(x)$, and
\item \label{enum:holonomy-affine} { for every $\beta \in \Delta_1$, $x \in X$ and $x' \in W^\beta(x)$  sufficiently close to $x$, there exists a H\"older continuous holonomy $f_{x,x'} : G_{\Rig} \cdot x \to G_{\Rig} \cdot x'$ defined by $f_{x,x'}(z) = W^\beta_{\loc}(z) \cap G_{\Rig} \cdot x'$.} %there exists an automorphism $\varphi_{x,x'} : G_{\Rig} \to G_{\Rig}$ such that $W^\beta_{\loc}(g \cdot x) \cap G_{\Rig} \cdot x' = \varphi_{x,x'}(g) \cdot x'$.
\end{enumerate}
\end{theorem}

\begin{remark}
It may help to understand the theorem by seeing its features in an example, which is descibed in Section \ref{sec:big-headache-ex}.
\end{remark}

\begin{remark}
Theorem \ref{thm:big-headache} gives $X$ a structure very similar to that of a $G$-bundle over $Y$. The fibers are all $G$-homogeneous spaces with lattice stabilizers, but the transition maps from one fiber to another may not be {\it translations} in $G$. Instead, they will be affine maps in $G$. On the other hand, since the full affine group of $G$ does not act on $G / \Gamma$ (as many automorphisms of $G$ will move the lattice $\Gamma$), we may not make it an $\Aff(G)$-bundle, either.
\end{remark}

%Fix an $\R^k$-periodic orbit $p \in X$, and notice that as in Proposition \ref{prop:fix-point-set}, the saturation of $p$ by

The rest of Section \ref{sec:main-structure} is devoted to the proof of Theorem \ref{thm:big-headache}. We begin in Section \ref{sec:fiber-homogeneous} by identifying a certain class of weights and corresponding flows along their corresponding coarse Lyapunov foliations which will generate the group $G_{\Rig}$. These will exactly be the weights $\beta$ for which conditions \ref{HR:a} and \ref{HR:b} can be deduced (for that specific $\beta$). We will then build the group action of $G_{\Rig}$ from these flows, which are all uniformly transverse to every rank one factor, and in fact parameterize their common fiber.

The usual intertwining properties with the $\R^k$-action, together with the description of $G_{\Rig}$ as the common fiber of all rank one factors imply most of Theorem \ref{thm:big-headache}. The only conclusion which will not be immediate is (\ref{enum:holonomy-affine}), which is shown in Section \ref{sec:holonomies-affine}.

{

\subsection{The Rigid Fiber}
\label{sec:fiber-homogeneous}

Recall Remark \ref{rem:homogeneous-weights}, and the sets  \[\Delta_{\Rig} := \set{ \beta \in \Delta : \ker \beta \mbox{ has a dense orbit or is dense in the fiber of some circle factor}}\] and $\Delta_1 = \Delta \setminus \Delta_{\Rig}$. By Lemma \ref{lem:TNS}, $\Delta_1$ consists of a union of symplectic pairs $\pm \alpha$. For each such pair, pick one such $\alpha$ and let $\Delta_1^+$ be the set of such choices. Then by Theorem \ref{thm:main-anosov}, it follows that after passing to a finite cover, for each $\alpha \in \Delta_1^+$, there exists a 3-manifold $Y_\alpha$ with Anosov flow $\psi_t^\alpha$, submersion $\pi_\alpha : X \to Y_\alpha$, and homomorphism $\sigma_\alpha : \R^k \to \R$ such that for every $x \in X$, $a \in \R^k$, $\pi_\alpha(a \cdot x) = \psi_{\sigma_\alpha(a)}^\alpha(\pi_\alpha(x))$. Consider $\pi = \prod_{\alpha \in \Delta_1^+} \pi_\alpha : X \to \prod_{\alpha \in \Delta_1^+} Y_\alpha$.

Recall Definition \ref{def:weight-bracket}.

\begin{lemma}
\label{lem:hom-ideal}
If $\beta_1,\beta_2 \in \Delta_{\Rig}$ are linearly independent, then $[\beta_1,\beta_2] \subset \Delta_{\Rig}$.
\end{lemma}

\begin{proof}
Recall that $\alpha \in [\beta_1,\beta_2]$ if and only if $\rho^{\beta_1,\beta_2}_\alpha(s,t,x) \not \equiv 0$. Furthermore, the functions $\rho^{\beta_1,\beta_2}_\alpha$ are defined via geometric commutators. If $\alpha \in \Delta_1$, then there exists an associated factor map $\pi_\alpha : X \to Y_\alpha$, and the foliations $W^{\beta_1},W^{\beta_2}$ are contained in the foliation $W^H$ that appears in Section \ref{sec:case1} (specifically, Lemma \ref{lem:foliation1}). In particular, it follows that $W^{\beta_1}$ and $W^{\beta_2}$ are contained in some $W^s_a \cap W^H$, and thus $\rho^{\beta_1,\beta_2}_\alpha \equiv 0$.
\end{proof}

%For each $\beta \in \Delta_{\Hom}$, there are norms $\norm{\cdot}_{\beta,x}$ on $TW^\beta_x$ for every $x \in X$ which satisfy the conclusions of \ref{HR:a} and vary H\"older continuously. In particular, we can use these norms to construct flows $\eta^\beta$ parameterizing the foliations $W^\beta$ as in Part II.

Notice that if $\beta \in \Delta_{\Rig}$, the action of $\ker \beta$ either has a dense orbit or has orbits dense in the fibers of some circle factor by definition. This allows us to construct the metrics as in \ref{HR:a} using arguments identical to the proof of Theorem \ref{thm:main-cartan} (see Section \ref{sec:part1-proofs}). Let $\mc P_{\Rig} \curvearrowright X$ denote the action of the free product $\mc P_{\Rig} = \R^{\abs{\Delta_{\Rig}}}$ induced by the flows $\eta^\beta$.
Let $H_0 = \bigcap_{\alpha \in \Delta_1} \ker \alpha$, and $\hat{\mc P}_{\Rig} = H_0 \ltimes \mc P_{\Rig}$. Recall the terminology of Definition \ref{def:stand-gens} and Proposition \ref{prop:symplectic-possibilities}.

\begin{lemma}
\label{lem:H0-commutators}
If $\beta \in \Delta_{\Rig}$ is such that $-\beta \in \Delta$ and $\mc P_{\set{\beta,-\beta}}$ factors through the action of a Lie group $\mc G_{\set{\beta,-\beta}} \curvearrowright X$, then the neutral element of $\Lie(\mc G_{\set{\beta,-\beta}})$ (if it is nontrivial) acts in the same way as a one parameter subgroup of $H_0$. 
\end{lemma}

\begin{proof}
We will show that the neutral element of $\mc G_{\set{\beta,-\beta}}$ lies in $\ker \alpha$ for every $\alpha \in \Delta_1$. Fix such an $\alpha$ and notice that the orbits $\mc G_{\set{\beta,-\beta}} \cdot x$ are contained in $\pi_\alpha^{-1}(\pi_{\alpha}(x))$ (which is the leaf of the foliation $W^H$ as in Lemma \ref{lem:foliation1} for $H = \ker \alpha$). Therefore, since any $a \not\in \ker \alpha$ projects to some time-$t$ map of the Anosov flow on $Y_\alpha$, and the orbit of the neutral element must be contained in the fibers, it must be contained in $\ker \alpha$. Since $\alpha$ was an arbitrary element of $\Delta_1$, the neutral element must be contained in $H_0$. 
\end{proof}

 %and $\hat{\mc P}_{\Hom} = H_0 \ltimes \mc P_{\Hom}$ denote the semidirect product, with associated action $\hat{\mc P}_{\Hom} \curvearrowright X$.

\begin{proposition}
\label{prop:hom-fibers}
The action of $\hat{\mc P}_{\Rig} \curvearrowright X$ factors through the locally free action of a Lie group $G_{\Rig}$, and for each $a \in \R^k$, there exists an automorphism $\psi_a$ of $G_{\Rig}$ such that for every $g \in G_{\Rig}$, $ag \cdot x = \psi_a(g)a \cdot x$.
\end{proposition}

\begin{proof}
We claim that the conclusions of Sections \ref{sec:fibers} and \ref{sec:symplectic} apply. Indeed, we have the higher rank assumptions \ref{HR:a} and \ref{HR:b} for a fixed weight or pair of weights in $\Delta_{\Rig}$. The arguments of Section \ref{sec:fibers} apply verbatim using Lemma \ref{lem:hom-ideal} to verify the inductive procedure works using only commutators from $\Delta_{\Rig}$. Furthermore, by Lemma \ref{lem:H0-commutators}, we may follow the constructions around Lemma \ref{lem:CD-normal} to identify any element that appears as a neutral element of $\mc G_{\set{\beta,-\beta}}$ with an element of $H_0$ to build the group $G_{\Rig}$ after making such identifications.

 Therefore, the action of $\hat{\mc P}_{\Rig}$ has relations as in Definition \ref{def:const-pairwise}, and one may make the same algebraic manipulations to put an arbitrary word in $\hat{\mc P}_{\Rig}$ in the form of Proposition \ref{stable unstable cycle decomposition}. Then the proof of Corollary \ref{cor:GhatLie} works verbatim, as local product structure still implies that no small cycle consisting of a single stable, single unstable, and single $\R^k$-orbit foliation exists.
 
 To see the intertwining of the $G_{\Rig}$ action by the $\R^k$ action is by automorphisms, notice that if $\sigma \in \hat{\mc P}_{\Rig}$ is a cycle at $x$, then $a \cdot \sigma$ is a cycle at $a \cdot x$, and $a$ acts by automorphisms on each coarse Lyapunov generator. The automorphisms then come from Proposition \ref{prop:integrality}.
\end{proof}

A surprising feature of the map $\pi$ is that it may not be onto. This occurs when one takes the restriction of a product action to a generic subgroup. See Section \ref{sec:embedded-ex}. However, the following lemma shows that it is always a submersion onto its image.

\begin{lemma}
\label{lem:product-factor}
$Y := \pi(X)$ is a $C^{r'}$ submanifold of $\prod_{\alpha \in \Delta_1} Y_\alpha$ and $\pi : X \to Y$ is a $C^{r'}$ submersion, and the induced action of $\R^k$ on $Y$ is totally Cartan.
\end{lemma}

\begin{proof}
Since each $\pi_\alpha$ is a submersion, and $\ker d\pi_\alpha$ is exactly the sum of the tangent space to the orbits of $\ker \alpha$ with $\bigoplus_{\beta \not= \pm\alpha} TW^\beta$, it follows that $\pi$ is a map of constant rank, and $\ker d\pi$ is exactly the tangent space to the orbits of $H_0$ and the sum of $TW^\beta$, $\beta \in \Delta_{\Rig}$.

Note that even though $\pi$ has constant rank, its image may not be a submanifold. { Rather, it is at this point only an immersed submanifold, possibly with self-crossings.} It will be a submanifold if we can prove that every point has a unique family of manifold charts. We may produce such charts by pushing forward the image of an open neighborhood in $X$ and using the implicit function theorem. These will give a manifold structure as long as the following is satisfied:
%To see that the image is a submanifold, it suffices to show that 
if $x_1,x_2 \in X$ satisfy $\pi(x_1) = \pi(x_2)$, then for any neighborhood $U$ of $x_1$, %then $d\pi(x_1)$ and $d\pi(x_2)$ have the same image. 
there exists a neighborhood $V$ of $x_2$ such that $\pi(V) \subset \pi(U)$.

To see this, note that a neighborhood of $\pi(x_1)$ is exactly given by the saturation of $\pi(x_1)$ by a local stable and unstable manifold of $\pi(x_1)$ for some Anosov $a$, together with the image of the $\sigma(\R^k)$-orbit foliation. It is then clear that images of sufficiently small neighborhoods of $V$ are contained in such neighborhoods of $\pi(x_1)$.

It is clear that the induced action of $\R^k$ on $Y$ is totally Cartan since the stable and unstable foliations of each Anosov flow on $Y_\alpha$ are images of coarse Lyapunov foliations on $X$. Such foliations are intersections of stable manifolds for a collection of Anosov elements $a_i$ acting on $X$ by definition.. Such Anosov elements are Anosov on $Y$ and the intersections of their stable and unstable manifolds give the desired description of the stable and unstable manifolds of each Anosov flow as common stable manifolds of the $\sigma(\R^k)$-action.
\end{proof}

\begin{corollary}
\label{cor:full-fiber}
If $x \in X$ , then $\hat{\mc P}_{\Rig} \cdot x$ { is the connected component of} $\pi^{-1}(\pi(x))$.
\end{corollary}

\begin{proof}
Since $T(\pi^{-1}(x))$ is exactly the sum of the tangent spaces to $H_0$-orbits and $\bigoplus_{\beta \in \Delta_{\Hom}} TW^\beta$, which is exactly the tangent space of $G$-orbits, $\hat{\mc P}_{\Hom}$ are exactly the connected components of the fibers of $\pi$.
\end{proof}

\subsection{Intertwining properties of holonomies}
\label{sec:holonomies-affine}

In this section, we show condition (\ref{enum:holonomy-affine}) of Theorem \ref{thm:big-headache}, which completes its proof.

Fix a coarse Lyapunov foliation $W^\beta$, $\beta \in \Delta_1$ and some $y,y' \in Y_\beta$ such that $y \in W^u(y')$, where $W^u$ is the unstable foliation of the Anosov flow on $Y_\beta$ (we assume without loss of generality that $\beta$ corresponds to the unstable foliation and $-\beta$ to the stable). Then $y$ and $y'$ determine a holonomy $f_{y,y'} : \pi_{\beta}^{-1}(y) \to \pi_{\beta}^{-1}(y')$, by setting $f_{y,y'}(z)$ to be the unique (local) intersection of $W^\beta(z)$ with $\pi_\beta^{-1}(y')$.

\begin{lemma}
\label{lem:hom-holonomy}
For any $z \in \pi_{\beta}^{-1}(y)$, $f_{y,y'}(G_{\Rig} \cdot z) = G_{\Rig}\cdot f_{y,y'}(z)$.
\end{lemma}

\begin{proof}
First, notice that if $b \in \ker \beta$, then $f_{y,y'}(b\cdot z) = b \cdot f_{y,y'}(z)$, since $\ker \beta$ preserves the fibers of $\pi_\beta$ and leaves the foliation $W^\beta$ invariant.
We then claim that it suffices to show that if $\alpha \in \Delta_{\Rig}$, then $f_{y,y'}(W^\alpha(z)) \subset G_{\Rig} \cdot f_{y,y'}(z)$ for every $z \in X$. Indeed,  since $G_{\Rig}$ is generated by $H_0 \subset \ker \beta$ and all such $\alpha$, the claim follows inductively.

We now show that $f_{y,y'}(W^\alpha(z)) \subset G_{\Rig} \cdot f_{y,y'}(z)$ for every $z \in X$. Observe that if $z \in X$ and $z_1 \in W^\alpha(z)$, then $f_{y,y'}(z_1)$ is reached from $z$ by first moving along the $W^\alpha$ foliation to get to $z_1$, then along the $W^\beta$ foliation by definition. Therefore, we may use the weak geometric commutator of Lemma \ref{lem:geom-commutator1} to reach $f_{y,y'}(z_1)$ by first moving along the $W^\beta$ foliation, then along the $W^\alpha$ foliation, and finally along foliations in $D(\alpha,\beta)$. However, since the $W^\beta$ foliation on $Y_\beta$ can first be lifted to $Y$, and the other $W^{\beta'}$ foliations, $\beta' \in D(\alpha,\beta) \cap \Delta_1$, commute with $W^\beta$ on $Y$, it follows that any { of the weights in $\Delta_1 \cap D(\alpha,\beta)$} must have a trivial leg in the weak geometric commutator {  of the $\alpha$- and $\beta$-leaves}. This implies that $z_1 \in G_{\Rig} \cdot f_{y,y'}(z)$, as claimed. 
\end{proof}

\section{Centralizer Structure}
\label{sec:centralizer}

{The results of this section analyze the centralizer structure for an $\R^k$ action, and identifies an important finite subgroup. %The example of Section \ref{sec:embedded-ex} shows that a splitting is not possible without them.}
%\subsection{Structure of a general centralizer of a totally Cartan Action}
%\label{sec:general-centralizer}
Let $\R^k \curvearrowright X$ be a cone transitive, totally Cartan action with trivial Starkov component. %, let $\Lambda $ be the centralizer of $\R^k$ in $\Diff^1(X)$. 
Fix an $\R^k$-periodic orbit $p$. Consider $V = \R^{\abs{\Delta}}$ as a vector space, whose entries represent the logarithms of eigenvalues of the derivatives of a map along the foliations $W^{\beta_1},\dots,W^{\beta_{\abs{\Delta}}}$ at $p$. %Then let $E \subset V$ be the set of logs of eigenvalues of $\left(da|_{W^\beta_1}(p),\dots,da|_{W^{\beta_{\abs{\Delta}}}}\right)$, where $a$ ranges over $\Stab_{\R^k}(p)$. 
%\begin{lemma}
%The subgroup $E \subset V$ is discrete.
%\end{lemma}
%\begin{proof}
 %Indeed, if there existed $e_n \in E$ converging to $0 \in V$, then there would exist $a_n \in \R^k$ such that $da_n(p) \to \id$. Let $v_n = a_n/\norm{a_n}$. Notice that since the $a_n$ belong to the discrete subgroup $\Stab_{\R^k}(p)$, the elements $dv_n$ are even closer to $\id$. Then if $v_{n_k} \to v$, $v$ must lie in the Starkov component. Since this is assumed to be trivial, we get a contradiction, so $E$ does not accumulate on $0 \in V$. That is, $E$ is discrete.
%\end{proof}
Let $\Lambda_p \subset Z_{\Diff^1}(\R^k \curvearrowright X)$ denote the subgroup of the centralizer of the $\R^k$ action consisting of elements fixing $p$. Define a homomorphism $\psi : \Lambda_p \to V$ %as follows: if $\lambda \R^k \in \Gamma = \Lambda' / \R^k$, choose any $a \in \R^k$ such that $\lambda a \cdot p = p$. Then 
by setting $\psi(\lambda) = \left(\log {d\lambda}|_{TW^{\beta_1}}(p),\dots,\log {d\lambda}|_{TW^{\beta_{\abs{\Delta}}}}(p)\right)$.% Since we have quotiented $V$ by $E$, this is independent of the choice of $a$, and hence well-defined. 

\begin{lemma}
\label{lem:psi-injective}
The homomorphism $\psi$ is injective.
\end{lemma}

\begin{proof}
The proof is very similar to that of Proposition \ref{prop:fix-point-set} and Lemma \ref{lem:starkov-factor}. Suppose that $\psi(\lambda) = 0$. %Then we may choose a representative $\lambda \in \Lambda'$ of $\Gamma$ such that $\lambda(p) = p$ and $d\lambda =\id$. 
Then ${d\lambda}(p) = \id$. Since $\lambda$ acts by linear transformations in the normal forms coordinates on the $W^\beta$-leaves, $\lambda$ also fixes the leaves $W^\beta(p)$ pointwise for every $\beta \in \Delta$. We may inductively show that $\lambda$ also fixes {any point reached by a path in the $W^\beta$, $\beta \in \Delta$, and $\R^k$-orbit foliations}, which is all of $X$. Therefore, $\lambda = \id$ and the map is injective.
\end{proof}

The following lemma is immediate by jointly diagonalizing the derivatives of elements $a \in \R^k$ such that $a \cdot q = q$ to obtain coefficients of the Lyapunov functionals, and pushing the metric around with the desired rescaling along its $\R^k$-orbit. The procedure is analogous to the metric constructed when there is a circle factor in the proof of Theorem \ref{thm:main-cartan} in Section \ref{sec:part1-proofs}.

\begin{lemma}
\label{lem:periodic-exact}
If $q$ is an $\R^k$-periodic point in $X$ and $\beta$ is a weight, there exists some $c_q \in \R_+$ and a norm $\norm{\cdot}_\beta = \norm{\cdot}_{\beta,aq}$ on $E^\beta_{aq}$, $a \in \R^k$ such that for every $b \in \R^k$ and $v \in E^\beta_{aq}$, $\norm{b_*(v)}_\beta = e^{c_q\beta(b)}\norm{v}_\beta$ .
\end{lemma}

Let $F$ be any finite collection of periodic points in $X$ with distinct $\R^k$-periodic {orbits}, one of which is $p$. Let $\Lambda_p^F \subset \Lambda_p$ be the subgroup of $\Lambda_p$ defined by:

\[ \Lambda_p^F = \set{ f \in \Lambda_p : f(\R^k \cdot q) = \R^k \cdot q \mbox{ for every }q \in F}.\]

Then $\Lambda_p^F$ has finite index in $\Lambda_p$, since for each $q \in F$, there are only finitely many other periodic orbits with the same $\R^k$-stabilizer. We now define a new homomorphism $\Psi_{F} : \Lambda_p^F \to V^{\abs{F}+1}$. To $p$, and each point in $F$, we associate an element of $V$. For $p$, we use $\psi$. For $q \in F$, associate the element $\left(\log \norm{\lambda_*|_{W^{\beta_1}}(q)}_{\beta_1},\dots,\log \norm{\lambda_*|_{W^{\beta_{\abs{\Delta}}}}(q)}_{\beta_{\abs{\Delta}}}\right)$.

\begin{proposition}
\label{prop:Psi-injective}
$\Psi_F$ is an injective homomorphism.
\end{proposition}

\begin{proof}
That it is injective follows from Lemma \ref{lem:psi-injective}. To see that $\Psi_{F}$ is a homomorphism, suppose that $f,g \in \Lambda_p^F$. It suffices to show that $\Psi_{F}$ is a homomorphism in each of the $V$-components. Fix $q \in F$ and $\beta \in \Delta$. Then:

\[  (fg)_*|_{W^{\beta}}(q) = f_*|_{W^{\beta}}(g(q)) \cdot g_*|_{W^\beta}(q). \]

Upon taking norms and logarithms, this proves that $\Psi_{F}$ is a homomorphism, provided that we show that $f_*|_{W^\beta}(g(q)) = f_*|_{W^\beta}(q)$ for every $f,g \in \Lambda_p^F$. Notice that since $g \in \Lambda_p^F$, it suffices to show that $\norm{f_*|_{W^\beta}}_\beta$ is constant along the $\R^k$ orbit of $q$. To see this, notice that since $fa = af$, $(fa)_*|_{W^\beta} = (af)_*|_{W^\beta}$. Therefore, $\norm{f_*|_{W^\beta}(aq)}_\beta \cdot e^{c_q\beta(a)} = e^{c_q\beta(a)} \cdot \norm{f_*|_{W^\beta}(q)}_\beta$, and $\norm{f_*|_{W^\beta}(\cdot)}_\beta$ is constant along the $\R^k$-orbit of $q$ as claimed.
\end{proof}

{
\begin{corollary}
\label{cor:finite-order}
Assume that $\lambda \in Z_{\Diff^1}(\R^k \curvearrowright X)$, and that for every $i$, $\norm{d\lambda^n|_{TW^{\beta_i}}}$ is uniformly bounded above and below in $n$. Then $\lambda$ is a finite-order diffeomorphism.
\end{corollary}

\begin{proof}
Fix $p$ and $F$ as discussed after Lemma \ref{lem:periodic-exact}. Since for $p$ and each $q \in F$, there are only finitely many periodic orbits with the same period, there exists some $n$ such that $\lambda^n$ fixes the $\R^k$-orbit of $p$ and the $\R^k$-orbit of $q$ for every $q \in F$. Choose some $a \in \R^k$ such that $(a\lambda^n) \cdot p = p$, so that $a\lambda^n \in \Lambda^F_p$. Notice that by the assumption on derivatives, $\Psi_F(a\lambda^n) = \Psi_F(a)$, so $a\lambda^n = a$ by Proposition \ref{prop:Psi-injective}. It follows that $\lambda^n = \id$ and $\lambda$ is finite order.
\end{proof}
}

%Recall the definition of a slow foliation for an element $a \in \R^k$ given in Section \ref{sec:holonomy-action}.

%\begin{lemma}
%\label{lem:slow-space}
%If $W^\alpha$ is a slow foliation of some $a$, and the coordinates of the bundle are given by $E^s_a \cong \R_\alpha \times \R^{m-1}$, then if $f$ is an arbitrary transformation commuting $a$, the normal form of a stable leaf at $x$ is given by $f_x(t,v) = (\mu_x t, h_x(v) + g_x(t))$, where $h_x$ and $g_x$ are polynomial. 
%\end{lemma}

%\begin{proof}
%{  Prove using normal forms, referencing...}
%\end{proof}

We now fix a choice of $F$. For each weight $\beta$, find an element $a$ such that $W^\beta$ is the slow foliation for $a$ (this is possible by Lemma \ref{lem:smooth-holonomy}), and again let $\widehat{W^\beta}$ denote the complementary foliation within $W^u_a(x)$ for a suitable choice of $a$. Since periodic points are dense by Lemma \ref{lem:cone-transitive}, for each $\beta \in \Delta$, we may pick another periodic point $q$ very close to some point $x \in W^\beta(p)$ distinct from $p$ at distance $\ve_0 > 0$ from $p$ (so that both $q$ and $x$ lie in some chart around $p$ in which local product structure holds). By local product structure,%One may pick $q$ close enough so that
 there is a unique local transverse intersection $W^{cs}_{a,\operatorname{loc}}(q) \cap W^u_{a,\operatorname{loc}}(p) = \set{y}$ and since $d(p,x) = \ve_0$, by choosing $q$ sufficiently close to $x$, we may assume that $\widehat{W^\beta}(y) \cap W^\beta(p) = \set{x'}$ has $x' \not= p$. (again, see Lemma \ref{lem:smooth-holonomy}). %, with first component (corresponding to $E^\beta$) nonzero in the normal forms chart.
Pick points $x$ and $q$ for every $\beta \in \Delta$, and let $F$ denote the collection of such points $q$, together with $p$.

\begin{lemma}
\label{lem:discrete-centralizer}
{For the choice of $F$ established above,} $\Psi_F(\Lambda_p^F)$ is a discrete subgroup of $V^{\abs{F}+1}$.
%If $\R^k \curvearrowright X$ is a transitive, totally Cartan action, $\Lambda$ is the centralizer of $\R^k$ in $\Diff^1(X)$, $p$ is an $\R^k$-periodic orbit, and $\Lambda_p \subset \Lambda$ is the subgroup of $\Lambda$ which fixes $p$, then $\Lambda_p$ is discrete in $\Diff^1(X)$. 
\end{lemma}

\begin{proof}
 %Choose a regular element $a \in \R^k$ so that $\beta$ is the slow foliation with respect to $a$. 
%{  For a given periodic point $p \in F$, we assume the structures discussed in the paragraph preceding the lemma are established.} 
Suppose that $\lambda \in \Lambda_p^F$ {satisfy that $\Psi_F(\lambda)$} is very close to $0 \in V^{\abs{F}+1}$, so that $\norm{\lambda_*|_{W^\gamma}(q)}_\gamma$ is very close to 1 for every $\gamma$ and at every $q \in F$. 

{Fix such a $q$. We claim that} for $\lambda$ such that $\Psi_{F}(\lambda)$ is sufficiently close to 0, $\lambda$ fixes the intersection of $W^{cs}_{a,\operatorname{loc}}(q) \cap W^u_{a,\operatorname{loc}}(p) = \set{y}$. {Indeed, $\lambda(y)$ will still belong to this intersection, as $p$ is fixed by $\lambda$ and the center-stable manifold of $q$ is fixed by $\lambda$. Furthermore, $y$ does not move much by the smallness of $\Psi_F(\lambda)$.} Since the intersection is locally unique we get that $\lambda$ fixes $y$. Then $y \in \widehat{W^\beta}(x')$ for some $x' \not= p$, and since $\widehat{W^\beta}(y)$ and $W^\beta(p)$ are both invariant and have derivatives very close to one, their local intersection is also unique and preserved. But then $\lambda(x') = x'$, so $\lambda$ fixes two points on the same $\beta$ leaf. This implies that $\lambda$ fixes the entire $\beta$ leaf since it must be linear in the normal forms coordinates. So ${d\lambda}|_{TW^\beta}(p) = 1$. %By Lemma \ref{lem:slow-space}, since the first coordinate must be preserved and is nonzero by construction, we conclude that $d\lambda|_{W^\beta}(p) = 1$.
Since this can be arranged for each $\beta$, we get that $\psi(\lambda) = 0$ and hence by Lemma \ref{lem:psi-injective}, we conclude that $\lambda = \id$. Therefore, $\Psi(\Lambda_p^F)$ is discrete in $V^{\abs{F}+1}$.
\end{proof}

{By  Lemma \ref{lem:psi-injective}, $\Psi_F|_{\Lambda_p^F}$ is an isomorphism onto its image, and $\Lambda_p^F$ is therefore a free abelian group of finite rank by Lemma \ref{lem:discrete-centralizer}.} Since $p \in F$ is an $\R^k$-periodic point, $\Lambda_p^F \cap \R^k \cong \Z^k$. Therefore, $\Lambda^{F} := \Lambda_p^F \cdot \R^k$ is a finite index subgroup of $Z_{\Diff^1}(\R^k \curvearrowright X)$, and is exactly the set of elements which fix the orbits of elements of $F$, not necessarily fixing $p$.

\begin{lemma}
\label{lem:injective-derivative}
$\Psi_F$ extends to an injective homomorphism from $\Lambda^{F}$ to $V^{\abs{F}+1}$.
\end{lemma}

\begin{proof}
Notice that $\Lambda^F$ is an abelian group isomorphic to $\R^k \times \Z^\ell$ for some $\ell$, and that $\Lambda^F_p \cong \Z^k \times \Z^\ell$ is a lattice in $\Lambda^F$. Since $\Psi$ is now an injective homomorphism from a lattice to a lattice, it must extend to an injective homomorphism on the corresponding vector spaces.
\end{proof}

Define 

\begin{eqnarray*} Q := \big\lbrace \lambda \in Z_{\Diff^1}(\R^k \times \Z^\ell \curvearrowright X) : d\lambda^n(x)|_{TW^\beta} \mbox{ is uniformly bounded,}\hspace{1cm} \\ \hspace{1cm}\mbox{for all }n \in \N, 
 x \in X \mbox{ and } \beta \in\Delta\big\rbrace.
\end{eqnarray*}

\begin{corollary}
\label{cor:q-finite}
$Q$ is a finite group.
\end{corollary}

\begin{proof}
Let $\tilde{F}$ denote the set of periodic orbits $g \cdot (\R^k \cdot x)$, where $\R^k \cdot x \in F$ and $g \in Z_{\Diff^1}(\R^k \times \Z^\ell)$. Notice that this is still a finite set, since for each possible period, there are only finitely many orbits with that period. Let $S(\tilde{F})$ denote the set of permutations of the finite set $\tilde{F}$, which itself is a finite group. Then there exists a homomorphism $\phi : Q \to S(\tilde{F})$ which tracks how the orbits of $\tilde{F}$ are permuted. We claim that $\phi$ is injective, which implies that $Q$ is finite.

Indeed, by construction, if $\lambda \in \ker \phi$, then $\lambda$ fixes every orbit of $F \subset \tilde{G}$. Therefore, $\lambda \in \Lambda^F$. However, since $d\lambda^n$ is uniformly bounded, we get that $\Psi_F(\lambda) = 0$. By Lemma \ref{lem:injective-derivative}, $\lambda = \id$. Therefore, $\ker \phi = \set{\id}$ and $\phi$ is injective, as claimed.
\end{proof}

\section{Detecting direct products}
\label{sec:prod-struct}

Fix a cone transitive $C^r$ totally Cartan  action $\R^k \curvearrowright X$  with trivial Starkov component, with { $r = 2$} or $r = \infty$. { We continue assuming that $r' = (1,\theta)$ or $\infty$, respectively.} We assume throughout this section that $\pi : X \to Y$ is a $C^{r'}$ rank one factor, with homomorphism $\sigma : \R^k \to \R$ and Anosov flow $\psi_t : Y \to Y$ such that $\pi(a \cdot x) = \psi_{\sigma(a)}(\pi(x))$. Let $\pm \alpha$ denote the symplectic pair of weights whose corresponding foliations project to the stable and unstable manifolds of $\psi_t$, and %{ Assume that the action of $\sigma(\R^k)$ on $Z$ is embedded in the product of Anosov flows, and that there is a locally free Lie group action $G \curvearrowright X$ such that $G \cdot x = \pi^{-1}(x)^\circ$ for every $x \in X$, and the action of $\ker \sigma$ on $\pi^{-1}(x)^\circ$ is homogeneous in the coordinates provided by $G$.
% Finally, assume that every coarse Lyapunov foliation $W^\beta$ either has its leaves everywhere tangent to the fibers of $\pi$ (we denote the set of such coarse exponents $\beta$ by $\Delta_F$), or has its leaves nowhere tangent to the fibers of $\pi$ (we denote the set of these coarse exponents by $\Delta_1$). This is exactly the output of Proposition \ref{prop:hom-fibers} and Lemmas \ref{lem:product-factor} and \ref{cor:full-fiber}.} % We may assume this because either no non-Kronecker rank one factor exists, or Proposition \ref{prop:case1-finished} holds.} Let
 $H_0$ denote $\ker \alpha$, so that $H_0$ has a dense orbit in each $\pi^{-1}(y)$, $y \in Y$. We require the following definition to understand the behavior of $H_0$ on each fiber $\pi^{-1}(y)$. %, and for ease of notation, let $r' = (1,\theta)$ or $\infty$ depending on whether $r = 2$ or $\infty$, respectively. 

\begin{definition}
\label{def:splitting}
Let $\pm \alpha \in \Delta$ be a pair of weights such that the $W^{\pm \alpha}$-foliations  descend to the stable and unstable foliations of some $C^{r'}$ rank one factor $Y$, $r' = (1,\theta)$ or $\infty$. We say that $\alpha$ is a {\it splitting weight} if $L_\alpha := \displaystyle\bigcap_{\beta \not= \pm \alpha} \ker \beta$ is nontrivial.
\end{definition}

\begin{lemma}
\label{lem:direct-implies-splitting}
Let $\R^k = \R \times \R^{k-1} \curvearrowright Y \times \bar{X}$ be the direct product of an Anosov flow on $Y$ and a totally Cartan action $\R^{k-1} \curvearrowright \bar{X}$ with trivial Starkov component. Then if $\alpha(t,a) = t$ for $(t,a) \in { \R} \times \R^{k-1}$, $\alpha$ is a splitting weight.
\end{lemma}

\begin{proof}
Notice that if $\beta \in\Delta$, then either $\beta = \pm \alpha$ and $W^\beta$ is a stable or unstable foliation on $Y$, or $W^\beta$ is a foliation in $\bar{X}$. Since the action is a product action, the Anosov flow on $Y$ is isometric on $W^\beta$. That is, $\R \times \set{0} \subset \ker \beta$ for all $\beta \in \Delta$. Therefore, $\R \times \set{0} \subset L_\alpha$ is nontrivial.
\end{proof}

Much of the rest of this section is proving the converse to Lemma \ref{lem:direct-implies-splitting}, which is true up to finite cover (Theorem \ref{thm:splitting-implies-direct}).

\begin{lemma}
\label{lem:splitting-equivalence}
If $\R^k \curvearrowright X$ is a cone transitive action with trivial Starkov component, and $\pm \alpha$ is a pair of weights descending to some rank one factor $Y$, then the following are equivalent:

\begin{enumerate}
\item $\alpha$ is a splitting weight,
\item $\alpha$ is not a linear combination of weights from $\Delta \setminus \set{\pm \alpha}$, %For every $\beta_1,\beta_2 \in \Delta \setminus \set{\pm \alpha}$, $s\beta_1 + t\beta_2 \not= \alpha$ for all $s,t\in\R$,
\item $\dim(L_\alpha) = 1$, {  and}
\item $\R^k = \ker \alpha \oplus L_\alpha$.
\end{enumerate}

Furthermore, if $\alpha$ is a splitting weight, the action of $\ker \alpha$ on $\pi^{-1}(y)$ is totally Cartan for every $y \in Y$.
\end{lemma}

\begin{proof}
Assume that $\alpha$ is a splitting weight (ie, (1) holds), so that $L_\alpha$ is nontrivial. We will show (2). %Let $\beta_1,\beta_2 \in \Delta \setminus \set{\pm \alpha}$. 
Assume for a contradiction that $\alpha = \sum t_i\beta_i$ is a linear combination of the weights $\beta_i$. Then for any $a \in L_\alpha$, $\alpha(a) = \sum t_i\beta_i(a) = 0$. Therefore, $a \in \ker \beta$ for all $\beta \in \Delta$ (including $\beta = \alpha$), and is therefore part of the Starkov component. Since the Starkov component is assumed to be trivial, this is a contradiction. Therefore, we have (2).

If (2) holds, we claim that $\dim(L_\alpha) = 1$ (ie, (3) holds). Consider the subspace $V \subset (\R^k)^*$, which is the linear span of the weights $\beta \in \Delta \setminus \set{\pm \alpha}$. By (2), $\alpha \not\in V$, so $\dim V < k$. On the other hand, $V \oplus \R \alpha = (\R^k)^*$, since the set of $a \in \R^k$ killed by $V \oplus \R \alpha$ is exactly the Starkov component. Therefore, $\dim(V) = k-1$. But $L_\alpha = \set{ a \in \R^k : \beta(a) = 0 \mbox{ for all } \beta \in V}$, and the standard duality implies that $\dim(L_\alpha) + \dim(V) = k$. Therefore, $\dim(L_\alpha) = 1$.

The equivalence of (3) and (4) is immediate once we know that $\ker \alpha \cap L_\alpha = \emptyset$. This follows from the assumption that the action has trivial Starkov component. It is also immediate that (4) implies (1), since $\dim(\ker \alpha)= k - 1$.

We now show that if $\alpha$ is a splitting weight, then the $\ker \alpha$ action is totally Cartan. Indeed, notice that $T\pi^{-1}(y) = T(\ker \alpha) \oplus \bigoplus_{\beta \not= \pm \alpha} E^\beta$. Since the action of $\ker \alpha$ is the restriction of the $\R^k$ action, we recover the totally Cartan property immediately once we show that no two weights of the $\ker \alpha$ action are positively proportional. That is, the $\ker \alpha$ action is not totally Cartan if and only if there exist distinct weights $\beta_1,\beta_2 \in \Delta \setminus \set{\pm \alpha}$ such that $\beta_1|_{\ker \alpha} = \beta_2|_{\ker \alpha}$. This occurs if and only if $(\beta_1-\beta_2)|_{\ker \alpha} = 0$. But then $\beta_1 - \beta_2$ is proportional to $\alpha$ on $\R^k$. If $\beta_1 - \beta_2 = 0$ on $\R^k$, then they are the same weight, which is a contradiction. Otherwise, $\beta_1 - \beta_2 = c\alpha$, $c \not= 0$, and $\alpha$ is a linear combination of weights in $\Delta \setminus \set{\pm \alpha}$.
\end{proof}

%\begin{lemma}
%If $\alpha$ is a splitting weight, and $y' \in W^\alpha(y) \subset Y$, then the holonomy map $f_{y,y'}^\alpha : \pi^{-1}(y) \to \pi^{-1}(y')$ is $C^r$.
%\end{lemma}

%{  Fix some $\beta \in \Delta_1$ and consider the foliations $W^\beta_X$, and their projections $W^\beta_Z$. Each $W^\beta_Z$, $\beta \in \Delta_1$, is a foliation by the assumption that each $W^\beta_X$ is nowhere tangent to the fibers. Fix some leg along $W^\beta_Z$ which begins and ends at some $z,z' \in Z$, respectively. Notice that this leg defines a holonomy $f_{z,z'} : \pi^{-1}(z) \to \pi^{-1}(z')$, respectively, by lifting this $W^\beta_Z$-leg to paths in the $W^\beta_X$ foliations in $X$.

\begin{lemma}
\label{lem:holonomies-smooth}
Let $\R^k \curvearrowright X$ be a cone transitive, totally Cartan action with rank one factor $\pi : X \to Y$ as described above, and assume the corresponding weight $\alpha$ is a splitting weight. Then if the $W^\alpha$-foliation on $X$ projects to the stable foliation on $Y$, and $y' \in W^s(y)$, then the holonomy maps $f_{y,y'}^\alpha : \pi^{-1}(y) \to \pi^{-1}(y')$ along the $W^\alpha$-foliations satisfy

\begin{itemize}
\item $f_{y,y'}^\alpha$ commutes with the $\ker \alpha$-action, 
%\item $f_{y,y'}^\alpha(z) \in W^\alpha(z)$ for all $z \in \pi^{-1}(y)$,
\item $f_{y,y'}^\alpha$ is a $C^{r'}$ diffeomorphism, %, and
\item $f_{y,y'}^\alpha$ is $C^{r'}$ %varies H\"older continuously in the $C^r$ topology 
as $y'$ varies along $W^s(y)$,
\item $f_{y,y'}^\alpha$ varies continuously in the $C^{r'}$-topology in the variables $y$ and $y'$.
\end{itemize} 
\end{lemma}

\begin{proof}
The first claim is immediate from the fact that $\ker \alpha$ preserves the fibers of $\pi$, and intertwines the leaves of $W^\alpha$. Since the local holonomies are defined as the intersection of $\pi^{-1}(y')$ and $W^\alpha_{\loc}(y)$, we get that $\ker \alpha$ commutes with the holonomies.

To see that it is a $C^r$-diffeomorphism, notice that we have assumed that $\alpha$ is a splitting weight. Therefore, we may choose $a \in L_\alpha$ such that $\alpha(a) = -1$. Choose some $b \in \R^k$ such that $a+tb$ is Anosov %not contained in any Lyapunov hyperplane
 for sufficiently small $t \in \R$. Then consider a small number $\delta$ such that both $a+\delta b$ and $a - \delta b$ are Anosov, and the splitting

\[ TX = E^\alpha \oplus \left( T\mc O \oplus \bigoplus_{\beta \not= \pm \alpha} E^\beta \right) \oplus E^{-\alpha} \]
is an invariant partially hyperbolic splitting, where $E^\alpha$ is expanding, $E^{-\alpha}$ is contracting, and the remaining bundles are neutral for $a$. Write 

\[ T\pi^{-1}(y) = T\mc O' \oplus \bigoplus_{\substack{\beta(b) > 0 \\ \beta \not= \pm \alpha}} E^\beta \oplus \bigoplus_{\substack{\beta(b) < 0 \\ \beta \not= \pm \alpha}} E^\beta = T\mc O' \oplus \hat{E}^u \oplus \hat{E}^s,\] where $\mc O'$ is the tangent bundle to the $\ker \alpha$-orbit foliation. Notice that the second and third terms of the sum are integrable, and the leaves are exactly the unstable and stable manifolds of $a + \delta b$ in the fibers, considering $a +\delta b$ as a transformation of the fiber bundle. Furthermore, the element $a + \delta b$ has stable distribution $E^\alpha \oplus \hat{E}^s$ on $X$, and since $a \in \ker\alpha$, for $\delta$ sufficiently small, $\hat{E}^s$ is the slow foliation and $E^\alpha$ is the fast foliation for $a + \delta b$, as in Definition \ref{def:slow-foliation} and subsequent Remark \ref{rem:fast-distribution}. Therefore, $E^\alpha$ is uniformly $C^r$ along $\hat{W}^s$, the foliation tangent to $\hat{E}^s$, and the corresponding $W^\alpha$ holonomies are uniformly $C^r$ along $\hat{W}^s$. Considering $a - \delta b$,  by the same arguments, we conclude that the $W^\alpha$ holonomies are uniformly $C^r$ along $\hat{W}^u$. Since they commmute with the $\ker \alpha$-action, they are immediately $C^r$ along $\mc O'$.

Now, since $\hat{W}^s$ and $\mc O'$ are complementary foliations inside the weak stable manifolds subordinate to the fiber, Journ\'{e}'s Theorem \ref{thm:journe} implies that the $W^\alpha$-holonomies are smooth along the weak stable manifolds. Since the weak stable manifolds together with $\hat{W}^u$ are complementary in $X$, it follows again Theorem \ref{thm:journe} that the holonomies are uniformly $C^{r'}$ as maps from one fiber to another.

To see the smoothness of $f_{y,y'}^\alpha$ as $y'$ varies along $W^s(y)$, notice that we know that each individual $f_{y,y'}^\alpha$ is a $C^{r'}$ transformation, and that $\pi^{-1}(W^s(y))$ has a canonical  $C^{r'}$ manifold structure as an immersed submanifold of $X$. Therefore, we may pick a $C^{r'}$ metric on the leaf $W^s(y)$, which corresponds to a flow which parameterizes the leaf at unit speed. This flow lifts to a flow on $\pi^{-1}(W^s(y))$, and each time $t$ map is a $C^{r'}$ transformation by Theorem \ref{thm:journe}, as it is $C^{r'}$ along its orbits and each fiber, which form uniformly transverse H\"older foliations. Since each time $t$ map of the local flow is $C^{r'}$, it follows that the local flow is $C^{r'}$ (see, eg, \cite[Corollary, p. 202]{montgomery-zippin}). Hence the holonomies $f_{y,y'}^\alpha$, the time $t$ maps of the flow, are $C^{r'}$ as $y'$ varies along $W^s(y)$.

To see the global continuity of $f_{y,y'}^\alpha$ in the $C^{r'}$ topology, we note that by combining with the previous claim, it suffices to show that $f_{y,y'}^\alpha$ varies continuously as $y$ and $y'$ are moved along their $\psi_t$-orbits and $W^u$-holonomies. It is clear using the intertwining properties of the $W^\alpha$-leaves with the $\R^k$-action that $f_{\psi_t(y),\psi_t(y')}^\alpha = (ta) f_{y,y'}^\alpha (-ta)$, where $a \in L_\alpha$ is the unique element which maps to $\psi_1$. It is immediate from this equation that as $t \to 0$, $f_{\psi_t(y),\psi_t(y')}^\alpha \to f_{y,y'}^\alpha$ in the $C^{r'}$ topology.

So we must show that $f_{y,y'}^\alpha$ is continuous along $W^u(y)$. That is, we let $y_1 \in W^u(y)$, and $y_1' \in W^{cs}(y_1)$ be the unique local intersection with $W^u(y')$. Then since $y_1' \in W^{cs}(y_1)$, there exists a unique $y_2 \in W^s(y_1)$ such that $y_1'= \psi_t(y_2)$ for some small $t$. We wish to show that $f_{y_1,y_1'}^\alpha$ converges to $f_{y,y'}^\alpha$ in the $C^{r'}$ topology as $y_1 \to y$. To this end, notice that continuous variation in the $C^0$ topology follows immediately from the continuity of the $W^\alpha$-foliation on $X$. So we must show that the derivatives are close. Fix $z \in \pi^{-1}(y)$, and let 

\[z' = f_{y,y'}^\alpha(z), \qquad  z_1 = f_{y,y_1}^{-\alpha}(z), \qquad z_2 = f_{y_1,y_2}^\alpha(z_1), \quad \mbox{and} \quad z_1' = f_{y',y_1'}^{-\alpha}(z'). \]

We claim that there exists some $a \in \R^k$ covering $\psi_t$ such that $z_1' = a \cdot z_2$. Indeed, we may use local product structure to connect $z_1'$ and $z_2$ by first moving along a piece of $\R^k$ orbit, then along $W^\alpha$, then along $W^{-\alpha}$, then along the foliations $\hat{W}^s$ and $\hat{W}^u$. Because $y_1' = \psi_t(y_2)$, it follows that the $W^{\pm \alpha}$ legs are trivial and the $\R^k$-orbit leg  covers $\psi_t$. Finally, pick an element $b' \in \ker \alpha$ such that $\hat{W}^u = W^u_b$ and $\hat{W}^s = W^s_b$. Then $d((tb') \cdot z_1',(tb')\cdot z_2)$ is uniformly small, since both $z_1'$ and $z_2$ are connected to $z$ by a short ($\alpha$, $-\alpha$)-path. It follows that the $\hat{W}^u$-leg of the path connecting $z_1'$ and $z_2$ is trivial. By considering $-b'$, we conclude that the $\hat{W}^s$-leg is also trivial. Hence the only nontrivial leg is the $\R^k$-orbit, which must cover $\psi_t$.

Now, as $y_1 \to y$, we know that $z_1 \to z$, $a \to 0$ and $z_1' \to z'$. If we can show that
\begin{equation}
\label{eq:contractible-loop} f_{y_1,y_2}^\alpha = -a \of f_{y',y_1'}^{-\alpha} \of f_{y,y'}^\alpha \of  f_{y_1,y}^{-\alpha}, \end{equation}
we may conclude continuity in the $C^{r'}$ topology from the $C^{r'}$ variation along the leaves themselves. To see this, note that we may obtain \eqref{eq:contractible-loop} by showing that going along the loop of holonomies
\[ F = f_{y',y}^\alpha \of f_{y_1',y'}^{-\alpha} \of a \of f_{y_1,y_2}^\alpha \of f_{y,y_1}^{-\alpha} \]
 yields the identity action on the fiber. Since we know that each holonomy intertwines the $\ker \alpha$-action, going along a loop centralizes it. Since $\alpha$ is a splitting weight, the action on the fiber is totally Cartan by Lemma \ref{lem:splitting-equivalence}. Hence, its centralizer is a discrete extension of $\ker \alpha$ by Lemma \ref{lem:discrete-centralizer}. Since this loop of holonomies can be contracted, it follows that going around a loop of holonomies acts via an element of $\ker \alpha$ (ie , $F \in \ker \alpha$). But since we chose $a$ so that the loop closes up at $z$ (ie, $F(z) = z$), it follows that $F = \id$. Hence, the loop is trivial and \eqref{eq:contractible-loop} holds.
\end{proof}

\begin{lemma}
\label{lem:isometric-holonomies}
{ If $\alpha$ is a splitting weight, there exists a family of metrics $\norm{\cdot}_x : T_x[\pi^{-1}(\pi(x))] \to \R$ %on $T\pi^{-1}(y)$ %$TW^{\pm \alpha}$
 such that $\norm{\cdot}_x$ %is $C^\infty$ along $\pi^{-1}(y)$ for every $y \in Y$,
  is H\"older as $x$ varies in $X$, and invariant under $L_\alpha$ and the $W^{\pm \alpha}$ holonomies.}
\end{lemma}

\begin{proof}
 The proof follows the proof of the Livsic theorem. Fix a point $y_0 \in Y$ such that the orbit of $y_0$ under the Anosov flow $\psi_t$ on $Y$ is dense, and an arbitrary $C^{r'}$ metric $\norm{\cdot}_0$ on $\pi^{-1}(y_0)$. Recall that by Lemma \ref{lem:splitting-equivalence}, the action of $\ker \alpha$ on $\pi^{-1}(y_0)$ is totally Cartan. Therefore, if we consider the group 

\begin{multline*}
Q = \{ \lambda \in Z_{\Diff^1}(\ker \alpha \curvearrowright { \pi^{-1}(y_0)}) : d\lambda^n(x)|_{TW^\beta} \mbox{ is uniformly bounded,}\\ \mbox{for all }n \in \N, x \in \pi^{-1}(y_0) \mbox{ and } \beta \in\Delta\setminus \set{\pm \alpha} \},
\end{multline*}
we know that $Q$ is finite by Corollary \ref{cor:q-finite}. We first prove the lemma when %for every $y \in Y$ the $\ker\alpha$-action on $\pi^{-1}(y)$ has no finite order elements of its centralizer
$Q = \set{e}$, then deduce the lemma in general.

 Denote $y_t = \psi_t(y_0)$ and $a \in L_\alpha$ to be the unique element which projects to $\psi_1$. Consider the metrics $\norm{v}_t = \norm{(-ta)_*v}_0$ on $\pi^{-1}(y_t)$. We claim that if there exists $t_k \to \infty$ such that $y_{t_k} \to y$, then $\norm{\cdot}_{t_k}$ converges to a well-defined metric on $\pi^{-1}(y)$, and that this family of metrics is unique and satisfies the properties described by the lemma.
 
Indeed, the maps $t_ka : \pi^{-1}(y_0) \to \pi^{-1}(y_{t_k})$ all have uniformly bounded derivatives by Lemma \ref{lem:uniformly bounded derivative} and the definition of $L_\alpha$. They therefore converge as uniformly Lipschitz maps to a map $\tau : \pi^{-1}(y_0) \to \pi^{-1}(y)$. We claim that $\tau$ is also $C^{r'}$. Indeed, notice that each $t_ka$ preserves each coarse Lyapunov foliation, and restricted to a given coarse Lyapunov leaf $W^\beta(z) \subset \pi^{-1}(y)$, is linear in the normal forms coordinates from $W^\beta(z)$ to $W^\beta(t_ka \cdot z)$ (see Section \ref{sec:normal-forms}). Hence if $t_ka \cdot z \to z'$, $\lim t_ka : W^\beta(z) \to W^\beta(z')$ is linear in the normal forms coordinates, and hence $C^{r'}$. Therefore, restricted to each coarse Lyapunov leaf contained in $\pi^{-1}(y_0)$, the map $\tau$ is $C^{r'}$. By iteratively applying Journ\'{e}'s theorem (Theorem \ref{thm:journe}) as in the end of Theorem \ref{lem:foliation1}, it follows that $\tau$ is $C^{r'}$ as a map from $\pi^{-1}(y_0)$.

We therefore define $\norm{v}_y = \norm{d\tau^{-1}(v)}_0$ when $v \in T\pi^{-1}(y)$. This is well-defined as long as $\tau$ is unique. But if $\tau$ and $\tau'$ are both obtained as limits of maps $t_ka$ and $s_ka$, respectively, then $\tau^{-1} \of \tau'$ is a map from $\pi^{-1}(y_0)$ to itself which commutes with the $\ker \alpha$ action. Furthermore, it is clearly an element of $Q$, since it a composition of limiting elements of $L_\alpha$. Since we have assumed $Q = \set{e}$, it follows that $\tau = \tau'$ and the norm is uniquely defined.
 
We claim that the metrics $\norm{\cdot}_y$ are H\"older continuous. It suffices to show that the metrics restricted to the bundles $E^\beta$, $\beta \not= \pm \alpha$ are H\"older continuous, since together with the $\ker \alpha$-orbit foliation, they span $T\pi^{-1}(y)$ for every $y$. Consider a reference norm $\norm{\cdot}_0$ on $X$. Since each $E^\beta$ is 1-dimensional, we may consider the functions $f_\beta : X \to \R$ defined by the relation $\norm{v}_y = f_\beta(y) \norm{v}_0$ for any $v \in E^\beta$.  We will show that $f_\beta$ is H\"older continuous.

Indeed, suppose that $x,x' \in X$ are two close points, and consider $y = \pi(x)$ and $y' = \pi(x')$. Then if $t_k$ and $s_k$ are sequences such that $y_k = \psi_{t_k}(y_0) \to y$ and $y_k' = \psi_{s_k}(y_0) \to y'$, $\psi_{s_k-t_k}(y_k) = y_k'$. Since $y_k$ and $y_k'$ are converging to $y$ and $y'$ which are close, the Anosov closing lemma applies and we may choose a periodic point $p_k$ with the following properties:

\begin{itemize}
\item $y_k$ is reached from $p_k$ using a short $W^s$- leg, then a short $W^u$-leg. The lengths of the $W^u$- and $W^s$-legs depends H\"older continuously on $d(y_k,y_k')$.
\item If $L_k$ is the period of $p_k$ then $\abs{L_k - (s_k-t_k)} < Cd(y_k,y_k')^{\theta_0}$ for some $\theta_0 \in (0,1)$.
\item There exists some $C_1 > 0$ and $\lambda > 1$ such that $d(\psi_t(y_k),\psi_t(p_k)) < C_1\lambda^{-\min\set{t,L_k-t}}d(y_k,y_k')$ for all $t \in [0,L_k]$.
\end{itemize}

Now, let $q_k \in \pi^{-1}(p_k)$ be the point which is obtained from $x_k$ by applying the $W^{\alpha}$- and $W^{-\alpha}$-holonomies of the fixed $W^s$- and $W^u$-path connecting $y_k$ and $p_k$. Then since $p_k$ is $\psi_t$-periodic, $q_k$ is $L_\alpha$-periodic, since $L_ka : \pi^{-1}(p) \to \pi^{-1}(p)$ has uniformly bounded derivatives and therefore belongs to $Q = \set{\id}$. Since the points $\psi_t(p_k)$ and $\psi_t(y_k)$ stay close on $Y$ for $t \in [0,L_k]$, and $W^\alpha$- and $W^{-\alpha}$-holonomies are intertwined by $L_\alpha \subset \R^k$, if follows that $q_k$ is connected to $L_ka \cdot x_k$ by a short path consisting of a single $W^\alpha$-leg and single $W^{-\alpha}$-leg (it is the pushforward of the connection between $x_k$ and $q_k$). 

Summarizing, if $\rho_k$ denotes the path in $Y$ connecting $p_k$ to $y_k$, and $\rho' = (\psi_{L_k})_*{\rho_k}$ denotes the pushforward of this path connecting $\psi_{L_k}(y_k)$ and $p_k$, then both $\rho_k$ and $\rho_k'$ are paths whose lengths have a H\"older estimate in $d(y_k,y_k')$, and $(L_ka) \of f_{\rho_k} = f_{\rho_k'} \of (L_ka)$, where $f_{\rho_k}$ and $f_{\rho_k'}$ are the holonomy actions. Since $L_ka$ acts as the identity on $\pi^{-1}(p_k)$, it follows that $L_ka : \pi^{-1}(y_k) \to \pi^{-1}(\psi_{L_k}(y_k))$ is exactly equal to $f_{\rho_k'} \of f_{\rho_k}^{-1}$, and its derivatives are H\"older close to the identity by Lemma \ref{lem:holonomies-smooth}. Finally, since $\abs{(s_k-t_k)-L_k} < Cd(y_k,y_k')^{\theta_0}$, it follows that the derivative of $(s_k-t_k)a$ on $\pi^{-1}(y_k)$ is close to the identity with H\"older estimates. In particular, we get that 

\begin{multline*}
 f_\beta(x) = \lim_{k \to \infty} \norm{d(t_ka)|_{E^\beta}(x_0)} = \lim_{ k \to \infty} \norm{d(t_k -s_k)a|_{E^\beta}(as_k \cdot x_0)}\cdot \norm{d(s_ka)|_{E^\beta}(x_0)} \\ \sim (1\pm Cd(x,x')^{\theta_0})\cdot f_\beta(x').
 \end{multline*}
Hence $f_\beta$ is H\"older and we have verified the desired properties of the metric.

To finish the construction of the metric in general, we remark that if $Q \not= \set{e}$, then while $Q$ does not have a global action on the manifold, it does induce an equivalence relation in the following way: we say that $x_1,x_2 \in X$ are related if $\pi(x_1) = \pi(x_2)$ and there exists some $x \in \pi^{-1}(y_0)$, $q \in Q$ and sequences $t_k,s_k \in \R$ such that $x_1 = \lim t_ka \cdot x$ and $x_2 = \lim s_k a \cdot (qx)$. This is an equivalence relation in which every equivalence class has cardinality $\abs{Q}$. We may therefore pass to the finite factor induced by the equivalence relation yielding a space in which the maps preserving $\pi^{-1}(y_0)$ obtained as limits of multiples of $a$ are all the identity (ie, on the factor, $Q = \set{e})$. We may therefore build a metric as above, and lift it to the covering as necessary.

We have therefore produced a family of metrics which are $L_\alpha$-invariant. Note that this immediately implies invariance under holonomies: if $x' \in W^\beta(x)$ for some $\beta$ linearly independent with $\alpha$ and $x,x'$ cover $y,y' \in Y$, respectively, choose some $a \in \ker \alpha$ such that $\beta(a) < 0.$ The holonomy satisfies
\[ f_{y,y'}^\alpha(x,x') =  a^{-n} f_{a^ny,a^ny'}^\alpha a^n. \]
Since $a^{\pm n}$ preserves the constructed metric, and the middle term  in the right hand side is a holonomy which converges to the identity transformation in $C^{r'}$ by Lemma \ref{lem:holonomies-smooth}, the transformation is an isometry.
%Indeed, if $t_ka$ and $s_ka$ are both sequences determining candidates, then $\psi_{s_k-t_k}(y_{t_k}) = y_{s_k}$. For sufficiently large $k$, $y_{t_k}$ and $y_{s_k}$ are close, so there exists a $\psi$-periodic orbit $y'$ by the Anosov closing lemma. Then $u_ka$ is a $C^r$ map from $\pi^{-1}(y')$ to itself commuting with the $\ker \alpha$-action for some $u_k$ close to $s_k-t_k$. Furthermore, since $a \in L_\alpha$, $a \in \ker \Psi_F$, where $\Psi_F$ is as in Lemma \ref{lem:injective-derivative} (notice that $a \in \Lambda^F$ since we have assumed that there are no finite-order elements of the centralizer). Therefore, since $\Psi_F$ is injective, $u_ka$ acts by the identity on $\pi^{-1}(y')$. Therefore, $(s_k-t_k)a$ acts by a map close to the identity, and converges to the identity as $k \to \infty$. This implies that the corresponding metric is well-defined.
\end{proof}

\begin{theorem}
\label{thm:splitting-implies-direct}
Let $\pi : X \to Y$ determine a $C^{r'}$ rank one factor, and $\alpha$ be the weight for which $W^{\pm \alpha}$ descends to the stable and unstable manifolds of the Anosov flow $\psi_t$ on $Y$. Assume that $\alpha$ is a splitting weight, $a \in L_\alpha$ is chosen so that $\pi(a\cdot x) = \psi_1(\pi(x))$ for all $x \in X$. Then if $y_0 \in Y$, and $\bar{X} = \pi^{-1}(y_0)$, there exists a finite cover $\tilde{Y}$ of $Y$ with corresponding lifted flow $\tilde{\psi}_t$, and a finite-to-one $C^{r'}$ covering map $H : \tilde{Y} \times \bar{X} \to X$ such that for all $t \in \R$ and $b \in \ker\alpha$:

\[ H(\tilde{\psi}_t(y),b\cdot x) = (ta+b) \cdot H(y,x). \]
\end{theorem}

\begin{proof}
We first construct $\tilde{Y}$ in the following way: let $\mc P_{y_0}$ denote the set of paths with $W^s$, $W^u$, and $\psi_t$-orbit legs based at $y_0$, and $\mc C_{y_0} \subset \mc P_{y_0}$ denote the set of cycles in such paths. Then $Y$ is canonically identified with the set $\mc P_{y_0} / \mc C_{y_0}$. Furthermore, $\mc C_{y_0}$ acts on $\bar{X} = \pi^{-1}(y_0)$ by compositions of the corresponding $\pm \alpha$-holonomies and the flow $L_\alpha$.

Notice that by Lemma \ref{lem:smooth-holonomy} and Lemma \ref{lem:isometric-holonomies}, the action of $\mc C_{y_0}$ is $C^{r'}$ and isometric in some H\"older norm on $\bar{X}$ and commutes with the $\ker \alpha$ action on $\bar{X}$. Hence it belongs to the finite group $Q$ of Corollary \ref{cor:q-finite}. That is, there exists a homomorphism $\tau : \mc C_{y_0} \to Q$ which associates to a cycle its holonomy action on the fiber. Notice that the space $\mc P_{y_0}$ carries a canonical topology by embedding each combinatorial pattern of paths with combinatorial length $n$ into $Y^n$, and making identifications with a pattern of lower complexity when a leg of a path is trivial. With this topology, the set $\mc C_{y_0}$ is closed, since the map associating a path to its endpoint is clearly continuous.

Furthermore, $\mc C_{y_0}$ is locally path connected. Indeed, if $\gamma$ is a cycle in $W^u$, $W^s$ and $\psi_t$-orbits on $Y$, let $\gamma_t$ denote its retract, so that $\gamma_0 = \gamma$, $\gamma_1$ is the trivial path and $\gamma_t$ is a path which goes up to time $t$ along $\gamma$, after parameterizing $\gamma$ as a cycle from $[0,1]$. Then while $\gamma_t$ is not a cycle, if it is contained in a sufficiently small neighborhood, using local product structure, there exists a uniquely defined, continuously varying path $\rho_t$ connecting the endpoint of $\gamma_t$ with $y_0$. Then the one-parameter family $\tilde{\gamma}_t = \rho_t * \gamma_t$ connects the path $\gamma$ to the trivial path in the space $\mc C_{y_0}$. 

 Finally, since the holonomies and maps $ta \in L_\alpha$ vary continuously with the basepoint, the homomorphism $\tau$ is continuous.

 Let $\bar{\mc C} = \ker \tau$, and notice that since $Q$ is finite, $\bar{\mc C}$ is a finite index subgroup of $\mc C_{y_0}$ containing the path component of $\mc C_{y_0}$. Hence $\tilde{Y} = \mc P_{y_0} / \bar{\mc C}$ is a finite extension of $Y$ by construction. Hence, $\tilde{Y}$ carries a canonical $C^{r'}$ manifold structure by using local charts from $Y$.

Now, we may build our map $H$. Indeed, if $y \in \tilde{Y}$ and $x \in \bar{X}$, then $y$ is represented by some path $\rho$ in the $W^u$, $W^s$ and $\psi_t$-orbit foliations based at $y_0 \in Y$. Let $f_\rho : \pi^{-1}(y_0) \to \pi^{-1}(e(\rho))$ denote the induced composition of holonomy maps. Then define $H(y,x) = f_\rho(x)$.

We claim that $H$ is well-defined and satisfies the conclusions of the theorem. We first show that it is well-defined. Indeed, if $\rho'$ is another path representing $y \in \tilde{Y}$, then $\rho' * \rho^{-1} \in \bar{\mc C}$. By definition, $\rho' * \rho^{-1}$ acts trivially on $\bar{X}$, and we conclude that $f_\rho = f_\rho'$.

We now check the intertwining properties. Indeed, if $\rho$ is a path on $Y$, let $\theta_{[0,t]}$ denote the path with one leg moving distance $t$ along the $\psi_t$-orbit foliation, which begins at the endpoint of $\rho$. Then $\theta_{[0,t]} * \rho$ represents the point $\tilde{\psi}_t(y)$ on $\tilde{Y}$. Hence:
\[ H(\tilde{\psi}_t(y),x) = f_{\theta_{[0,t]} * \rho}(x) = f_{\theta_{[0,t]}} \of f_\rho(x) = ta \cdot H(t,x). \]
Similarly, if $b \in \ker \alpha$, then by Lemma \ref{lem:smooth-holonomy} and the fact that $a$ commutes with $\ker \alpha$, $b \of f_\rho = f_\rho \of b$ for any path $\rho \in \mc P_{y_0}$. Therefore,
\[ H(y,b\cdot x) = f_\rho(b\cdot x) = b \cdot f_\rho(x) = b\cdot H(y,x). \]
These two intertwining properties imply the combined intertwining property.

We will now show that $H$ is finite-to-one. Indeed, observe that if $H(y_1,x_1) = H(y_2,x_2)$, then $f_{\rho_1}(x_1) = f_{\rho_2}(x_2)$ for some representatives $\rho_i$ of $y_i$, $i= 1,2$. This implies that $f_{\rho_2}^{-1} \of f_{\rho_1}(x_1) = x_2$, and that $\sigma = \rho_2^{-1} * \rho_1$ is a cycle on $Y$. The action on the fiber belongs to $\tau(\mc C_{y_0}) \subset Q$, a finite group. This finite group element must determine the point $x_2$ from $x_1$. Hence, $(y_2,x_2) = q \cdot (y_1,x_1)$ for some $q \in Q$, and the map $H$ is finite-to-one.

Finally, we show that $H$ is $C^{r'}$. We prove this using Journe's theorem \ref{thm:journe}, by showing that is is uniformly $C^{r'}$ along $\tilde{Y}$ and $\bar{X}$. That it is $C^{r'}$ along $\bar{X}$ follows because for a fixed $\rho$, the composition of holonomies $f_\rho$ is $C^{r'}$ by Lemma \ref{lem:holonomies-smooth}. That it is $C^{r'}$ along $\tilde{Y}$ follows again from iterating Theorem \ref{thm:journe}: restricted to each $W^{s/u}(y) \times \set{x}$, the map is $C^{r'}$ since the $W^{\pm \alpha}$ leaves are $C^{r'}$-embedded submanifolds, and the smooth structure on $Y$, and hence each $W^{s/u}(y)$-leaf, is inherited from the projection map $\pi : X \to Y$. Then since they jointly integrate with the $L_\alpha$-leaf foliation, which has $C^{r'}$-embedded leaves, it follows that the weak unstable manifolds are $C^{r'}$-embedded. Our arguments have shown that the $W^s$ and $W^{cu}$ leaves jointly integrate (since they are exactly the images of the $\tilde{Y}$-leaf in $\tilde{Y} \times \bar{X}$), so applying Theorem \ref{thm:journe} implies that each $\tilde{Y}$ leaf is $C^{r'}$ embedded.
\end{proof}

}

\section{Applications of the Main Theorems}
\label{sec:corollaries}

\begin{proof}[Proof of Corollary \ref{cor:Zk-actions}]
Let $\Z^\ell \curvearrowright X$ be a cone transitive, totally Cartan action. Let $\tilde{X}$ denote the suspension space of $X$, equipped with a transitive, totally Cartan action of $\R^\ell$. By Lemma \ref{lem:starkov-factor}, the Starkov component $S \subset \R^\ell$ acts through a free torus action. In particular, it must be a rational subtorus in the torus projection $p : \tilde{X} \to \mathbb{T}^\ell$ determining the suspsension. Let $k = \ell - \dim(S)$, and choose any rational $k$-dimensional subspace $V$ transverse to $S$. Then let $a_1,\dots,a_k \in \Z^\ell$ be generators of this subspace. Notice that since the first return action of $S$ is by a finite group, and we wish to embed only a finite index subgroup, it suffices to consider the case of trivial Starkov component.

Rather than suspending $\Z^\ell \curvearrowright X$, we suspend the action of $\langle a_1,\dots,a_k\rangle$. Call the resulting space $\bar{X}$, and notice that by construction, the $\R^k$ action on $\bar{X}$ has trivial Starkov component. %One may suspend further if necessary to get a self-centralizing $\R^{m}$ action, $m \ge k$, on a space $\bar{Y}$ in which the action of $\R^k$ embeds as a rational hyperplane (as in the start of the proof of Theorem \ref{thm:full-classify}). 
The space $\bar{X}$ is still a suspension of an action on $X$ (since it is the suspension of a discrete action, on an existing suspension space) with a factor map $\bar{p} : \bar{X} \to \mathbb{T}^k$. %Then, after passing to a finite cover if necessary, $\bar{X}$ 

We claim that the generators of each Anosov flow on a 3-manifold, and the span of the generators of the homogeneous action, are all rational subspaces in $\R^{m}$. Indeed, notice that by Theorem \ref{thm:big-headache}, there exists a factor $\pi : \bar{X} \to Y_1 \times \dots \times Y_n$, which projects to a product of Anosov flows. By Lemma \ref{lem:susp}, each $Y_i$ corresponds to a rank one factor of the action of $\Z^k$ on $X$, and there exists a projection $\pi_X : X \to Z_1 \times \dots \times Z_n$, where $Y_i$ is a suspension of an Anosov diffeomorphism on $Z_i$. The dynamics along the fibers of $\pi_X$ is algebraic in the sense of Theorem \ref{thm:big-headache}, and since there exists an Anosov element, the group parameterizing the fibers must be nilpotent. Finally, the spaces $Z_i$ are all diffeomorphic to $\mathbb{T}^2$, since they are 2-manifolds carrying an Anosov diffeomorphism.
% Then $Y_i$ carries a generator of the Anosov flow which must be closed in the torus factor of $\bar{Y}$ since $Y_i$ is an Anosov flow (its closure must be a subtorus, and if its dimension was more than one, the rank of the factor would be greater than one). Therefore, the generator of the flow on the space $Y_i$ is rational. Similarly, the generators of the action of $G / \Gamma$ must correspond to a rational subspace of the torus factor $\mathbb{T}^m$, as it must be exactly complementary to the sum of the directions for each $Y_i$. This immediately implies that every $Y_i$, as well as the homogeneous flow on $G/ \Gamma$, are suspensions. After taking a finite index subgroup of $\Z^{m}$ we may without loss of generality assume that the smallest rational generators of the Anosov flows and action on $G / \Gamma$ are the standard generators of $\Z^{m}$. Thus the first return action must be isomorphic to a product of Anosov diffeomorphisms of 2-manifolds, and an affine Anosov action on a homogeneous space. Since the reductive component of the homogeneous space must be trivial by the Anosov condition, and any Anosov diffeomorphism of an orientable 2-manifold must be on a torus, we conclude the corollary.
\end{proof}

\begin{proof}[Proof of Corollary \ref{cor:global}] 
Suppose $G$ is a semisimple real linear Lie group of noncompact type without compact factors, and let $\Gamma \subset G$ be a cocompact lattice which projects densely onto any rank one factor of $G$, $M$ be a connected compact subgroup which intersects $\Gamma$ only at $\set{e}$, and $X = M \backslash G / \Gamma$ be the corresponding double homogeneous space.    Notice that if some finite cover of the action on $X$ is smoothly conjugate to a Weyl chamber flow, so is the original action on $X$ (since the action is determined by vector fields which commute with the deck transformations of the cover, so always lift and descend). So we allow ourselves to pass to finite covers throughout the proof.

Passing to a subgroup of finite index, we may assume that $\Gamma$ is torsion-free by Selberg's Lemma (see, eg, \cite[Theorem 4.8.2]{Morris-arithmetic}). %\cite[Lemma 9.1(?)]{mostow-book}.
If $K$ denotes a maximal compact subgroup of $G$ containing $M$, then $\Gamma$ acts freely and properly discontinuously on $K\setminus G$, the symmetric space attached to $G$. Through an application of the Borel density theorem, there is a decomposition of $G$ into an almost direct product of semisimple Lie groups, so that $G$ is finitely covered by $G_1 \times G_2 \times \dots \times G_n$ and if $\Gamma_i$ is the intersection of the image of $G_i$ in $G$ with $\Gamma$, then $\Gamma_i$ is a lattice in $G$, $\Gamma = \Gamma_1 \cdot \Gamma_2 \cdot \dots \cdot \Gamma_n$ and $\Gamma_i \cap \Gamma_j \subset Z(G)$ (see, for instance \cite[Proposition 4.3.3]{Morris-arithmetic} or \cite[Corollary 5.19]{raghunathan1972}). By assumption, the rank of each $G_i$ is at least 2. Let $M_i$ denote the intersection of $M$ with $G_i$, and $X_i$ be the symmetric space $K_i \backslash G_i$, where $K_i$ is a maximal compact subgroup of $G_i$. Since the rank of $G_i$ is at least 2 for every $i$, the dimension of $X_i$ is at least 4.  Furthermore, every $X_i$ is aspherical, and thus $\Gamma_i = \pi_1(K_i\setminus G_i /\Gamma_i)$.  By the Serre long exact sequence for a fibration, there is a surjection  $\alpha_i :  \pi _1(G_i/\Gamma_i) \mapsto \pi _1(K_i \setminus G_i/\Gamma_i) =\Gamma_i$ whose kernel is $\pi_1(K_i)$.  

Since we have assumed that $\Gamma$ intersects $M$ only at $\set{e}$ and $M \subset K$, we may do a similar analysis for the quotients $M_i \backslash G_i$, where $M_i \backslash G_i / \Gamma$ is seen as a $M_i \backslash K_i$-fiber bundle over $K_i \backslash G_i / \Gamma$. Accordingly, we have a short exact sequence associated to their fundamental groups:

\[ 1 \to \pi_1(K_i) / \pi_1(M_i) \to \pi_1(M_i \backslash G_i / \Gamma_i) \to \Gamma_i \to 1. \]

Consider a   $C^2$ cone transitive, totally Cartan $\R^k$-action on $M \backslash G / \Gamma$.  After passing to a finite cover, by Theorem \ref{thm:big-headache}, some finite cover of $X$ can be written as a homogeneous fiber bundle over a product of 3-manifolds $Y_1 \times \ldots \times Y_k$. % and a homogeneous space $H/\Lambda$.  
Hence $\pi _1(M \backslash G/\Gamma)$ has $\pi _1 (Y_1) \times \ldots \times \pi _1 (Y_k)$ as a factor, and $\pi _1 (H/\Lambda)$ as a normal subgroup.  

We claim that $k = 0$, so that no such rank one factor exists. Indeed, if there was such a factor, $\pi_1(Y_1)$ would be a group with exponential growth, since $Y_1$ supports an Anosov flow \cite{plante-thurston}. In particular, it is nonabelian. Let $\bar{p}_i : \pi_1(M_i \backslash G_i/\Gamma_i) \to \pi_1(Y_1)$ be restriction of the map induced by the projection onto $Y_1$, and $p_i$ denote its pullback to $\pi_1(G_i / \Gamma_i)$. By the Margulis normal subgroup theorem for universal covers (see, eg, \cite[Part 3]{margulis-volume}), $\ker p_i$ is either central or cofinite in $\pi_1(M_i \backslash G_i/\Gamma_i)$, which is a lattice in $\tilde{G_i}$, the universal cover of $G_i$. 

If $\ker p_i$ is cofinite, $p_i$ has a finite group as its image. But $\pi_1(Y_1)$ has no finite subgroups, since by \cite{verjovsky}, the universal cover of $Y_1$ is $\R^3$ and any group acting properly discontinuously and cocompactly on $\R^3$ must have cohomological dimension equal to 3. Since any group with torsion has infinite cohomological dimension, we conclude that $p_i$ is the trivial homomorphism.

Now assume that $\ker p_i$ is central. 
Note that since $\pi_1(Y_1)$ has no torsion elements, any torsion elements of $\pi_1(G_i / \Gamma_i)$ are contained in $\ker p_i \cap \pi_1(K_i)$. Therefore, they form a subgroup, since they are contained in the abelian group $Z(\pi_1(G_i/\Gamma_i))$. If $\Upsilon_i$ is the torsion subgroup of $\pi_1(G_i / \Gamma_i)$, $\Upsilon_i$ is a normal subgroup contained in both $\ker p_i$ and $\pi_1(K_i)$. Let $\Lambda_i = \pi_1(G_i/\Gamma_i) / \Upsilon_i$, so that $\Lambda_i$ is still a central extension of $\Gamma_i$, but without torsion, and $\Lambda_i$ is also an extension of $\pi_1(Y_1)$. We therefore have the following exact sequences:

\begin{equation}
\label{eq:cd1}
1 \to \pi_1(K_i) / \Upsilon_i \to \Lambda_i \to \Gamma_i \to 1,
\end{equation}

and

\begin{equation}
\label{eq:cd2}
1 \to (\ker p_i / \Upsilon_i) \to \Lambda_i \to \pi_1(Y_1) \to 1.
\end{equation}

The cohomological dimension of $\Gamma_i$ is the dimension of the symmetric space associated to $G_i$, which is strictly larger than 3 by assumption. %Furthermore, the center of $\Gamma_i$ is finite since we have assumed that $G_i$ is linear, so $\Upsilon_i \supset Z(\Lambda_i)$. 
Furthermore, $\ker p_i \subset Z(\pi_1(G_i/\Gamma_i))$, since by assumption, $\ker p_i$ is central in $\pi_1(G_i / \Gamma_i)$, and $Z(\pi_1(G_i / \Gamma_i)) \subset \pi_1(K_i)$. Since $\ker p_i / \Upsilon_i$ and $\pi_1(K_i) / \Upsilon_i$ are torsion-free, countable abelian groups of finite rank, it follows that they are isomorphic to $\Z^{\ell'}$ and $\Z^\ell$, respectively, with $\ell' \le \ell$ (since $\ker p_i$ embeds into $\pi_1(K_i)$). Since $\pi_1(Y_1)$ acts freely and properly discontinuously on the aspherical 3-manifold, $\tilde{Y_1}$ (which is aspherical by \cite[Lemma 5.5, p.147]{bautista-morales}), we conclude that the cohomological dimension of $\pi_1(Y_1)$ is 3. In particular, by \cite[Theorem 5.5]{bieri-book}, the cohomological dimension of $\Lambda$ is $\ell + \dim(K\backslash G / \Gamma) > \ell + 3$ from \eqref{eq:cd1}, and $\ell' + 3$ from \eqref{eq:cd2}. Since $\ell' \le \ell$, this is not possible, and we have reached a contradiction. Hence $p_i$ is trivial for every $i$, and hence $\pi_1(Y_1)$ is the trivial group. This is a contradiction since $\pi_1(Y_1)$ must be a group with exponential growth by \cite{plante-thurston}.

In conclusion, we see that the $\R^k$ action on $M\backslash G/\Gamma$ does not have any rank 1 factors, and $M \backslash G/\Gamma $ is homeomorphic to $H/\Lambda$. We assume without loss of generality that $H$ is simply connected, and claim that $H$ is semisimple. Indeed, if not, then by Theorem \ref{thm:homo-classification}, we may consider the Levi factor $H^{L}$ of $H$ (which has no compact factors), and the solvradical $H^s$, and $\Lambda^s = H^s \cap \Lambda$ as a normal subgroup of $\Lambda$, and $\Lambda^L = \Lambda / \Lambda^s$. Since $\Lambda = \pi_1(H /\Lambda)$ is isomorphic to $\pi_1(M \backslash G / \Gamma)$, we may consider the isomorphic image of $Z(\pi_1(M \backslash G / \Gamma))$ in $\Lambda$.

We claim that the center of $\Lambda$ is exactly $Z(\Lambda^L) \cdot \Lambda^s$. Indeed, it is clear that the center must be contained in this group, so we must show that $Z(\Lambda^L) \cdot \pi_1(H^s)$ is central in $\Lambda$. This is immediate for $Z(\Lambda^L)$, since central elements of semisimple groups always act trivially in finite-dimensional representations. To see that $\Lambda^s$ is central, we consider its isomorphic image in $\pi_1(M \backslash G / \Gamma)$. By considering its factors onto each irreducible component, and again applying the Margulis normal subgroup theorem, it follows that it must be central as well.

Now, the quotient of $\pi_1(M \backslash G / \Gamma)$ by its center is exactly $\Gamma$, and the quotient of $\Lambda$ by its center is $\Lambda^L$. Therefore, by Mostow rigidity \cite[Theorem A']{mostow-book}, $G$ is locally isomorphic to $H^L$, so $\dim(G) = \dim(H^L)$.  Since
\[ \dim(G) \ge \dim(M \backslash G / \Gamma) = \dim(H) \ge \dim(H^L) = \dim(G), \]
with equalities holding exactly when $M = \set{e}$ and $H^s = \set{e}$, respectively, we conclude that $H$ is semisimple and $M$ is trivial.

%by considering its projection onto $\pi_1(M_i \backslash G_i /\Gamma_i)$ and again applying the Margulis normal subgroup theorem, we conclude that $\Lambda^s \subset \prod \pi_1(K_i) / \pi_1(M_i)$. Therefore, $\Lambda^L = \Lambda / \Lambda^s$ is a lattice in $H^L$, and is isomorphic to the quotient of $\Gamma$ by the isomorphic image of $\Lambda^s$.

%By Mostow rigidity \cite[Theorem A']{mostow-book}, $G$ is isomorphic to $H$ and $\Gamma$ to $\Lambda$.   
Therefore, we know that there is an affine diffeomorphism between $H /\Lambda$ and $G / \Gamma \, (=M \backslash G / \Gamma)$. Conjugating the dynamics on $H/\Lambda$ by this diffeomorphism,  our $\R^k$ totally Cartan action is $C^{1,\theta}$-conjugate to a homogeneous 
subaction of $G$ on $G/\Gamma$ by a subgroup $A$ of $G$ (up to an automorphism of $\R^k$).  Since the $\R^k$-action is Anosov, $A$ is a split Cartan of $G$. Since the coarse Lyapunov spaces are 1-dimensional, so are the root spaces of $G$.  Hence $G$ itself is $\R$-split. 

Finally, if the action is $C^{\infty}$, so is the conjugacy by Theorem \ref{thm:big-main}.  
\end{proof}

\begin{proof}[Proof of Corollary \ref{cor:semisimple-actions}]
Let $G$ be a semisimple Lie group, and $G \curvearrowright X$ be a locally free action of $G$ such that the restriction of the action to a split Cartan subgroup $A \subset G$ is totally Cartan. Notice then that the bundle which is tangent to the $G$-orbits as an invariant sub-bundle yields an $A$-invariant subbundle of $TX$. Since the action is by a Cartan subgroup of $G$, the root splitting of $G$ with respect to $A$ must coincide with the coarse Lyapunov splitting, and the centralizer of the Cartan must be trivial. Furthermore, $G$ has no compact factors. Therefore, $G$ is $\R$-split, and the root subgroups of $G$ are coarse Lyapunov subgroups.

We claim that the action of $A$ on $X$ has no rank one factors (up to finite cover). Indeed, suppose that $\pi : X \to Y$ is a submersion onto a 3-manifold, $\psi_t : Y \to Y$ is an Anosov flow and $\sigma : A \to \R$ is a homomorphism satisfying $\pi(a \cdot x) = \psi_{\sigma(a)}(\pi(x))$ for every $a \in A$ and $x \in X$. Then there exists a unique symplectic pair of coarse Lyapunov exponents $\pm \alpha$ such that $E^\alpha$ and $E^{-\alpha}$ map to the stable and unstable distributions $E^s$ and $E^u$ for $\psi_t$, respectively. We claim that $\alpha$ is a root of $G$. Indeed, if $\alpha$ were not a root of $G$, then every root subgroup of $G$ would be contained in the fibers of $\pi$. Since the root subgroups of $G$ generate $A$, this would imply that the $A$-orbits were contained in the fiber as well. This is a contradiction to cone transitivity, so $\alpha$ must be a root of $G$.

Let $\hat{\Delta}$ denote the set of roots of $G$ which are not proportional to $\pm \alpha$. Then since they are contained in the fibers of the projection onto $Y$, they are normalized by the root subgroups corresponding to $\pm \alpha$. Hence the group generated by the root subgroups in $\hat{\Delta}$ is a normal subgroup of $G$. Since $G$ is semisimple, it has a transverse subgroup, so the root subgroups corresponding to $\pm \alpha$ generate a factor of $G$ locally isomorphic to $PSL(2,\R)$. Since we assume that no such factor exists, we have arrived at a contradiction, and there is no rank one factor of the $A$-action on $X$.

%Since the roots of a semisimple Lie algebra span the dual space to any Cartan subalgebra, the Starkov component must be trivial. Therefore, some finite cover of the action is $C^{1,\theta}$-embedded in a direct product of rank one systems and a homogeneous action. We claim that the Anosov flow $\varphi_i^t : Y_i \to Y_i$ is a homogeneous flow as well. Indeed, let $a \in \R^k$ be a the generator of such a flow, and notice that since the roots of $G$ generate the dual space, there exists a root $\alpha$ such that $\alpha(a) \not= 0$. Then the pair $\pm \alpha$ is unique, since the flow on $Y_i$ splits as a direct product. Since the flows $\eta^{\pm \alpha}$ must be the action of unipotent subgroups of $G$, and  the factor $Y_i$ is generated by the homogeneous flows $ta$, $\eta^\alpha_t$ and $\eta^{-\alpha}_t$, they must generate a copy of (a cover of) $PSL(2,\R)$. So the factors $Y_i$ are homogeneous.

Finally, we apply Theorem \ref{thm:big-main}. Notice that the relations on the group $G$ automatically determine the commutator and symplectic relations in the path group $\hat{\mc P}$ used in the proofs of Part II (ie, they determine the pairwise cycle structures (Definition \ref{def:const-pairwise} for weights corresponding to roots). Therefore, the group $G$ embeds into the group giving the homogeneous structure provided by the dynamics, and the $G$-action, as well as the $A$-action, is homogeneous.
\end{proof}

%\begin{proof}[Proof of Corollary \ref{cor:lattice-actions}]
%By Corollary \ref{cor:Zk-actions}, the manifold $X$ must be a nilmanifold. Then by \cite[Corollary 1.8(2,3)]{Brown:2015aa}, a finite index subgroup of $\Gamma$ is smoothly conjugated to an affine action.
%Corollary \ref{cor:lattice-actions} follows from Corollary \ref{cor:semisimple-actions} in the same way that Corollary \ref{cor:Zk-actions} follows from Theorem \ref{thm:full-classify}, by suspending the $\Gamma$-action to a $G$-action. Notice that the condition that the exponents on $X$ are not proportional to the roots of $G$ is used to obtain that the suspended $G$-action is totally Cartan.
%\end{proof}

\color{black}

\bibliographystyle{plain}
\bibliography{allbib}

\end{document}